\documentclass[10pt, a4paper, onecolumn]{IEEEtran}

\usepackage{graphicx}
\usepackage{amsmath}
\usepackage{amsfonts}
\usepackage{amssymb}
\usepackage{color}
\usepackage{rotating}
\usepackage{caption}
\usepackage{subcaption}
\IEEEoverridecommandlockouts

\newtheorem{lemma}{Lemma}
\newtheorem{theorem}{Theorem}
\newtheorem{corollary}{Corollary}
\newtheorem{definition}{Definition}

\newtheorem{proposition}{Proposition}
\newtheorem{problem}{Problem}

\title{An approximate solution to the decentralized two-controller infinite-horizon scalar {LQG} problem: Part I- fast dynamics}

\author{Se Yong Park (separk@eecs.berkeley.edu), Anant Sahai (sahai@eecs.berkeley.edu)
\thanks{A part of the results in this paper was presented in Conference on Decision and Control, 2012~\cite{Park_easier}. The authors are with the department of Electrical Engineering and Computer Sciences at the University of California at Berkeley.}
}

\begin{document}
\maketitle

\begin{abstract}
We consider scalar decentralized average-cost infinite-horizon LQG problems with two controllers, focusing on the fast dynamics case when the (scalar) eigenvalue of the system is large. It is shown that the best linear controllers' performance can be an arbitrary factor worse than the optimal performance. We propose a set of finite-dimensional nonlinear controllers, and prove that the proposed set contains an easy-to-find approximately optimal solution that achieves within a constant ratio of the optimal quadratic cost.
The insight for nonlinear strategies comes from revealing the relationship between information flow in control and wireless information flow.
More precisely, we discuss a close relationship between the high-SNR limit in wireless communication and fast-dynamics case in decentralized control, and justify how the proposed nonlinear control strategy can be understood as exploiting the generalized degree-of-freedom gain in wireless communication theory.
For a rigorous justification of this argument, we develop new mathematical tools and ideas. To reveal the relationship between infinite-horizon problems and  generalized MIMO Witsenhausen's counterexamples, we introduce the idea of \textit{geometric slicing}. To analyze the nonlinear strategy performance, we  introduce an approximate-comb-lattice model for the relevant random variables.
\end{abstract}

\section{Introduction}
One of the biggest successes in stochastic control theory is the LQG (linear quadratic Gaussian) problem with a single controller.
The solution of the LQG problem contributed two big ideas to classical control theory~\cite{KumarVaraiya}: The first is the optimality of linear controllers. This fact allows designers to confidently focus on finite-dimensional linear strategies without worrying about the infinite-dimensional strategy space. The second is the optimality of the Certainty-Equivalent-Controllers (CEC). Without loss of optimality, we can first estimate states and then control the system as if the estimated states were the true states. This is also called the estimation and control separation principle.

Even if the optimality results were restricted to single-controller LQG problems, their philosophical contribution was not limited to them. Lots of related but different control areas --- including nonlinear system control and adaptive control --- accepted these principles and focused on essentially linear controllers, and separated estimation from control. In this sense, the LQG problems form a conceptual foundation in control theory.

However, this beautiful result on the LQG problem with a single controller fails as soon as we introduce more than one controller. Following convention, we call a problem with a single controller a centralized problem, and one with multiple controllers a decentralized problem. The famous Witsenhausen's counterexample~\cite{Witsenhausen_Counterexample} demonstrates that nonlinear strategies outperform linear strategies even in a simple finite-horizon decentralized LQG problem. Later, Ho, Kastner, and Wong~\cite{Ho_Teams} qualitatively argued that the need for nonlinear controllers stems from ``signaling" --- we will also use the term ``implicit communication" interchangeably --- between decentralized controllers. Finding the optimal nonlinear strategy in most decentralized problems is known to be a non-convex infinite-dimensional problem~\cite{Sandell_Survey}, for which we do not have a well-developed theory.

Yet, it is still interesting to consider the average-cost infinite-horizon decentralized LQG problem, which is the natural extension of \cite[p.93]{KumarVaraiya}.
\begin{align}
&\mathbf{x}[n+1]=\mathbf{A}\mathbf{x}[n]+ \sum_i \mathbf{B_i}\mathbf{u_i}[n] + \mathbf{w}[n] \nonumber\\
&\mathbf{y_i}[n]=\mathbf{C_i}\mathbf{x}[n]+ \mathbf{v_i}[n] \nonumber
\end{align}
Here, the underlying random variables $\mathbf{x}[0]$, $\mathbf{w}[n]$, and $\mathbf{v_i}[n]$ are independent Gaussian. The objective is to minimize the asymptotic average cost:
\begin{align}
\limsup_{N \rightarrow \infty} \frac{1}{N} \sum_{0 \leq n < N}& \mathbb{E}[\mathbf{x}^*[n]\mathbf{Q} \mathbf{x}[n]]+\sum_{i} \mathbb{E}[\mathbf{u_i}^*[n]\mathbf{R_i} \mathbf{u_i}[n]] \nonumber
\end{align}
where $\mathbf{Q}\succeq \mathbf{0}$, $\mathbf{R_i}\succeq \mathbf{0}$, and each $\mathbf{u_i}[n]$ is a causal function of $\mathbf{y_i}[n]$ alone.
This paper considers the simplest toy case among these infinite-horizon
decentralized LQG problems, a scalar system with two-controllers. As should be expected, linear controllers are not optimal. The crux of decentralized LQG problems, nonconvex optimization over infinite-dimensional spaces, is still there and finding the optimal solution seems impossible. Instead of trying to solve the problem exactly, we solve it \textbf{approximately} to within a constant factor of the optimal cost.

\label{sec:review}
\begin{figure*}[htbp]
\begin{center}
\includegraphics[width=5in]{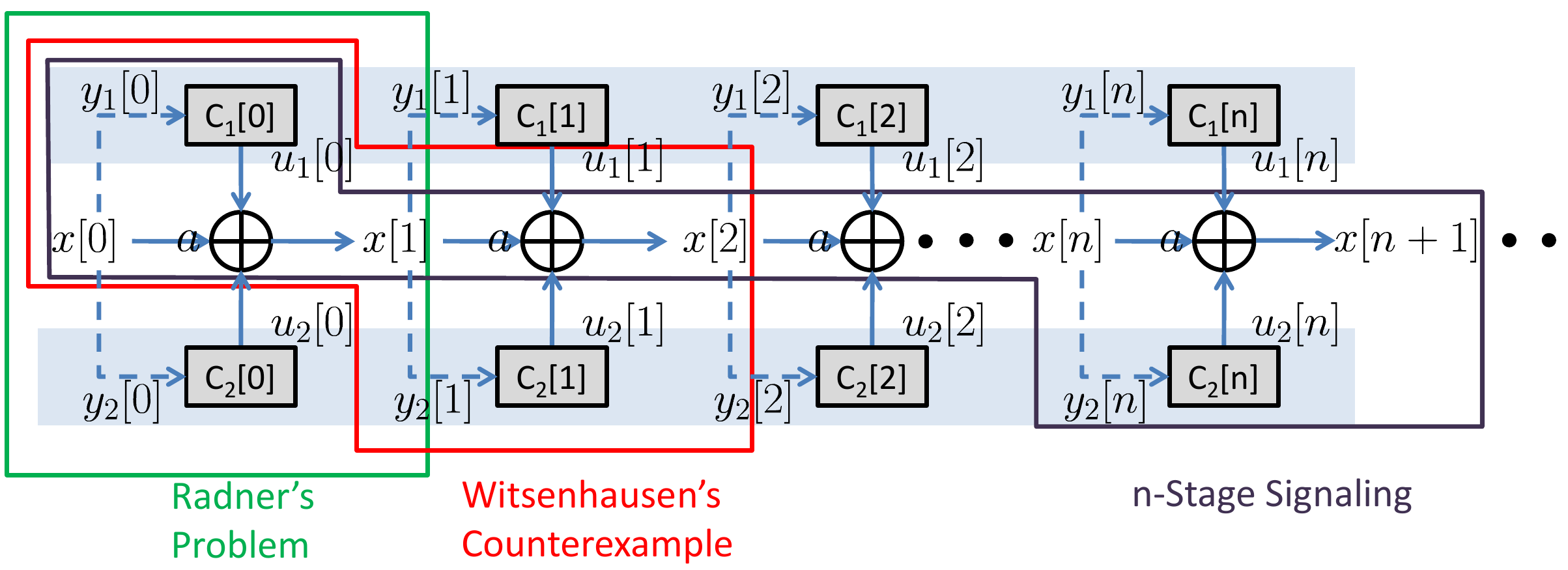}
\caption{Relationship between Radner's problem~\cite{Radner_Team}, Witsenhausen's counterexample~\cite{Witsenhausen_Counterexample}, and the infinite-horizon scalar LQG problem with two controllers}
\label{fig:dyn}
\end{center}
\end{figure*}

\subsection{Literature Review and Intellectual Context}
Until the late 40s, control and communication were considered in a unified framework under the name of \textit{cybernetics}. According to Wiener~\cite{Wiener}, cybernetics is defined as `the scientific study of control and communication in the animal and the machine.' However, Shannon's revolutionary paper detached communication problems as its own field of interest. Since then, control and communication have grown as two separated areas.

Now, control theory which successfully addressed fundamental problems in centralized control is facing the decentralized challenge (non-convex infinite-dimensional optimization problems). This challenge divided related control research in two major directions.

The first direction is finding those special cases under which a
linear strategy {\em is} optimal --- or almost equivalently, finding
cases where the problem is convex. Radner's pioneering
paper~\cite{Radner_Team} considered the case when the controllers act
simultaneously and the dynamics of the system terminates after one
time step, as shown in Fig.~\ref{fig:dyn}. Signaling is intuitively impossible by
problem construction. Therefore, linear controllers are optimal in
this case in spite of the problem being decentralized. Witsenhausen
found another sufficient condition for linear optimality called the
nested information pattern~\cite{Witsenhausen_Separation}. The condition tells if all information is shared with one step time delay by explicit communication, there is no need to implicitly communicate the information and linear strategies are optimal.
Later, this concept was generalized by Yuksel to stochastic nestedness~\cite{Yuksel_Nested}.

More recently, Rotkowitz and Lall~\cite{Rotkowitz_characterization}
proposed an algebraic condition for convexity of the problem called ``quadratic invariance." The condition finds sparsity constraints on the controller which remains convex even after Youla's parametrization~\cite{Youla}. There is a lot of on-going research in this direction~\cite{Shah_Partial,Lessard_state} that has refined our understanding and also revealed much about the structure of optimal controllers in these special cases where linear controllers are optimal. However, all of these quadratic-invariance structured problems also have no signaling incentive and the information patterns are nested~\cite{Rotkowitz_information}.\footnote{There are a few special cases when a linear controller is optimal but cannot be explained in the context of signaling incentive. Especially, in \cite{Bansal_affine}, Bansal and Basar found that when input cost and state disturbance measures match, a linear controller is optimal. Likewise, in communication theory where encoder and decoder can be thought of as distributed controllers, it is well known that linear is optimal when the source and channel distributions and cost measures match~\cite{Gastpar_Tocode}.}

On the other hand, the second direction studies general cases when linear strategies are not optimal.
\cite{Nayyar_Decentralized} discussed the structure of the optimal controllers in general decentralized problems, and \cite{Wu_Witsenhausen} found the mathematical properties of the optimal strategy for Witsenhausen's counterexample.
However, these results do not give quantifiable results, and to get such results we have to study the effect of implicit communication~\cite{Ho_Teams, Witsenhausen_Counterexample}.

Basically, most of quantifiable results focus on Witsenhausen's counterexample. As we can see in
Fig.~\ref{fig:dyn}, in Witsenhausen's counterexample the two
controllers act in different time slots and may try to
communicate. Exploiting implicit communication between the controllers makes
nonlinear strategies outperform linear ones. Mitter and
Sahai found that linear strategies can be arbitrarily bad compared to
nonlinear strategies~\cite{Sahai_Witsenhausen}. Many researchers including \cite{Ho_Witsenhausen,Baglietto_Numerical,Shamma_Learning,Karlsson_Iterative} tried computer-based exhaustive search in finding the optimal strategy.  Finally, Grover \textit{et al.}
 showed that signaling-based nonlinear strategies approximately
achieve the optimal cost to within a constant
ratio~\cite{Pulkit_Witsen}. This paper continues this approach, and
can be considered as a direct descendent of \cite{Pulkit_Witsen}. In
fact, there is a close relationship between Witsenhausen's
counterexample and the scalar infinite-horizon LQG problem considered
in this paper. We will revisit this point in
Section~\ref{sec:proof}, and see that the infinite-horizon LQG problem can be thought of as interlocking of a series of generalized Witsenhausen's counterexamples.

Another not directly but conceptually related branch of the second direction is ``Control over Communication Channels"~\cite{Tatikonda_Control,Sahai_Anytime,Yuksel_Optimal,Elia_Bode,Nair_Communication,Minero_Stabilization}, which tries to quantify explicit information flow for control. They introduce an explicit communication link and measure the amount of information flow required to control the system.
One of their main results is that in scalar systems we need at least the communication rate, (log of eigenvalue) bits, to stabilize the system~\cite{Tatikonda_Control}. Later, this concept was extended to nonlinear filtering~\cite{Martins_Feedback}. In this paper, we will see the underlying relationship to decentralized control problems.

On the other hand, communication theory (especially, wireless communication theory) has made quite a lot of quantifiable results in their network communication problems. Since communication problems are decentralized in nature, the exact characterization of the capacity has been open for most of communication networks which involve many nodes. However, they still made a progress by dividing cases based on SNR (Signal-to-Noise Ratio), bringing linear algebraic ideas and concepts to the problems, and solving the problems approximately. Especially, Avestimehr \textit{et al.} considered relay communication problem with arbitrary large number of nodes, and successfully characterize the capacity within a constant gap. At the heart of this progress, there are the concepts of \textit{generalized degree of freedom (d.o.f.)} and \textit{binary deterministic models}. In \cite{Salman_Wireless}, Avestimehr \textit{et al.} idealized bit levels as different antennas. By conceptualizing each bit level as different subspaces, they could apply linear algebraic concepts and ideas for much precise analysis. By expanding the concept of d.o.f. (essentially, the rank of linear spaces) to different bit levels, they could the capacity of wireless communication networks within a constant gap.

The main contribution of this paper is the parallelism between information flows in decentralized LQG control and those in wireless communication theory. We will see that just as wireless communication theory divides cases depending on SNR, decentralized LQG problems can be divided based on the eigenvalue of the systems. Moreover, we will find the relevant bottleneck in decentralized LQG problems using the idea of  `geometric slicing', which we believe is a proper analogy to the information-theoretic cutset bound~\cite{Cover} in a dynamic-programming context.

The rest of the paper is organized as follows: We formally state the problem and the main results in Section~\ref{sec:main}.
Section~\ref{sec:intui} gives the underlying intuitions behind the results.
In Section~\ref{sec:proof}, \ref{sec:upperbound}, \ref{sec:prooflower}, \ref{sec:proof:ratio}, we will convert these intuitions into formal proofs, and introduce proof ideas for that. Section~\ref{sec:wireless} discusses the fundamental relationship between wireless communication theory and decentralized LQG problems. Finally, Section~\ref{sec:discuss} concludes the paper.

\section{Problem Statement and Main Results}
\label{sec:main}
Throughout this paper, we will discuss the scalar infinite-horizon decentralized LQG problems with two controllers.
\begin{problem}[scalar infinite-horizon decentralized LQG problems with two controllers]
\begin{align}
& x[n+1]=ax[n]+b_1 u_1[n]+b_2 u_2[n]+w[n]  \label{eqn:system}\\
& y_1[n]=c_1 x[n]+ v_1[n]  \\
& y_2[n]=c_2 x[n]+ v_2[n]
\end{align}
Here, $u_1[n]$ and $u_2[n]$ must be causal functions of $y_1[n]$ and
$y_2[n]$ respectively,
\textit{i.e.}~$u_1[n]=f_{1,n}(y_1[0],\cdots,y_1[n])$ and
$u_2[n]=f_{2,n}(y_2[0],\cdots,y_2[n])$. Following the traditional LQG
problem formulation, the objective is minimizing an average quadratic cost:
\begin{align}
\limsup_{N \rightarrow \infty} \frac{1}{N}
\sum_{0 \leq n < N} q \mathbb{E}[x^2[n]] + r_1 \mathbb{E}[u_1^2[n]] + r_2 \mathbb{E}[u_2^2[n]]. \label{eqn:part11}
\end{align}
Here, $q \geq 0$, $r_1 \geq 0$, $r_2 \geq 0$ and the underlying random variables are independent Gaussian, \textit{i.e.}~$x[0] \sim \mathcal{N}(0, \sigma_0^2)$, $w[n] \sim \mathcal{N}(0,\sigma_w^2)$, $v_1[n] \sim \mathcal{N}(0,\sigma_{v1}^2)$ and $v_2[n] \sim \mathcal{N}(0,\sigma_{v2}^2)$.
\label{prob:lqg2}
\end{problem}
Figure~\ref{fig:dyn} shows a pictorial description of the problem by introducing duplicated nodes across different time-steps and thus unraveling the dynamics.\footnote{The idea of unraveling the system by introducing duplicated nodes across different time-steps was also used to study network information flows~\cite{Ahlswede_Network}. As in \cite{Ahlswede_Network}, we will see the unraveling of the dynamics will be helpful to find the information bottleneck of the system.}
First, without loss of generality, we put a series of assumptions on the problems.

Assumption (a): $b_1=b_2=1$.

Assumption (b): $c_1=c_2=1$.

Assumption (c): $\sigma_w^2=1$.

Assumption (d): $\sigma_{v1} \leq \sigma_{v2}$.

Assumptions (a), (b) do not lose generality since we can rescale $u_1, u_2$ and $y_1, y_2$ respectively. Assumption (c) doesn't lose generality since we can rescale the system equation by $\frac{1}{\sigma_w}$. Assumption (d) doesn't lose generality because it is simply a way of deciding which controller is $1$, and which is $2$.

Therefore, throughout this paper we will consider the following problem:
\begin{problem}[Normalized decentralized LQG problem for Problem~\ref{prob:lqg2}]
\begin{align}
& x[n+1]=ax[n]+u_1[n]+u_2[n]+w[n]  \label{eqn:systemsimple}\\
& y_1[n]=x[n]+ v_1[n]  \\
& y_2[n]=x[n]+ v_2[n]
\end{align}
where $x[0] \sim \mathcal{N}(0,\sigma_0^2)$, $w[n] \sim \mathcal{N}(0,1)$, $v_1[n] \sim \mathcal{N}(0,\sigma_{v1}^2)$, $v_2[n] \sim \mathcal{N}(0,\sigma_{v2}^2)$. The control objective is minimizing the average cost of \eqref{eqn:part11}.
\label{prob:lqg}
\end{problem}

Even though this problem is the simplest decentralized infinite-horizon LQG problem, as we will see in Proposition~\ref{prop:1}, linear strategies are not optimal and the optimization problem becomes infinite-dimensional and non-convex. Here, we follow the approximation approach of \cite{Pulkit_Witsen}, which
itself inherits from \cite{Salman_Wireless} and the spirit of approximate algorithms in computer science theory. We propose a set of
finite-dimensional function spaces that are guaranteed to contain an approximately optimal solution. Therefore, if we optimize only over the proposed finite-dimensional function spaces, the solution achieves the optimal performance within a constant ratio regardless of the problem parameters, $a$, $q$, $r_1$, $r_2$, $\sigma_0$, $\sigma_{v1}$, and $\sigma_{v2}$. In this paper, we first consider the fast-dynamics case when the single eigenvalue of the system is large ($|a| \geq 2.5$) and discuss the relationship with high-SNR in wireless communication theory. The slow-dynamics case when the single eigenvalue of the system is small ($|a| < 2.5$) is discussed in \cite{Park_Approximation_Journal_Partii}, and the relationship with low-SNR in wireless communication theory is also revealed.\footnote{Here, we did not optimize for the best ratio, and the explicit number $2.5$ is arbitrary. We could have written Theorem~\ref{thm:1}, \ref{thm:2} with any fixed number like $|a|=2,3,5,6,\cdots$ which may result in a different ratio. However, as $|a|$ increases, the ratio between linear and nonlinear strategy cost goes to infinity, and the gain by considering nonlinear strategies becomes larger.}

The first set of controllers is two naive memoryless linear strategies, which simply zero-force the state.
\begin{definition}[Linear Strategy $L_{lin,bb}$]
$L_{lin,bb}$ is the set of the following two controllers:\\
(i) $u_1[n]=-a y_1[n]$, $u_2[n]=0$. \\
(ii) $u_1[n]=0$, $u_2[n]=-ay_2[n]$.
\label{def:lin}
\end{definition}

The second set is a two-parameter $(s,d)$ nonlinear strategy set for implicit communication between two controllers.
\begin{definition}[$s$-Stage Signaling Strategy $L_{sig,s}$]
For a given $s \in \mathbb{N}$, $L_{sig,s}$ is the set of all controllers which can be written as the following form for some $d > 0$,
\begin{align}
&u_1[n]=-aR_d(y_1[n]) \label{eqn:1}\\
&u_2[n]=-a(Q_{a^s d}(y_2[n]-R_{a^s d}(\sum_{1 \leq i \leq s} a^{i-1}u_2[n-i]))+R_{a^s d}(\sum_{1 \leq i \leq s} a^{i-1}u_2[n-i])).\label{eqn:2}
\end{align}
Here, $Q_x(y):=x \lfloor \frac{y}{x} + \frac{1}{2} \rfloor$ and $R_x(y):=y-Q_x(y)$. These quantities represent the quantization level and remainder when $y$ is divided by $x$ respectively, \textit{i.e.}~let $y=q \cdot x + r$ for $q \in \mathbb{Z}$ and $r \in [-\frac{x}{2},\frac{x}{2})$. Then, $Q_x(y)=q\cdot x$ and $R_x(y)=r$. (We also put $u_1[n]=0$ and $u_2[n]=0$ for $n < 0$.)
\label{def:sig}
\end{definition}
We will give the intuition behind this strategy in Section~\ref{sec:intui}. Roughly speaking, in the strategy set $L_{sig,s}$ the first controller implicitly communicates its observation to the second controller with delay $s$ by making the state easier to estimate. This strategy is essentially a multi-stage generalization of the quantization strategy~\cite{Pulkit_Witsen} used for Witsenhausen's counterexample. Notice that the strategy requires remembering the past $s$ control inputs.

Now, we can state the main theorem of this paper, which tells us that when $|a| \geq 2.5$ the optimization only over $L_{lin,bb}$ and $L_{sig,s}$ is enough to give a constant-ratio optimal strategy among all possible strategies.

\begin{theorem}
Consider the decentralized LQG problem shown in Problem~\ref{prob:lqg}. Let $L'=L_{lin,bb} \cup \bigcup_{s \in \mathbb{N}} L_{sig,s}$ and $L$ be the set of all measurable causal strategies. There exists a constant $c \leq 1.5 \times 10^5$ such that for all $|a| \geq 2.5$, $q$, $r_1$, $r_2$, $\sigma_0$, $\sigma_{v1}$ and $\sigma_{v2}$,
\begin{align}
\frac{
\underset{u_1,u_2 \in L'}{\inf}
\underset{N \rightarrow \infty}{\limsup}
\frac{1}{N}
\underset{0 \leq n < N}{\sum}
\mathbb{E}[qx^2[n]+r_1 u_1^2[n]+r_2 u_2^2[n]]
}
{
\underset{u_1,u_2 \in L}{\inf}
\underset{N \rightarrow \infty}{\limsup}
\frac{1}{N}
\underset{0 \leq n < N}{\sum}
\mathbb{E}[qx^2[n]+r_1 u_1^2[n]+r_2 u_2^2[n]]
}
\leq c. \nonumber
\end{align}
\label{thm:1}
\end{theorem}
\begin{proof}
See Section~\ref{sec:proof:ratio}.
\end{proof}

Since measurability is the minimal condition required to even have the problem make any sense, $\inf_{u_1,u_2 \in L}$ implies a minimization over all possible control strategies. Thus, in the rest of the paper, we will simply write it as $\inf_{u_1,u_2}$.

For the proof, we give explicit and computable upper and lower bounds on the optimal cost, and prove that they are within a constant ratio. In Lemma~\ref{ach:lemmabb} of page~\pageref{ach:lemmabb}, we will see that the linear strategies $L_{lin,bb}$ give the following upper bounds on the minimal average cost.
\begin{align}
&\underset{u_1,u_2}{\inf}
\underset{N \rightarrow \infty}{\limsup}
\frac{1}{N}
\underset{0 \leq n < N}{\sum}
\mathbb{E}[qx^2[n]+r_1 u_1^2[n]+r_2 u_2^2[n]] \leq 
q(a^2 \sigma_{v1}^2 + 1)+r_1( a^4 \sigma_{v1}^2 + a^2 \sigma_{v1}^2 + a^2).\\
&\underset{u_1,u_2}{\inf}
\underset{N \rightarrow \infty}{\limsup}
\frac{1}{N}
\underset{0 \leq n < N}{\sum}
\mathbb{E}[qx^2[n]+r_1 u_1^2[n]+r_2 u_2^2[n]] \leq
q(a^2 \sigma_{v2}^2 + 1)+r_2( a^4 \sigma_{v2}^2 + a^2 \sigma_{v2}^2 + a^2).
\end{align}
In Lemma~\ref{ach:lemma9} of page~\pageref{ach:lemma9}, we will see that the signaling strategies $L_{sig,s}$ give the following upper bounds.
\begin{align}
&\underset{u_1,u_2}{\inf}
\underset{N \rightarrow \infty}{\limsup}
\frac{1}{N}
\underset{0 \leq n < N}{\sum}
\mathbb{E}[qx^2[n]+r_1 u_1^2[n]+r_2 u_2^2[n]] \\
& \leq 
\inf_{(d,w_1) \in S_{U,1}}
q D_{U,1}(d,w_1) + r_1 \frac{a^2 d^2}{4} + r_2( 8a^2 D_{U,1}(d,w_1) + \frac{7}{2} a^{2(s+1)}d^2 + 4 a^2 \sigma_{v2}^2).
\end{align}
where the definitions of $D_{U,1}(d,w_1)$ and $S_{U,1}$ are available in Lemma~\ref{ach:lemma9} of page~\pageref{ach:lemma9}.

For the lower bounds, we will see four different bounds in Lemma~\ref{cov:lemma3} of page~\pageref{cov:lemma3} and Lemma~\ref{cov:lemma2} of page~\pageref{cov:lemma2}. Thus, the optimal cost of Problem~\ref{prob:lqg} is lower bounded as follows.
\begin{align}
&\inf_{u_1,u_2} \limsup_{N \rightarrow \infty} \frac{1}{N}
\sum_{0 \leq n < N} q \mathbb{E}[x^2[n]] + r_1 \mathbb{E}[u_1^2[n]] + r_2 \mathbb{E}[u_2^2[n]]\\
& \geq \max_{1 \leq i \leq 4}\sup_{(k_1,k_2,k,\sigma_{v2}',\alpha,\Sigma) \in S_{L,i}}\inf_{\widetilde{P_1},\widetilde{P_2} \geq 0} q D_{L,i}(\widetilde{P_1},\widetilde{P_2};k_1,k_2,k,\sigma_{v2}',\alpha,\Sigma) + r_1 \widetilde{P_1} + r_2 \widetilde{P_2}.
\end{align}
Here, the definitions of $D_{L,1}$ and $S_{L,1}$ are available in Lemma~\ref{cov:lemma3} of page~\pageref{cov:lemma3}. The remaining definitions of $D_{L,i}$ and $S_{L,i}$ for $2 \leq i \leq 4$ are shown in Lemma~\ref{cov:lemma2} of page~\pageref{cov:lemma2}.

Finally, in Section~\ref{sec:proof:ratio} of page~\pageref{sec:proof:ratio} we will compare these upper and lower bounds, and prove that they are within a constant ratio.

To prove a similar result for the slow-dynamics case ($|a| < 2.5$), we need a further set of single-controller optimal strategies.
These strategies are linear strategies which can be parametrized by a single parameter $k$.
\begin{definition}
$L_{lin,kal}$ is the set of all controllers which can be written in either one of the two following forms for some $k \in \mathbb{R}$\\
(i) $u_1[n]=-k\mathbb{E}[x[n]|y_1^n,u_1^{n-1}]$, $u_2[n]=0$.\\
(ii) $u_1[n]=0$, $u_2[n]=-k\mathbb{E}[x[n]|y_2^n,u_2^{n-1}]$.
\label{def:lin2}
\end{definition}
Here, $\mathbb{E}[x[n]|y_1^n,u_1^{n-1}]$ and $\mathbb{E}[x[n]|y_2^n,u_2^{n-1}]$ can be easily computed by Kalman filtering once $k$ is fixed.\footnote{Since $u_i^{n-1}$ is known to the controller, we can compensate for the past control inputs and treat it as an open-loop system. The estimation problem in the open-loop system is well-known Kalman filtering. This concept is called the control-estimation separation principle in the control community.}

The results of \cite{Park_Approximation_Journal_Partii} reveal that when $|a| \leq 2.5$, optimization over $L_{lin,kal}$ is enough to give a constant-ratio optimal strategy among all possible strategies.
\begin{theorem}
There exists $c \leq 6 \cdot 10^6$ such that for all $|a| \leq 2.5$, $q$, $r_1$, $r_2$, $\sigma_0$, $\sigma_{v1}$ and $\sigma_{v2}$,
\begin{align}
\frac{
\underset{u_1,u_2 \in L_{lin,kal}}{\inf}
\underset{N \rightarrow \infty}{\limsup}
\frac{1}{N}
\underset{0 \leq n < N}{\sum}
\mathbb{E}[qx^2[n]+r_1 u_1^2[n]+r_2 u_2^2[n]]
}
{
\underset{u_1,u_2}{\inf}
\underset{N \rightarrow \infty}{\limsup}
\frac{1}{N}
\underset{0 \leq n < N}{\sum}
\mathbb{E}[qx^2[n]+r_1 u_1^2[n]+r_2 u_2^2[n]]
}
\leq c. \nonumber
\end{align}
\label{thm:2}
\end{theorem}
\begin{proof}
See \cite{Park_Approximation_Journal_Partii}.
\end{proof}

By Theorem~\ref{thm:1} and \ref{thm:2}, we can achieve the optimal cost to within a factor of
$6 \cdot 10^6$ by optimizing only over $L_{lin,kal}$ and $L_{sig,s}$, which only involves single and two parameter optimization problems respectively. We believe that the factor here is coming from our proof strategy and the gap will be far smaller in practice.

The optimal parameters for the proposed strategy sets in Definition~\ref{def:lin}, \ref{def:sig}, \ref{def:lin2} are not difficult to find.
The optimization over $L_{lin,kal}$ is a centralized LQG problem and it is well known that the optimal $k$ can be easily found by Riccati equations~\cite{KumarVaraiya}. In Proposition~\ref{prop:2} of page~\pageref{prop:2} and Proposition~\ref{prop:3} of page~\pageref{prop:3}, we will see that the parameter $s$ in $L_{sig,s}$ can be selected based on the problem parameters. Particularly, we can use $s =\lceil \frac{\ln \sigma_{v2}^2 - \ln ( \max (1, a^2 \sigma_{v1}^2))}{2 \ln a} \rceil$, so that $a^{2(s-1)} \max(1,a^2 \sigma_{v1}^2) < \sigma_{v2}^2 \leq a^{2s} \max(1,a^2 \sigma_{v1}^2)$. Moreover, Corollary~\ref{rat:ach} of page~\pageref{rat:ach} gives a simple analytic upper bound on the performance of $L_{sig,s}$, which has only two local optima as $d$ varies. Therefore, both optimization problems are easily solvable.

However, the true implication of Theorem~\ref{thm:1} and \ref{thm:2} is that they reveal the essential skeletons of an optimal strategy. Since the original optimization problems are infinite-dimensional and non-convex, it is not even clear how and where to start a computer-based search. By revealing the minimal strategy for approximately optimal performance, these results might give an initial point to start optimization for further performance refinements. More importantly, as we will see in later sections, the proposed strategy sets are intuitively interpretable and understandable.

\section{Intuition: Deterministic Model Interpretation}
\begin{figure*}[htbp]
\begin{center}
\includegraphics[width=6in]{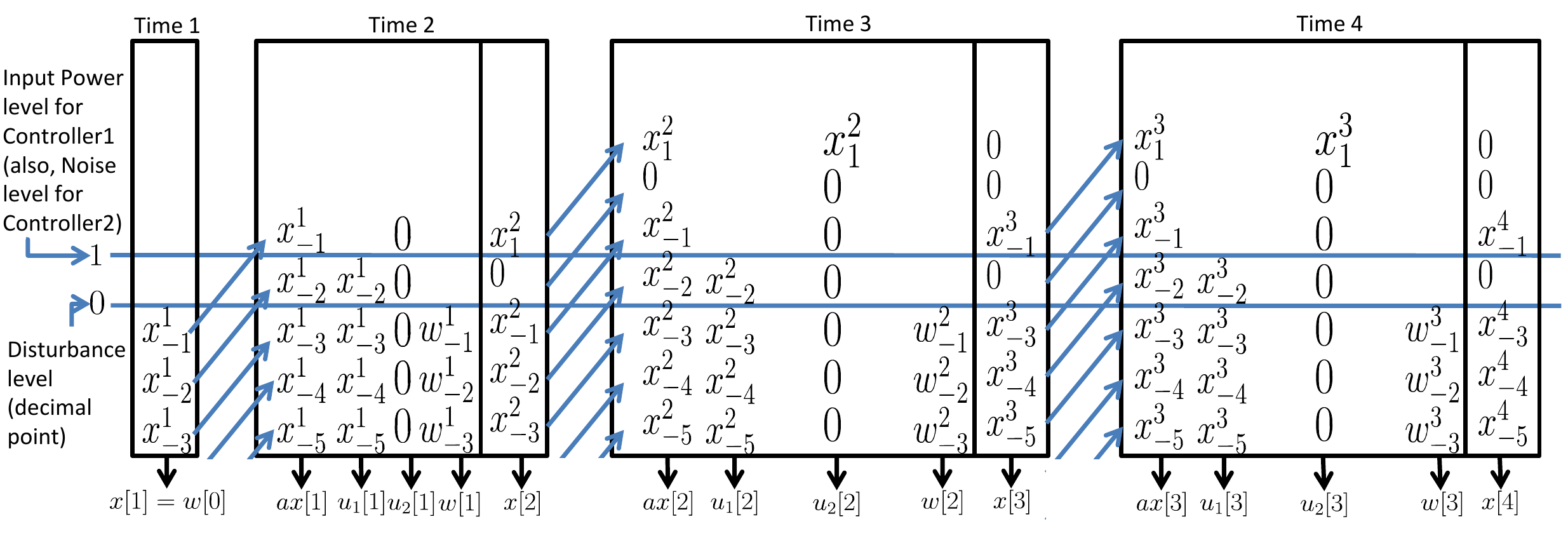}
\caption{Deterministic Model Interpretation of Nonlinear Control Strategies $L_{sig,1}$}
\label{fig:nonlin}
\end{center}
\end{figure*}

\label{sec:intui}
\begin{figure*}[htbp]
\begin{center}
\includegraphics[width=6in]{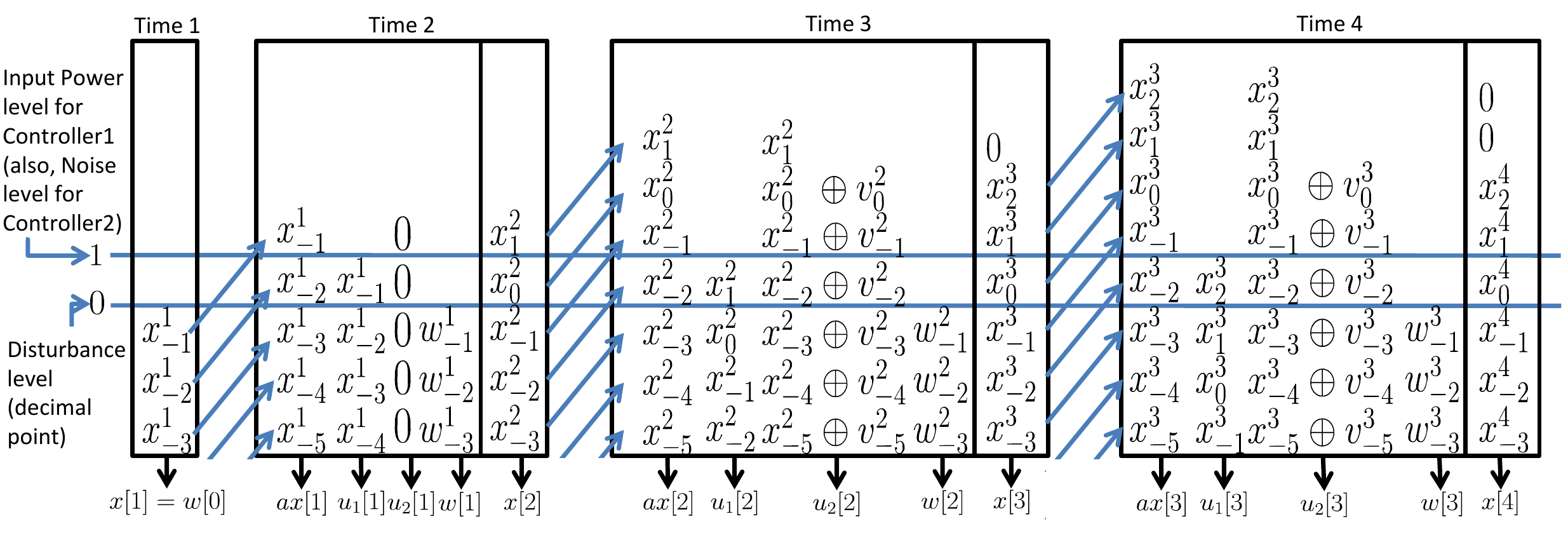}
\caption{Deterministic Model Interpretation of Linear Control Strategies $L_{lin}'$}
\label{fig:lin}
\end{center}
\end{figure*}

After reading the problem statement and the main result, readers may wonder\\
(1) Why is the proposed set of nonlinear strategies enough to achieve a constant ratio from the optimal cost?\\
(2) Why are linear strategies not enough to achieve a constant ratio from the optimal? \\

In this section, we will give an intuitive answer for these questions based on a linear deterministic model in the spirit of Avestimehr, Diggavi and Tse~\cite{Salman_Wireless}, which has already proved to be useful in understanding some control problems~\cite{Pulkit_Witsen,Ranade_implicit}.

The point of these linear deterministic models is to simplify and idealize real arithmetic, and allow us to take a linear view of nonlinearity.
The idea is to consider real numbers in binary expansion and then to simplify arithmetic by eliminating carries. For example, if we have a number $5$ we write it as  $101$ in binary expansion, and likewise we write $\frac{1}{4}$ as $0.01$. If we have a random variable $X$ which is uniform on $[0,4)$, we can write it as $b_1 b_0. b_{-1} b_{-2} \cdots $ in binary expansion where $b_i$ are i.i.d.~Bernoulli($\frac{1}{2}$) random variables on $\{ 0,1\}$.

Since uniform random variables are so simple, we idealize Gaussian random variables as uniform random variables. For a given Gaussian random variable with zero mean and variance $\sigma^2$, we caricature it as a uniform random variable on $[0,\sigma)$ and use the same deterministic model for the uniform random variable. For example, a Gaussian random variable $\mathcal{N}(0,4^2)$ is caricatured as $b_1 b_0 . b_{-1} b_{-2} \cdots$.

Then, we simplify the arithmetic on binary representations. Addition and subtraction are approximated by bitwise XOR --- thereby ignoring the carry effect. For $B'=b_1' b_0' . b_{-1}' b_{-2}' \cdots$ and $B''=b_1'' b_0'' . b_{-1}'' b_{-2}'' \cdots$, both $B'+B''$ and $B'-B''$ are approximated by $(b_1' \oplus b_1'') (b_0' \oplus b_0'') . (b_{-1}' \oplus b_{-1}'') (b_{-2}' \oplus b_{-2}'') \cdots$. Since we are modeling addition and subtraction in the same way, we ignore the sign of the numbers and consider $x$ and $-x$ to be the same in deterministic models and we will assume every number is positive from now on.

Multiplication is approximated by a bit shift. For example, $B' \times 4 $ and $B' \slash 4$ are equal to $b_1' b_0' b_{-1}' b_{-2}' . b_{-3}' \cdots$ and $0. b_1' b_0' b_{-1}' \cdots $ respectively. If we restrict multipliers and dividers to be $2^n$, this agrees with conventional multiplication and division.

For further discussion, it will be helpful to define the (binary) index and level for binary expansions. For a given binary expansion $B=\cdots b_1 b_0. b_{-1} b_{-2}$, the \textit{index} $i$ bit of $B$ indicates $b_i$. It is natural to call the bits $b_i, b_{i+1}, b_{i+2}, \cdots $ as the bits above level $i$, and the bits $b_{i-1}, b_{i-2}, b_{i-3}, \cdots$  as the bits below level $i$. To clarify this point, we define the \textit{level} $i$ as the imaginary line between two sequential bits $b_i$ and $b_{i-1}$. Thus, the decimal point corresponds to the level $0$. 

We also denote the \textit{upper-level} of $B$ as the minimum level $l$ such that all bits above the level $l$ are $0$, i.e. $b_i=0$ for all $i \geq l$. For example, the upper level of $3$ is $2$ and the upper level of $4$ is $3$. When $B$ is a random variable, we denote the \textit{upper level} $l$ of $B$ as the worst case bound, i.e. the minimum $l$ such that $b_i=0$ for all $i \geq l$ with probability $1$. Therefore, the upper level of a uniform random variable on $[0,4)$ is $2$ since $4$ is not included in the interval.

Now, we can come up with the corresponding binary deterministic counterpart for the LQG problems of Problem~\ref{prob:lqg}. To simplify the discussion, we will assume the first controller has perfect observations, and the second controller has no input cost. Like \cite{Pulkit_Witsen}, we will consider the state disturbance minimization problem for given control power constraints, instead of the weighted long-term cost minimization problem.
\begin{problem}[Binary Deterministic Model for Problem~\ref{prob:lqg}]
Let $a', \sigma_{v2}', p_1' \in \mathbb{Z}$ be the problem parameters. For time index $n \geq 0$ and binary index $i$, the deterministic system dynamics is given as follows:
\begin{align}
&x^{n+1}_{i}= x^n_{i-a'} \oplus u^n_{1,i} \oplus u^n_{2,i}  \oplus w^n_i\\
&y^{n}_{1,i}=x^{n}_i \\
&y^{n}_{2,i}=x^{n}_i \oplus v^{n}_i
\end{align}
Here, $x^0_i = 0$ for all $i$. For all $n$, $w^n_i$ are $0$ for all $i \geq 0$ and i.i.d. Bernoulli $\frac{1}{2}$ on $\{ 0,1\}$ for all $i < 0$. For all $n$, $v^n_i$ are $0$ for all $i \geq \sigma_{v2}'$ and i.i.d. Bernoulli $\frac{1}{2}$ on $\{ 0,1\}$ for all $i < \sigma_{v2}'$. The $v^n_i$ are independent from the $w^n_i$. $u^n_{1,i}$ and $u^n_{2,i}$ are causal functions of $y^n_{1,i}$ and $y^n_{2,i}$ respectively. The first controller has a ``power" limit, $u^n_{1,i}=0$ for all $n \geq 0$ and $i \geq p_1'$.

The goal of control is to minimize the upper-level $d$ on the worst state distortion, i.e. minimizing $d$ such that $x^n_i=0$ for all $i \geq d$ and $n \geq 0$ with probability $1$.
\label{prob:det}
\end{problem}
We can notice that $x^{n}_i$, $u^n_{1,i}$, $u^n_{1,i}$, $w^n_i$,  $v^n_i$ correspond to $x[n]$, $u_1[n]$, $u_2[n]$, $w[n]$, $v[n]$ of Problem~\ref{prob:lqg} respectively. Therefore, we will use the latter terms for a compact representation of the bits in Problem~\ref{prob:det}. 
Moreover, since the parameters of Problem~\ref{prob:det} are given in the binary levels of amplitude, they have the following relationship with those of Problem~\ref{prob:lqg}: $a=2^{a'}$, $\sigma_{v2}^2=2^{2 \sigma_{v2}'}$, $\mathbb{E}[u_1^2[n]] \leq 2^{2 p_1'}$. Through the rest of discussion, we will focus on the case, $a'=2, \sigma_{v2}'=\frac{a'}{2}, p_1'=\frac{a'}{2}$. Therefore, the corresponding parameters in LQG Problem~\ref{prob:lqg} are $a=2^{2}$, $\sigma_{w}^2=1$, $\sigma_{v2}^2=2^{2}$, $\mathbb{E}[u_1^2[n]] \leq 2^{2}$.

Based on this deterministic model, we will answer the first question, `why the proposed strategy is approximately optimal'. First, we can easily derive the following lower bound on the state disturbance.
\begin{proposition}
When $a'=2, \sigma_{v2}'=\frac{a'}{2}, p_1'=\frac{a'}{2}$ in Problem~\ref{prob:det}, the minimum upper-level $d$ on the state distortion level has to be at least $2$.
\label{prop:converse}
\end{proposition}
\begin{proof}
We can easily see that for $n \geq 1$, the distortion level of $x[n]$ is at least $0$ since $w[n-1]$ with upper-level $0$ is added at each time step. At time $n+1$, this distortion will be shifted by two bits up, and the upper-level of the distortion becomes $2$. However, the first controller cannot touch the bits above the level $1$ and so the first controller cannot reduce the distortion level. Moreover, at the second controller, any bits below level $1$ are is corrupted by i.i.d. Bernoulli($\frac{1}{2}$) observation noise. Therefore, the second controller's observation is independent from the unknown bits sitting below the level $1$. Consequently, there is no action it can take to draw that bit to $0$.

Neither controller can reduce the distortion bit sitting between the level $1$ and $2$, so with a positive probability this bit can be non-zero. Therefore, the upper-level of $x[n+1]$ has to be at least $2$, i.e. $d \geq 2$.
\end{proof}

In fact, the following proposition shows this lower bound is actually achievable.
\begin{proposition}
Consider Problem~\ref{prob:det} with $a'=2, \sigma_{v2}'=\frac{a'}{2}, p_1'=\frac{a'}{2}$. Given the first controller's observation $y_1[n]= \cdots y^n_{1,1} y^n_{1,0}. y^n_{1,-1} \cdots$ and the second controller's observation $y_2[n]= \cdots y^n_{2,1} y^n_{2,0}. y^n_{2,-1} \cdots$, let the first and second controller's control input be
\begin{align}
&u_1[n] =  y^n_{1,-2}. y^n_{1,-3} y^n_{1,-4}\cdots\\
&u_2[n] =  \cdots y^n_{2,2} y^n_{2,1}000.0\cdots
\end{align}
Then, this strategy can achieve the optimal upper-level on the state distortion, $d=2$.
\label{prop:detoptimal}
\end{proposition}
\begin{proof}
Since we already know the minimal $d \geq 2$ from Proposition~\ref{prop:converse}, it is enough to show that the proposed strategy can achieve $d=2$.

Figure~\ref{fig:nonlin} shows the resulting dynamics when we actually use this strategy. Since the initial state $x[0]=0$, at time 1 both controller's input is also $0$ and $w[0]$ is only term that contributes to $x[1]$. Thus, $x[1]$ can be represented by $0.x^1_{-1} x^1_{-2} x^1_{-3} \cdots$ in the deterministic model where each bit is i.i.d. Bernoulli $\frac{1}{2}$ in $\{0,1\}$.

At time 2, $x[1]$ is shifted two-bits up to generate $x^1_{-1} x^1_{-2} . x^1_{-3} \cdots$. Since the first controller's observation $y_1[1]$ is equal to $x_1[n]$, its control input is $ x^1_{-2} . x^1_{-3} \cdots$. After being corrupted by the noise $v_2[1]$, the second controller's observation $y_2[1]$ becomes $v^1_0. ( x^{1}_{-1} \oplus v^1_{-1}) (x^{1}_{-2} \oplus v^1_{-2}) \cdots$ which is independent from the state, and as a result $u_2[1]=0$. When all of these are added, the second bit of the state canceled by the first controller's input. As we can see in Figure~\ref{fig:nonlin} the state $x[2]$ results in $x^2_1 0 . x^2_{-1} x^2_{-2} \cdots$ where each bit is i.i.d. Bernoulli $\frac{1}{2}$ except for the $0$ in the $0$th position.

At time 3, the first controller does essentially the same operation as time 2, canceling the lower bits below the level $1$. However, the second controller's observation has larger level than before, $y_2[2]= x^2_1 (v^2_0) . (x^2_{-1} \oplus v^2_{-1}) \cdots $. Thus, $u_2[2]$ becomes $x_1^2 000.0 \cdots$.
When we are adding these values, the first bit of the state is canceled by the second controller's input and the third bit of the state is canceled by the first controller's input. As Figure~\ref{fig:nonlin} shows, the resulting state $x[3]$ is  $x^3_1 0 . x^3_{-1} x^3_{-2} \cdots$, which is essentially the same as $x[2]$.

Therefore, we arrive in steady state and repeating the control strategy always gives the state with the same upper-level $2$. This finishes the proof.
\end{proof}

So, we have an optimal scheme for the deterministic model. Let's apply the insights that we learnt from the deterministic model to the original LQG problem.
The first controller's strategy of Proposition~\ref{prop:detoptimal} can be understood as a sequence of two operations. The first operation is extracting the lower bits of $y_1[n]$ and thus generating $0.0y_{1,-2} y_{1,-3}\cdots$. To mimic this, we can simply divide $y_1[2]$ by $0.1$ (in binary) and take the remainder. Using Definition~\ref{def:sig}, this can be written as $R_d(y_1[2])$ with $d=0.1$.  The second operation is shifting the bits up to generate $u_1[2]$, which is just multiplication by a constant ($-a$ to be exact). Therefore, $u_1[2]=-a R_d(y_1[2])$.

The second controller's strategy of Proposition~\ref{prop:detoptimal} can be understood as a sequence of two operations. The first operation is extracting the higher bits of $y_2[2]$ and thus generating $y_{2,m} \cdots y_{2,1}0.0 \cdots$. For this, we can divide $y_2[2]$ by $10$ (in binary) and take the quotient (exactly speaking, quotient multiplied by divisor). Using Definition~\ref{def:sig}, this can be written as $Q_{d'}(y_2[2])$ with $d'=ad=10$. The second operation is shifting two bits up as before, which is multiplication. Therefore, $u_2[2]=-a Q_{d'}(y_2[2])$.

Compared to \eqref{eqn:1} and \eqref{eqn:2}, this strategy is essentially equivalent to $1$-stage signaling strategy except for some minor terms in $u_2[n]$. Therefore, in nonlinear strategies $L_{sig,s}$, $u_1[n]$ tries to cancel the lower bits in $ax[n]$ by exploiting its better observation and $u_2[n]$ tries to cancel higher bits in $ax[n]$ exploiting its less expensive input cost.

Now, we understand why the proposed strategy might be approximately optimal. We can move on to the next question, `why linear is not enough for constant-ratio optimality'. Let's first remind ourselves of the linear operations in binary deterministic models. Addition and subtraction correspond to bitwise XOR. Multiplication and division by a constant correspond to shifting bits up and down. Let's revisit Problem~\ref{prob:det} with these restrictions on the strategy, and understand why we cannot achieve the optimal performance with linear strategies.

\begin{proposition}
Consider Problem~\ref{prob:det} with $a'=2, \sigma_{v2}'=\frac{a'}{2}, p_1'=\frac{a'}{2}$. Let's restrict the controller strategies to be the following forms: For some $k_{i,j}, k_{i,j}' \in \mathbb{Z}$ and for all $i \in \mathbb{Z}$ and $n \in \mathbb{Z}^+$,
\begin{align}
&u^n_{1,i} = y^0_{1,i+k_{0,n}} \oplus y^1_{1,i+k_{1,n}} \oplus \cdots \oplus y^n_{1,i+k_{n,n}} \\
&u^n_{2,i} = y^0_{2,i+k_{0,n}'} \oplus y^1_{2,i+k_{1,n}'} \oplus \cdots \oplus y^n_{2,i+k_{n,n}'}
\end{align}
Under this constraint on control strategy, the minimal upper-level $d$ on the state distortion is $3$.
\label{prop:detlin}
\end{proposition}

Intuitively this proposition is obvious. As we saw in Proposition~\ref{prop:converse}, the first controller is input power is strictly less than the distortion level $2$. When we restrict the strategy to be linear, the first controller cannot cancel any bits in the state. Therefore, the second controller is the only controller that can control the state. The second controller can only see the bits above level $1$, and after one time step, the distortion level will become $3$. Let's clarify this point more carefully.

Time 1 is the same as the proof of Proposition~\ref{prop:detoptimal}. However, at time 2 the first controller cannot cancel the lower bits any more. The only allowed operations are shifting the bits in each observation and taking XOR between them. As we can see in Figure~\ref{fig:lin}, within the power constraint the first controller cannot but shift at least one-level down the bits in $y_1[1]$, and may choose $u_1[1]=x^1_{-1}.x^1_{-2}x^1_{-3} \cdots$. As we discussed in Proposition~\ref{prop:detoptimal}, the second controller's observation is independent from the state and the optimal $u_2[1]$ is $0$. Therefore, as we can see in Figure~\ref{fig:lin} no bits cancel with each other, and $x[2]=x^2_{1}x^2_{0}.x^2_{-1} \cdots$ where each bits are i.i.d. Bernoulli $\frac{1}{2}$.

At time 3, due to the same reason, the best feasible input for the first controller is $u_1[2]=x^2_{1}.x^2_{0}x^2_{-1}\cdots$ and cannot cancel any bits in the state. Meanwhile, the second controller's observation with additive noise is $y_2[2]=x^2_{1} (x^2_{0} \oplus v^2_{0}) . (x^2_{-1} \oplus v^2_{-1}) \cdots $. Therefore, to cancel the first bits of the state, the second controller shifts two-bits up in $y_2[2]$ and chooses $u_2[2]=x^2_{1} ~(x^2_{0} \oplus v^2_{0}) ~(x^2_{-1} \oplus v^2_{-1}) ~(x^2_{-2} \oplus v^2_{-2}) . ~(x^2_{-3} \oplus v^2_{-3}) \cdots$. When these are added, the first bit of the state cancels and the resulting state $x[3]$ has three bits above the decimal point as we can see in Figure~\ref{fig:lin}.

By repeating this procedure, we can see that after the transient states of time step $1,2,3$, the plant and controllers stay in steady state.  From Figure~\ref{fig:lin}, in steady state, $x[n]$ has three bits above the decimal point. Therefore, compared to the optimal performance without the linear controller constraint, we can see one-bit-level performance degradation. This degradation comes from the inefficient use of the first controller input. In other words, the first controller cancels the lower bits of the state in the optimal strategy while it cannot cancel any bits in linear strategies.

Let's consider the Gaussian counterpart of the previous results. As we discussed earlier, the corresponding parameters in the original LQG problem is $\sigma_{v1}^2=0$, $\sigma_{v2}^2=\Theta (a)$, $\mathbb{E}[u_1^2[n]] \leq \Theta(a)$, $\mathbb{E}[u_2^2[n]] \leq \infty$. We will consider the minimum state distortion as $a$ goes to infinity. From Proposition~\ref{prop:detoptimal}, we can expect that the optimal state distortion is $\mathbb{E}[x^2[n]] \leq O(a^2)$ with these parameters.\footnote{Exactly speaking, the optimal state distortion is $\mathbb{E}[x^2[n]] \leq O(a^2 \log a)$ with input power constraint $\mathbb{E}[u_1^2[n]] \geq \Theta(a \log a)$. The is due to the fact that unlike uniform random variables Gaussian random variables can be arbitrary large with logarithmically decreasing probability. Later, this effect will be captured by large deviation ideas, and turns out to be crucial to get constant-ratio optimality. We will discuss more about this issue in Section~\ref{sec:caveat}.} From Proposition~\ref{prop:detlin}, we can expect that the state distortion is $\mathbb{E}[x^2[n]] \geq \Omega(a^3)$ when we restrict control strategies to be linear. Here, we can see the ratio between the optimal cost and the linear strategy cost goes to infinity as $a$ grows.

Even if the discussion so far focused on minimizing the state distortion under the power constraints, the result can be easily converted to the weighted long-term cost. Let's choose the parameters of Problem~\ref{prob:lqg} as $q=1$, $r_1=a$, $r_2=0$, $\sigma_0=0$, $\sigma_{v1}^2=0$, and $\sigma_{v2}^2=a$, i.e.
\begin{align}
\underset{u_1,u_2}{\inf}
\underset{N \rightarrow \infty}{\limsup}
\frac{1}{N}
\underset{0 \leq n < N}{\sum}
\mathbb{E}[ x^2[n]+a u_1^2[n]]\nonumber
\end{align}
If $\mathbb{E}[x^2[n]] \leq O(a^2)$ when $\mathbb{E}[u_1^2[n]]\leq \Theta(a)$ as we predicted, the optimal weighted cost has to be $O(a^2)$.
However, if we restrict the control strategies to be linear, $\mathbb{E}[x^2[n]]$ will be $\Omega(a^3)$ with the same power constraint according to our conjecture. Therefore, we need at least $\mathbb{E}[u_1^2[n]] \geq \Theta(a^2)$ to make $\mathbb{E}[x^2[n]] \leq O(a^2)$. In either case\footnote{One may wonder why we do not consider the cases between $\mathbb{E}[u_1^2[n]]=\Theta(a)$ and $\mathbb{E}[u_1^2[n]]=\Theta(a^2)$, for example $\mathbb{E}[u_1^2[n]] = \Theta(a^{\frac{3}{2}})$. The reason comes from the limitation of these bit-wise deterministic models, precision. For $a=4$, we will write $u_1[n]$ as $u_{0}^n.u_{-1}^n u_{-2}^n \cdots $ in binary when $\mathbb{E}[u_1^2[n]]=a$, and as $u_{1}^n u_{0}^n. u_{-1}^n \cdots $ when $\mathbb{E}[u_1^2[n]]=a^2$. When $\mathbb{E}[u_1^2[n]]=a^{\frac{3}{2}}$, we have to choose either one of these two. We choose the former in this paper, so we cannot resolve the difference between $\mathbb{E}[u_1^2[n]]=a$ and $\mathbb{E}[u_1^2[n]]=a^{\frac{3}{2}}$.}, the weighted cost is $\Omega(a^3)$.

Formally, the following proposition formalizes this insight and proves the ratio between the optimal cost and linear strategy cost actually diverges in Gaussian problems.
\begin{proposition}
Let $L_{lin}'$ be the set of all linear time-varying controllers which can be written in the following form:
\begin{align}
&u_1[n]= \sum_{i \leq n} k_{n,i} y_1[i], \nonumber \\
&u_2[n]= \sum_{i \leq n} k'_{n,i} y_2[i]. \nonumber
\end{align}
Consider Problem~\ref{prob:lqg} with parameters $q=1$, $r_1=a$, $r_2=0$, $\sigma_0^2=0$, $\sigma_{v1}^2=0$, and $\sigma_{v2}^2=a$. Then, we have
\begin{align}
\frac
{
\underset{u_1,u_2 \in L_{lin}'}{\inf}
\underset{N \rightarrow \infty}{\limsup}
\frac{1}{N}
\underset{0 \leq n < N}{\sum}
\mathbb{E}[q x^2[n]+r_1 u_1^2[n]]
}
{
\underset{u_1,u_2 \in L_{sig,1}}{\inf}
\underset{N \rightarrow \infty}{\limsup}
\frac{1}{N}
\underset{0 \leq n < N}{\sum}
\mathbb{E}[q x^2[n]+r_1 u_1^2[n]]
}
\rightarrow \infty \nonumber
\end{align}
as $a \rightarrow \infty$.
\label{prop:1}
\end{proposition}
\begin{proof}
See Appendix~\ref{app:proposition}.
\end{proof}

From the discussion above, we can see that the first controller with better observations is ``signaling" to the second controller (with worse observations) through the control actions. However, the notion of communication here is different from the conventional one. In conventional communication problems,
the transmitter has access to a source (but cannot change it) and reduces the uncertainty about the source at the destination by explicitly sending information about the source.

However, in control systems the source is the state, and the ``transmitter"(which is a controller) can change the source itself by control action. Therefore, it can reduce the uncertainty of the source and make the source easier to estimate at the destination. Then, the destination will have a better idea about the source even without receiving any explicit information.  This generalized notion of communication is the one happening between the first and the second controller.

Moreover, we can also see the delay of the communication is crucial in control problems, while this is usually ignored in traditional information theory.
In Figure~\ref{fig:nonlin}, the second controller has to wait until the disturbance is amplified above its observation noise level, which causes a $1$-step delay between two controllers. However, as we increase the observation noise level of the second controller, the second controller has to wait longer until the disturbance is amplified enough and this will result in a longer ``delay" between the two controller's actions.

In Section~\ref{sec:upperbound}, we will explore this point by relating the infinite-horizon LQG problem to control problems with different time horizons.
As we saw in Figure~\ref{fig:dyn}, Radner's problem~\cite{Radner_Team} and Witsenhasuen's counterexample~\cite{Witsenhausen_Counterexample} are sub-blocks of the infinite-horizon LQG problem. We will see later in Section~\ref{sec:upperbound} that the scheme discussed here is a $1$-step-delay implicit communication scheme which essentially solves Witsenhausen's counterexample. In general, we may need up to an $s$-step-delay implicit communication to solve $s$-stage MIMO Witsenhausen's counterexamples.

\subsection{Caveat: Deterministic Model does not work for Radner's Problem}
\label{sec:caveat}

\begin{figure*}[htbp]
\begin{center}
        \begin{subfigure}[b]{0.3\textwidth}
                \centering
                \includegraphics[width=\textwidth]{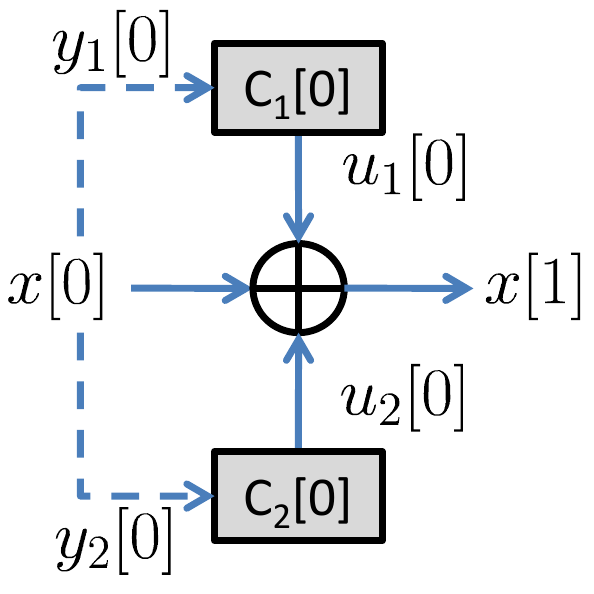}
                \caption{}
                \label{fig:radner1}
        \end{subfigure}
        \begin{subfigure}[b]{0.4\textwidth}
                \centering
                \includegraphics[width=\textwidth]{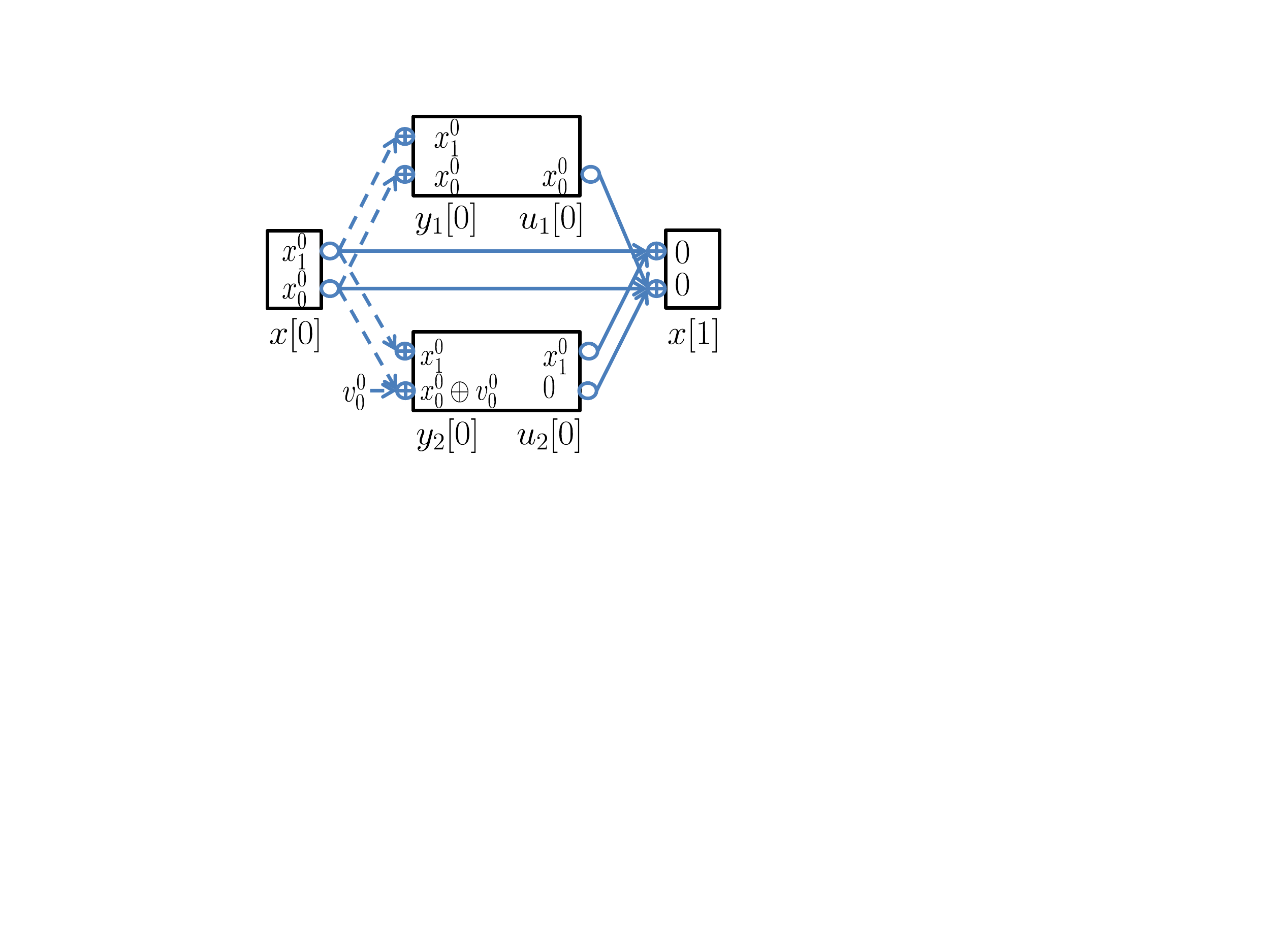}
                \caption{}
                \label{fig:radner2}
        \end{subfigure}
\caption{(a) Radner's Problem  and  (b) the corresponding binary deterministic model. Here, the binary deterministic model can fail to correctly predict the optimal strategy and the optimal performance.}
\label{fig:radner}
\end{center}
\end{figure*}

\begin{figure*}[htbp]
\begin{center}
        \begin{subfigure}[b]{0.35\textwidth}
                \centering
                \includegraphics[width=\textwidth]{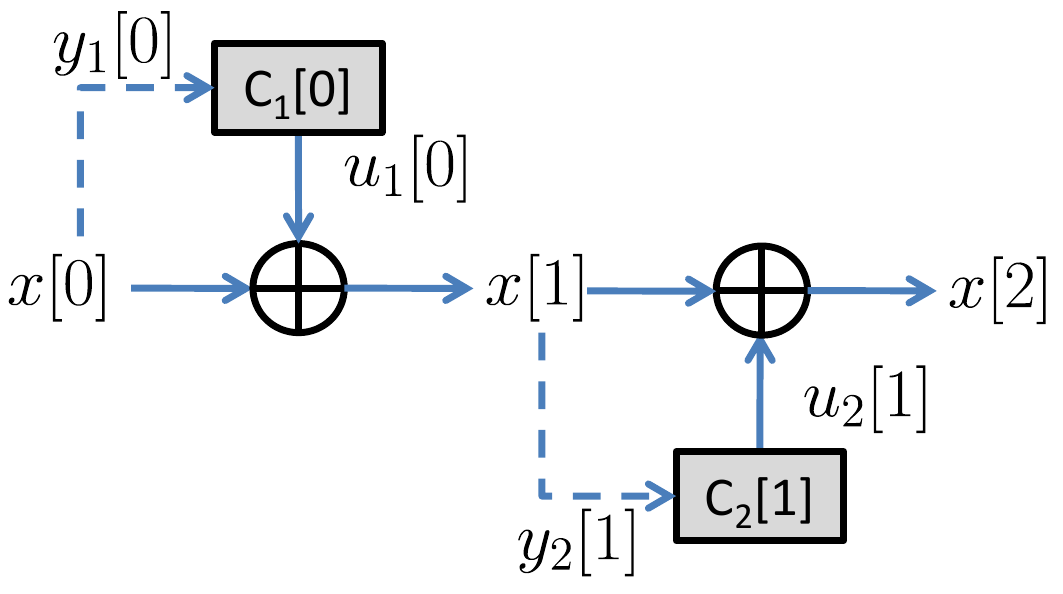}
                \caption{}
                \label{fig:witsen1}
        \end{subfigure}
        \begin{subfigure}[b]{0.45\textwidth}
                \centering
                \includegraphics[width=\textwidth]{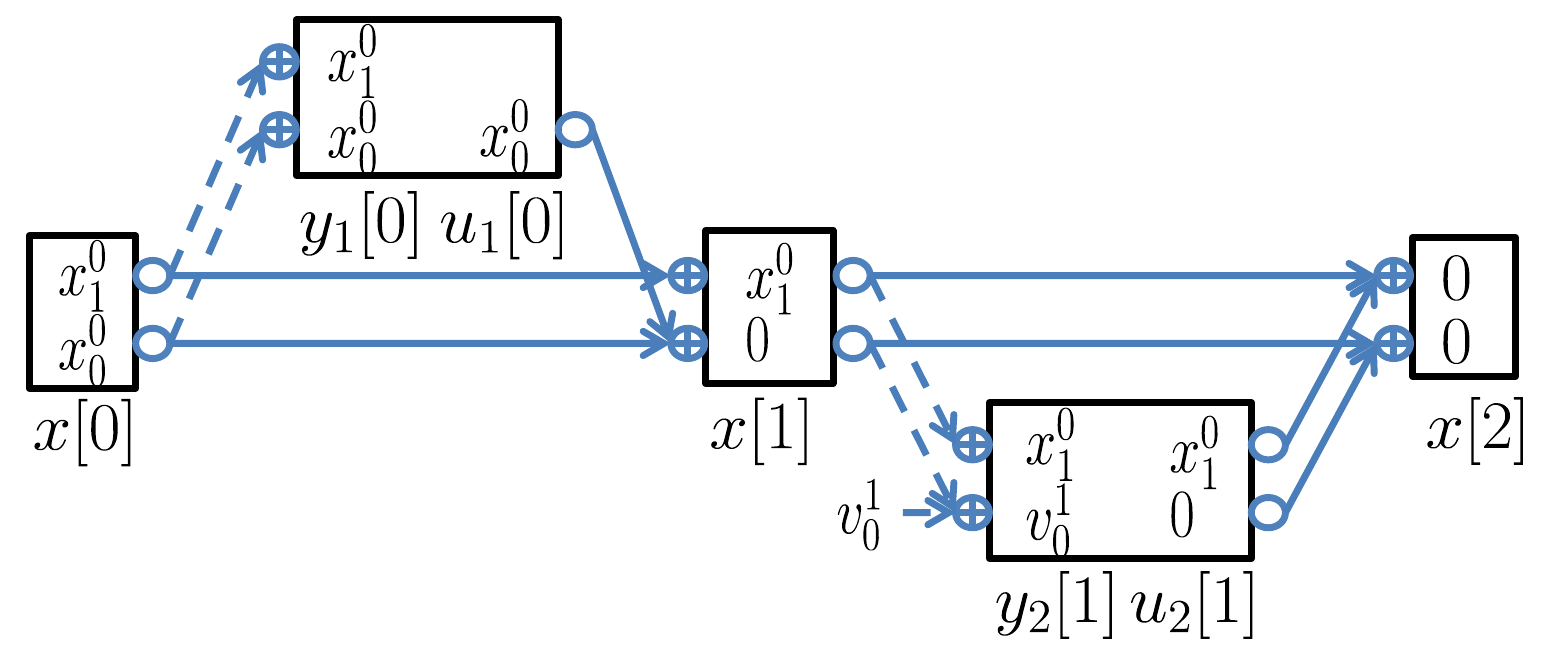}
                \caption{}
                \label{fig:witsen2}
        \end{subfigure}
\caption{(a) Witsenhausen's Counterexample  and  (b) the corresponding binary deterministic model . Here, the binary deterministic model does approximately predict the optimal strategy and the optimal performance.}
\label{fig:witsen}
\end{center}
\end{figure*}

Even though we explained the result based on the binary deterministic model, it is just a simplified model for intuition and we should not naively believe that the same results always hold in Gaussian models as well. In fact, we will show that in Radner's problem~\cite{Radner_Team} the deterministic model fails to correctly predict the behavior of Gaussian problems.

In \cite{Radner_Team}, Radner considered the following problem of Figure~\ref{fig:radner1}:
$x[0]$, $v_1[0]$, $v_2[0]$ are independent Gaussian random variables with zero mean and variance $\sigma_{0}^2$, $\sigma_{v1}^2$, $\sigma_{v2}^2$ respectively. Let $y_1[0]:=x[0] + v_1[0]$, $y_2[0]:=x[0] + v_2[0]$, $u_1[0] := f_1(y_1[0])$, $u_2[0] := f_2(y_2[0])$ and $x[1] := x[0] + u_1[0] + u_2[0]$. The control objective is minimizing $\mathbb{E}[q  x[1]^2+r_1  u_1[0]^2 + r_2  u_2[0]^2]$. And he proved that a linear controller is optimal.

Later, Witsenhausen found that if we shift the second controller by one time-step, the problem is fundamentally different and the optimal controller is not linear~\cite{Witsenhausen_Counterexample}. Figure~\ref{fig:witsen1} shows the counterexample: $x[0]$, $v_1[0]$, $y_1[0]$, $u_1[0]$ are the same as Radner's problem. However, $x[1] := x[0] + u_1[0]$, the second controller observes $y_2[1]:=x[1]+v_2[1]$ where $v_2[1]$ is Gaussian with zero mean and variance $\sigma_{v2}^2$, and $u_2[1] := f_2(y_2[1])$, $x[2]=x[1]+u_2[1]$. The control objective is minimizing $\mathbb{E}[q x[2]^2+r_1  u_1[0]^2 + r_2  u_2[1]^2]$.

At a high level, this difference can be understood in terms of implicit communication. Radner's problem is a single-stage problem. Even if one controller sends some information, it is impossible for the other controller to receive the information at the next time step. Therefore, implicit communication between the controllers is impossible, and it is widely believed that if this is the case, then linear is optimal~\cite{Witsenhausen_Separation,Yuksel_Nested,Shah_Partial,Lessard_state, Rotkowitz_information}. However, Witsenhausen's counterexample is a two-stage problem. If the first controller sends some information, the second controller can receive this information at the next time step. Therefore, implicit communication is possible, and nonlinear strategies which are good at this implicit communication can outperform linear strategies.

Let's revisit these problems using the binary deterministic models. Like in Section~\ref{sec:intui}, we will give a perfect observation to the first controller and allow unbounded input power for the second controller. The goal of control is minimizing the state disturbance for a given input power constraint.

Binary deterministic model counterparts of Radner's problem and Witsenhausen's counterexample, shown in Figure~\ref{fig:radner2} and \ref{fig:witsen2} respectively, are formulated as follows.
\begin{problem}[Binary Deterministic Model for Radner's Problem]
For binary level index $i$, the deterministic system dynamics is given as follows:
\begin{align}
& x^1_{i} = x^0_{i} \oplus u^0_{1,i} \oplus u^0_{2,i},\\
& y^0_{1,i}=x^0_{i}, \\
& y^0_{2,i}=x^0_{i} \oplus v^0_{i}.
\end{align}
Here, $x^0_{i}$ are $0$ for all $i \geq 2$ and Bernoulli $\frac{1}{2}$ on $\{0,1 \}$ for all $i < 2$. $v^0_{i}$ are $0$ for all $i \geq 1$ and Bernoulli $\frac{1}{2}$ on $\{0,1\}$ for all $i < 1$. $u^0_{1,i}$ and $u^0_{2,i}$ are functions of $y^0_{1,i}$ and $y^0_{2,i}$ respectively. The first controller has a power limit,  $u^0_{1,i}=0$ for all $i \geq 1$.

The goal of the control is to minimize the final state distortion level $d$, i.e. minimizing $d$ such that $x^1_i=0$ for all $i \geq d$.
%
\label{prob:detrad}
\end{problem}
Here, we can easily notice that $x_i^0, x_i^1, u_{1,i}^0, u_{2,i}^0, y_{1,i}^0, y_{2,i}^0, v_i^0$ correspond to $x[0], x[1], u_1[0], u_2[0], y_1[0], y_2[0], v_2[0]$ in the original Radner's problem.

\begin{problem}[Binary Deterministic Model for Witsenhausen's Counterexample]
For binary level index $i$, the deterministic system dynamics is given as follows:
\begin{align}
&x^1_i = x^0_i \oplus u^0_{1,i} \\
&x^2_i = x^1_i \oplus u^1_{2,i} \\
&y^0_{1,i} = x^0_i \\
&y^1_{2,i} = x^1_i \oplus v^1_{i}
\end{align}
Here, $x^0_i$ are $0$ for all $i \geq 2$ and Bernoulli $\frac{1}{2}$ on $\{0,1\}$ for all $i < 2$. $v^1_i$ are $0$ for all $i \geq 1$ and Bernoulli $\frac{1}{2}$ on $\{0,1 \}$ for all $i < 1$. $u^0_{1,i}$ and $u^1_{2,i}$ are functions of $y^0_{1,i}$ and $y^1_{2,i}$ respectively. The first controller has a power limit, $u^0_{1,i}=0$ for all $i \geq 1$.

The goal of the control is to minimize the final state distortion level $d$, i.e. minimizing $d$ such that $x^2_i=0$ for all $i \geq d$.
%
\label{prob:detwit}
\end{problem}
Here, we can easily notice that $x_i^0, x_i^1, x_i^2, u_{1,i}^0, u_{2,i}^1, y_{1,i}^0, y_{2,i}^1, v_i^0$ correspond to $x[0], x[1], x[2], u_1[0], u_2[1], y_1[0], y_2[1], v_2[1]$ in the original Witsenhausen's problem.

As we can see in Figure~\ref{fig:radner2} and Figure~\ref{fig:witsen2}, essentially the same scheme that we discussed in Section~\ref{sec:intui} can be also used in both deterministic problems to give the optimal cost. The first controller cancels the lower bits $x_{20}$ and the second controller cancels the higher bits $x_{10}$ at the next time step.
\begin{proposition}
At time $n$, let the first controller's observation be $y_1[n]= \cdots y_{1,1}^n y_{1,0}^n. y_{1,-1}^n \cdots$ in binary expansion. Likewise, the second controller's observation is $y_2[n]=\cdots y_{2,1}^n y_{2,0}^n. y_{2,-1}^n \cdots$ in binary expansion. Then, the following control strategy achieves $d=-\infty$ (i.e. the final state is zero.) in both Problem~\ref{prob:detrad} and \ref{prob:detwit} and is optimal in both problems.
\begin{align}
&u_1[n] = y_{1,0}^n. y_{1,-1}^n \cdots \\
&u_2[n] =\cdots y_{2,2}^n y_{2,1}^n 0. 00 \cdots
\end{align}
\label{prop:witsencost}
\end{proposition}
\begin{proof}
Immediately follows from Figure~\ref{fig:radner2} and Figure~\ref{fig:witsen2}.
\end{proof}

As we discussed in Section~\ref{sec:intui} the corresponding strategy in the reals is a nonlinear strategy. However, linear is optimal in Radner's problem. How can this be? Clearly, the nonlinear strategy is not even approximately achieving the cost that the binary deterministic model promises. The binary deterministic model \textbf{fails to predict} the optimal control strategy and the optimal cost of the real Gaussian Radner's problem. The reason for this is the binary deterministic model ignores the carry-over in addition and subtraction which is actually happening in real Gaussian problems. In fact, we can see the difference between $y_2[0]$ of Figure~\ref{fig:radner2} and $y_2[1]$ of Figure~\ref{fig:witsen2}. The second bit of $y_2[0]$ of Figure~\ref{fig:radner2} is $x_{20} \oplus v_{10}$ which causes carry-over in the reals, while the second bit of $y_2[1]$ of Figure~\ref{fig:witsen2} is $v_{10}$ is just $v_{10}$. Therefore, the bitwise separation ignoring the carry-over results in an overly optimistic conclusion in binary deterministic models. A linear view of nonlinearity is too simplified in this case. To clarify this point, in \cite{Park_Thesis} we propose a simple deterministic model ---which we named a `ring model'--- that takes into account of the carry-over effect, and succeeds in capturing the behavior of the optimal costs in Radner's problem. We refer to \cite{Park_Thesis} for further details on this point.

In fact, even in Witsenhausen's counterexample there is a small gap between the predicted cost and actual LQG cost, even though the deterministic model correctly predicts the approximately optimal strategy. As we can see in Proposition~\ref{prop:witsencost}, in the deterministic model the final state is $0$ as long as the first controller's input power is greater than the second controller's noise level. In the corresponding LQG problem, the final cost turns out to be only an exponentially decreasing function of the first controller's input power. However, the underlying reason for this gap is different from that in Radner's problem. This gap in Witsenhausen's counterexample comes from the tail of Gaussian random variables and the finite-dimensionality of the problem.\footnote{
In infinite-dimensional problems, the laws of large numbers guarantee that Gaussian random variables behave typically and the probability that they can be arbitrary large asymptotically goes to zero. Therefore, we can drive the final cost to $0$ with bounded first controller's input power, and the cost predicted by the deterministic model is actually achievable. To capture the finite-dimentionality of the problem, we have to use large deviation ideas. We refer to \cite{Grover_distributed, Choudhuri_witsenhausen} for further details.}
While all disturbances are bounded with probability $1$ in  deterministic models, in LQG problems Gaussian random variables can be arbitrary large with an exponentially decreasing probability. This results in a logarithmic gap between the costs in Witsenhausen's counterexample. However, unlike in Radner's problem this gap is only logarithmic and the insights that we gain from the deterministic models are still useful in the original LQG problems.

Therefore, we can rightfully say that deterministic models predict the essential behavior of Witsenhausen's counterexample, while it fails to predict for Radner's problem.

\begin{figure*}[htbp]
\begin{center}
\includegraphics[width=3in]{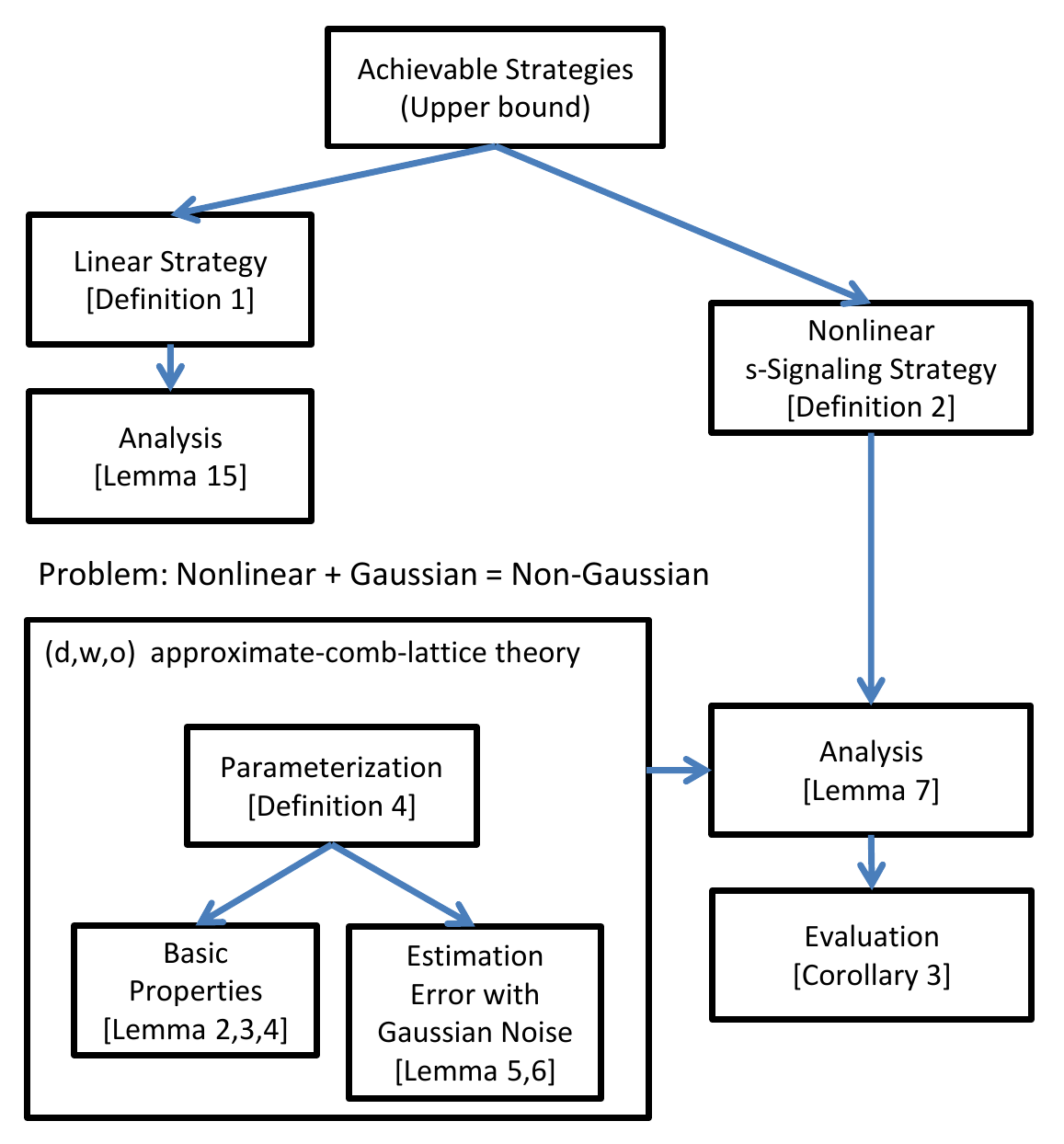}
\caption{Flow diagram of the ideas for the upper bound on the control performance}
\label{fig:proofflow}
\end{center}
\end{figure*}

\begin{figure*}[htbp]
\begin{center}
\includegraphics[width=6in]{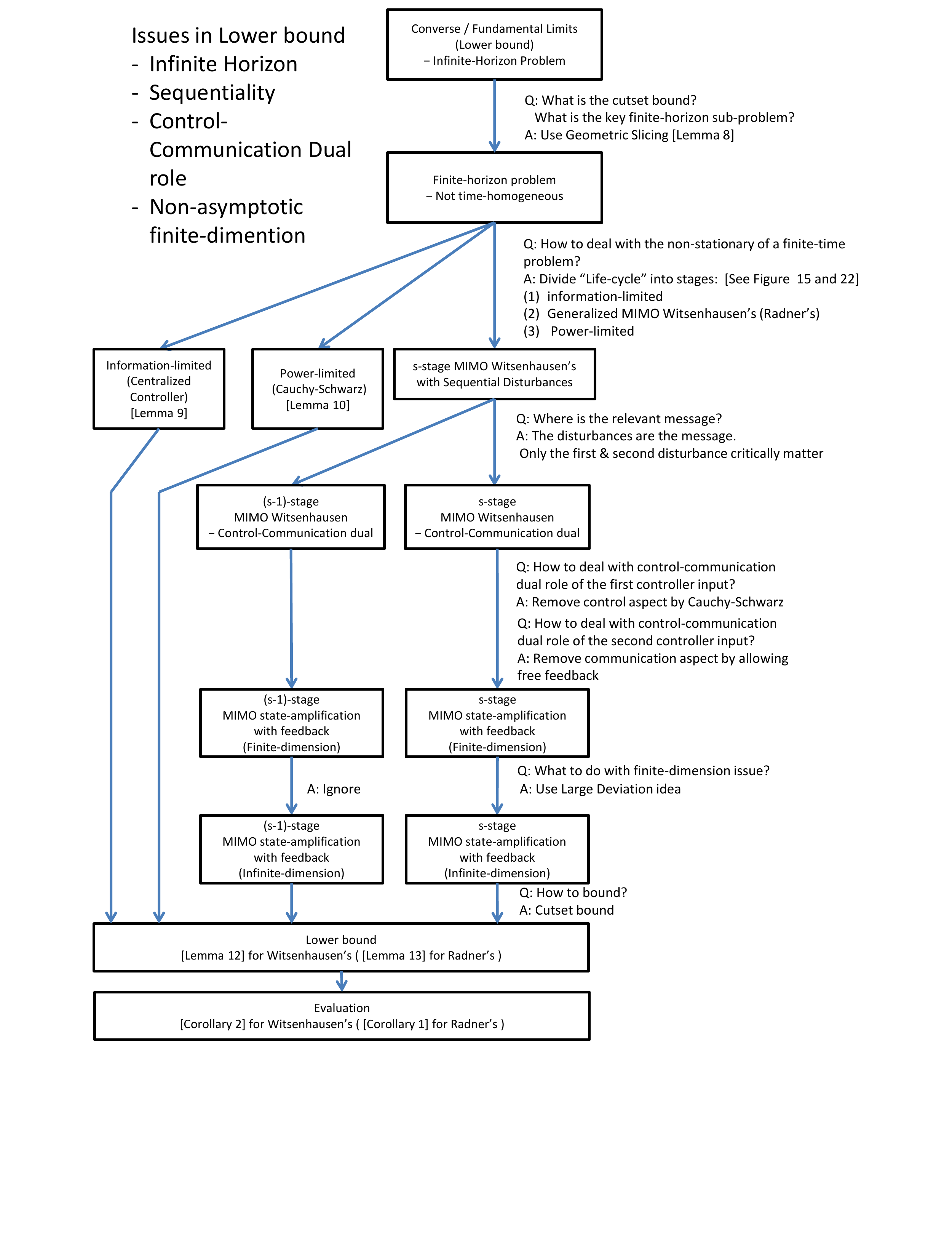}
\caption{Flow diagram of the ideas for the lower bound on the control performance}
\label{fig:proofflow2}
\end{center}
\end{figure*}

\section{Proofs and Proof Ideas: Outline}
\label{sec:proof}
The formal proof of the main result is separated into three parts. We will give upper and lower bounds on the optimal cost, and then compare them to show that they are within a constant ratio.

Figure~\ref{fig:proofflow} shows the proof idea flow for the upper bound\footnote{This corresponds to achievability arguments in information theory.} on the optimal cost. This is done by analyzing specific control strategies. First, it is easy to analyze linear strategies by simply tracking mean and variance. For nonlinear strategies, it can be tricky since mean and variance do not characterize non-Gaussian random variables. Therefore, in Section~\ref{sec:dwo}, we will introduce a mini theory to analyze quantization-based strategies, which we call $(d,w,o)$-approximate-comb-lattice theory. Section~\ref{sec:dwoanalysis} will actually analyze the nonlinear strategy performance based on this theory.

To show that we cannot do much better, we also have to find a lower bound\footnote{This corresponds to converse arguments in information theory.} on the cost. Figure~\ref{fig:proofflow2} shows the flow of ideas in the proof for the lower bound. The key idea is identifying the informational bottleneck of a problem and figuring out the information relaying between the controllers. 
In information theory, to figure out the informational bottleneck of the system, we partition the nodes and apply cutset bounds~\cite{Salman_Wireless,Cover}. However, here rather than simply partitioning the nodes, we expand the system in time and must divide the infinite-horizon problem into finite-horizon ones. The geometric slicing idea (Figure~\ref{fig:geometric2}) is introduced for this.

Now, we have a finite-horizon problem. However, unlike infinite-horizon problems where the effect of transients can be amortized over infinitely many stationary states, the transient states are the essence of a finite-horizon problem and therefore the problem is non-stationary. To handle this issue, we divide the resulting finite-horizon problem into three sub time-intervals --- childhood, youth and old age, so to speak. Figure~\ref{fig:division} (or Figure~\ref{fig:division2}) shows the division of time intervals. In ``childhood", we do not have enough information about the state, so we will call this interval information-limited. In ``old age", we do not have enough power to control the state too well, so we will call this interval power-limited.
Between these two ---in ``youth"--- something interesting is happening and we will call this interval a MIMO Witsenhaussen's (or Radner's for Figure~\ref{fig:division2}) interval.

In this interesting interval, the first controller is power-limited and the second controller is information-limited, which is essentially the same issue as in Witsenhausen's counterexample. In fact, we will relate this interval to an $s$-stage MIMO Witsenhausen's counterexample where a new disturbance is added at each time step. Then, the question becomes what are the critical disturbances among these? We will see that only first and second disturbances matter, and we can relax to simpler problems which are $s$-stage and $(s-1)$-stage MIMO Witsenhausen's with only one disturbance. However, still these problems are difficult due to the dual role of controller actions. The controller actions can be used to control the states, but at the same time they can be used to communicate some information to the other controller. This control-communication dual role of controller actions makes the problem hard.

To tackle this issue, we remove the control role from the first controller, and thereby the first controller will behave like a transmitter in communication problems. On the other hand, we remove the communication role from the second controller by allowing free feedback, and thereby the second controller will behave like a receiver in communication problems. In this way, we can reduce the problem to MIMO state-amplification with feedback, which generalizes the problem shown in \cite{Kim_amplification}. However, the resulting problem is finite dimensional, and information-theoretic results for infinite-dimensional problems can possibly give loose bounds~\cite{Pulkit_Witsen}. In fact, we have to adapt large deviation ideas to the $s$-stage MIMO state-amplification problem for this reason.\footnote{This is the same issue and idea that we discussed in Section~\ref{sec:caveat} for Witsenhausen's counterexample.} Now, we can apply simple information-theoretic cutset bounds to the final communication problems and derive lower bounds.

Before we discuss the proof details, we first convert the weighted long-term average cost optimization problem to an optimization problem with average power constraints as we did in Section~\ref{sec:intui}. The original control objective is minimizing the weighted cost of the state disturbance and the controller input powers. However, it will be useful to consider minimizing the state given an average bound on the input powers. Formally, the problem is written as follows.

\begin{problem}[Decentralized LQG problem with average power constraints]
Consider the same dynamics as Problem~\ref{prob:lqg}. But, now the control objective is minimizing the state disturbance $D(P_1,P_2)$ for given input power constraints $P_1, P_2 \in \mathbb{R}^+$. We will say the power-disturbance tradeoff, $D(P_1,P_2)$ is achievable if and only if there exist causal control strategies $u_1[n], u_2[n]$ such that
\begin{align}
&\limsup_{N \rightarrow \infty}\frac{1}{N}\sum_{n=1}^{N} \mathbb{E}[x^2[n]] \leq D(P_1,P_2),  \\
&\limsup_{N \rightarrow \infty}\frac{1}{N}\sum_{n=1}^{N} \mathbb{E}[u_1^2[n]] \leq P_1, \\
&\limsup_{N \rightarrow \infty}\frac{1}{N}\sum_{n=1}^{N} \mathbb{E}[u_2^2[n]] \leq P_2.
\end{align}
\label{prob:power}
\end{problem}
Lemma~\ref{rat:lemma1} will relate the weighted-cost problem, Problem~\ref{prob:lqg}, and the power-constraints problem, Problem~\ref{prob:power}, telling us that if we can approximately solve the latter we can also approximately solve the former. To characterize $D(P_1,P_2)$ approximately, we will come up with lower and upper bounds on $D(P_1,P_2)$. Since we are only aiming for an approximate solution, in the discussion for intuitions and interpretations we will focus on the scaling and ignore the constants.

The following Cauchy-Schwarz style inequalities will be helpful to get bounds.
\begin{lemma}
For arbitrary correlated random variables $X_1,\cdots,X_n$, the following inequality holds:
\begin{align}
(\sqrt{\mathbb{E}[X_1^2]} - \sqrt{\mathbb{E}[X_2^2]} \cdots - \sqrt{\mathbb{E}[X_n^2]})_+^2 \leq \mathbb{E}[(X_1 + \cdots + X_n)^2] \leq ( \sqrt{\mathbb{E}[X_1^2]} + \cdots + \sqrt{\mathbb{E}[X_n^2]} )^2
\leq n (\mathbb{E}[X_1^2]+ \cdots + \mathbb{E}[X_n^2])
\end{align}
\label{ach:lemmacauchy}
\end{lemma}
\begin{proof}
\begin{align}
&\mathbb{E}[(X_1 + \cdots + X_n)^2] \\
&=\mathbb{E}[X_1^2] + \cdots + \mathbb{E}[X_n^2] + 2 \mathbb{E}[X_1 X_2] + \cdots + 2 \mathbb{E}[ X_{n-1} X_n]\\
&\leq \mathbb{E}[X_1^2] + \cdots + \mathbb{E}[X_n^2] + 2 \sqrt{\mathbb{E}[X_1^2] \mathbb{E}[X_1^2]} + \cdots + 2 \sqrt{\mathbb{E}[X_1^2] \mathbb{E}[X_n^2]} \\
&= (\sqrt{\mathbb{E}[X_1^2]} + \cdots + \sqrt{\mathbb{E}[X_n^2]})^2 \\
&\leq n (\mathbb{E}[X_1^2]+ \cdots + \mathbb{E}[X_n^2])
\end{align}
where all inequalities follow from Cauchy-Schwarz.
\begin{align}
&\mathbb{E}[(X_1 + \cdots + X_n)^2] \\
&= \mathbb{E}[X_1^2] + 2 \mathbb{E}[X_1(X_2 + \cdots + X_n)] + \mathbb{E}[(X_2+ \cdots + X_n)^2] \\
&\geq \mathbb{E}[X_1^2] - 2 \sqrt{\mathbb{E}[X_1^2] \mathbb{E}[(X_2 + \cdots + X_n)^2]} + \mathbb{E}[(X_2+ \cdots + X_n)^2] \\
&= (\sqrt{\mathbb{E}[X_1^2]}-\sqrt{\mathbb{E}[(X_2 + \cdots + X_n)^2]})^2 \\
&\geq (\sqrt{\mathbb{E}[X_1^2]}-\sqrt{\mathbb{E}[X_2^2]}  \cdots - \sqrt{ \mathbb{E}[X_n^2]})_+^2
\end{align}
where all inequalities follow from Cauchy-Schwarz.
\end{proof}

\begin{figure}[htbp]
\begin{center}
\includegraphics[width=3in]{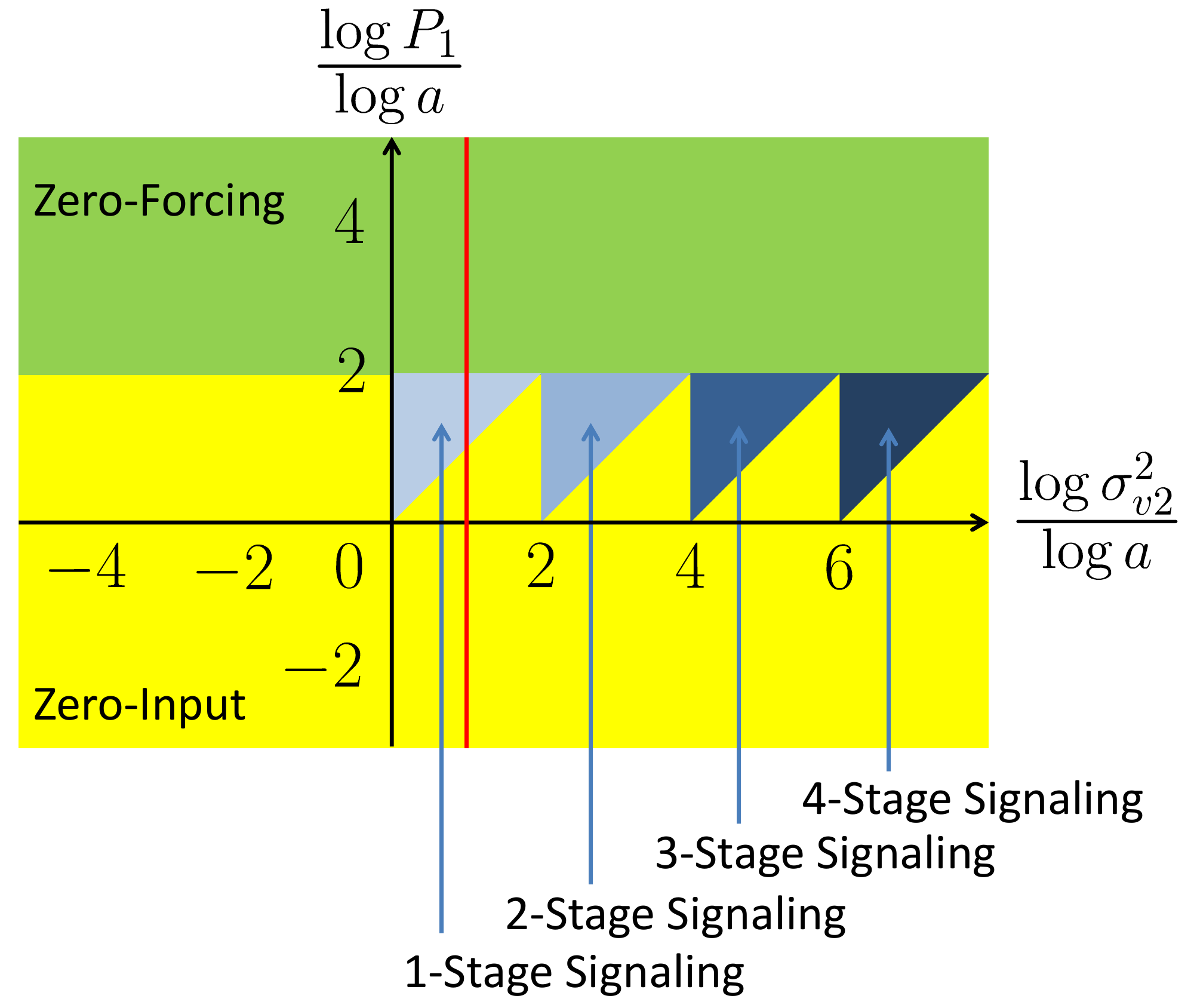}
\caption{Approximately optimal strategies for given $P_1$ and $\sigma_{v2}^2$ when $\sigma_{v1}^2=0$ and $P_2 = \infty$.}
\label{fig:region}
\end{center}
\end{figure}

\begin{figure}[htbp]
\begin{center}
\includegraphics[width=3in]{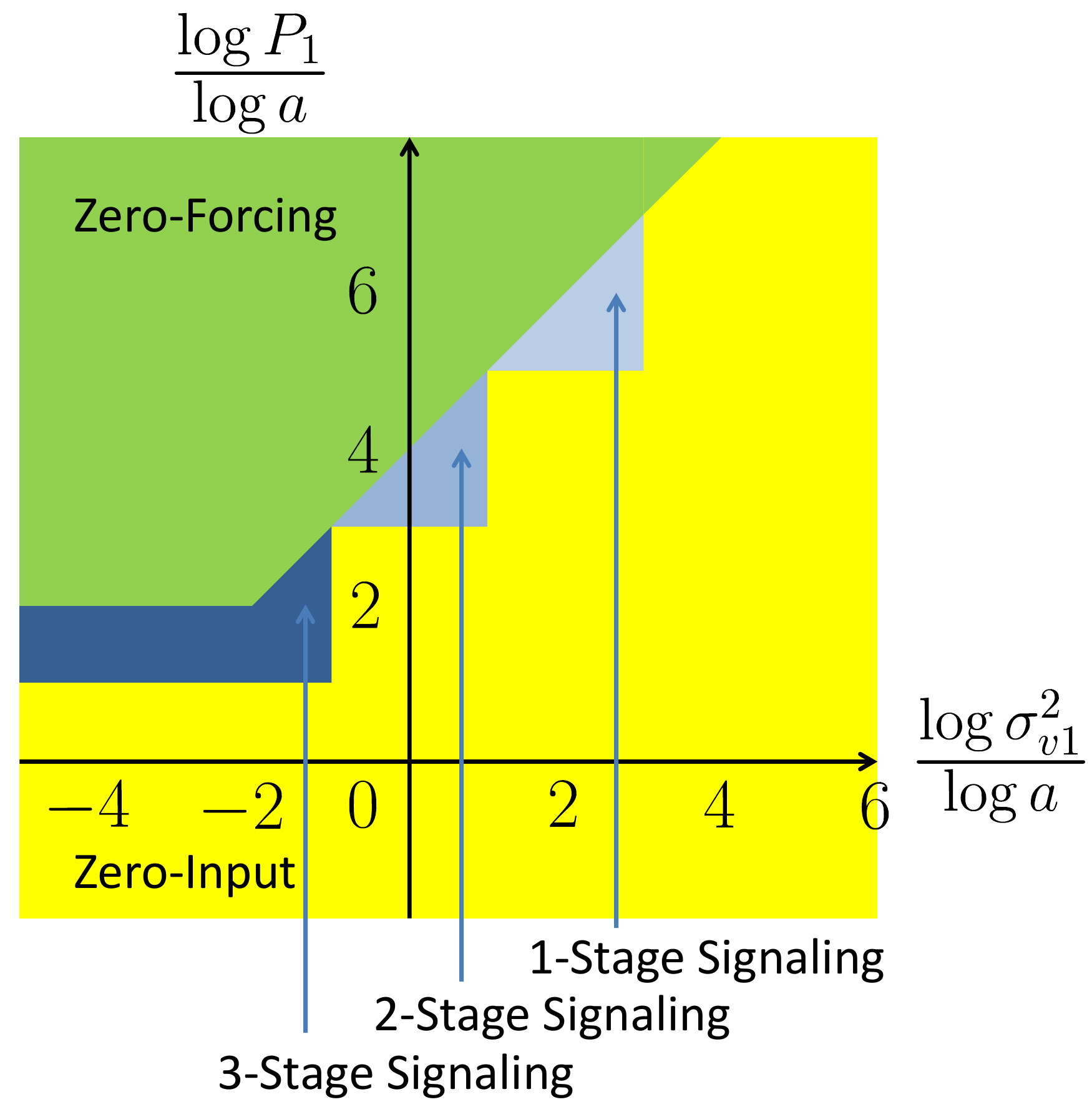}
\caption{Approximately optimal strategies for given $P_1$ and $\sigma_{v1}^2$ when $\sigma_{v2}^2 = a^5$ and $P_2 =\infty$.}
\label{fig:region2}
\end{center}
\end{figure}

\begin{figure}[htbp]
\begin{center}
\includegraphics[width=3in]{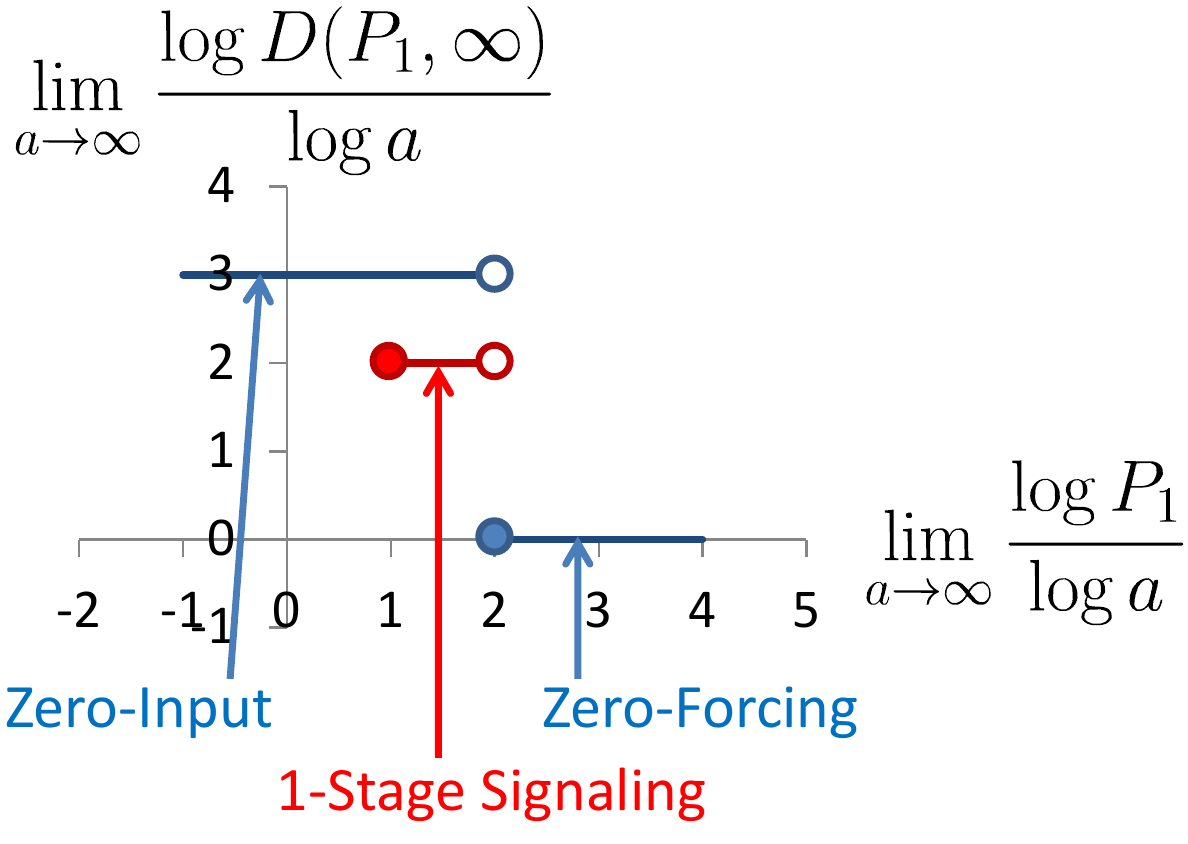}
\caption{The minimum state disturbance $D(P_1,P_2)$ when $\sigma_{v1}^2=0$, $\sigma_{v2}=a$ and $P_2=\infty$ as a function of $P_1$. The red line indicates the cost achievable by the $1$-stage signaling strategy. The blue line indicates the cost achievable by linear strategies. As we can see this performance plot corresponds to that of the red line in Figure~\ref{fig:region}.}
\label{fig:dof2}
\end{center}
\end{figure}

\begin{figure}[htbp]
\begin{center}
\includegraphics[width=3in]{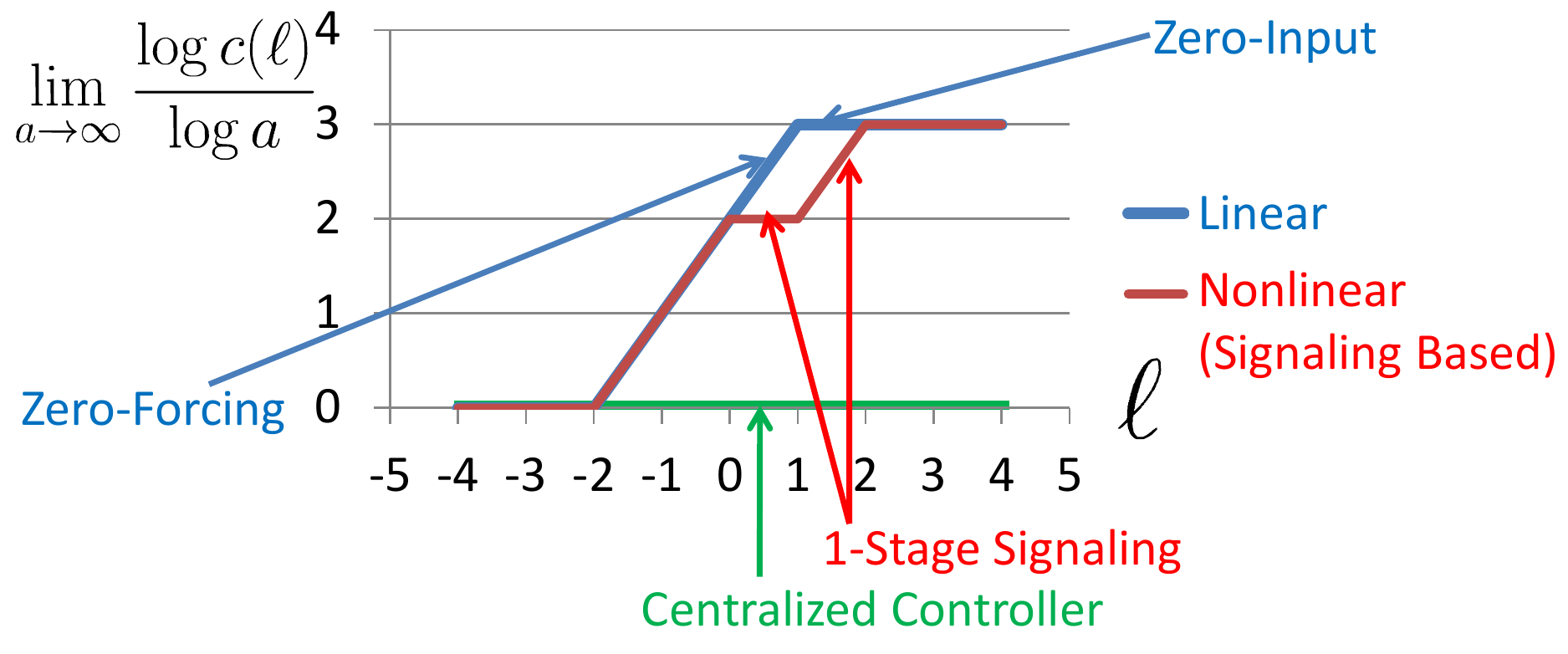}
\caption{The optimal weighted average cost for $\sigma_{v1}^2=0$, $\sigma_{v2}^2=a$, $q=1$, $r_1=a^l$, $r_2=0$. The red line indicates the optimal cost among all possible strategies. The blue line indicates the optimal cost among only linear strategies. The green line indicates the cost of the centralized controller which has both observations and can control both inputs. As $l$ varies, the optimal strategy traverses the red line of Figure~\ref{fig:region}.}
\label{fig:dof}
\end{center}
\end{figure}

\section{Proofs and Proof Ideas: Upper bound on the optimal cost}
\label{sec:upperbound}
To come up with an upper bound on $D(P_1,P_2)$, we should propose enough achievable control strategies for approximate optimality and analyze their performances.

As we discussed in Section~\ref{sec:caveat}, a $1$-stage signaling strategy ($L_{sig,1}$) for the infinite-horizon problem (shown in Figure~\ref{fig:nonlin}) and the nonlinear strategy for Witsenhausen's counterexample (shown in Figure~\ref{fig:witsen}) are essentially equivalent. The first controller implicitly communicates its observation to the second controller by forcing the lower bits to be zero. This point can be visually understood in Figure~\ref{fig:dyn} by noticing that Witsenhausen's counterexample is indeed a sub-block of the infinite-horizon problem.

However, there is a significant difference between these two problems --- the time-horizon. Witsenhausen's counterexample terminates after $2$-time steps, while the system keeps running in infinite-horizon problems. Therefore, more issues arise when we are designing controllers for infinite-horizon problems.

First, since the system keeps running in infinite horizon problems, the implicit communication also has to keep happening. In Figure~\ref{fig:dyn}, $C_1[1]$ communicates to $C_2[2]$, $C_1[2]$ communicates to $C_2[3]$, and so on. In other words, an infinite-horizon problem can be thought as a series of Witsenhausen counterexamples. Because of this interlocking of Witsenhausen's blocks, the effect of one problem can propagate to subsequent ones. To handle this interference between interlocked problems, we introduced the $R_{a^s d}(\sum_{1 \leq i \leq s} a^{i-1}u_2[n-i])$ terms in the $s$-stage signaling policy in Definition~\ref{def:sig}.

The second difference is that since we have a longer time horizon, $C_1[0]$ does not have to communicate to $C_2[1]$ of the next time step. It can communicate with longer delay to $C_2[2]$, $C_2[3]$, $\cdots$. In general, $C_{1}[0]$ can communicate to $C_2[s]$ as we can see in Figure~\ref{fig:dyn}.
In fact, the $s$-stage signaling strategy of Definition~\ref{def:sig}, $L_{sig,s}$, enables $C_{1}[0]$ to communicate with $C_2[s]$, and the infinite-horizon problem is decomposed into a series of interlocked `$s$-stage MIMO Witsenhausen's counterexamples'.

Let's take a careful look at these signaling strategies, and understand which strategy has to be used for certain parameters of Problem~\ref{prob:lqg}.
For simplicity, we first consider the extreme case when the first controller has a perfect observation and the second controller has no power constraint just like Section~\ref{sec:intui}. In other words, $\sigma_{v1}^2=0$ and $P_2=\infty$. Here, we will be making references to the binary deterministic perspective on the problem.\\

\subsubsection{When $\sigma_{v1}^2=0$ and $P_2 = \infty$}

Figure~\ref{fig:region} summarizes which strategy has to be used for a given $\sigma_2^2$ and $P_1$. 
First, we can notice that if the first controller has enough power then it does not need any help from the second controller. At each time step the disturbance $w[n]$ is added, it is observable at the next time step $n+1$ by the first controller when its power is amplified by $a^2$. Therefore, if $P_1 \geq a^2$ the first controller can remove the disturbance by itself by choosing $u_1[n]=-ay_1[n]$. We will call this a zero-forcing strategy from the first controller's point of view. On the other hand, at each time a new state disturbance $w[n]$ with variance $1$ is added. Therefore, when $P_1 \leq 1$ most of the first controller's input will be masked by the additional disturbance $w[n]$. Therefore, in this case $u_1[n]=0$ is approximately optimal, and we will call this a zero-input strategy from the first controller's point of view.

Therefore, the question is ``what should the first controller do when $P_1$ is between these two extreme values?" As we discussed before, the first controller can implicitly communicate its perfect observation to the second controller by canceling the bits which are not observable by the second controller. This idea can be implemented when the bits below the second controller's noise level are observed by the first controller at previous time steps. For example, in Figure~\ref{fig:nonlin}, $x^{1}_{-2}$ of time step $2$, the bit below the noise level of the second controller, is observed by the first controller at time step $1$, one time step before.

Then, what is the condition for the first controller to observe the disturbance one time step before in the original LQG problems? We can notice that at each time the disturbance $w[n]$ is amplified by $a$ and its variance becomes $a^2$ after one time step. Therefore, when $1 \leq \sigma_{v2}^2 \leq a^2$ the bits below the second controller's noise level are observed by the first controller at $1$ time step before.

What is the minimum power required for the first controller to cancel all the bits below the second controller's noise level $\sigma_{v2}^2$? As we can guess\footnote{We can also conjecture this from the deterministic model in Figure~\ref{fig:nonlin}. In Figure~\ref{fig:nonlin}, the first controller's input power level and the second controller's noise level is the same.}, the answer is $\sigma_{v2}^2$.
In sum, for $1$-stage signaling to be actually useful, the parameters of the LQG problems have to be $1 \leq \sigma_{v2}^2 \leq a^2$ and $\sigma_{v2}^2 \leq P_1 \leq a^2$. When $P_1 \geq a^2$, zero-forcing is approximately optimal, and when $0 \leq P_1 \leq \sigma_{v2}^2$, zero-input is approximately optimal.

In general, when $a^{2(s-1)} \leq \sigma_{v2}^2 \leq a^{2s}$ for some $s \in \mathbb{N}$, the bits below the second controller's noise level can be ``previewed" by the first controller at $s$ time steps before, and the first controller's power has to be larger than $\frac{\sigma_{v2}^2}{a^{2(s-1)}}$ to actually cancel those bits. Therefore, in this case when $P_1 \geq a^2$, zero-forcing is approximately optimal, when $\frac{\sigma_{v2}^2}{a^{2(s-1)}} \leq P_1 \leq a^2$, $s$-stage signaling is approximately optimal, and when $0 \leq P_1 \leq \frac{\sigma_{v2}^2}{a^{2(s-1)}}$, zero-input is approximately optimal.

On the other hand, when $\sigma_{v2}^2 \leq 1$, it corresponds to dividing the infinite-horizon problem to a series of Radner's problems.\footnote{In Section~\ref{sec:proof:ratio}, we will name this case as the weakly-degraded-observation case, while the remaining case is named as the strongly-degraded-observation case.} The first controller will observe the bits below the second controller's noise level without any delay, so it gets no preview. As we discussed in Section~\ref{sec:caveat}, we cannot expect a significant gain from nonlinear strategies when two controllers are acting simultaneously on essentially the same quality of observations. Therefore, in this case, a linear strategy is enough to achieve constant-ratio optimality. We will revisit this point when we are discussing lower bounds in Section~\ref{sec:prooflower}\\

\subsubsection{When $\sigma_{v1}^2>0$}
So far, we limited ourselves to $\sigma_{v1}^2=0$ and $P_2=\infty$. Let's first consider the case when $\sigma_{v1}^2 > 0$. 

If we take a careful look at the previous case of $\sigma_{v1}^2 =0$, the bits that the first controller actually uses are those between the power level $1$ and $a^{-2}$. The bits below $a^{-2}$ are useless since at the next time step, they will be masked by the new disturbance. Therefore, as long as $\sigma_{v1}^2 \leq a^{-2}$, the first controller can observe all its useful bits and the previous argument does not change.

Then, what is happening in the case when $\sigma_{v1}^2 \geq a^{-2}$? First, let's ask what is the minimum power $P_1$ for the first controller to zero-force the state. The disturbance is amplified by $a^2$ at each time step, and the bits below $\sigma_{v1}^2$ are not observable by the first controller. Therefore, by the time the first controller observes the effect of the disturbance, the state's variance becomes $a^2 \sigma_{v1}^2$. To actually cancel it at the next time step, the first controller's power has to be greater than $a^4 \sigma_{v1}^2$, i.e. $P_1 \geq a^4 \sigma_{v1}^2$.

When the input power is smaller than $a^4 \sigma_{v1}^2$, it has to use signaling strategies. So when can we use the $s$-stage signaling strategy? To use the $s$-stage signaling strategy, the first controller has to observe the bits below the second controller's noise level at least $s$ time steps before. Therefore, $\sigma_{v1}^2$ has to be less than $\frac{\sigma_{v2}^2}{a^{2s}}$. Since a longer stage signaling requires smaller power, we will use an $s$-stage signaling strategy when $\frac{\sigma_{v2}^2}{a^{2(s+1)}} < \sigma_{v1}^2 \leq \frac{\sigma_{v2}^2}{a^{2s}}$. Then, what is the minimum power to use $s$-stage signaling strategy? Since the first controller has to cancel the bits below $\frac{\sigma_{v2}^2}{a^{2s}}$ at the next time step, $P_1$ has to be greater than $\frac{a^2 \sigma_{v2}^2}{a^{2s}}$. When $P_1$ is less than this, the first controller uses the zero-input strategy.

Summarizing the conclusions so far, let $s = \lceil \frac{\ln \sigma_{v2}^2 - \ln ( \max (1, a^2 \sigma_{v1}^2))}{2 \ln a} \rceil $ so that
\begin{align}
a^{2(s-1)}\max(1,a^2 \sigma_{v1}^2) < \sigma_{v2}^2 \leq a^{2s} \max(1, a^2 \sigma_{v1}^2).
\end{align} Then 
(i) When $P_1 \geq \max(a^2, a^4 \sigma_{v1}^2)$, the zero-forcing strategy\\
(ii) When $\frac{\sigma_{v2}^2}{a^{2(s-1)}} \leq P_1 \leq \max(a^2, a^4 \sigma_{v1}^2)$, the $s$-stage signaling strategy\\
(iii) When $P_1 \leq \frac{\sigma_{v2}^2}{a^{2(s-1)}}$, the zero-input strategy\\
are approximately optimal respectively.\\

\subsubsection{When $P_2 < \infty$}
Let's consider when the second controller also has a power constraint $P_2$. When the first controller is zero-forcing the state, the second controller does not have to control and the power constraint $P_2$ does not change the result. When the first controller is either applying signaling or zero input strategy, the second controller has to stabilize the system. By the definition, $D(P_1, \infty)$ is the smallest state disturbance we can expect. Therefore, $P_2$ has to be essentially greater than $a^2 D(P_1, \infty)$ to cancel the state at the next time step and stabilize the system. In fact, this turns out to be sufficient, too.

\subsection{Generalized d.o.f. Performance}
\label{sec:generalizeddof}
Now, we have approximately optimal strategies. In this section, we will see how their performances scales as the problem parameters vary. More precisely, we will increase the various problem parameters in different scales, and see how the control cost scales as a function of the problem parameters. In spirit, this measure of the performance corresponds to the generalized d.o.f. in wireless communication~\cite{Etkin_Interference,Salman_Wireless} where the SNRs of different antennas are allowed to scale differently. The more fundamental connection with wireless communication theory will be discussed in Section~\ref{sec:wireless}.

Figure~\ref{fig:dof2} shows how the minimum state disturbance of the proposed strategies scales as $a$ goes to infinity. Precisely, in Problem~\ref{prob:power} we fix $a=a$, $\sigma_{v1}^2=0$, $\sigma_{v2}^2=a$, $P_2=\infty$, and explore how $D(P_1,P_2)$ scales in $a$ when $P_1$ scales differently in $a$. From the problem parameters, we can easily see this cost plot corresponds to the cost of the red line ($\sigma_{v2}^2=a$) in Figure~\ref{fig:region}. As we discussed before, between zero-forcing and zero-input linear strategies, the nonlinear $1$-stage signaling strategy performs better.

So far the discussion is from the power-disturbance point of view. However, the original weighted cost problem is essentially the same since the optimal strategy will have some corresponding control input powers. Let's consider the system equation \eqref{eqn:systemsimple} with $a=a, \sigma_{v1}^2=0, \sigma_{v2}^2=a$ and the average cost \eqref{eqn:part11}
\begin{align}
&c(l)=\inf_{u_1,u_2} \limsup_{N \rightarrow \infty}\frac{1}{N} \sum_{0 \leq n < N} \mathbb{E}[x^2[n]+a^l u_1^2[n]]
\nonumber
\end{align}
Figure~\ref{fig:dof} shows how the average cost scales as $a$ goes to infinity for different values of $l$. As we change $l$, the optimal solution follows the red line ($\sigma_{v2}^2=a$) in Figure~\ref{fig:region}. \\
(i) When $l$ is small ($l \leq 0$), the input cost of the first controller is inexpensive and the zero-forcing strategy is optimal up to scaling.\\
(ii) When $l$ is large ($l \geq 2$), the input cost is expensive and the zero-input strategy is approximately optimal. \\
(iii) Between these two extremes ($0 \leq 1 \leq 2$), we need nonlinear $1$-stage signaling strategy and it is approximately optimal. \\
As we can see in Figure~\ref{fig:dof}, the average cost of linear and optimal nonlinear strategy scales differently in $a$. Therefore, the performance ratio between these two diverges to infinity, which was formally stated in Proposition~\ref{prop:1}. Moreover, in Figure~\ref{fig:dof} we can also see a naive lower bound on the cost derived by allowing a centralized controller is too loose to give a constant ratio optimality. Thus, we have to improve both upper and lower bound to prove a constant ratio optimality.

It is worth to mention that figuring out this generalized d.o.f. cost is not enough to guarantee constant-ratio optimality, since the logarithmic scaling (caused by the tail of the Gaussian random variables) in $a$ does not appear in the generalized d.o.f. cost. For example, the first term shown in the lower bound given by (c) of Corollary~\ref{rat:lemma3} cannot be captured in the generalized d.o.f. cost.

\subsection{(d,w,o) Approximate-Comb-Lattice Theory}
\label{sec:dwo}
\begin{figure}[htbp]
\begin{center}
\includegraphics[width=3in]{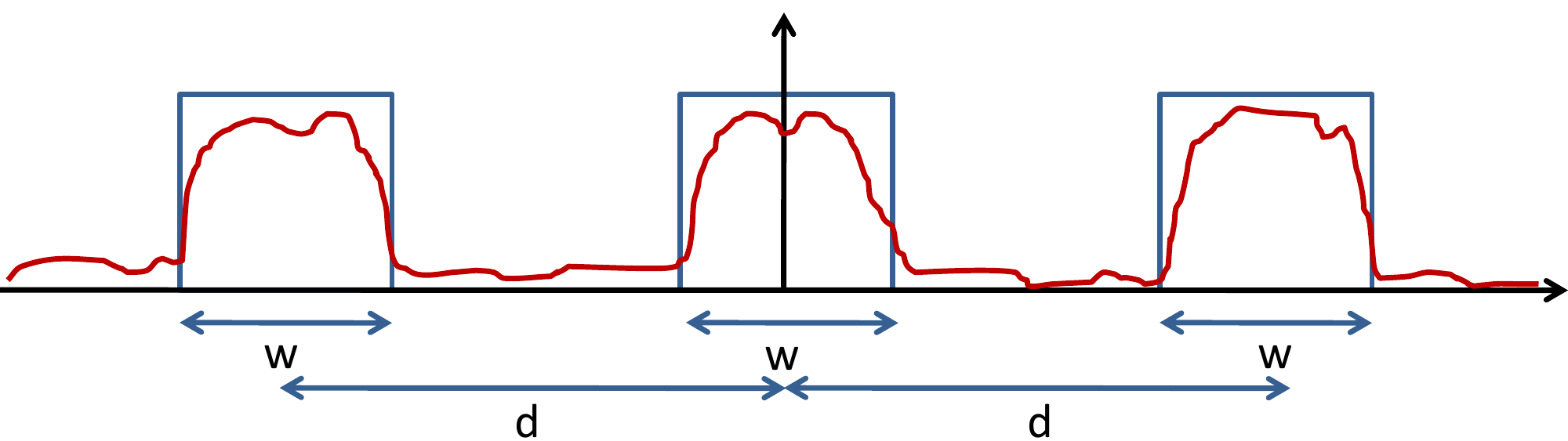}
\caption{Pictorial description of Definition~\ref{def:box}}
\label{fig:ach}
\end{center}
\end{figure}

So far, we understand the approximately optimal strategies and intuitively why they have to be used for given problem parameters. Now, we have to analyze their performances. Unlike linear strategies, nonlinear strategies make the random variables (the state, observations and inputs) non-Gaussian. Thus, the mean and variance is not enough to describe the distribution of random variables, and the exact description of the distribution requires a potentially infinite number of parameters. Therefore, we have to come up with an approximate description involving only a finite number of parameters. To this end, we propose the following definition which will turn out to be useful in analyzing quantization-based signaling strategies.
\begin{definition}
Let $X$ be a random variable, $d$ be nonzero, and $w,o$ be nonnegative reals with $|d| > w$. We say $X \leq_{df} (d,w,o)$ if
\begin{align}
\mathbb{P}\{ X \not\in \bigcup_{i \in \mathbb{Z}} [i \cdot d- \frac{w}{2},i \cdot d+ \frac{w}{2}] \} \leq o.
\end{align}
\label{def:box}
\end{definition}
Figure~\ref{fig:ach} pictorially shows this definition. When a random variable stays in one of the boxes with width $w$, the event will be considered typical. When a random variable falls outside the boxes, the event will be considered atypical and measured by outage probability $o$. Notice that after quantization, the probability mass will concentrate in a sequence of boxes. When $d=\infty$, we have only one box centered at $0$.

Let's study properties of this definition. The first lemma tells what happens when we add two random variables.
\begin{lemma}
Let $X_1$, $X_2$ and $X_3$ be arbitrary correlated random variables. If $X_1 \leq_{df} (d_1,w_1,o_1)$, $X_2 \leq_{df} (\infty,w_2,o_2)$ and $X_3 \leq_{df} (d_1,w_3,o_3)$ then
\begin{align}
&X_1 + X_2 \leq_{df} (d_1,w_1+w_2,o_1+o_2)\\
&X_1 + X_3 \leq_{df} (d_1,w_1+w_3,o_1+o_3)
\end{align}
\label{ach:lemma1}
\end{lemma}
\begin{proof}
\begin{align}
&\mathbb{P}\{ X_1 + X_2 \not\in \bigcup_{i \in \mathbb{Z}}[i \cdot d_1 - \frac{w_1 + w_2}{2}, i \cdot d_1 + \frac{w_1 + w_2}{2}] \}\\
&\leq \mathbb{P}\{ X_1 + X_2 \not\in \bigcup_{i \in \mathbb{Z}}[i \cdot d_1 - \frac{w_1 + w_2}{2}, i \cdot d_1 + \frac{w_1 + w_2}{2}] , X_2 \in [- \frac{w_2}{2},\frac{w_2}{2}]\}+  \mathbb{P}\{X_2 \not\in [- \frac{w_2}{2},\frac{w_2}{2}] \} \\
&\leq  \mathbb{P}\{ X_1 \not\in \bigcup_{i \in \mathbb{Z}}[i \cdot d_1 - \frac{w_1 }{2}, i \cdot d_1 + \frac{w_1 }{2}] \}+
o_2 \\
&\leq o_1 + o_2
\end{align}
The second part follows similarly since when we add two points from the lattice points spaced by $d$, the resulting point is also in that lattice.
\end{proof}
The second lemma tells what happens when we multiply a random variable by a constant.
\begin{lemma}
Let $X \leq_{df} (d,w,o)$ and $k>0$. Then,
\begin{align}
k X \leq_{df} (k d, k w, o).
\end{align}
\label{ach:lemma2}
\end{lemma}
\begin{proof}
\begin{align}
&\mathbb{P}\{ k X \not\in \bigcup_{i \in \mathbb{Z}}[i \cdot kd - \frac{kw}{2}, i \cdot kd + \frac{kw}{2} ] \} \\
&=\mathbb{P}\{ X \not\in \bigcup_{i \in \mathbb{Z}}[i \cdot d - \frac{w}{2}, i \cdot d + \frac{w}{2} ] \} \\
&\leq o.
\end{align}
\end{proof}

The next lemma tells the variance of remainder is only smaller than the original random variable.
\begin{lemma}
For all random variable $X$ and nonzero $d$, we have
\begin{align}
\mathbb{E}[R_d(X)^2] = \mathbb{E}[(X-Q_d(X))^2] \leq \mathbb{E}[X^2].
\end{align}
\label{ach:lemma3}
\end{lemma}
\begin{proof}
For a real $x$, let $x=q \cdot d +r $ for $q \in \mathbb{Z}$ and $r \in [-\frac{d}{2},\frac{d}{2})$. Then,
\begin{align}
x^2&=(q \cdot d  +r )^2 = q^2 d^2 + 2 qdr +r^2 \\
&=|qd| (|qd|+2 sgn(qd) \cdot r) + r^2
\end{align}
When $q=0$, $x^2 = r^2$.\\
When $q \neq 0$, since $q \in \mathbb{Z}$ we have $x^2  \geq |qd|(|d| -  2 |r|) + r^2 \geq r^2$.

Therefore,
\begin{align}
\mathbb{E}[X^2] \geq \mathbb{E}[R_d(X)^2]. 
\end{align}
\end{proof}

Since all underlying random variables of interest are Gaussian, we will relate Gaussian distributions with our parameterization.
\begin{lemma}
Let $Q(x):=\frac{1}{\sqrt{2 \pi}} \int^{\infty}_{x} \exp(-\frac{u^2}{2}) du$.
Then, $Q(x) \sim \frac{1}{\sqrt{ 2 \pi } x} \exp(-\frac{x^2}{2})$. More precisely, for $\forall x >0$
\begin{align}
\frac{1}{\sqrt{2\pi}}\left( \frac{1}{x} -\frac{1}{x^3} \right) \exp\left( - \frac{x^2}{2} \right) \leq Q(x) \leq \frac{1}{\sqrt{2\pi}x} \exp\left( -\frac{x^2}{2} \right).
\end{align}
Moreover, when $X$ is Gaussian with zero mean and variance smaller than $\sigma^2$, for all $w \geq 0$
\begin{align}
X \leq_{df} (\infty, w, 2 \cdot Q(\frac{w}{2 \sigma})).
\end{align}
\label{ach:lemma4}
\label{ach:lemma5}
\end{lemma}
\begin{proof}
For the first part, see \cite{Feller}. The second part directly follows from the definition.
\end{proof}

The next lemma bounds the MMSE error of a quantized random variable when it is corrupted by Gaussian observation noise.
\begin{lemma}
Let $X$ and $V$ be independent random variable where $X \leq_{df} (d,w,o)$ with $|d|> w$ and $V$ is a Gaussian random variable with zero-mean and variance $\sigma^2$. Then,
\begin{align}
&\mathbb{E}[(X-Q_d(X+V))^2] \\
&\leq \mathbb{E}[(X-Q_d(X))^2]+
\sum_{1 \leq i \leq \infty}
(i |d|+\frac{w}{2})^2 \cdot 2Q(\frac{(2i-1)|d|-w}{2\sigma})\\
&+
o\cdot
\left((d+\frac{d}{2})^2  + \sum_{2 \leq i \leq \infty} (i \cdot d+\frac{d}{2})^2 \cdot 2Q_d(\frac{(i-1)|d|}{\sigma}) \right).
\end{align}
\label{ach:lemma8}
\end{lemma}
\begin{proof}
For convenience, let $d > 0$. $d<0$ can be proved by replacing $d$ with $|d|$.
Denote $\mathcal{T}_{d,w}=\bigcup_{i \in \mathbb{Z}}[i \cdot d -\frac{w}{2}, i \cdot d + \frac{w}{2}]$.
\begin{align}
\mathbb{E}[(X-Q_d(X+V))^2] & = \mathbb{E}[(X-Q_d(X+V))^2| X \in \mathcal{T}_{d,w} ]
\mathbb{P}\{ X \in \mathcal{T}_{d,w} \} \\
&+ \mathbb{E}[(X-Q_d(X+V))^2|X \in \mathcal{T}_{d,w}^c]
\mathbb{P}\{ X \in \mathcal{T}_{d,w}^c \}\\
&\leq \mathbb{E}[(X-Q_d(X+V))^2| X \in \mathcal{T}_{d,w} ] \\
& + \mathbb{E}[(X-Q_d(X+V))^2|X \in \mathcal{T}_{d,w}^c] \cdot o
\end{align}
Notice that when $X \in \mathcal{T}_{d,w}$ and $|V| < \frac{d-w}{2}$, $Q_d(X)=Q_d(X+V)$. When $X \in \mathcal{T}_{d,w}$ and $|V| < d+\frac{d-w}{2}$, $Q_d(X)=Q_d(X+V) \pm d$ and so on.
Therefore,
\begin{align}
&\mathbb{E}[(X-Q_d(X+V))^2|X \in \mathcal{T}_{d,w}]=\mathbb{E}[(Q_d(X)+R_d(X)-Q_d(X+V))^2|X \in \mathcal{T}_{d,w}] \\
&\leq \mathbb{E}[(X-Q_d(X))^2] +(d+\frac{w}{2})^2 \cdot 2Q(\frac{d-w}{2 \sigma})+(2d+\frac{w}{2})^2 \cdot 2Q(\frac{3d-w}{2 \sigma}) + \cdots
\end{align}

Moreover, for all $x$ when $|V| < d$, $Q_d(x)-Q_d(x+V) = -d , 0, d$. When $|V| < 2d$, $Q_d(x)-Q_d(x+V) = -2d ,-d , 0, d, 2d$ and so on. Therefore, since $|R_d(\cdot)| \leq \frac{d}{2}$,
\begin{align}
&\mathbb{E}[(X-Q_d(X+V))^2|X \in \mathcal{T}_{d,w}^c]
=\mathbb{E}[(Q_d(X)-Q_d(X+V)+R_d(X))^2 |X \in \mathcal{T}_{d,w}^c]\\
&\leq (d+\frac{d}{2})^2  + (2d+\frac{d}{2})^2 \cdot 2Q_d(\frac{d}{\sigma}) + (3d+\frac{d}{2})^2 \cdot 2Q_d(\frac{2d}{\sigma}) + \cdots
\end{align}

Therefore,
\begin{align}
&\mathbb{E}[(X-Q_d(X+V))^2] \\
&\leq \mathbb{E}[(X-Q_d(X))^2]+(d+\frac{w}{2})^2 \cdot 2Q(\frac{d-w}{2 \sigma})+(2d+\frac{w}{2})^2 \cdot 2Q(\frac{3d-w}{2\sigma}) + \cdots \\
&+
o\cdot
((d+\frac{d}{2})^2  + (2d+\frac{d}{2})^2 \cdot 2Q_d(\frac{d}{\sigma}) + (3d+\frac{d}{2})^2 \cdot 2Q_d(\frac{2d}{\sigma}) + \cdots).
\end{align}
\end{proof}

\subsection{Analysis of Signaling Strategies}
\label{sec:dwoanalysis}
Now, we are ready to analyze the performance of the $s$-stage signaling strategy. In the $s$-stage signaling strategy, the first controller imposes a lattice structure on $x[n]$, but the second controller's action can possibly break this lattice structure. However, the second controller knows all its past inputs, so it can exploit the imposed lattice structure by compensating for its past inputs. More precisely, we will see that $x[n]-R_{a^s d}(\sum_{1 \leq i \leq s}a^{i-1}u_2[n-i])$ ---with the compensation term, $R_{a^s d}(\sum_{1 \leq i \leq s}a^{i-1}u_2[n-i])$--- has a lattice structure, and the second controller will observe this quantized state with an observation noise $v_2[n]$. In spirit, the idea and analysis in this section is similar to that in \cite{Park_constant}.

Before we state the lemma, we introduce a definition to compare multiple numbers. For $a_1, \cdots a_n, b_1, \cdots, b_n \in \mathbb{R}$, we say $(a_1, \cdots, a_n) \leq (b_1, \cdots, b_n)$ if and only if $a_1 \leq b_1$, $\cdots$, $a_n \leq b_n$.

\begin{lemma}
For a given $s \in \mathbb{N}$, let $S_{U,1}$ be the set of $(d,w_1)$ such that 
\begin{align}
&d >0, w_1 >0,\\
&|a|^s d - (|a|^{s-1}d \frac{|a|}{|a|-1} + w_1) > 0. 
\end{align}
The bound $D_{U,1}(d,w_1)$ is defined as 
\begin{align}
D_{U,1}(d,w_1)&:=
2 a^{2s}(2 (\frac{d}{2})^2 (\frac{1}{1-\frac{1}{|a|}})^2 + 2 (\frac{1}{1-\frac{1}{a^2}}) +2 a^2 \sigma_{v1}^2) \\
&+ \sum_{1 \leq i \leq \infty} 4 a^2 ( i |a|^s d + \frac{|a|^{s-1}d \frac{|a|}{|a|-1}+w_1}{2})^2 Q(\frac{(2i-1)|a|^s d - (|a|^{s-1}d \frac{|a|}{|a|-1}+w_1)}{2\sigma_{v2}}) \\
&+8 a^2 Q(\frac{w_1}{2 \sqrt{a^{2(s-1)}\frac{a^2}{a^2-1}+a^{2s}\sigma_{v1}^2}})
\sum_{1 \leq i \leq \infty} (i |a|^s d + \frac{|a|^s d}{2})^2 Q(\frac{(i-1) |a|^s d}{\sigma_{v2}})\\
& + 2(a^2 (\frac{d}{2})^2)+1. \label{eqn:lemma9:2}
\end{align}
Let $|a|>1$. Then, for all $s$ and $(d,w_1) \in S_{U,1}$, the $s$-stage signaling strategy of Definition~\ref{def:sig} can achieve the following Power-Disturbance tradeoff of Problem~\ref{prob:power}.
\begin{align}
(D(P_1,P_2),P_1,P_2) \leq ( D_{U,1}(d,w_1) , \frac{a^2 d^2}{4} , 8a^2 D_{U,1}(d,w_1) + \frac{7}{2} a^{2(s+1)}d^2 + 4 a^2 \sigma_{v2}^2)
\end{align}
\label{ach:lemma9}
\end{lemma}

\begin{proof}
For notational simplicity, we only consider $a > 1$. The proof for $a < -1$ can be obtained by replacing $a$ with $|a|$.

By the definition of $s$-stage signaling strategies,
\begin{align}
&u_1[n]=-aR_d(y_1[n])\\
&u_2[n]=-a(Q_{a^s d}(y_2[n]-R_{a^s d}(\sum_{1 \leq i \leq s} a^{i-1}u_2[n-i]))+R_{a^s d}(\sum_{1 \leq i \leq s} a^{i-1}u_2[n-i]))
\end{align}
Therefore, for all $n$ we have
\begin{align}
x[n+1]&=ax[n]+u_1[n]+u_2[n]+w[n] \\
&=ax[n]-a(Q_{a^s d}(y_2[n]-R_{a^s d}(\sum_{1 \leq i \leq s}a^{i-1}u_2[n-i]))+R_{a^s d}(\sum_{1 \leq i \leq s} a^{i-1}u_2[n-i]))+u_1[n]+w[n]\\
&=a(
\underbrace{x[n]-R_{a^s d}(\sum_{1 \leq i \leq s}a^{i-1}u_2[n-i])}_{:=X[n]}
- Q_{a^s d}(
\underbrace{y_2[n]-R_{a^s d}(\sum_{1 \leq i \leq s}a^{i-1}u_2[n-i]))}_{:=Y_2[n]} )+u_1[n]+w[n] \label{eqn:ageq43}
\end{align}
First, we will prove that for all $n \geq s$, $X[n]$ has a lattice structure. Then, $Y_2[n]$ is $X[n]+v_2[n]$, so we can use Lemma~\ref{ach:lemma8} that analyzes the estimation error of quantized random variables.

For $n \geq s$, we have
\begin{align}
&X[n]\\
&=x[n]-R_{a^s d}(\sum_{1 \leq i \leq s}a^{i-1}u_2[n-i]) \\
&= a^s x[n-s] + \sum_{1 \leq i \leq s} a^{i-1} u_1[n-i] + \sum_{1 \leq i \leq s} a^{i-1}u_2[n-i] + \sum_{1 \leq i \leq s} a^{i-1} w[n-i]-R_{a^s d}(\sum_{1 \leq i \leq s}a^{i-1}u_2[n-i])\\
&= (a^s x[n-s] + a^{s-1}u_1[n-s])+\sum_{1 \leq i \leq s} a^{i-1}u_2[n-i]-R_{a^s d}(\sum_{1 \leq i \leq s}a^{i-1}u_2[n-i])\\
&\quad +\sum_{1 \leq i \leq s-1} a^{i-1} u_1[n-i] + \sum_{1 \leq i \leq s} a^{i-1} w[n-i]\\
&= (a^s x[n-s] - a^s R_d(y_1[n-s]))+\sum_{1 \leq i \leq s} a^{i-1}u_2[n-i]-R_{a^s d}(\sum_{1 \leq i \leq s}a^{i-1}u_2[n-i])\\
&\quad +\sum_{1 \leq i \leq s-1} a^{i-1} u_1[n-i] + \sum_{1 \leq i \leq s} a^{i-1} w[n-i]\\
&= (a^s (x[n-s]+v_1[n-s]-v_1[n-s]) - a^s R_d(y_1[n-s]))+\sum_{1 \leq i \leq s} a^{i-1}u_2[n-i]-R_{a^s d}(\sum_{1 \leq i \leq s}a^{i-1}u_2[n-i])\\
&\quad +\sum_{1 \leq i \leq s-1} a^{i-1} u_1[n-i] + \sum_{1 \leq i \leq s} a^{i-1} w[n-i] \\
&=
(a^s y_1[n-s] - a^s R_d(y_1[n-s]))+\sum_{1 \leq i \leq s} a^{i-1}u_2[n-i]-R_{a^s d}(\sum_{1 \leq i \leq s}a^{i-1}u_2[n-i])\\
&\quad +\sum_{1 \leq i \leq s-1} a^{i-1} u_1[n-i] + \sum_{1 \leq i \leq s} a^{i-1} w[n-i]
-a^s v_1[n-s].
\end{align}
Here, by Lemmas~\ref{ach:lemma1} and \ref{ach:lemma2} we have
\begin{align}
&a^s y_1[n-s] - a^s R_d(y_1[n-s]) \leq_{df} (a^s d, 0, 0), \label{eqn:ageq41}\\
&\sum_{1 \leq i \leq s} a^{i-1}u_2[n-i]-R_{a^s d}(\sum_{1 \leq i \leq s}a^{i-1}u_2[n-i]) \leq_{df} (a^s d, 0, 0),\label{eqn:ageq42}\\
&\sum_{1 \leq i \leq s} a^{i-1}w[n-i] - a^s v_1[n-s] \sim \mathcal{N}(0, \sum_{1 \leq i \leq s} a^{2(i-1)}+a^{2s}\sigma_{v1}^2), \\
&\sum_{1 \leq i \leq s-1} a^{i-1}u_1[n-i] \leq_{df} ( \infty, ad + a^2 d+ \cdots + a^{s-1}d,0).
\end{align}
The first and second term have a lattice structure. The third and fourth term can be thought as bounded disturbances.

Since
\begin{align}
& \sum_{1 \leq i \leq s} a^{2(i-1)}+a^{2s}\sigma_{v1}^2 \leq a^{2(s-1)}\frac{a^2}{a^2-1}+a^{2s}\sigma_{v1}^2 \\
& ad + a^2 d+ \cdots + a^{s-1}d = a^{s-1}d(1+\frac{1}{a}+ \cdots + \frac{1}{a^{s-2}})  \leq a^{s-1}d \frac{a}{a-1}
\end{align}
by Lemma~\ref{ach:lemma1}, Lemma~\ref{ach:lemma5} we conclude for all $w_1 \geq 0$
\begin{align}
\sum_{1 \leq i \leq s-1} a^{i-1}u_1[n-i]+\sum_{1 \leq i \leq s} a^{i-1}w[n-i] \leq (\infty, a^{s-1}d \frac{a}{a-1}+w_1, 2 \cdot Q(\frac{w_1}{2 \sqrt{a^{2(s-1)}\frac{a^2}{a^2-1}+a^{2s}\sigma_{v1}^2}})) \label{eqn:ageq430}
\end{align}
Applying Lemma~\ref{ach:lemma1} to \eqref{eqn:ageq41}, \eqref{eqn:ageq42}, \eqref{eqn:ageq430} gives
\begin{align}
X[n] \leq (a^s d, a^{s-1}d \frac{a}{a-1}+w_1, 2 \cdot Q(\frac{w_1}{2 \sqrt{a^{2(s-1)}\frac{a^2}{a^2-1}+a^{2s}\sigma_{v1}^2}})).
\end{align}
Therefore, we can see that $X[n]$ $(n \geq s)$ has a lattice structure. Then, we will analyze the performance of the estimator of $X[n]$ using Lemma~\ref{ach:lemma8}.

First, for $n \geq s$ we have the following inequality.
\begin{align}
&\mathbb{E}[(X[n]-Q_{a^s d}(X[n]))^2]\\
&=\mathbb{E}[(x[n]-R_{a^s d}(\sum_{1 \leq i \leq s}a^{i-1}u_2[n-i])-Q_{a^s d}(x[n]-R_{a^s d}(\sum_{1 \leq i \leq s}a^{i-1}u_2[n-i])))^2]\\
&=\mathbb{E}[(x[n]-R_{a^s d}(\sum_{1 \leq i \leq s}a^{i-1}u_2[n-i])-Q_{a^s d}(\sum_{1 \leq i \leq s}a^{i-1}u_2[n-i])\\
&\quad -Q_{a^s d}(x[n]-R_{a^s d}(\sum_{1 \leq i \leq s}a^{i-1}u_2[n-i])-Q_{a^s d}(\sum_{1 \leq i \leq s}a^{i-1}u_2[n-i])))^2]\\
&=\mathbb{E}[(x[n]-\sum_{1 \leq i \leq s}a^{i-1}u_2[n-i]-Q_{a^s d}(x[n]-\sum_{1 \leq i \leq s}a^{i-1}u_2[n-i]))^2]\\
&= \mathbb{E}[(\sum_{1 \leq i \leq s-1} a^{i-1}u_1[n-i]+ \sum_{1 \leq i \leq s} a^{i-1}w[n-i]-a^s v_1[n-s] \\
&\quad -Q_{a^s d}(\sum_{1 \leq i \leq s-1} a^{i-1}u_1[n-i]+ \sum_{1 \leq i \leq s} a^{i-1}w[n-i]-a^s v_1[n-s])
)^2]\\
&\leq \mathbb{E}[(\sum_{1 \leq i \leq s-1} a^{i-1}u_1[n-i]+ \sum_{1 \leq i \leq s} a^{i-1}w[n-i]-a^s v_1[n-s])^2]\\
&(\because Lemma~\ref{ach:lemma3})\\
&\leq 2\mathbb{E}[(\sum_{1 \leq i \leq s-1} a^{i-1}u_1[n-i])^2]+ 2\mathbb{E}[(\sum_{1 \leq i \leq s} a^{i-1}w[n-i]-a^s v_1[n-s])^2]\\
&(\because Lemma~\ref{ach:lemmacauchy})\\
&\leq 2(\sqrt{\mathbb{E}[u_1^2[n-1]]}+\cdots+\sqrt{a^{2(s-2)}\mathbb{E}[u_1^2[n-s+1]]})^2 + 2\mathbb{E}[(\sum_{1 \leq i \leq s} a^{i-1}w[n-1]-a^s v_1[n-s])^2]\\
&(\because Lemma~\ref{ach:lemmacauchy})\\
&\leq 2(\frac{ad}{2})^2 a^{2(s-2)} (\frac{1}{1-\frac{1}{a}})^2 + 2a^{2(s-1)}(\frac{1}{1-\frac{1}{a^2}})+2 a^{2s}\sigma_{v1}^2\\
&(\because \mbox{Definition of $u_1[n]$})\\
&= a^{2(s-1)}(2 (\frac{d}{2})^2 (\frac{1}{1-\frac{1}{a}})^2 + 2 (\frac{1}{1-\frac{1}{a^2}})+2a^2 \sigma_{v1}^2 )
\end{align}
Therefore, by Lemma~\ref{ach:lemma8} we can bound the estimation error as follows.
\begin{align}
&\mathbb{E}[(X[n]-Q_{a^sd}(Y_2[n]))^2]\\
&=\mathbb{E}[(X[n]-Q_{a^sd}(X[n]+v_2[n]))^2]\\
&\leq a^{2(s-1)}(2 (\frac{d}{2})^2 (\frac{1}{1-\frac{1}{a}})^2 + 2 (\frac{1}{1-\frac{1}{a^2}})+2a^2 \sigma_{v1}^2 ) \\
&+(a^s d + \frac{a^{s-1}d \frac{a}{a-1}+w_1}{2})^2 2Q(\frac{a^s d - (a^{s-1}d \frac{a}{a-1}+w_1)}{2\sigma_{v2}}) \\
&+(2 a^s d + \frac{a^{s-1}d \frac{a}{a-1}+w_1}{2})^2 2Q(\frac{3 a^s d - (a^{s-1}d \frac{a}{a-1}+w_1)}{2\sigma_{v2}}) + \cdots \\
&+2 Q(\frac{w_1}{2 \sqrt{a^{2(s-1)}\frac{a^2}{a^2-1}+a^{2s}\sigma_{v1}^2}}) (
(a^s d + \frac{a^s d}{2})^2 + (2a^s d + \frac{a^s d}{2})^2 2 Q(\frac{a^s d}{\sigma_{v2}})+ (3a^s d + \frac{a^s d}{2})^2 2 Q(\frac{2a^s d}{\sigma_{v2}}) +\cdots
)
\end{align}
Finally, by plugging the above equation into \eqref{eqn:ageq43} we conclude for all $n \geq s$,
\begin{align}
\mathbb{E}[x^2[n+1]] &= \mathbb{E}[ (
a(X[n]-Q_{a^sd}(Y_2[n]))+u_1[n]+w[n])^2] \\
& \leq
2 \mathbb{E}[ (
a(X[n]-Q_{a^sd}(Y_2[n])))^2] + 2 \mathbb{E}[u_1^2[n]] + \mathbb{E}[w^2[n]] \\
& \leq 2 a^{2s}(2 (\frac{d}{2})^2 (\frac{1}{1-\frac{1}{a}})^2 + 2 (\frac{1}{1-\frac{1}{a^2}}) +2 a^2 \sigma_{v1}^2) \\
&+ 2 a^2 (a^s d + \frac{a^{s-1}d \frac{a}{a-1}+w_1}{2})^2 2Q(\frac{a^s d - (a^{s-1}d \frac{a}{a-1}+w_1)}{2\sigma_{v2}}) \\
&+ 2 a^2 (2 a^s d + \frac{a^{s-1}d \frac{a}{a-1}+w_1}{2})^2 2Q(\frac{3 a^s d - (a^{s-1}d \frac{a}{a-1}+w_1)}{2\sigma_{v2}}) + \cdots \\
&+4 a^2 Q(\frac{w_1}{2 \sqrt{a^{2(s-1)}\frac{a^2}{a^2-1}+a^{2s}\sigma_{v1}^2}}) (
 (a^s d + \frac{a^s d}{2})^2 + (2a^s d + \frac{a^s d}{2})^2 2 Q(\frac{a^s d}{\sigma_{v2}})+ (3a^s d + \frac{a^s d}{2})^2 2 Q(\frac{2a^s d}{\sigma_{v2}}) +\cdots
)\\
& + 2(a^2 (\frac{d}{2})^2)+1 \label{eqn:ageq431}
\end{align}
Moreover, by \eqref{eqn:ageq43} $\mathbb{E}[x^2[n]]$ is bounded for any $n < s$. Therefore, the L.H.S. of \eqref{eqn:ageq431} is an upper bound on $D(P_1,P_2)$.

For all $n$, we also have
\begin{align}
\mathbb{E}[u_1^2[n]] \leq \frac{a^2 d^2 }{4} \label{eqn:ageq432}
\end{align}
which is an upper bound on $P_1$.

Before we bound $\mathbb{E}[u_2^2[n]]$, we first notice that by Lemma~\ref{ach:lemmacauchy}
\begin{align}
&\mathbb{E}[(Q_{a^s d}(y_2[n]-R_{a^s d}(\sum_{1 \leq i \leq s}a^{i-1}u_2[n-1])))^2] \\
&=\mathbb{E}[(y_2[n]-R_{a^s d}(\sum_{1 \leq i \leq s}a^{i-1}u_2[n-1]-R_{a^s d}(y_2[n]-R_{a^s d}(\sum_{1 \leq i \leq s}a^{i-1}u_2[n-1])))^2] \\
&\leq 2\mathbb{E}[(y_2[n]-R_{a^s d}(\sum_{1 \leq i \leq s}a^{i-1}u_2[n-1]))^2] + 2(\frac{a^s d}{2})^2 \\
&= 2\mathbb{E}[(x[n]-R_{a^s d}(\sum_{1 \leq i \leq s}a^{i-1}u_2[n-1]))^2]+2 \sigma_{v2}^2 + 2(\frac{a^s d}{2})^2 \\
&\leq 4 \mathbb{E}[x^2[n]]+4(\frac{a^s d}{2})^2  +2 \sigma_{v2}^2+ 2(\frac{a^s d}{2})^2
\end{align}
Therefore, for all $n$
\begin{align}
\mathbb{E}[u_2^2[n]] &= a^2 \mathbb{E}[(Q_{a^s d}(y_2[n]-R_{a^s d}(\sum_{1 \leq i \leq s}a^{i-1}u_2[n-1]))+R_{a^s d}(\sum_{1 \leq i \leq s}a^{i-1}u_2[n-i]))^2] \\
&\leq a^2(2 \mathbb{E}[(Q_{a^s d}(y_2[n]-R_{a^s d}(\sum_{1 \leq i \leq s}a^{i-1}u_2[n-1])))^2] + 2 (\frac{a^s d}{2})^2) \\
&\leq a^2( 8 \mathbb{E}[x^2[n]]+8(\frac{a^s d}{2})^2  +4 \sigma_{v2}^2+ 4(\frac{a^s d}{2})^2 + 2(\frac{a^s d}{2})^2 ) \\
&\leq 8a^2 \mathbb{E}[x^2[n]] + \frac{7}{2} a^{2(s+1)}d^2 + 4 a^2 \sigma_{v2}^2 \label{eqn:ageq433}
\end{align}
which gives an upper bound on $P_2$.
Therefore, by \eqref{eqn:ageq431}, \eqref{eqn:ageq432}, \eqref{eqn:ageq433} the lemma is proved.
\end{proof}

\section{Proofs and Proof Ideas: Lower bound on the optimal cost}
\label{sec:prooflower}
\begin{figure}[htbp]
\begin{center}
\includegraphics[width=2in]{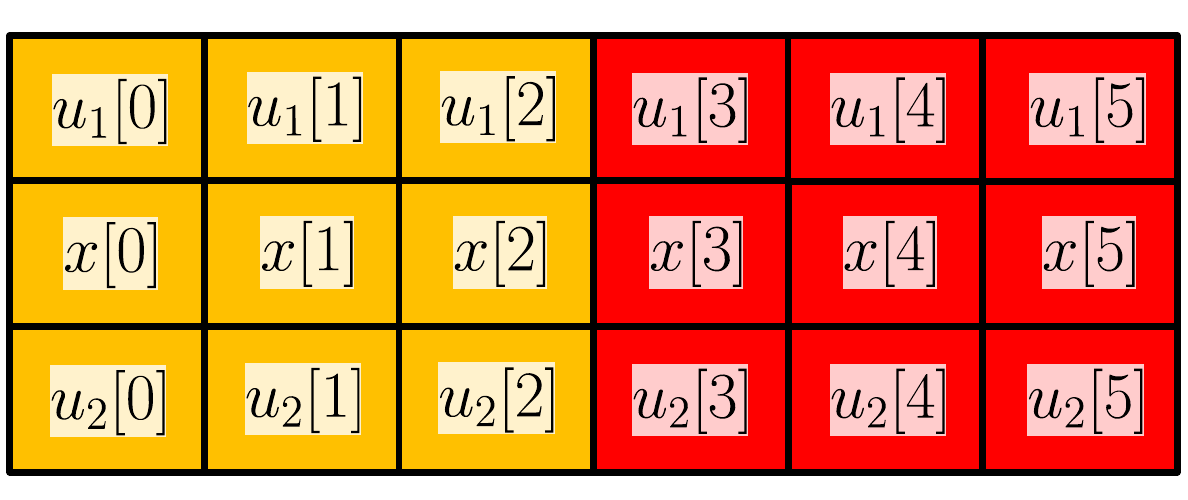}
\caption{Naive truncation idea to divide an infinite-horizon problem to finite-horizon sub-problems. This idea fails to give a constant-ratio lower bound.}
\label{fig:geometric}
\end{center}
\end{figure}

\begin{figure}[htbp]
\begin{center}
\includegraphics[width=2in]{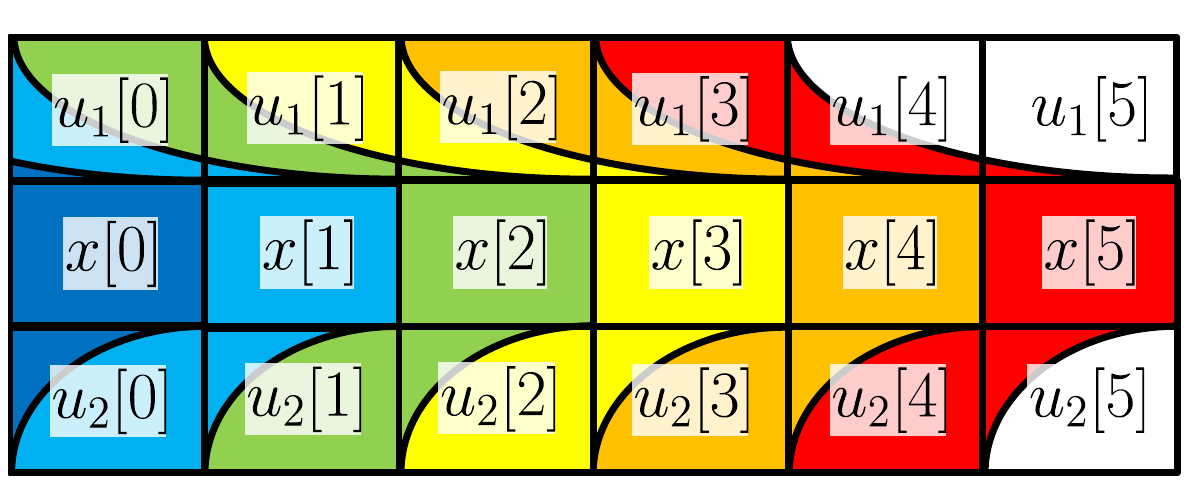}
\caption{Geometric Slicing idea to divide an infinite-horizon problem to finite-horizon sub-problems. This idea successfully gives a lower bound tight to within a constant ratio.}
\label{fig:geometric2}
\end{center}
\end{figure}

In this section, we will study the lower bound on the optimal cost and understand why it is impossible to outperform the proposed strategies by an arbitrary factor. In Section~\ref{sec:upperbound}, we discussed the relationship between the infinite-horizon problem of this paper and finite-horizon problems. The first idea for the lower bound is to make this idea formal, i.e. dividing the infinite-horizon problem into a sequence of finite-horizon problems.

\subsection{Geometric Slicing of Infinite-Horizon Problems} 
\label{sec:lower:geo}
Let's say we want to divide the infinite-horizon problem into sub-problems with time-horizon $3$. 
A naive way of dividing the problem is truncation, which is pictorially described in Figure~\ref{fig:geometric}.
The total cost $\sum_{i=1}^N q\mathbb{E}[x^2[n]]+ r_1 \mathbb{E}[u_1^2[n]] + r_2 \mathbb{E}[u_2^2[n]]$ can be divided into $3$-time-horizon problems. The first problem is minimizing $q\mathbb{E}[x^2[0]+x^2[1]+x^2[2]]+r_1 \mathbb{E}[u_1^2[0]+u_1^2[1]+u_1^2[2]]+r_2 \mathbb{E}[u_2^2[0]+u_2^2[1]+u_2^2[2]]$. The second problem is minimizing $q\mathbb{E}[x^2[3]+x^2[4]+x^2[5]]+r_1 \mathbb{E}[u_1^2[3]+u_1^2[4]+u_1^2[5]]+r_2 \mathbb{E}[u_2^2[3]+u_2^2[4]+u_2^2[5]]$, and so on. However, in this approach we can find only $\frac{N}{3}$ sub-problems out of $N$ times, which turns out not to be enough to prove constant-ratio optimality.

The main reason why the truncation idea gives too loose a bound is that in order to decouple the sub-problems from each other, we have to start each sub-problem with initial state $0$ because that is the best possible initial state. But then, we have to wait long enough until the state disturbance $w[n]$ is amplified enough. So, in each sub-problem, the state cost at the final time step is the only one that is large enough. We end up penalizing the state only for $\frac{N}{3}$-time steps, while the actual cost penalizes the state for $N$-time steps. In general, if we truncate the problem to $s$-time-horizon problems, the resulting bound will be loose by a factor of $s$. In fact, to find a matching lower bound for the $s$-stage signaling strategy, we have to divide the infinite-horizon problem into $s$-time-horizon problems. Therefore, the lower bounds based on the truncation idea will be too loose as $s$ goes to infinity.

The next idea of `geometric slicing' solves this by introducing interlocking sub-problems and penalizing the state at every time step.
Figure~\ref{fig:geometric2} shows the idea pictorially. For example, we can slice the problem to $3$-time horizon problems as follows. The first problem is minimizing $q\mathbb{E}[x^2[2]]+r_1 \mathbb{E}[\frac{1}{2}u_1^2[0]+\frac{1}{4}u_1^2[1]+\frac{1}{8}u_1^2[2]] + r_2 \mathbb{E}[\frac{1}{2}u_2^2[1]+\frac{1}{4}u_2^2[2]]$. The second problem is minimizing $q\mathbb{E}[x^2[3]]+r_1 \mathbb{E}[\frac{1}{2}u_1^2[1]+\frac{1}{4}u_1^2[2]+\frac{1}{8}u_1^2[3]] + r_2 \mathbb{E}[\frac{1}{2}u_2^2[2]+\frac{1}{4}u_2^2[3]]$, and so on. Here, notice that $u_1^2[1]$ shows up in both problems but it does not cause any difficulty since the weights form a geometric sequence and the sum is less than $1$. Therefore, we are slicing the problem using geometric sequences, and that is where the name of the idea come from. In this way, we can extract $N$ sub-problems out of an $N$-time-horizon problem. The sub-problems can be formally written as follows.

\begin{problem}[Geometrically-Sliced Finite-horizon LQG problem for Problem~\ref{prob:lqg}]
Let the system equations, the problem parameters, the underlying random variables, and the restrictions on the controllers be given exactly the same as Problem~\ref{prob:lqg}. However, now the control objective is for given $0 < \alpha < 1$, $k,k_1,k_2 \in \mathbb{N}$ $(k_1 \leq k, k_2 \leq k)$, minimizing the finite-horizon cost
\begin{align}
\inf_{u_1,u_2} q \mathbb{E}[x^2[k]] + r_1(1-\alpha)(\sum_{k_1 \leq  i \leq k-1} \alpha^{i-k_1}\mathbb{E}[u_1^2[i]]) + r_2(1-\alpha)(\sum_{k_2 \leq  i \leq k-1} \alpha^{i-k_2}\mathbb{E}[u_2^2[i]]).
\end{align}
\label{prob:slice}
\end{problem}
Even if the system can run for infinite time, the cost terminates after the time step $k$. Therefore, this problem is effectively a finite-horizon problem. 
The next lemma shows the cost of this finite-horizon problem is a lower bound to the original infinite-horizon cost of Problem~\ref{prob:lqg}.

\begin{lemma}
Let the system equations, the problem parameters, the underlying random variables, and the restrictions on the controllers be given as in Problem~\ref{prob:lqg}. When $\sigma_0^2=0$, for all $0<\alpha<1$, $k,k_1,k_2 \in \mathbb{N}$ $(k_1 \leq k, k_2  \leq k)$, the infinite-horizon cost of Problem~\ref{prob:lqg} is lower bounded by the finite-horizon cost of Problem~\ref{prob:slice}, i.e.
\begin{align}
&\inf_{u_1,u_2} \limsup_{N \rightarrow \infty} \frac{1}{N}
\sum_{0 \leq n \leq N-1} (q \mathbb{E}[x^2[n]]+r_1\mathbb{E}[u_1^2[n]]+r_2\mathbb{E}[u_2^2[n]])\\
&\geq \inf_{u_1,u_2} q \mathbb{E}[x^2[k]] + r_1(1-\alpha)(\sum_{k_1 \leq  i \leq k-1} \alpha^{i-k_1}\mathbb{E}[u_1^2[i]]) + r_2(1-\alpha)(\sum_{k_2 \leq  i \leq k-1} \alpha^{i-k_2}\mathbb{E}[u_2^2[i]]). \label{eqn:geo:1}
\end{align}
Here, when $k_1=k$ or $k_2=k$ the second or third term in the lower bound vanishes.

Furthermore, both costs are increasing functions of $\sigma_0^2$ and when $\sigma_0^2=0$, $u_1[0]=0$ and $u_2[0]=0$ are optimal for both.
\label{lem:geo}
\end{lemma}
\begin{proof}
Let's first prove that for all finite-horizon and infinite horizon problem, the average cost is an increasing function in $\sigma_0^2$.
\begin{proposition}
Let $x'[0]$ and $x''[0]$ be independent random variables, and $x'[0]$ has zero mean. Consider two systems where the system equations are given by Problem~\ref{prob:lqg}. However, the initial state of the first system is $x'[0]+x''[0]$ while the initial state of the second system is $x'[n]$. Except for the initial states, both systems have the same underlying random variables $w[n]$, $v_1[n]$, $v_2[n]$ as those in Problem~\ref{prob:lqg}. We denote the variables of the first system as $x[n]$, $u_i[n]$, $y_i[n]$, and those of the second system as $\bar{x}[n]$, $\bar{u}_i[n]$, $\bar{y}_i[n]$. Then, the following inequality is true.
\begin{align}
&\inf_{u_1, u_2} \frac{1}{N} \sum_{0 \leq n \leq N-1} (q \mathbb{E}[x^2[n]] + r_1 \mathbb{E}[u_1^2[n]] + r_2 \mathbb{E}[u_2^2[n]])\geq \inf_{\bar{u}_1, \bar{u}_2}
\frac{1}{N} \sum_{0 \leq n \leq N-1} 
(q \mathbb{E}[\bar{x}^2[n]]+ r_1 \mathbb{E}[\bar{u}_1^2[n]]+ r_2 \mathbb{E}[\bar{u}_2^2[n]]).
\end{align}
\label{prop:genie}
\end{proposition}
\begin{proof}
Since both systems are coupled with each other except the initial state, we will reduce the first system to the second system by giving $x''[0]$ as side-information. Define $L_g$ as the set of strategies for the first system which depend on its own observations and $x''[0]$, i.e. $L_g:= \{(u_1[n], u_2[n]) : u_1[n]=f_{1,n}(y_1[0],\cdots,y_1[n], x''[0]), u_2[n]=f_{2,n}(y_2[0],\cdots,y_2[n], x''[0])\}$.  Likewise, define $L_g'$ as the set of strategies for the second system which depend on its own observations and $x''[0]$, i.e. $L_g' := \{(\bar{u}_1[n], \bar{u}_2[n]) : \bar{u}_1[n]=f_{1,n}'(\bar{y}_1[0], \cdots ,  \bar{y}_1[n], x''[0]), f_{2,n}'(\bar{y}_1[0], \cdots ,  \bar{y}_1[n], x''[0]) \}$.

Further, define $u_i'[n]:=u_i[n]-\mathbb{E}[u_i[n]|x''[0]]$ and $u_i''[n]:=\mathbb{E}[u_i[n]|x''[0]]$. Then, we can lower bound the average cost as follows.
\begin{align}
&\inf_{u_1, u_2} \frac{1}{N} \sum_{0 \leq n \leq N-1} (q \mathbb{E}[x^2[n]] + r_1 \mathbb{E}[u_1^2[n]] + r_2 \mathbb{E}[u_2^2[n]])\\
&
\overset{(A)}{\geq} \inf_{u_1, u_2 \in L_{g}} \frac{1}{N} \sum_{0 \leq n \leq N-1} (q \mathbb{E}[x^2[n]] + r_1 \mathbb{E}[u_1^2[n]] + r_2 \mathbb{E}[u_2^2[n]])\\
&
\overset{(B)}{=} \inf_{u_1, u_2 \in L_{g}} \frac{1}{N} \sum_{0 \leq n \leq N-1} 
(q \mathbb{E}[(a^n x[0]+ a^{n-1}w[0] + \cdots + w[n-1]\\
&+a^{n-1}u_1[0]+ \cdots + u_1[n-1] + a^{n-1}u_2[0]+ \cdots + u_2[n-1])^2] \\
&+ r_1 \mathbb{E}[u_1^2[n]] + r_2 \mathbb{E}[u_2^2[n]])\\
&
\overset{(C)}{=} \inf_{u_1, u_2 \in L_{g}} \frac{1}{N} \sum_{0 \leq n \leq N-1} 
q \mathbb{E}[(a^n x'[0]+ a^{n-1}w[0] + \cdots + w[n-1]\\
&+a^{n-1}u_1'[0]+ \cdots + u_1'[n-1] + a^{n-1}u_2'[0]+ \cdots + u_2'[n-1]\\
&+ a^n x''[0] +a^{n-1}u_1''[0]+ \cdots + u_1''[n-1] + a^{n-1}u_2''[0]+ \cdots + u_2''[n-1])^2]\\
&+ r_1 \mathbb{E}[(u_1'[n]+u_1''[n])^2] + r_2 \mathbb{E}[(u_2'[n]+u_2''[n])^2]\\
&
\overset{(D)}{=} \inf_{u_1, u_2 \in L_{g}} \frac{1}{N} \sum_{0 \leq n \leq N-1} 
q \mathbb{E}[(a^n x'[0]+ a^{n-1}w[0] + \cdots + w[n-1]\\
&+a^{n-1}u_1'[0]+ \cdots + u_1'[n-1] + a^{n-1}u_2'[0]+ \cdots + u_2'[n-1])^2]\\
&+ q\mathbb{E}[(a^n x''[0] +a^{n-1}u_1''[0]+ \cdots + u_1''[n-1] + a^{n-1}u_2''[0]+ \cdots + u_2''[n-1])^2]\\
&+ r_1 \mathbb{E}[u_1'[n]^2]+r_1\mathbb{E}[u_1''[n]^2] + r_2 \mathbb{E}[u_2'[n]^2]+r_2\mathbb{E}[u_2''[n]^2]\\
&\geq \inf_{u_1, u_2 \in L_{g}} \frac{1}{N} \sum_{0 \leq n \leq N-1} 
q \mathbb{E}[(a^n x'[0]+ a^{n-1}w[0] + \cdots + w[n-1]\\
&+a^{n-1}u_1'[0]+ \cdots + u_1'[n-1] + a^{n-1}u_2'[0]+ \cdots + u_2'[n-1])^2]\\
&+ r_1 \mathbb{E}[u_1'[n]^2]+ r_2 \mathbb{E}[u_2'[n]^2] \\
&\overset{(E)}{\geq} \inf_{\bar{u}_1, \bar{u}_2 \in L_g'}
\frac{1}{N} \sum_{0 \leq n \leq N-1} 
(q \mathbb{E}[\bar{x}^2[n]]+ r_1 \mathbb{E}[\bar{u}_1^2[n]]+ r_2 \mathbb{E}[\bar{u}_2^2[n]])\\
&= \inf_{\bar{u}_1, \bar{u}_2 \in L_g'}
\frac{1}{N}
\mathbb{E}[
 \sum_{0 \leq n \leq N-1} 
(q \mathbb{E}[\bar{x}^2[n]]+ r_1 \mathbb{E}[\bar{u}_1^2[n]]+ r_2 \mathbb{E}[\bar{u}_2^2[n]])
| x''[0]]\\
&\geq \inf_{x'} 
\inf_{\bar{u}_1, \bar{u}_2 \in L_g'}
\frac{1}{N}
\mathbb{E}[
 \sum_{0 \leq n \leq N-1} 
(q \mathbb{E}[\bar{x}^2[n]]+ r_1 \mathbb{E}[\bar{u}_1^2[n]]+ r_2 \mathbb{E}[\bar{u}_2^2[n]])
| x''[0]=x']\\
&\overset{(F)}{=} \inf_{\bar{u}_1, \bar{u}_2 \in L}
\frac{1}{N} \sum_{0 \leq n \leq N-1} 
(q \mathbb{E}[\bar{x}^2[n]]+ r_1 \mathbb{E}[\bar{u}_1^2[n]]+ r_2 \mathbb{E}[\bar{u}_2^2[n]])
\end{align}
(A): $L \subseteq L_g$.\\
(B): By the system dynamics of \eqref{eqn:systemsimple}.\\
(C): Definitions of $x'[0]$, $x''[0]$, $u_i'[n]$, $u_i''[n]$.\\
(D): Since $x'[0], w[n]$ are zero mean and independent from $x''[0]$, they are orthogonal to $x''[0]$, $u_1''[n]$, $u_2''[n]$. Moreover, by definition, $u_1'[n]$, $u_2'[n]$ are orthogonal to $x''[0]$, $u_1''[n]$, $u_2''[n]$.\\
(E): To justify this, we will show that for all $n$, $y_i[0], \cdots y_i[n], x''[0]$ are functions of $\bar{y}_i[0], \cdots \bar{y}_i[n], x''[0]$. Therefore, there exists $\bar{u}_1[n], \bar{u}_2[n]$ such that $\bar{u}_1[n]=u_1'[n]$, $\bar{u}_2[n]=u_2'[n]$. 

First, when $n=1$, the claim is obvious since $y_i[0]=\bar{y}_i[0]+x''[0]$. Thus, $y_i[0], x''[0]$ are functions of $\bar{y}_i[0], x''[0]$. Moreover, since $u_i'[n]$ are functions of $y_i[0]$ and $x''[0]$, we can find $\bar{u}_i[n]$ such that $\bar{u}_i[n] = u_i'[n]$.

Let's say the claim holds until $n-1$. Then, we have 
\begin{align}
y_i[n] =& a^n x'[0] + a^{n-1} w[0]+ \cdots + w[n-1] \\
&+ a^{n-1} u_1'[0] + \cdots + u_1'[n-1] \\
&+ a^{n-1} u_2'[0] + \cdots + u_2'[n-1] + v_i[n]+ g(x''[0]) \\
=& a^n x'[0] + a^{n-1} w[0]+ \cdots + w[n-1] \\
&+ a^{n-1} \bar{u}_1[0] + \cdots + \bar{u}_1[n-1] \\
&+ a^{n-1} \bar{u}_2[0] + \cdots + \bar{u}_2[n-1] + v_i[n]+ g(x''[0]) \\
=& \bar{y}_i[n] +g(x''[0]) 
\end{align}
where $g(x''[0]):=a^n x''[0]+a^{n-1} \mathbb{E}[u_1[0]|x''[0]] + \cdots + \mathbb{E}[u_1[n-1]|x''[0]]+a^{n-1} \mathbb{E}[u_2[0]|x''[0]] + \cdots + \mathbb{E}[u_2[n-1]|x''[0]]$, and
the send equality comes from the induction hypothesis. Therefore, $y_i[n]$ is a function of $\bar{y}_i[n], x''[0]$, and we can find $\bar{u}_i[n]$ such that $\bar{u}_i[n]=u_i'[n]$. This proves the claim by induction.\\

(F): Since in $L_g'$ the strategies can depend on $x''[0]$.

Therefore, the proposition is true.
\end{proof}

Here, we can notice that the proof holds for all quadratic costs. Therefore, by setting $x'[0] \sim \mathcal{N}(0,\sigma_0'^2)$ and $x''[0] \sim \mathcal{N}(0,\sigma_0''^2)$, we can prove the costs in \eqref{eqn:geo:1} are increasing functions on $\sigma_{0}^2$. We can also easily justify that when $x_0[0]=0$, $u_1[0]=u_2[0]=0$ is the optimal input by symmetry.

Then, let's prove the inequality of \eqref{eqn:geo:1}.
\begin{align}
&\inf_{u_1,u_2} \limsup_{N \rightarrow \infty} \frac{1}{N}
\sum_{0 \leq n \leq N-1} (q \mathbb{E}[x^2[n]]+r_1\mathbb{E}[u_1^2[n]]+r_2\mathbb{E}[u_2^2[n]])\\
&\overset{(a)}{\geq} \limsup_{N \rightarrow \infty} \inf_{u_1,u_2}
\frac{1}{N}
\sum_{0 \leq n \leq N-1} (q \mathbb{E}[x^2[n]]+r_1\mathbb{E}[u_1^2[n]]+r_2\mathbb{E}[u_2^2[n]])\\
&\overset{(b)}{\geq} \limsup_{N \rightarrow \infty}
\inf_{u_1, u_2} \frac{1}{N}(
q \mathbb{E}[x^2[k]] + r_1(1-\alpha)(\sum_{k_1 \leq  i \leq k-1} \alpha^{i-k_1}\mathbb{E}[u_1^2[i]]) + r_2(1-\alpha)(\sum_{k_2 \leq  i \leq k-1} \alpha^{i-k_2}\mathbb{E}[u_2^2[i]])\\
&\quad+ q \mathbb{E}[x^2[k+1]] + r_1(1-\alpha)(\sum_{k_1 \leq  i \leq k-1} \alpha^{i-k_1}\mathbb{E}[u_1^2[i+1]]) + r_2(1-\alpha)(\sum_{k_2 \leq  i \leq k-1} \alpha^{i-k_2}\mathbb{E}[u_2^2[i+1]])+ \cdots \\
&\quad+ q \mathbb{E}[x^2[N-1]] + r_1(1-\alpha)(\sum_{k_1 \leq  i \leq k-1} \alpha^{i-k_1}\mathbb{E}[u_1^2[i+N-k-1]]) + r_2(1-\alpha)(\sum_{k_2 \leq  i \leq k-1} \alpha^{i-k_2}\mathbb{E}[u_2^2[i+N-k-1]]) \\
&\overset{(c)}{\geq} \limsup_{N \rightarrow \infty}
 \frac{1}{N}( \inf_{u_1, u_2}
q \mathbb{E}[x^2[k]] + r_1(1-\alpha)(\sum_{k_1 \leq  i \leq k-1} \alpha^{i-k_1}\mathbb{E}[u_1^2[i]]) + r_2(1-\alpha)(\sum_{k_2 \leq  i \leq k-1} \alpha^{i-k_2}\mathbb{E}[u_2^2[i]])\\
&\quad+ \inf_{u_1, u_2}q \mathbb{E}[x^2[k+1]] + r_1(1-\alpha)(\sum_{k_1 \leq  i \leq k-1} \alpha^{i-k_1}\mathbb{E}[u_1^2[i+1]]) + r_2(1-\alpha)(\sum_{k_2 \leq  i \leq k-1} \alpha^{i-k_2}\mathbb{E}[u_2^2[i+1]])+ \cdots \\
&\quad+ \inf_{u_1, u_2} q \mathbb{E}[x^2[N-1]] + r_1(1-\alpha)(\sum_{k_1 \leq  i \leq k-1} \alpha^{i-k_1}\mathbb{E}[u_1^2[i+N-k-1]]) + r_2(1-\alpha)(\sum_{k_2 \leq  i \leq k-1} \alpha^{i-k_2}\mathbb{E}[u_2^2[i+N-k-1]]) \label{eqn:slicing:1}\\
&\overset{(d)}{\geq} \limsup_{N \rightarrow \infty}
\frac{N-k}{N}(\inf_{u_1, u_2}
q \mathbb{E}[x^2[k]] + r_1(1-\alpha)(\sum_{k_1 \leq  i \leq k-1} \alpha^{i-k_1}\mathbb{E}[u_1^2[i]]) + r_2(1-\alpha)(\sum_{k_2 \leq  i \leq k-1} \alpha^{i-k_2}\mathbb{E}[u_2^2[i]])\\
&\overset{(e)}{=} \inf_{u_1, u_2}q \mathbb{E}[x^2[k]] + r_1(1-\alpha)(\sum_{k_1 \leq  i \leq k-1} \alpha^{i-k_1}\mathbb{E}[u_1^2[i]]) + r_2(1-\alpha)(\sum_{k_2 \leq  i \leq k-1} \alpha^{i-k_2}\mathbb{E}[u_2^2[i]]).
\end{align}
(a): $\inf \sup \geq \sup \inf$. \\
(b): We can easily check that the sum of the weight for each input cost, $\mathbb{E}[u_1^2[n]]$ or $\mathbb{E}[u_2^2[n]]$ is less than $(1-\alpha)(1+\alpha + \alpha^2 + \cdots)$ which is $1$.\\
(c): $\inf_{x} f(x) + g(x) \geq \inf_{x} f(x) + \inf_{x'} g(x')$.\\
(d): The second minimization problem in \eqref{eqn:slicing:1} can be thought as a one-time-step shift of the first minimization problem, i.e. $x[1]$ of the second problem corresponds to the initial state $x[0]=0$ of the first problem. Therefore, by putting $x'[0]=0$ and $x''[0]=x[1]$ in Proposition~\ref{prop:genie}, the first problem's cost is smaller than the second problem's cost. Likewise, we can prove that the first problem's cost is a lower bound for all other problems' cost.\\
(e): $\limsup_{N \rightarrow \infty} \frac{N-k}{N} = 1$.
\end{proof}

Conceptually, this idea of geometric slicing can be thought of as an interesting variant on how discounted dynamic programming~\cite{Bertsekas} is used to study average-cost dynamic programming. 

\begin{figure}[htbp]
\begin{center}
\includegraphics[width=4.6in]{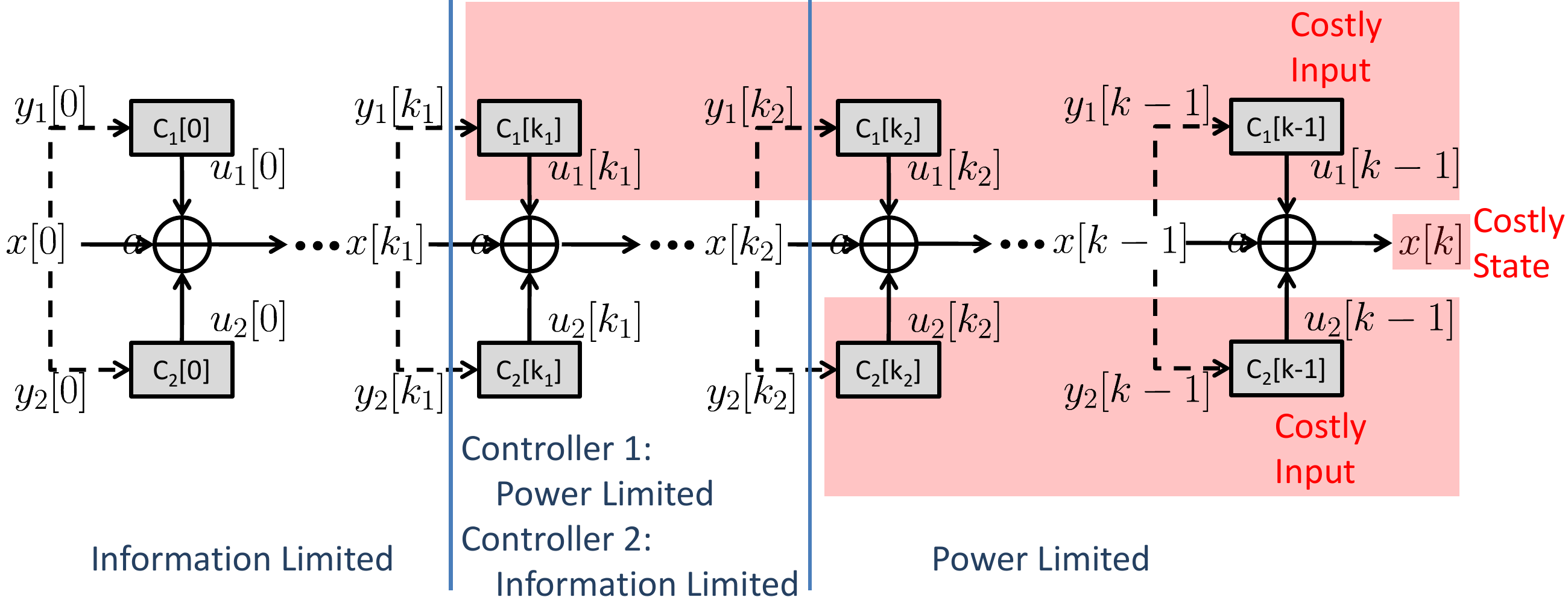}
\caption{The general finite-horizon problem structure which can give a lower bound for $s$-stage signaling strategies. The problem consists of three time intervals. In the first time interval, both controllers are information-limited. In the second time interval, the first contoller is power-limited and the second controller is information-limited. In the third time interval, both controllers are power-limited.}
\label{fig:division}
\end{center}
\end{figure}

\subsection{Finite-Horizon LQG Problems: Three-Stage Division}
\label{sec:lower:three}
Now, we can divide the infinite-horizon problem to finite-horizon problems. Figure~\ref{fig:division} shows the finite-horizon problem that gives a lower bound approximately matching with $s$-stage signaling strategies. As we discussed in Figure~\ref{fig:proofflow2}, the resulting problem is not stationary and to tackle this issue we will divide the time-horizon into three intervals: (1) information-limited interval, (2) MIMO Witsenhausen's interval, (3) power-limited interval.

Let's first state the power-distortion tradeoff version of the finite-horizon problem of Problem~\ref{prob:slice}.
\begin{problem}[Finite-Horizon LQG problem with discounted power constraints] Let's consider the same system and parameters as Problem~\ref{prob:slice}. But, now the control objective is minimizing the final state disturbance $D_F(P_1,P_2)$ for given input power constraints  $P_1, P_2 \in \mathbb{R}^+$. In other words, we solve
\begin{align}
D_F(P_1,P_2) &=\inf_{u_1 ,u_2} \mathbb{E}[x^2[k]]\\
s.t. & \sum_{k_1 \leq  i \leq k-1} \alpha^{i-k_1}\mathbb{E}[u_1^2[i]] \leq P_1 \\
& \sum_{k_2 \leq  i \leq k-1} \alpha^{i-k_2}\mathbb{E}[u_2^2[i]] \leq P_2.
\end{align}
\label{prob:slicepower}
\end{problem}

Here we can see four parameters that characterize the performance of controllers, $\sigma_{v1}^2$, $P_1$, $\sigma_{v2}^2$, $P_2$. The importance of these parameters becomes different depending on which interval they lie in.

The information-limited interval ---which corresponds to the time steps between $0$ and $k_1$ in Problem~\ref{prob:slicepower} and Figure~\ref{fig:division}--- is introduced to handle the case when $\sigma_{v1}^2$ is large. Since $\sigma_{v2}^2 \geq \sigma_{v1}^2$, in this interval both controllers have very noisy observations and we can allow arbitrarily large power to both controllers. 
In fact, in Figure~\ref{fig:division} we can see in this interval both controllers do not have any input costs. Therefore, the important parameters are $\sigma_{v1}^2$ and $\sigma_{v2}^2$. Furthermore, even the cost of the centralized controller (with access to both observations $y_1[n]$ and $y_2[n]$) gives a reasonable bound. Essentially, what this interval is doing is waiting until the variance of the state disturbances grows enough --- till $\sigma_{v1}^2$ up to scaling.

On the other hand, the power-limited interval --- which corresponds to the time steps between $k_2$ and $k$ in Problem~\ref{prob:slicepower} and Figure~\ref{fig:division} --- is introduced to handle the case when both controllers do not have enough power to stabilize the system. Therefore, in this interval the important parameters are $P_1$ and $P_2$. We will even give a perfect observation of $x[n]$ to both controllers by setting $\sigma_{v1}^2=0$ and $\sigma_{v2}^2=0$. In this interval, we will keep running the system by making $k$ arbitrarily large, and prove that $\mathbb{E}[x^2[k]]$ must diverge to infinity given that the previous interval ended up with a too large $x[k_2]$.

Between these two intervals --- the time steps between $k_1$ and $k_2$ in Problem~\ref{prob:slicepower} and Figure~\ref{fig:division} --- each controller faces a different situation. The first controller has enough information about the state but it does not have enough power. The second controller has enough power but it does not have enough information. Therefore, the important parameters of this interval are $P_1$ and $\sigma_{v2}^2$. So, we will allow a perfect observation to the first controller by setting $\sigma_{v1}^2=0$ and infinite power to the second controller by setting $P_2 = \infty$. In other words, the first controller is power limited and the second controller is information limited. This situation is exactly the same as that of Witsenhausen's counterexample which we discussed in Section~\ref{sec:caveat}. Therefore, we will call this interval an $s$-stage MIMO Witsenhausen's interval and discuss it in Section~\ref{sec:mimowitsen} in more detail.

Let's convert these ideas into formal proofs. As we mentioned, we will bound the cost in the information-limited interval by analyzing a centralized controller with both observations $y_1[n]$ and $y_2[n]$ when there is only initial disturbance $w[0]$.
\begin{lemma}
Let $w[0] \sim \mathcal{N}(0,1)$, $v_1[n] \sim \mathcal{N}(0,\sigma_{v1}^2)$, $v_2[n] \sim \mathcal{N}(0,\sigma_{v2}^2)$ be independent Gaussian random variables. Let
\begin{align}
&y_1[n]=a^{n-1}w[0]+v_1[n],\\
&y_2[n]=a^{n-1}w[0]+v_2[n].
\end{align}
Then,
\begin{align}
\mathbb{E}[(a^{k-1}w[0]-\mathbb{E}[a^{k-1}w[0]|y_1[1:k_1],y_2[1:k_1]])^2]=\frac{a^{2(k-1)}\sigma_{v1}^2}{(1+\frac{\sigma_{v1}^2}{\sigma_{v2}^2})(\frac{a^{2(k_1-1)}(1-a^{-2 k_1})}{1-a^{-2}})+\sigma_{v1}^2}.
\end{align}
\label{cov:lemma1}
\end{lemma}

\begin{proof}
Notice that
\begin{align}
&y_1[n]=a^{n-1}w[0]+v_1[n]\\
&\frac{\sigma_{v1}}{\sigma_{v2}}y_2[n]=\frac{\sigma_{v1}}{\sigma_{v2}}a^{n-1}w[0]+\frac{\sigma_{v1}}{\sigma_{v2}}v_2[n]
\end{align}
Since maximum-ratio combining is a sufficient statistic (See \cite{Tse} for instance), the sufficient statistic $y_s$ of $y_1[1:k_1-1], y_2$ for estimating $w[0]$ is given as:
\begin{align}
y_s &= \sum_{1 \leq n \leq k_1} a^{n-1}y_1[n] + \sum_{1\leq n \leq k_1} \frac{\sigma_{v1}}{\sigma_{v2}} a^{n-1} ( \frac{\sigma_{v1}}{\sigma_{v2}} y_2[n])\\
&=\sum_{1 \leq n \leq k_1} a^{n-1}(a^{n-1}w[0]+v_1[n])+ \sum_{1 \leq n \leq k_1} \frac{\sigma_{v1}}{\sigma_{v2}}a^{n-1}
(\frac{\sigma_{v1}}{\sigma_{v2}}a^{n-1}w[0]+\frac{\sigma_{v1}}{\sigma_{v2}}v_2[n] )\\
&=(\sum_{1 \leq n \leq k_1}(a^{2(n-1)}+\frac{\sigma_{v1}^2}{\sigma_{v2}^2}a^{2(n-1)}) )w[0]+(\sum_{1 \leq n \leq k_1}a^{n-1}v_1[n]+\sum_{1 \leq n \leq k_1} \frac{\sigma_{v1}^2}{\sigma_{v2}^2}a^{n-1}v_2[n])
\end{align}
The estimation error for $a^{k-1}w[0]$ is
\begin{align}
&\mathbb{E}[(a^{k-1}w[0]-\mathbb{E}[a^{k-1}w[0]|y_1[1:k_1],y_2[1:k_1]])^2]\\
&=\mathbb{E}[(a^{k-1}w[0]-\mathbb{E}[a^{k-1}w[0]|y_s])^2]\\
&=\mathbb{E}[(a^{k-1}w[0])^2]- \mathbb{E}[a^{k-1}w[0] y_s](\mathbb{E}[y_s^2])^{-1}\mathbb{E}[a^{k-1}w[0] y_s]\\
&=\frac{\mathbb{E}[(a^{k-1}w[0])^2]\mathbb{E}[y_s^2]-\mathbb{E}[a^{k-1}w[0] y_s]^2 }{\mathbb{E}[y_s^2]}\\
&=\frac{a^{2(k-1)}((\sum_{1 \leq n \leq k_1}(a^{2(n-1)}+\frac{\sigma_{v1}^2}{\sigma_{v2}^2}a^{2(n-1)}))^2 + \sum_{1\leq n \leq k_1}a^{2(n-1)}\sigma_{v1}^2 + \sum_{1\leq n \leq k_1}(\frac{\sigma_{v1}^2}{\sigma_{v2}^2})^2 a^{2(n-1)}\sigma_{v2}^2)}
{(\sum_{1 \leq n \leq k_1}(a^{2(n-1)}+\frac{\sigma_{v1}^2}{\sigma_{v2}^2}a^{2(n-1)}))^2 + \sum_{1\leq n \leq k_1}a^{2(n-1)}\sigma_{v1}^2 + \sum_{1\leq n \leq k_1}(\frac{\sigma_{v1}^2}{\sigma_{v2}^2})^2 a^{2(n-1)}\sigma_{v2}^2}\\
&-
\frac{a^{2(k-1)}(\sum_{1 \leq n \leq k_1}(a^{2(n-1)}+\frac{\sigma_{v1}^2}{\sigma_{v2}^2}a^{2(n-1)}))^2}
{(\sum_{1 \leq n \leq k_1}(a^{2(n-1)}+\frac{\sigma_{v1}^2}{\sigma_{v2}^2}a^{2(n-1)}))^2 + \sum_{1\leq n \leq k_1}a^{2(n-1)}\sigma_{v1}^2 + \sum_{1\leq n \leq k_1}(\frac{\sigma_{v1}^2}{\sigma_{v2}^2})^2 a^{2(n-1)}\sigma_{v2}^2}\\
&=
\frac{a^{2(k-1)}( \sum_{1\leq n \leq k_1}a^{2(n-1)}\sigma_{v1}^2 + \sum_{1\leq n \leq k_1}(\frac{\sigma_{v1}^2}{\sigma_{v2}^2})^2 a^{2(n-1)}\sigma_{v2}^2)}
{(\sum_{1 \leq n \leq k_1}(a^{2(n-1)}+\frac{\sigma_{v1}^2}{\sigma_{v2}^2}a^{2(n-1)}))^2 + \sum_{1\leq n \leq k_1}a^{2(n-1)}\sigma_{v1}^2 + \sum_{1\leq n \leq k_1}(\frac{\sigma_{v1}^2}{\sigma_{v2}^2})^2 a^{2(n-1)}\sigma_{v2}^2}\\
&=
\frac{a^{2(k-1)}(\sigma_{v1}^2+\frac{\sigma_{v1}^4}{\sigma_{v2}^2})(\frac{a^{2(k_1-1)}(1-a^{-2k_1}) }{1-a^{-2}}) }
{(1+\frac{\sigma_{v1}^2}{\sigma_{v2}^2})^2 (\frac{a^{2(k_1-1)}(1-a^{-2 k_1})}{1-a^{-2}})^2 + (\sigma_{v1}^2+\frac{\sigma_{v1}^4}{\sigma_{v2}^2})\frac{a^{2(k_1-1)}(1-a^{-2k_1})}{1-a^{-2}} }\\
&=
\frac{a^{2(k-1)}\sigma_{v1}^2}{(1+\frac{\sigma_{v1}^2}{\sigma_{v2}^2})(\frac{a^{2(k_1-1)}(1-a^{-2k_1})}{1-a^{-2}})+\sigma_{v1}^2}.
\end{align}
\end{proof}

To bound the performance in the power-limited interval, we have to bound the influence of control inputs on the state with respect to their power constraints. By expanding $x[n]$ using the system equation~\eqref{eqn:systemsimple}, we can see $x[n]=\sum_{0 \leq  i  \leq n-1}  a^{n-1-i} w[i] + a^{n-1-i}u_1[i] + a^{n-1-i}u_2[i]$. Thus, the terms $\sum_{0 \leq i \leq n-1}a^{n-1-i}u_j[i]$ can be thought as the influence of control inputs to the state.  The following Cauchy-Schwarz style inequality bounds the variance of $\sum_{0 \leq i \leq n-1}a^{n-1-i}u_j[i]$ by the power constraint $\sum_{0 \leq i \leq n-1} \alpha^{i} \mathbb{E}[u_j^2[i]]$ imposed in Problem~\ref{prob:slicepower}.
\begin{lemma}
For arbitrary random variables $X_i$, $n \in \mathbb{N}$, $a,b \in \mathbb{R}$ $(|\frac{1}{a^2 b}| < 1)$, we have
\begin{align}
\mathbb{E}[(a^{n-1} X_0 + a^{n-2} X_1+ \cdots + X_{n-1} )^2] \leq
\frac{a^{2(n-1)}(1-(\frac{1}{a^2 b})^n)}{1-\frac{1}{a^2 b}}(\mathbb{E}[X_0^2] + b\mathbb{E}[X_1^2] + \cdots + b^{n-1}\mathbb{E}[X_{n-1}^2])
\end{align}
\label{lem:power}
\end{lemma}
\begin{proof}
\begin{align}
&\mathbb{E}[(a^{n-1}X_0 + a^{n-2} X_1 + \cdots + X_{n-1})^2] \\
&\leq (\sqrt{a^{2(n-1)}\mathbb{E}[X_0^2]} + \sqrt{a^{2(n-2)}\mathbb{E}[X_1^2]} + \cdots + \sqrt{\mathbb{E}[X_{n-1}^2]})^2\\
&\leq (a^{2(n-1)}+\frac{a^{2(n-2)}}{b}+\cdots+(\frac{1}{b})^{n-1}) (\mathbb{E}[X_0^2]+b\mathbb{E}[X_1^2]+\cdots+b^{n-1}\mathbb{E}[X_{n-1}^2]) \\
&=\frac{a^{2(n-1)}(1-(\frac{1}{a^2 b})^n)}{1-\frac{1}{a^2 b}}(\mathbb{E}[X_0^2]+b\mathbb{E}[X_1^2]+\cdots+b^{n-1}\mathbb{E}[X_{n-1}^2])
\end{align}
where all inequalities follow from Cauchy-Schwarz.
\end{proof}

\begin{figure}[htbp]
\begin{center}
\includegraphics[width=3in]{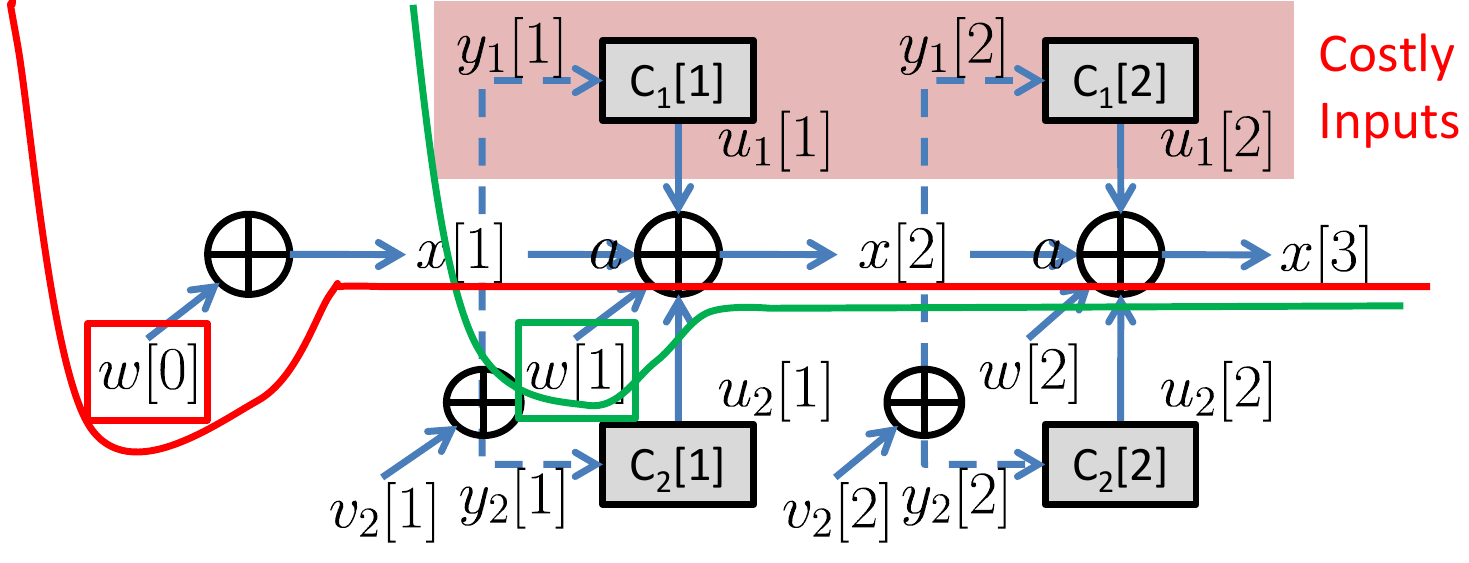}
\caption{Finite-horizon generalized MIMO Witsenhausen's counterexample. This problem gives the matching lower bound to $1$-stage signaling.}
\label{fig:finite1}
\end{center}
\end{figure}

\begin{figure}[htbp]
\begin{center}
\includegraphics[width=2.5in]{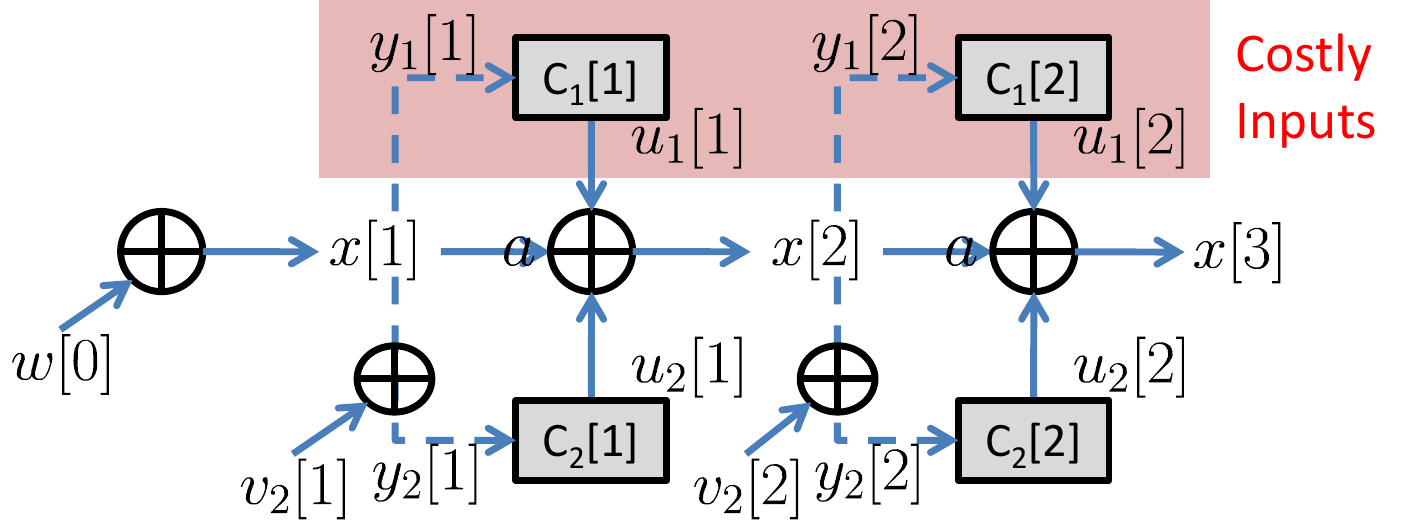}
\caption{The simplified problem that results from Figure~\ref{fig:finite1} by cutting the problem across the red line. Unlike the original problem, $w[0]$ is the only disturbance.}
\label{fig:finite2}
\end{center}
\end{figure}

\begin{figure}[htbp]
\begin{center}
\includegraphics[width=2in]{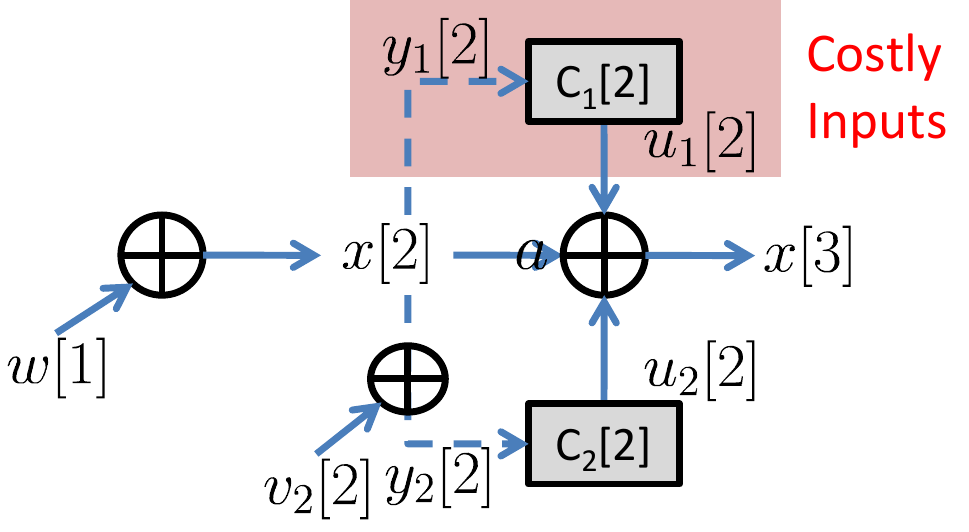}
\caption{The simplified problem that results from Figure~\ref{fig:finite1} by cutting the problem across the blue line. Unlike the original problem, $w[1]$ is the only disturbance.}
\label{fig:finite3}
\end{center}
\end{figure}

\begin{figure}[htbp]
\begin{center}
\includegraphics[width=2.5in]{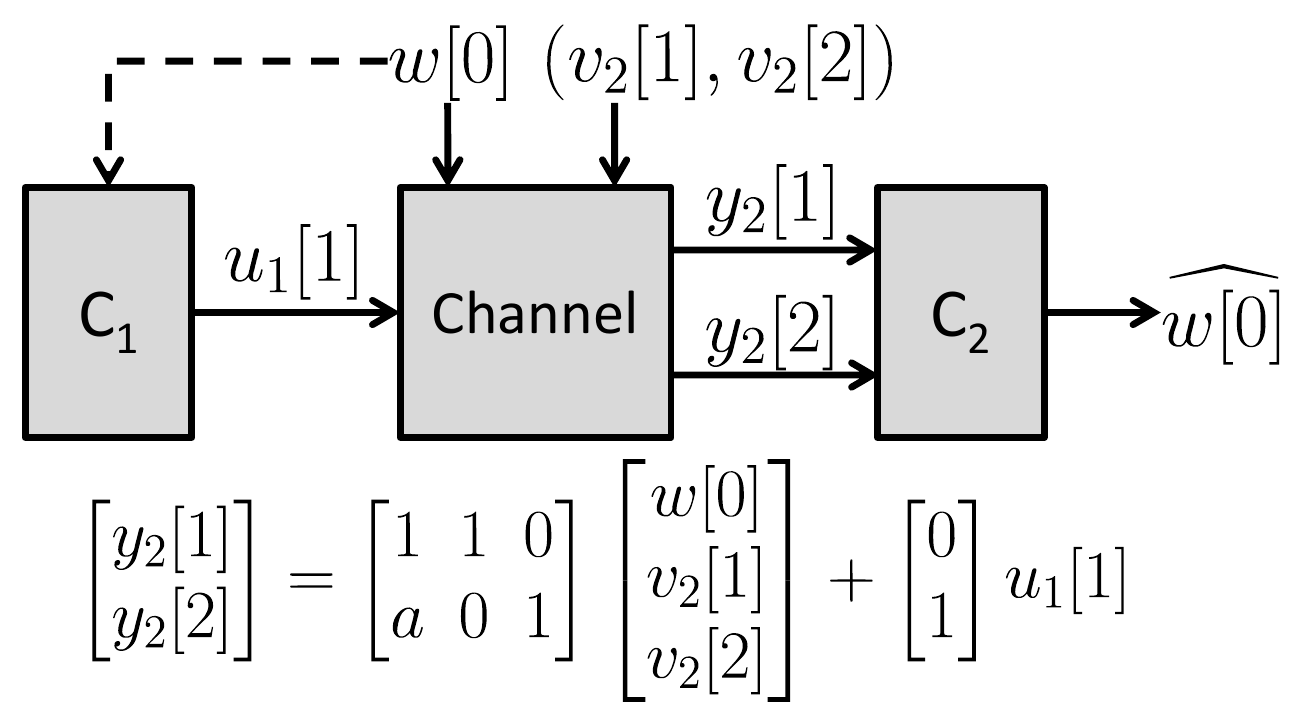}
\caption{A further simplified MIMO communication problem of Figure~\ref{fig:finite2}. This problem reflects the communication aspect of Figure~\ref{fig:finite2}.}
\label{fig:finite4}
\end{center}
\end{figure}


\subsection{$s$-stage MIMO Witsenhausen's interval: From multiple disturbances to a single disturbance}
\label{sec:mimowitsen}

Understanding the MIMO Witsenhausen's interval is necessary to find a matching lower bound to $s$-stage signaling strategies. Let's explicitly consider Problem~\ref{prob:power} with parameters $\sigma_{v1}^2 = 0$ and $P_2 = \infty$ and find the lower bound on $D(P_1,P_2)$ that approximately matches to the $1$-stage signaling strategy.

By selecting the parameters $k=3, k_1=1, k_2=3, \alpha=\frac{1}{2}$ in Problem~\ref{prob:slicepower}, we have the problem of minimizing $D_F(P_1,P_2)=\mathbb{E}[x^2[3]]$ with the power constraint $( \frac{1}{2}\mathbb{E}[u_1^2[1]] + \frac{1}{4}\mathbb{E}[u_1^2[2]]) \leq P_1$. This is a lower bound on $D(P_1,P_2)$.

Figure~\ref{fig:finite1} shows the resulting $2$-stage finite-horizon problem. As we can see the problem looks similar to Witsenhausen's one in Figure~\ref{fig:witsen1}. In fact, it can be thought as a multi-stage MIMO (multiple-input multiple-output) Witsenhausen's counterexample. Compared to the original Witsenhausen's counterexample, both controllers have observations and control inputs at every time step, and a new state disturbance $w[n]$ is added at every time step. Since the second controller's input cost is free, it can be considered as the receiver in a communication problem. From this perspective, the observation $y_2[1]$ can be considered as side-information at the receiver, and the input $u_2[1]$ can be imagined to be feedback from receiver to transmitter.

The first question that we have to answer to take this communication perspective is ``What is the relevant message in this communication problem?" Since the only uncertainty of the system is the state disturbance $w[n]$, the answer has to be the disturbance. However, since a new $w[n]$ is added at every time step, we have to find the critically relevant disturbance among them.

To understanding this issue, let's revisit the binary deterministic model of Section~\ref{sec:intui}. In Figure~\ref{fig:nonlin}, we can see $x[3]$ corresponds to  $00x_{-1}^{3}0.x_{-3}^3 x_{-4}^3\cdots$ in the binary deterministic model. We will divide this binary number into three parts. The first part is the first two bits $00$, the second part is the next two bits $x_{-1}^{3}0$, and the third part is the remaining bits $x_{-3}^3 x_{-4}^3\cdots$. If we track back the arrows of Figure~\ref{fig:nonlin}, we can see that these three parts originated from the different disturbances $w[0]$, $w[1]$, $w[2]$ respectively. Therefore, we can see that $w[2]$ is not a dominating disturbance since its bit level is much smaller than the other parts, and the dominant disturbances for $x[3]$ are $w[0]$ and $w[1]$. We will separate these two disturbances using the cutset idea in information theory.

The first cut gives every disturbance except $w[0]$ as side information to the second controller, i.e. we give $w[1], w[2]$ as side information. Figure~\ref{fig:finite2} shows the resulting problem, which is a $2$-stage MIMO Witsenhausen's counterexample with only one disturbance at the beginning. Likewise, the second cut gives $w[0], w[2]$ and reserves $w[1]$ inside the cut. Figure~\ref{fig:finite3} shows the resulting problem, which is a $1$-stage Radner's problem. Both problems are relaxations of the original problem, and any convex sum of their cost is also a lower bound to the cost of the original problem.

We already know how to solve Radner's problem in Figure~\ref{fig:finite3}. However, the problem in Figure~\ref{fig:finite2} is a generalized MIMO Witsenhausen's problem, which is even harder than the original one. The crux of the problem is the dual role of controllers' inputs. The input signals $u_1[n]$ and $u_2[n]$ can be used to cancel the state (control role) and at the same time to send information about their observations (communication role). Therefore, we will simplify the problem by removing the less important role.

The first controller has a perfect observation while its input cost is expensive. Therefore, it is better to use the control inputs to send information about the state. We will essentially remove the control role of the first controller input by using the Cauchy-Schwarz inequality. Meanwhile, the second controller has free input cost but blurry observations. Therefore, it is better to focus on the control role. We will remove the communication role of the second controller input by allowing free freeback from the second controller to the first. Therefore, the first controller reduces to a transmitter and the second controller reduces to a receiver.

Figure~\ref{fig:finite4} shows the pure MIMO communication problem we will get after removing the dual roles of the controllers from the problem of Figure~\ref{fig:finite3}. The first controller knows the exact state $w[0]$ and sends information through the input $u_1[1]$. Thus, the first controller is the transmitter and $u_1[1]$ is the transmitted signal.\footnote{Here, $u_1[2]$ cannot send any information to the second controller since communication requires at least one step delay from the transmitter to the receiver.} The second controller estimates the state $w[0]$ based on its observation $y_2[1], y_2[2]$. Therefore, the second controller is the receiver and $y_2[1], y_2[2]$ are the received signals.\footnote{The second controller can also feedback its observation through $u_2[1]$. However, this effect of feedback is negligible in this case, since the causal feedback information can only affect $u_2[2]$ at the transmitter. However, we will see the effect of feedback later in the more generalized problem of Figure~\ref{fig:reduced}.} We will use a simple information-theoretic cutset bound to bound the performance of this communication system, and eventually derive a lower bound approximately matching to the $1$-stage signaling strategy.

At this point, one may wonder why we need the lower bound of Figure~\ref{fig:finite2} and Figure~\ref{fig:finite4} which correspond to zeros in the binary deterministic model. It is because it is not zero in Gaussian real models. Binary deterministic models simplify Gaussian random variables as bounded uniform distributions. This simplification can be justified in an infinite-dimensional relaxation. However, in finite dimensions the simplification only approximately holds and the zeros in the binary deterministic model are actually exponentially decreasing small quantities in a Gaussian model. As shown in \cite{Pulkit_Witsen}, we will replace $v_2[n]$ of Figure~\ref{fig:finite4} by a test channel, adapting ideas of large-deviation theory. The problem of Figure~\ref{fig:finite4} gives a non-trivial lower bound that captures the exponentially decreasing small quantities that must occur because of the finite-dimensionality.

In general, we will see an $s$-stage MIMO Witsenhausen's counterexample in the second time interval of Figure~\ref{fig:division}. Following the same steps as above, we will reduce the problem to a pair of pure communication problems, $s$-stage and $(s-1)$-stage MIMO state-amplification with feedback.

\begin{figure*}[htbp]
\begin{center}
\includegraphics[width=3in]{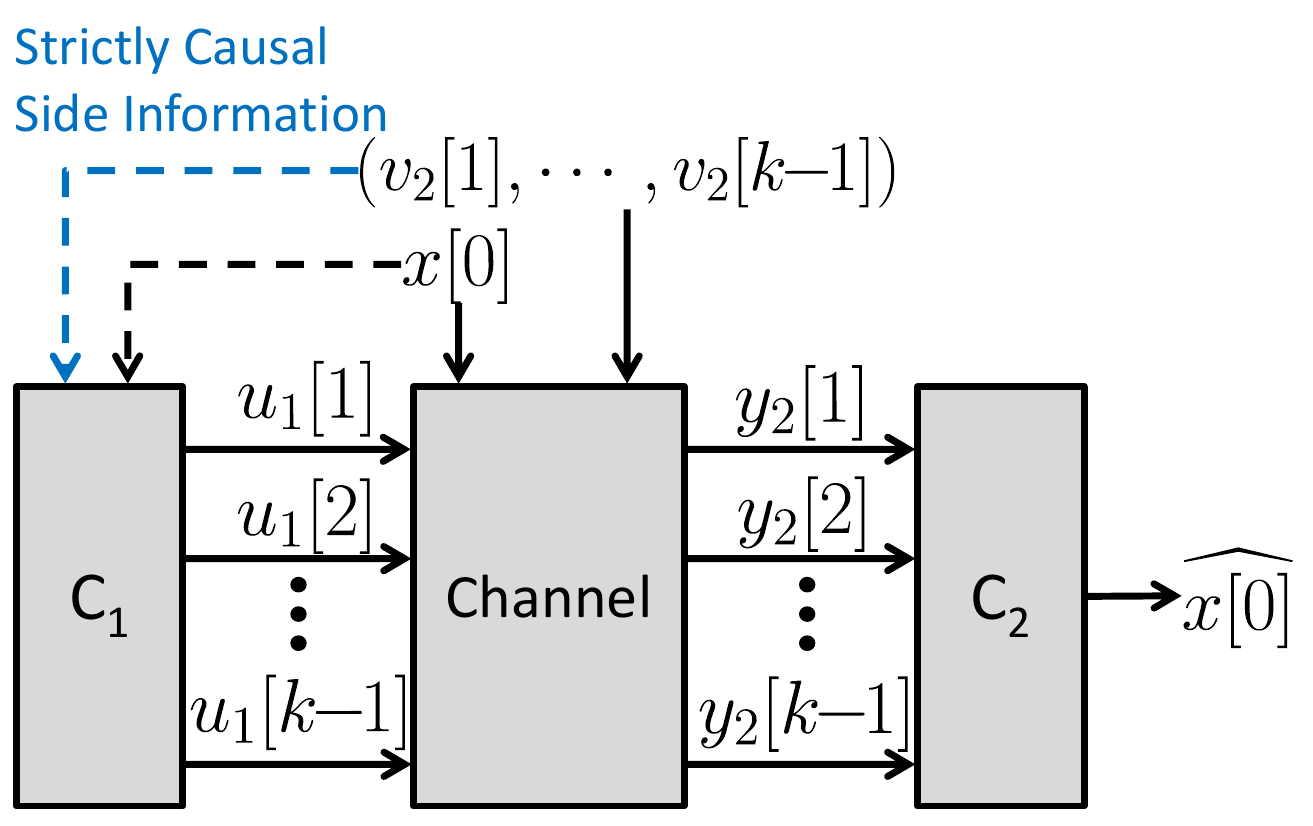}
\caption{$s$-stage MIMO state-amplification with feedback. This problem reflects the implicit communication aspect in the MIMO Witsenhausen's interval of Figure~\ref{fig:division}.}
\label{fig:reduced}
\end{center}
\end{figure*}

\begin{figure*}[htbp]
\begin{center}
\includegraphics[width=3in]{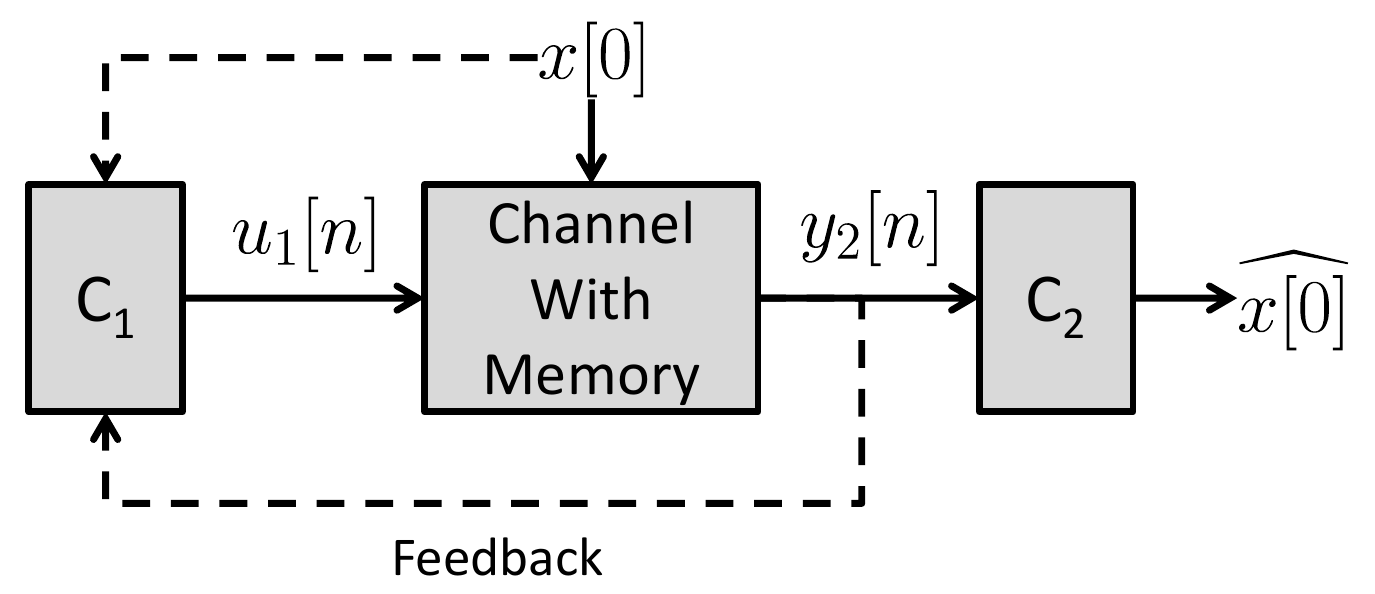}
\caption{An equivalent representation of $s$-stage MIMO state-amplification with feedback in Figure~\ref{fig:reduced}. The MIMO channel of Figure~\ref{fig:reduced} can be thought as a channel with memory.}
\label{fig:reduced2}
\end{center}
\end{figure*}

\subsection{$s$-stage MIMO state-amplification with feedback}
\label{sec:lower:amp}
Figure~\ref{fig:reduced} shows the $s$-stage MIMO state-amplification problem. As we discussed before, the first controller $C_1$ is the transmitter, and the second controller $C_2$ is the receiver. The transmitter knows the state $x[0]$ at the initial time and learns the channel noise $v_2[n]$ by causal feedback. The goal of communication is minimizing the estimation error of the state $x[0]$ at the receiver.

Let's formally state the $s$-stage MIMO state-amplification with feedback problem.
\begin{problem}[$s$-stage MIMO state-amplification with feedback]
Let the underlying random variables $x[0] \sim \mathcal{N}(0,\sigma_0^2)$ and $v_2[n] \sim \mathcal{N}(0,\sigma_{v}^2)$ be all independent.  These are the source and observation noise respectively. The transmitter's input $u_1[n]$ is a function on $x[0]$ and $v_2[1], \cdots,v_2[n-1]$, i.e.
\begin{align}
&u_1[1]=f_1(x[0])\\
&u_1[2]=f_2(x[0],v_2[1])\\
&\vdots\\
&u_1[k-1]=f_{k-1}(x[0],v_2[1:k-2])
\end{align}
The receiver's observations $y_2[n]$ are given as follows.
\begin{align}
&y_2[0]=x[0]+v_2[0]\\
&y_2[1]=ax[0]+u_1[0]+v_2[1]\\
&y_2[2]=a^2x[0]+au_1[0]+u_1[1]+v_2[2]\\
&\vdots\\
&y_2[k-1]=a^{k-1}x[0]+a^{k-2}u_1[0]+\cdots +u_1[k-2]+v_2[k-1]
\end{align}
The receiver generates an estimate $\widehat{x[0]}$ of the state $x[0]$ based on its received signal $y_2[1:k-1]$, i.e. $\widehat{x[0]}=g(y_2[1:k-1])$. The objective of the system is minimizing the quadratic estimation error, $\mathbb{E}[(x[0]-\widehat{x[0]})^2]$.
\end{problem}

This problem can be more compactly represented as Figure~\ref{fig:reduced2} by thinking of the MIMO channel as a channel with memory. As shown in \cite{Cover_feedback}, feedback only increases the capacity at most a half bit per time step. However, in this problem we are using the channels for $k$ time steps, so we still have to justify that the feedback does not increase the capacity too much. The following lemma explicitly computes an information-theoretic cutset bound for this communication problem and gives a reasonable bound on the rate-distortion tradeoff.
\begin{lemma}
Let's consider the problem of Figure~\ref{fig:reduced}.\\
(i) Let $x[0] \sim \mathcal{N}(0,\sigma_0^2)$ and $v_2[n] \sim \mathcal{N}(0,\sigma_{v}^2)$.
Let $w \in \mathbb{R}$ satisfy $|\frac{1}{a^2 w}|<1$ and the input power constraint is
\begin{align}
(1-w)\mathbb{E}[u_1^2[0]]+ (1-w)w\mathbb{E}[u_1^2[1]]+ \cdots + (1-w)w^{k-2}\mathbb{E}[u_1^2[k-2]] \leq P
\end{align}
Then, the estimation error of $x[0]$ based on $y_2[0:k-1]$ is lower bounded by
\begin{align}
\mathbb{E}[(x[0]-\widehat{x[0]})^2] \geq \frac{\sigma_0^2}{2^{2I_k}}
\end{align}
where
\begin{align}
I_k = \frac{k}{2} \log( 1 + \frac{1}{k\sigma_v^2}(\frac{2a^{2(k-1)} \sigma_0^2 }{1-a^{-2}}  + \frac{2 a^{k-2} }{1-a^{-2}} \frac{P}{(1-\frac{1}{ a^2 w})(1-w)})).
\end{align}
(ii) Consider the same problem as (i) except that $v_2[k-1] \sim \mathcal{N}(0,\sigma_v'^2)$, i.e. only the last observation noise variance is different. Then, the estimation error based on $y_2[0:k-1]$ is lower bounded by
\begin{align}
\mathbb{E}[(x[0]-\widehat{x[0]})^2] \geq \frac{\sigma_0^2}{2^{2I_k'}}
\end{align}
where
\begin{align}
I_k'=I_{k-1}
+\frac{1}{2}\log(1 + \frac{1}{\sigma_{v}'^2}(
2 a^{2(k-1)} \sigma_0^2 + 2\frac{a^{2(k-2)}}{1-\frac{1}{a^2 w}} \frac{P}{1-w})).
\end{align}
(iii) Consider the same problem as (ii) except that $v_2[k-1] \sim Unif[-\sigma_{v}',\sigma_{v}']$, i.e. the last observation is a uniform random variable. Then, the estimation error based on $y_2[0:k-1]$ is lower bounded by
\begin{align}
\mathbb{E}[(x[0]-\widehat{x[0]})^2] \geq \frac{\sigma_0^2}{2^{2I_k''}}
\end{align}
where
\begin{align}
I_k''=I_k'+\frac{1}{2}\log(\frac{\pi e }{2}).
\end{align}
\label{lem:reduced}
\end{lemma}
\begin{proof}
(i) First, we can lower bound the estimation error as follows:
\begin{align}
&\frac{1}{2}\log( 2 \pi e \mathbb{E}[(x[0] - \widehat{x[0]})^2])\\
&\geq h(x[0]-\widehat{x[0]}| y_2[0:k-1])\\
&=h(x[0]| y_2[0:k-1])\\
&=h(x[0])- I(x[0]; y_2[0:k-1]) \\
&\geq \frac{1}{2}\log(2 \pi e \sigma_0^2) - I(x[0]; y_2[0:k-1]). \label{eqn:reduced1}
\end{align}
We will upper bound the mutual information. Let's first upper bound the received signal power. Since $u_1[n]$ is a strictly causal function of $v_2[n]$,
\begin{align}
&\mathbb{E}[y_2^2[n]] \leq 2 \mathbb{E}[(a^n x[0])^2] + 2 \mathbb{E}[(a^{k-2}u_1[0]+u_1[n-1])^2] + \mathbb{E}[v_2^2[n]].\\
\end{align}
By Lemma~\ref{lem:power}, we have
\begin{align}
&\mathbb{E}[(a^{n-1}u_1[0]+\cdots+u_1[n-1])^2]\\
&\leq \frac{a^{2(n-1)}}{1 - \frac{1}{a^2 w}} \frac{P}{1-w}.
\end{align}
Therefore, the received signal power is upper bounded as
\begin{align}
&\mathbb{E}[y_2^2[n]] \leq 2 a^{2n} \sigma_0^2 + 2 \frac{a^{2(n-1)}}{1 - \frac{1}{a^2 w}} \frac{P}{1-w} + \mathbb{E}[v_2^2[n]].
\end{align}
Thus, we can conclude
\begin{align}
&\sum_{0 \leq n \leq k-1} \mathbb{E}[y_2^2[n]] \\
&\leq \sum_{0 \leq n \leq k-1} 2a^{2n} \sigma_0^2 + 2 \frac{a^{2(n-1)}}{1 - \frac{1}{a^2 w}} \frac{P}{1-w} + \sigma_v^2\\
&= 2( 1+ \cdots +a^{2(k-1)}) \sigma_0^2 + \sum_{0 \leq n \leq k-1}2 \frac{a^{2(n-1)}}{1 - \frac{1}{a^2 w}} \frac{P}{1-w} + k\sigma_v^2\\
&= 2a^{2(k-1)} \frac{1-a^{-2k}}{1-a^{-2}} \sigma_0^2 + 2 a^{k-2} \frac{1-a^{-2k}}{1-a^{-2}} \frac{P}{(1-\frac{1}{a^2 w})(1-w)} + k\sigma_v^2\\
&\leq   \frac{2a^{2(k-1)} \sigma_0^2 }{1-a^{-2}}  + \frac{2 a^{k-2} }{1-a^{-2}} \frac{P}{(1-\frac{1}{a^2 w})(1-w)} + k\sigma_v^2.
\end{align}
Using this, we can upper bound the mutual information.
\begin{align}
&I(x[0]; y_2[0:k-1]) \\
&\leq h(y_2[0:k-1]) - h(y_2[0:k-1]| x[0]) \\
&\leq \sum_{0 \leq n \leq k-1} h(y_2[n]) - \sum_{0 \leq n \leq k-1} h(v_2[n])\\
&\leq \sum_{0 \leq n \leq k-1} \frac{1}{2}\log( 2 \pi e \mathbb{E}[y_2[n]^2]) - \frac{k-1}{2} \log(2 \pi e \sigma_v^2) \\
&=\frac{1}{2} \log( \prod_{0 \leq n \leq k-1} \frac{\mathbb{E}[y_2[n]^2]}{\sigma_v^2})\\
&\leq \frac{1}{2} \log\left( ( \frac{1}{k} \sum_{0 \leq n \leq k-1} \frac{\mathbb{E}[y_2[n]^2]}{\sigma_v^2})^{k-1}\right) \\
&(\because \mbox{geometric mean and arithmetic mean})\\
&\leq \frac{1}{2} \log\left(( 1 + \frac{1}{k\sigma_v^2}(\frac{2a^{2(k-1)} \sigma_0^2 }{1-a^{-2}}  + \frac{2 a^{k-2} }{1-a^{-2}} \frac{P}{(1-\frac{1}{a^2 w})(1-w)}))^{k}\right) \label{eqn:reduced2}
\end{align}
The last term is $I_k$. By plugging \eqref{eqn:reduced2} into \eqref{eqn:reduced1}, we get
\begin{align}
\mathbb{E}[(x[0] - \widehat{x[0]})^2] \geq \frac{\sigma_0^2}{2^{2I_k}}
\end{align}
which finishes the proof.

(ii) We have
\begin{align}
&I( x[0]; y_2[0:k-1]) \\
&\leq h(y_2[0:k-1]) - h(y_2[0:k-1]| x[0]) \\
&\leq \sum_{0 \leq n \leq k-1} h(y_2[n]) - \sum_{0 \leq n \leq k-1} h(v_2[n])\\
&\leq \sum_{0 \leq n \leq k-1} \frac{1}{2}\log( 2 \pi e \mathbb{E}[y_2[n]^2]) - \frac{k-1}{2} \log(2 \pi e \sigma_v^2)
- \frac{1}{2} \log(2 \pi e \sigma_v'^2) \\
&=\frac{1}{2} \log( \prod_{0 \leq n \leq k-2} \frac{\mathbb{E}[y_2[n]^2]}{\sigma_v^2}) + \frac{1}{2}\log( \frac{\mathbb{E}[y_2[k-1]^2]}{\sigma_v'^2})\\
&\leq \frac{1}{2} \log( \frac{1}{k-1} \sum_{0 \leq n \leq k-2} \frac{\mathbb{E}[y_2[n]^2]}{\sigma_v^2})^{k-1} + \frac{1}{2}\log( \frac{\mathbb{E}[y_2[k-1]^2]}{\sigma_v'^2})\\
&(\because \mbox{geometric mean and arithmetic mean})\\
&\leq \frac{1}{2} \log( 1 + \frac{1}{(k-1)\sigma_v^2}(\frac{2a^{2(k-2)} \sigma_0^2 }{1-a^{-2}}  + \frac{2 a^{k-3} }{1-a^{-2}} \frac{P}{(1-\frac{1}{a^2 w})(1-w)}))^{k-1} \\
&+\frac{1}{2}\log(1 + \frac{1}{\sigma_{v}'^2}(
2 a^{2(k-1)} \sigma_0^2 + 2\frac{a^{2(k-2)}}{1-\frac{1}{a^2 w}} \frac{P}{1-w}
)). \label{eqn:reduced3}
\end{align}
The last term is $I_k'$. By plugging \eqref{eqn:reduced3} into \eqref{eqn:reduced1}, we get
\begin{align}
\mathbb{E}[(x[0] - \widehat{x[0]})^2] \geq \frac{\sigma_0^2}{2^{2I_k'}}
\end{align}
which finishes the proof.

(iii) We can repeat the proof of (ii) replacing the distribution of $v_2[k-1]$ by uniform.
\end{proof}
In this lemma, the bound of (ii) is tighter than that of (i) since it excludes the last observation in the arithmetic-geometric inequality, but it is harder to compute. We also allow the variance of the last observation noise to be different from the other ones, since we will replace it with another distribution to adapt large deviation ideas.\footnote{Even though large deviation ideas usually introduce a sequence of atypical noise, here the SNR of the last observation dominates the SNR of all the other observations. Thus, it is enough to introduce atypically large noise only to the last observation.}

\subsection{Lower bound on the optimal cost based on Witsenhausen's counterexample}

Now, we can combine the previous results to derive a lower bound that will approximately match with $s$-stage signaling strategies. We will derive a lower bound on the weighted average cost of Problem~\ref{prob:lqg}, i.e. we will find functions $D_{L,i}(\widetilde{P_1},\widetilde{P_2})$ such that
\begin{align}
\inf_{u_1,u_2} \limsup_{N \rightarrow \infty} \frac{1}{N}
\sum_{0 \leq n < N} q \mathbb{E}[x^2[n]] + r_1 \mathbb{E}[u_1^2[n]] + r_2 \mathbb{E}[u_2^2[n]] \geq \min_{\widetilde{P_1},\widetilde{P_2} \geq 0} q D_{L,i}(\widetilde{P_1},\widetilde{P_2}) + r_1 \widetilde{P_1} + r_2 \widetilde{P_2}.\end{align}

Here, the lower bounds $D_{L,i}(\widetilde{P_1},\widetilde{P_2})$ can be thought as a lower bound on $D(P_1,P_2)$, the power-disturbance tradeoff, of Problem~\ref{prob:power}. The first bound $D_{L,1}$ is given in the following lemma, and the rest will be given in Lemma~\ref{cov:lemma2} of page~\pageref{cov:lemma2}.


\begin{lemma}
Define $S_{L,1}$ as the set of $(k_1,k_2,k,\sigma_{v2}',\alpha,\Sigma)$ such that
\begin{align}
&k_1,k_2,k \in \mathbb{N}, \sigma_{v2}',\alpha,\Sigma \in \mathbb{R}_+, \\
&k_1 \geq 1, k_2 - k_1 -1 \geq 0, k \geq k_2,\\
&\sigma_{v2}' \geq 0, 0 \leq \alpha \leq 1,\\
&0 \leq \Sigma \leq 
\left\{\begin{array}{ll}
1 &\mbox{when }k_1 =1\\
\frac{a^{2(k_1-1)}\sigma_{v1}^2}{(1+\frac{\sigma_{v1}^2}{\sigma_{v2}^2})(\frac{a^{2(k_1-2)}(1-a^{-2(k_1-1)})}{1-a^{-2}})+\sigma_{v1}^2} & \mbox{when } k_1 \geq 2
\end{array}
\right.
\end{align}
We also define $D_{L,1}(\widetilde{P_1},\widetilde{P_2};k_1,k_2,k,\sigma_{v2}',\alpha,\Sigma )$ as follows:
\begin{align}
&D_{L,1}(\widetilde{P_1},\widetilde{P_2};k_1,k_2,k,\sigma_{v2}',\alpha,\Sigma) \\
&:=
\alpha (\sqrt{c \frac{a^{2(k-k_1)}\Sigma}{2^{2I'(\widetilde{P_1})}}}
- \sqrt{ c \frac{a^{2(k-k_1-1)}(1-(2.5a^{-2})^{k_2-k_1})}{1-2.5a^{-2}} \frac{\widetilde{P_1}}{1-2.5^{-1}} } \\
&- \sqrt{ \frac{a^{2(k-k_2-1)}(1-(2.5a^{-2})^{k-k_2})}{1-2.5a^{-2}} \frac{ 2.5^{k_2-k_1} \widetilde{P_1}}{1-2.5^{-1}} }
- \sqrt{ \frac{a^{2(k-k_2-1)}(1-(2.5a^{-2})^{k-k_2})}{1-2.5a^{-2}} \frac{\widetilde{P_2}}{1-2.5^{-1}} })_+^2 \label{eqn:witsenlower1} \\
&+(1-\alpha) ( \sqrt{\frac{a^{2(k-k_1-1)}\Sigma}{2^{2I''(\widetilde{P_1})}}}
-\sqrt{\frac{a^{2(k-k_1-2)}(1-(2.5a^{-2})^{k-k_1-1})}{1-2.5a^{-2}} \frac{2.5 \widetilde{P_1}}{1-2.5^{-1}}}\label{eqn:witsenlower2} \\
&-\sqrt{\frac{a^{2(k-k_2-1)}(1-(2.5a^{-2})^{k-k_2})}{1-2.5a^{-2}} \frac{\widetilde{P_2}}{1-2.5^{-1}}} )_+^2
+1 \label{eqn:witsenlower3}
\end{align}
where
\begin{align}
&I''(\widetilde{P_1})= \left \{
\begin{array}{ll}
\frac{{k_2-k_1-1}}{2} \log(1+\frac{1}{(k_2-k_1-1)\sigma_{v2}^2}(\frac{2a^{2(k_2-2-k_1)}}{1-a^{-2}}\Sigma + \frac{2a^{2(k_2-3-k_1)} }{1-a^{-2}} \frac{2.5\widetilde{P_1}}{(1-2.5a^{-2})(1-2.5^{-1})}))  & \mbox{if } k_2-k_1-1>0\\
0 & \mbox{if } k_2-k_1-1=0
\end{array} \right.
\\
&I'(\widetilde{P_1})= I''(\widetilde{P_1})+\frac{1}{2} \log(1+ \frac{1}{\sigma_{v2}'^2}(2a^{2(k_2-1-k_1)}\Sigma + 2 \frac{a^{2(k_2-2-k_1)}\widetilde{P_1}}{(1-2.5a^{-2})(1-2.5^{-1})})) + \frac{1}{2} \log( \frac{2 \pi e}{4} ) \mathbf{1}(\sigma_{v2} \neq \sigma_{v2}')\\
&c=\left\{
\begin{array}{ll}
\frac{2\sigma_{v2}'}{\sqrt{2\pi}\sigma_{v2}}\exp(-\frac{\sigma_{v2}'^2}{2 \sigma_{v2}^2}) & \mbox{if }\sigma_{v2} \neq \sigma_{v2}' \\
1 & \mbox{if } \sigma_{v2}= \sigma_{v2}'\\
\end{array}
\right. 
\end{align}
Let $|a| \geq 2.5$. Then, for all $q, r_1, r_2 \geq 0$, the minimum cost \eqref{eqn:part11} of Problem~\ref{prob:lqg} is lower bounded as follows:
\begin{align}
&\inf_{u_1,u_2} \limsup_{N \rightarrow \infty} \frac{1}{N}
\sum_{0 \leq n < N} q \mathbb{E}[x^2[n]] + r_1 \mathbb{E}[u_1^2[n]] + r_2 \mathbb{E}[u_2^2[n]]\\
& \geq \sup_{(k_1,k_2,k,\sigma_{v2}',\alpha,\Sigma) \in S_{L,1}}\min_{\widetilde{P_1},\widetilde{P_2} \geq 0} q D_{L,1}(\widetilde{P_1},\widetilde{P_2};k_1,k_2,k,\sigma_{v2}',\alpha,\Sigma) + r_1 \widetilde{P_1} + r_2 \widetilde{P_2}.
\end{align}
\label{cov:lemma3}
\end{lemma}

\begin{proof}
For simplicity, we assume $a\geq 2.5$, $k_1 \geq 2$, $k_2-k_1-1 > 0$, $k > k_2$, $\sigma_{v2} \neq \sigma_{v2}'$. The remaining cases when $a \leq -2.5$ or $k_1 = 1$ or $k_2 -k_1 -1 =0$ or $k=k_2$ or $\sigma_{v2} = \sigma_{v2}'$ easily follow with minor modifications.

$\bullet$ Geometric Slicing: We first apply the geometric slicing idea of Section~\ref{sec:lower:geo} to get a finite-horizon problem. By setting $\alpha=2.5^{-1}$ in Lemma~\ref{lem:geo}, the average cost is lower bounded by
\begin{align}
&\inf_{u_1, u_2}(
q \mathbb{E}[x^2[k]] \\
&+r_1 \underbrace{((1-2.5^{-1}) \mathbb{E}[u_1^2[k_1]] + (1-2.5^{-1})2.5^{-1}\mathbb{E}[u_1^2[k_1+1]]+ \cdots+(1-2.5^{-1}) 2.5^{-k+1+k_1} \mathbb{E}[u_1^2[k-1]] )}_{:=\widetilde{P_1}}\\
&+r_2 \underbrace{((1-2.5^{-1}) \mathbb{E}[u_2^2[k_2]] + (1-2.5^{-1})2.5^{-1}\mathbb{E}[u_2^2[k_2+1]]+ \cdots+(1-2.5^{-1}) 2.5^{-k+1+k_2} \mathbb{E}[u_2^2[k-1]] )}_{:=\widetilde{P_2}})
\end{align}
Here, we denote the second and the third terms as $\widetilde{P_1}$ and $\widetilde{P_2}$ respectively. As we mentioned in Figure~\ref{fig:finite1}, \ref{fig:finite2} and \ref{fig:finite3}, we will relax the problem in two different ways --- one with state disturbance $w[0]$ and the other one with $w[1]$. Let's start with the former.

$\bullet$ Large deviation idea: As mentioned in Section~\ref{sec:mimowitsen}, we will apply large deviation ideas\footnote{As mentioned before, large deviation theory replaces whole noise sequence as atypical one. However, for the simplicity of computation, we will only replace the last observation noise. Precisely, the Gaussian observation noise $v_2[k_2-1]$ will behave like a uniform observation noise with larger variance with a certain probability. Thus, we can replace $v_2[k_2-1]$ with a uniform random variable with larger variance by multiplying by the corresponding probability. See \cite{Pulkit_Witsen} for the details of the idea.} to $v_2[k_2-1]$. For this, we write $v_2[k_2-1]$ as a mixture of two independent random variables:
\begin{align}
v_2[k_2-1]=C \cdot v_2'[k_2-1] + (1-C) v_2''[k_2-1]
\end{align}
where $C, v_2'[k_2-1], v_2''[k_2-1]$ are independent random variables whose distributions are given as follows:
\begin{align}
&v_2'[k_2-1] \sim Unif[-\sigma_{v2}',\sigma_{v2}']\\
&f_{v_2''[k_2-1]}(v) = \left \{ \begin{array}{ll}
\frac{1}{1-c} \frac{1}{\sqrt{2\pi} \sigma_{v2}} \exp(- \frac{v^2}{2 \sigma_{v2}^2}) & \mbox{for }|v| > \sigma_{v2}' \\
\frac{1}{1-c} \frac{1}{\sqrt{2\pi} \sigma_{v2}} (\exp(- \frac{v^2}{2 \sigma_{v2}^2})-\exp(- \frac{\sigma_{v2}'^2}{2 \sigma_{v2}^2}) ) &\mbox{for }|v| \leq \sigma_{v2}'
\end{array} \right. \\
&C=
\left \{ \begin{array}{ll}
1 &\mbox{w.p. } c \\
0 &\mbox{w.p. } 1-c
\end{array} \right.
\end{align}
where $c = \frac{2\sigma_{v2}'}{\sqrt{2 \pi} \sigma_{v2}} \exp(- \frac{\sigma_{v2}'^2}{2 \sigma_{v2}^2})$.

$\bullet$ Three stage division: As mentioned in Section~\ref{sec:lower:three}, we will divide the finite-horizon problem into three time intervals. 
The following definitions of  $U_{ij}$ correspond to the first and second controller's input in these three intervals shown in Figure~\ref{fig:division}, where the index $i$ and $j$ represent the controllers and time intervals respectively.
\begin{align}
&W:=aw[k-2]+\cdots + a^{k-1}w[0]\\
&U_{11}:=a^{k-2} u_1[1] + \cdots + a^{k-k_1} u_1[k_1-1]\\
&U_{12}:=a^{k-k_1-1} u_1[k_1]+ \cdots + a^{k-k_2} u_1[k_2-1]\\
&U_{13} := a^{k-k_2-1} u_1[k_2]+ \cdots + u_1[k-1] \\
&U_{21}:=a^{k-2} u_2[1] + \cdots + a^{k-k_1} u_2[k_1-1]\\
&U_{22}:=a^{k-k_1-1} u_2[k_1]+ \cdots + a^{k-k_2} u_2[k_2-1]\\
&U_{23} := a^{k-k_2-1} u_2[k_2]+ \cdots + u_2[k-1] \\
&\overline{W} := (w[k-1],w[k-2],\cdots,w[1])
\end{align}
The goal in this proof is grouping control inputs into $U_{ij}$, where each $U_{ij}$ can be thought as either power-limited or information-limited inputs. By expanding $x[n]$, we reveal the effects of the controller inputs on the state, and then isolate (and bound) their effects according to their characteristics.

$\bullet$ Power-Limited Interval: Let's first handle the third power-limited interval using Cauchy-Schwarz inequalities.
Notice that
\begin{align}
x[k]&=w[k-1]+aw[k-2]+\cdots+a^{k-1}w[0]\\
&+ u_1[k-1]+a u_1[k-2] + \cdots + a^{k-2}u_1[1] \\
&+ u_2[k-1]+a u_2[k-2] + \cdots + a^{k-2}u_2[1]
\end{align}
Therefore, by Lemma~\ref{ach:lemmacauchy}
\begin{align}
&\mathbb{E}[x^2[k]]=\mathbb{E}[(W+U_{11}+U_{12}+U_{13}+U_{21}+U_{22}+U_{23})^2]+\mathbf{E}[w^2[k-1]]\\
&\geq (\sqrt{\mathbb{E}[(W+U_{11}+U_{12}+U_{21}+U_{22})^2]}-\sqrt{\mathbb{E}[U_{13}^2]}-\sqrt{\mathbb{E}[U_{23}^2]})^2_+ +1
\label{eqn:ageq412}
\end{align}
Here, we can notice that $\mathbb{E}[(W+U_{11}+U_{12}+U_{21}+U_{22})^2]$ is not affected by the controllers' inputs of the third interval.

$\bullet$ First controller's input in Witsenhausen's interval: We will also separate out the effect of the power-limited (first controller's) input in the second interval, $U_{12}$, and introduce large deviation ideas.
\begin{align}
&\mathbb{E}[(W+U_{11}+U_{12}+U_{21}+U_{22})^2]\\
&=\mathbb{E}[\mathbb{E}[(W+U_{11}+U_{12}+U_{21}+U_{22})^2|C]]\\
&=\mathbb{E}[(W+U_{11}+U_{12}+U_{21}+U_{22})^2|C=1]\mathbb{P}(C=1)+ \mathbb{E}[(W+U_{11}+U_{12}+U_{21}+U_{22})^2|C=0]\mathbb{P}(C=0)\\
&\geq \mathbb{E}[(W+U_{11}+U_{12}+U_{21}+U_{22})^2|C=1]\mathbb{P}(C=1) \\
&= c  \cdot \mathbb{E}[(W+U_{11}+U_{12}+U_{21}+U_{22})^2|C=1]\\
&\geq c ( \sqrt{\mathbb{E}[(W+U_{11}+U_{21}+U_{22})^2|C=1]} - \sqrt{\mathbb{E}[U_{12}^2|C=1]})_+^2 \label{eqn:ageq422}
\end{align}
Here, we can notice that by the causality of the system, $C$ only affects the inputs from $u_2[k_2-1]$ and $u_1[k_2]$. Thus, $u_2[1:k_2-2]$ and $u_1[1:k_2-1]$ are independent of $C$. We can also notice $\mathbb{E}[(W+U_{11}+U_{21}+U_{22})^2|C=1]$ has only information-limited inputs.

$\bullet$ Information-Limited Interval: Using Lemma~\ref{cov:lemma1}, we will bound the remaining uncertainty of the state after the information-limited interval. Since we will grant all disturbances except $w[0]$ as side-information, we denote the relevant observations as $y_1'[n]$ and $y_2'[n]$. Formally, let $y_1'[n]$, $y_2'[n]$, $W'$, $W''$, $U_{22}'$, $U_{22}''$ be as follows:
\begin{align}
&y_1'[n]:=a^{n-1}w[0]+v_1[n]\\
&y_2'[n]:=a^{n-1}w[0]+v_2[n]\\
&W' := W - \mathbb{E}[W |y_1'[1:k_1-1],y_2'[1:k_1-1], \overline{W}, C=1 ] \\
&W'' := \mathbb{E}[W |y_1'[1:k_1-1],y_2'[1:k_1-1], \overline{W}, C=1 ]\\
&U_{22}' := U_{22} - \mathbb{E}[U_{22}| y_1'[1:k_1-1],y_2'[1:k_1-1], \overline{W}, C=1]\\
&U_{22}'' := \mathbb{E}[U_{22}| y_1'[1:k_1-1],y_2'[1:k_1-1], \overline{W}, C=1]
\end{align}
Here we can notice $W, y_1'[1:k_1-1], y_2'[1:k_1-1], \overline{W}$ are independent of $C$ and
\begin{align}
W'=a^{k-1}w[0] - \mathbb{E}[a^{k-1}w[0] |y_1'[1:k_1-1],y_2'[1:k_1-1] ].
\end{align}
Since $w[0], y_1'[1:k_1-1], y_2'[1:k_1-1], \overline{W}$ are jointly Gaussian, $W'$ is independent from $y_1'[1:k_1-1], y_2'[1:k_1-1], \overline{W}$.
By Lemma~\ref{cov:lemma1} we have
\begin{align}
&\mathbb{E}[W'^2|C=1]\\
&=\mathbb{E}[(a^{k-1}w[0] - \mathbb{E}[a^{k-1}w[0] |y_1'[1:k_1-1],y_2'[1:k_1-1] ])^2]\\
&= \frac{a^{2(k-1)}\sigma_{v1}^2}{(1+\frac{\sigma_{v1}^2}{\sigma_{v2}^2})(\frac{a^{2(k_1-2)}(1-a^{-2(k_1-1)})}{1-a^{-2}})+\sigma_{v1}^2} \label{eqn:ageq423}
\end{align}
This lower bounds the uncertainty in the state due to $w[0]$ after the state has been observed through $y_1'[1:k_1-1]$ and $y_2'[1:k_1-1]$.

Note that $y_1[1:k_1-1],y_2[1:k_1-1],\overline{W}$  are functions of $y_1'[1:k_1-1],y_2'[1:k_1-1],\overline{W}$. Therefore, $U_{11}$ and $U_{21}$ are also functions of $y_1'[1:k_1-1],y_2'[1:k_1-1],\overline{W}$. 
Since $(W',U_{22}')$ are orthogonal to all functions of $(y_1'[1:k_1-1],y_2'[1:k_1-1],\overline{W})$, $(W',U_{22}')$ are also orthogonal to $(W'', U_{11}, U_{21}, U_{22}'')$. Moreover, since $\mathbb{E}[W'+U_{22}']=0$ and the conditioning on $C=1$ can be ignored due to causality, we can conclude
\begin{align}
&\mathbb{E}[(W+U_{11}+U_{21}+U_{22})^2|C=1]\\
&=\mathbb{E}[(W'+W''+U_{11}+U_{21}+U_{22}'+U_{22}'')^2 | C=1]\\
&=\mathbb{E}[(W' + U_{22}')^2| C=1] + \mathbb{E}[(W''+U_{11}+U_{21}+U_{22}'')^2 | C=1] \\
&\geq \mathbb{E}[(W' + U_{22}')^2| C=1].
\end{align}
In the last term, the effect of the information-limited interval inputs is separated out.

$\bullet$ Second controller's input in Witsenhausen's interval: We will bound the remaining uncertainty of the state   after it has been estimated by the second controller in the second time interval. For this, we will reduce the problem to the state amplification problem of Section~\ref{sec:lower:amp}, and apply Lemma~\ref{lem:reduced}.

$U_{22}'$ is a function of $y_2[1:k_2-1], y_1'[1:k_1-1], y_2'[1:k_1-1], \overline{W}$. Here, $y_1'[1:k_1-1], y_2'[1:k_1-1], \overline{W}$  are independent from $W'$ and $y_2[1:k_1-1]$ is a function of $y_1'[1:k_1-1], y_2'[1:k_1-1], \overline{W}$. Therefore, only $y_2[k_1:k_2-1]$ are dependent on $W'$. Moreover, $y_1[1:k_1-1]$ ---and therefore, $u_1[1:k_1-1]$--- is a function of $y_1'[1:k_1-1], y_2'[1:k_1-1], \overline{W}$, so they are also independent from $W'$. 

Now, we can subtract the independent part from $W'$ from the observation $y_2[k_1,k_2-1]$ without losing information about the state. First, consider $y_2[k_1]$.
\begin{align}
&y_2[k_1]-(w[k_1-1]+aw[k_1-2]+\cdots+a^{k_1-2}w[1])-\mathbb{E}[a^{k_1-1}w[0]|y_1'[1:k_1-1],y_2'[1:k_1-]]\\
&-(u_1[k_1-1]+au_1[k_1-2]+ \cdots + a^{k_1-2}u_1[1])\\
&-(u_1[k_1-1]+au_1[k_1-2]+ \cdots + a^{k_1-2}u_1[1])\\
&=a^{k_1-1}w[0]-\mathbb{E}[a^{k_1-1}w[0]|y_1'[1:k_1-1], y_2'[1:k_1-1]]+v_2[k_1] \\
&=a^{k_1-k}W' +  v_2[k_1]
\end{align}
Likewise, we can subtract the independent (from $W'$) part from $y_2[k_1+1]$. Furthermore, $u_2[k_1]$ can also be subtracted from $y_2[k_1+1]$ without losing information since the second controller already knows about $u_2[k_1]$. Thus, the information about $W'$ in $y_2[k_1+1]$ is
\begin{align}
&a^{k_1}w[0]-\mathbb{E}[a^{k_1}w[0]|y_1'[1:k_1-1], y_2'[1:k_1-1]]+u_1[k_1] + v_2[k_1+1]\\
&=a^{k_1-k+1}W' + u_1[k_1] + v_2[k_1+1].
\end{align}
In the same way, we can extract the relevant information about $W'$  from the observations $y_2[n]$.
It is worth to mention that conditioned on $C=1$, $v_2[k_2-1]$ is replaced by $v_2'[k_2-1]$, and thus the information about $W'$ in $y_2[k_2-1]$ is 
\begin{align}
&a^{k_2-2}w[0]-\mathbb{E}[a^{k_2-2}w[0]|y_1'[1:k_1-1],y_2'[1:k_1-1]]+u_1[k_2-1]+ au_1[k_2-2]+ \cdots +a^{k_2-k_1-1}u_1[k_1]+v_2'[k_2-1]\\
&=a^{k_2-k-1}W' + u_1[k_2-1]+ au_1[k_2-2]+ \cdots +a^{k_2-k_1-1}u_1[k_1]+v_2'[k_2-1].
\end{align}
Moreover, as we mentioned, the conditioning $C=1$ does not affect $u_1[k_1:k_2-1]$ by causality. We have
\begin{align}
\mathbb{E}[(1-2.5^{-1})u_1^2[k_1]+(1-2.5^{-1})2.5^{-1}u_1^2[k_1+1]+\cdots+(1-2.5^{-1})2.5^{-k_2+k_1+1}u_1^2[k_2-1]|C=1] \leq \widetilde{P_1} \leq 2.5 \widetilde{P_1}.
\end{align}
Therefore, we can see that after removing the independent (from $W'$) part from $y_2[k_1:k_2-1]$ the problem reduces to the state amplification problem of Section~\ref{sec:lower:amp}. By plugging $x[0]=a^{k_1-1}w[0]$, $\sigma_v=\sigma_{v2}$, $\sigma_{v}'=\sigma_{v2}'$, $k=k_2-k_1$, $w=2.5^{-1}$, $P=2.5 \widetilde{P_1}$ and $\sigma_0^2 =\Sigma$ (which comes from \eqref{eqn:ageq423}) in Lemma~\ref{lem:reduced} (iii), we have\footnote{Here, we have to use (iii) of Lemma~\ref{lem:reduced} instead of (i) since in the last observation the SNR (Signal-to-Noise ratio) is too big to apply an arithmetic-geometric inequality together with the previous observations.}
\begin{align}
\mathbb{E}[(W' + U_{22}')^2| C=1] \geq \frac{a^{2(k-k_1)}\Sigma}{2^{2I'(\widetilde{P_1})}}. \label{eqn:ageq414}
\end{align}

$\bullet$ Power-Limited Inputs: As mentioned before, causality implies $C$ is independent from $y_1[1:k_2-1]$ and thus $U_{12}$. Then, we can upper bound the power of the power-limited inputs.
\begin{align}
&\mathbb{E}[U_{12}^2|C=1]=\mathbb{E}[U_{12}^2] \\
&= \mathbb{E}[(a^{k-k_1-1}u_1[k_1]+\cdots + a^{k-k_2}u_1[k_2-1])^2]\\
&= a^{2(k-k_2)}\mathbb{E}[(a^{k_2-k_1-1}u_1[k_1]+\cdots + u_1[k_2-1])^2]\\
&\leq \frac{a^{2(k-k_1-1)}(1-(2.5a^{-2})^{k_2-k_1})}{1-2.5a^{-2}} (\mathbb{E}[u_1^2[k_1]]+\cdots+2.5^{-(k-k_1-1)}\mathbb{E}[u_1^2[k_2-1]])\\
&\leq \frac{a^{2(k-k_1-1)}(1-(2.5a^{-2})^{k_2-k_1})}{1-2.5a^{-2}} \frac{\widetilde{P_1}}{1-2.5^{-1}}  \label{eqn:ageq421}
\end{align}
where the first inequality comes from Lemma~\ref{lem:power} with parameters $a=a$ and $b=2.5^{-1}$. Likewise, by applying Lemma~\ref{lem:power} with paramters $a=a$ and $b=2.5^{-1}$, we have
\begin{align}
&\mathbb{E}[U_{13}^2]=\mathbb{E}[(a^{k-k_1-1}u_1[k_1]+\cdots+u_1[k-1])^2]\\
&\leq \frac{a^{2(k-k_2-1)}(1-(2.5a^{-2})^{k-k_2})}{1-2.5a^{-2}} (\mathbb{E}[u_1^2[k_2]]+ \cdots + 2.5^{-(k-k_2-1)}\mathbb{E}[u_1^2[k-1]])\\
&\leq \frac{a^{2(k-k_2-1)}(1-(2.5a^{-2})^{k-k_2})}{1-2.5a^{-2}} \frac{ 2.5^{k_2-k_1} \widetilde{P_1}}{1-2.5^{-1}} \label{eqn:ageq415}
\end{align}
and
\begin{align}
&\mathbb{E}[U_{23}^2]=\mathbb{E}[(a^{k-k_2-1}u_2[k_2]+\cdots+u_2[k-1])^2]\\
&\leq \frac{a^{2(k-k_2-1)}(1-(2.5a^{-2})^{k-k_2})}{1-2.5a^{-2}} (\mathbb{E}[u_2^2[k_2]]+\cdots+2.5^{-(k-k_2-1)}\mathbb{E}[u_2^2[k-1]])\\
&\leq \frac{a^{2(k-k_2-1)}(1-(2.5a^{-2})^{k-k_2})}{1-2.5a^{-2}} \frac{\widetilde{P_2}}{1-2.5^{-1}}. \label{eqn:ageq416}
\end{align}

$\bullet$ Lower bound from $w[0]$: Finally, by plugging \eqref{eqn:ageq422}, \eqref{eqn:ageq414}, \eqref{eqn:ageq421}, \eqref{eqn:ageq415}, \eqref{eqn:ageq416} into \eqref{eqn:ageq412}
\begin{align}
\mathbb{E}[x^2[k]]& \geq (\sqrt{c \frac{a^{2(k-k_1)}\Sigma}{2^{2I'(\widetilde{P_1})}}}
- \sqrt{ c \frac{a^{2(k-k_1-1)}(1-(2.5a^{-2})^{k_2-k_1})}{1-2.5a^{-2}} \frac{\widetilde{P_1}}{1-2.5^{-1}} } \\
&- \sqrt{ \frac{a^{2(k-k_2-1)}(1-(2.5a^{-2})^{k-k_2})}{1-2.5a^{-2}} \frac{ 2.5^{k_2-k_1} \widetilde{P_1}}{1-2.5^{-1}} }
- \sqrt{ \frac{a^{2(k-k_2-1)}(1-(2.5a^{-2})^{k-k_2})}{1-2.5a^{-2}} \frac{\widetilde{P_2}}{1-2.5^{-1}} })_+^2 +1 \label{eqn:ageq417}
\end{align}

$\bullet$ Lower bound from $w[1]$: As we mentioned in Figure~\ref{fig:finite1}, \ref{fig:finite2} and \ref{fig:finite3}, we will repeat the above derivation for $w[1]$ instead of $w[0]$.

Let's denote
\begin{align}
&\widetilde{U}_{11} := a^{k-2}u_1[1]+\cdots+a^{k-k_1-1}u_1[k_1]\\
&\widetilde{U}_{12} := a^{k-k_1-2}u_1[k_1+1]+\cdots+u_1[k-1]\\
&\widetilde{U}_{21} := a^{k-2}u_2[1]+\cdots+a^{k-k_1-1}u_2[k_1]\\
&\widetilde{U}_{22} := a^{k-k_1-2}u_2[k_1+1]+\cdots+a^{k-k_2}u_2[k_2-1] \\
&\widetilde{W} := (w[k-1], w[k-2], \cdots, w[2], w[0])
\end{align}
Compared with the previous case, $\widetilde{U}_{11}$ and $\widetilde{U}_{21}$ include extra input signals $u_1[k_1]$ and $u_2[k_1]$ since $w[1]$ is generated one time-step later than $w[0]$. $\widetilde{U}_{12}$ includes all power-limited inputs of the first controller.

Like before, $\widetilde{U}_{ij}$ groups the controller inputs into either information-limited or power-limited ones. Then, we will isolate the effect of the inputs $\widetilde{U}_{ij}$ to the state $x[n]$ according to their categories.

$\bullet$ Power-Limited Inputs: Like the previous case, we first isolate the power-limited inputs. However, unlike the previous case, we do not need to introduce any large deviation ideas. By Lemma~\ref{ach:lemmacauchy},
\begin{align}
\mathbb{E}[x^2[k]]&=\mathbb{E}[(W+\widetilde{U}_{11}+\widetilde{U}_{12}+\widetilde{U}_{21}+\widetilde{U}_{22}+U_{23})^2]+1 \\
&\geq (\sqrt{\mathbb{E}[(W+\widetilde{U}_{11}+\widetilde{U}_{21}+\widetilde{U}_{22})^2]} - \sqrt{\mathbb{E}[\widetilde{U}_{12}^2]} - \sqrt{\mathbb{E}[U_{23}^2]})_+^2 + 1  \label{eqn:ageq420}
\end{align}
Now, the resulting $\mathbb{E}[(W+\widetilde{U}_{11}+\widetilde{U}_{21}+\widetilde{U}_{22})^2]$ has only information-limited inputs.

$\bullet$ Information-Limited Interval: Like before, we will bound the remaining uncertainty of the state after the information-limited interval using Lemma~\ref{cov:lemma1}. Denote $\widetilde{y}_1[n]$ and $\widetilde{y}_2[n]$ as follows:
\begin{align}
&\widetilde{y}_1[1] := v_1[1]\\
&\widetilde{y}_2[1] := v_2[1]\\
&\mbox{For }n \geq 2 \\
&\widetilde{y}_1[n] := a^{n-2}w[1]+v_1[n]\\
&\widetilde{y}_2[n] := a^{n-2}w[1]+v_2[n]\\
&W_1' := W - \mathbb{E}[W|\widetilde{y}_1[1:k_1],\widetilde{y}_2[1:k_1],\widetilde{W}]\\
&W_1'' := \mathbb{E}[W|\widetilde{y}_1[1:k_1],\widetilde{y}_2[1:k_1],\widetilde{W}]\\
&\widetilde{U}_{22}' = \widetilde{U} - \mathbb{E}[\widetilde{U}_{22}| \widetilde{y}_1[1:k_1],\widetilde{y}_2[1:k_1],\widetilde{W}]\\
&\widetilde{U}_{22}'' =  \mathbb{E}[\widetilde{U}_{22}| \widetilde{y}_1[1:k_1],\widetilde{y}_2[1:k_1],\widetilde{W}]
\end{align}
Here we can notice
\begin{align}
W_1'= a^{k-2}w[1] - \mathbb{E}[a^{k-2}w[1]|\widetilde{y}_1[1:k_1],\widetilde{y}_2[1:k_1]]
\end{align}
Since $w[1], \widetilde{y}_1[1:k_1], \widetilde{y}_2'[1:k_1],\widetilde{W}$ are jointly Gaussian, $W_1'$ is independent from $\widetilde{y}_1[1:k_1], \widetilde{y}_2'[1:k_1],\widetilde{W}$. By Lemma~\ref{cov:lemma1} we have
\begin{align}
&\mathbb{E}[W_1'^2]\\
&= \mathbb{E}[(a^{k-2}w[1] - \mathbb{E}[a^{k-2}w[1]|\widetilde{y}_1[1:k_1],\widetilde{y}_2[1:k_1]])^2]\\
&= \frac{a^{2(k-2)} \sigma_{v2}^2}
{(1+\frac{\sigma_{v1}^2}{\sigma_{v2}^2})(\frac{a^{2(k_1-2)}(1-a^{-2(k_1-1)})}{1-a^{-2}})+\sigma_{v1}^2} \label{eqn:ageq424}
\end{align}
This lower bounds the remaining state disturbance due to $w[1]$ after it is observed by $\widetilde{y}_1[1:k_1]$ and $\widetilde{y}_2[1:k_1]$.

Note that $y_1[1:k_1], y_2[1:k_1], \widetilde{W}$ are functions of $\widetilde{y}_1[1:k_1], \widetilde{y}_2[1:k_1], \widetilde{W}$. Therefore, $\widetilde{U}_{11}$ and $\widetilde{U}_{21}$ are also functions of $\widetilde{y}_1[1:k_1], \widetilde{y}_2[1:k_1], \widetilde{W}$. By repeating the previous argument, we can conclude
\begin{align}
&\mathbb{E}[(W+\widetilde{U}_{11}+\widetilde{U}_{21}+\widetilde{U}_{22})^2]\\
&=\mathbb{E}[(W_1'+W_1''+ \widetilde{U}_{11}+\widetilde{U}_{21}+\widetilde{U}_{22}'+\widetilde{U}_{22}'' )^2]\\
&=\mathbb{E}[(W_1'+\widetilde{U}_{22}')^2] + \mathbb{E}[(W_1''+ \widetilde{U}_{11}+\widetilde{U}_{21} + \widetilde{U}_{22}'')^2]\\
&\geq \mathbb{E}[(W_1'+\widetilde{U}_{22}')^2].
\end{align}
In the last term, the effect of the information-limited inputs is separated out.

$\bullet$ Second controller's input in Witsenhausen's interval: Like before, we will reduce the problem to the state amplification problem of Section~\ref{sec:lower:amp}, and apply Lemma~\ref{lem:reduced}. Only the observations $y_2[k_1+1:k_2-1]$ are relevant to $W_1'$. Here, we also have the power constraint on $u_1$
\begin{align}
\mathbb{E}[(1-2.5^{-1})u_1^2[k_1+1]+ \cdots + 2.5^{-k_2+k_1+2}u_1^2[k_2-1] ] \leq 2.5 \widetilde{P_1}
\end{align}
Like before, after removing the independent (from $W_1'$) part from the observations $y_2[k_1+1:k_2-1]$, the problem reduces to the state amplification problem of Section~\ref{sec:lower:amp}. By plugging $x[0]=a^{k_1-1}w[1], \sigma_v=\sigma_{v2}, k=k_2-k_1-1, w=2.5^{-1}, P= 2.5 \widetilde{P_1}$ and $\sigma_0^2 = \Sigma$ (which comes from \eqref{eqn:ageq424}) to Lemma~\ref{lem:reduced} (i), we have\footnote{Unlike the previous part, we apply (i) of Lemma~\ref{lem:reduced} instead of (iii) since SNR is small enough for all observations.}
\begin{align}
\mathbb{E}[(W_1'+\widetilde{U}_{22}')^2] \geq \frac{a^{2(k-k_1-1)}\Sigma}{2^{2I''(\widetilde{P_1})}}. \label{eqn:ageq410}
\end{align}

$\bullet$ Lower bound from $w[1]$: By applying Lemma~\ref{lem:power} with the parameters $a=a$ and $b=2.5^{-1}$, we can upper bound the power of the power-limited inputs.
\begin{align}
&\mathbb{E}[\widetilde{U}_{12}^2]= \mathbb{E}[(a^{k-k_1-2}u_1[k_1+1]+\cdots+u_1[k-1])^2] \\
&\leq \frac{a^{2(k-k_1-2)}(1-(2.5a^{-2})^{k-k_1-1})}{1-2.5a^{-2}} \frac{2.5 \widetilde{P_1}}{1-2.5^{-1}} \label{eqn:ageq419}
\end{align}
Therefore, by plugging \eqref{eqn:ageq410}, \eqref{eqn:ageq419}, \eqref{eqn:ageq416} into \eqref{eqn:ageq420} we get
\begin{align}
\mathbb{E}[x^2[k]] &\geq ( \sqrt{\frac{a^{2(k-k_1-1)}\Sigma}{2^{2I''(\widetilde{P_1})}}}
-\sqrt{\frac{a^{2(k-k_1-2)}(1-(2.5a^{-2})^{k-k_1-1})}{1-2.5a^{-2}} \frac{2.5 \widetilde{P_1}}{1-2.5^{-1}}}\\
&-\sqrt{\frac{a^{2(k-k_2-1)}(1-(2.5a^{-2})^{k-k_2})}{1-2.5a^{-2}} \frac{\widetilde{P_2}}{1-2.5^{-1}}} )_+^2 +1 \label{eqn:ageq418}
\end{align}

$\bullet$ Final Lower bound: By \eqref{eqn:ageq417} and \eqref{eqn:ageq418}, for all $0 \leq \alpha \leq 1$
\begin{align}
&\mathbb{E}[x^2[k]] \\
&\geq
\alpha (\sqrt{c \frac{a^{2(k-k_1)}\Sigma}{2^{2I'(\widetilde{P_1})}}}
- \sqrt{ c \frac{a^{2(k-k_1-1)}(1-(2.5a^{-2})^{k_2-k_1})}{1-2.5a^{-2}} \frac{\widetilde{P_1}}{1-2.5^{-1}} } \\
&- \sqrt{ \frac{a^{2(k-k_2-1)}(1-(2.5a^{-2})^{k-k_2})}{1-2.5a^{-2}} \frac{ 2.5^{k_2-k_1} \widetilde{P_1}}{1-2.5^{-1}} }
- \sqrt{ \frac{a^{2(k-k_2-1)}(1-(2.5a^{-2})^{k-k_2})}{1-2.5a^{-2}} \frac{\widetilde{P_2}}{1-2.5^{-1}} })_+^2\\
&+(1-\alpha) ( \sqrt{\frac{a^{2(k-k_1-1)}\Sigma}{2^{2I''(\widetilde{P_1})}}}
-\sqrt{\frac{a^{2(k-k_1-2)}(1-(2.5a^{-2})^{k-k_1-1})}{1-2.5a^{-2}} \frac{2.5 \widetilde{P_1}}{1-2.5^{-1}}}\\
&-\sqrt{\frac{a^{2(k-k_2-1)}(1-(2.5a^{-2})^{k-k_2})}{1-2.5a^{-2}} \frac{\widetilde{P_2}}{1-2.5^{-1}}} )_+^2
+1
\end{align}
\end{proof}

In this lemma, the time-interval from $0$ to $k_1-1$ corresponds to the information-limited interval in Figure~\ref{fig:division}. The time-interval from $k_1$ to $k_2-1$ corresponds to the Witsenhausen's interval in Figure~\ref{fig:division}. The time-interval from $k_2$ to $k$ corresponds to the power-limited interval in Figure~\ref{fig:division}.

\begin{figure}[htbp]
\begin{center}
\includegraphics[width=4.6in]{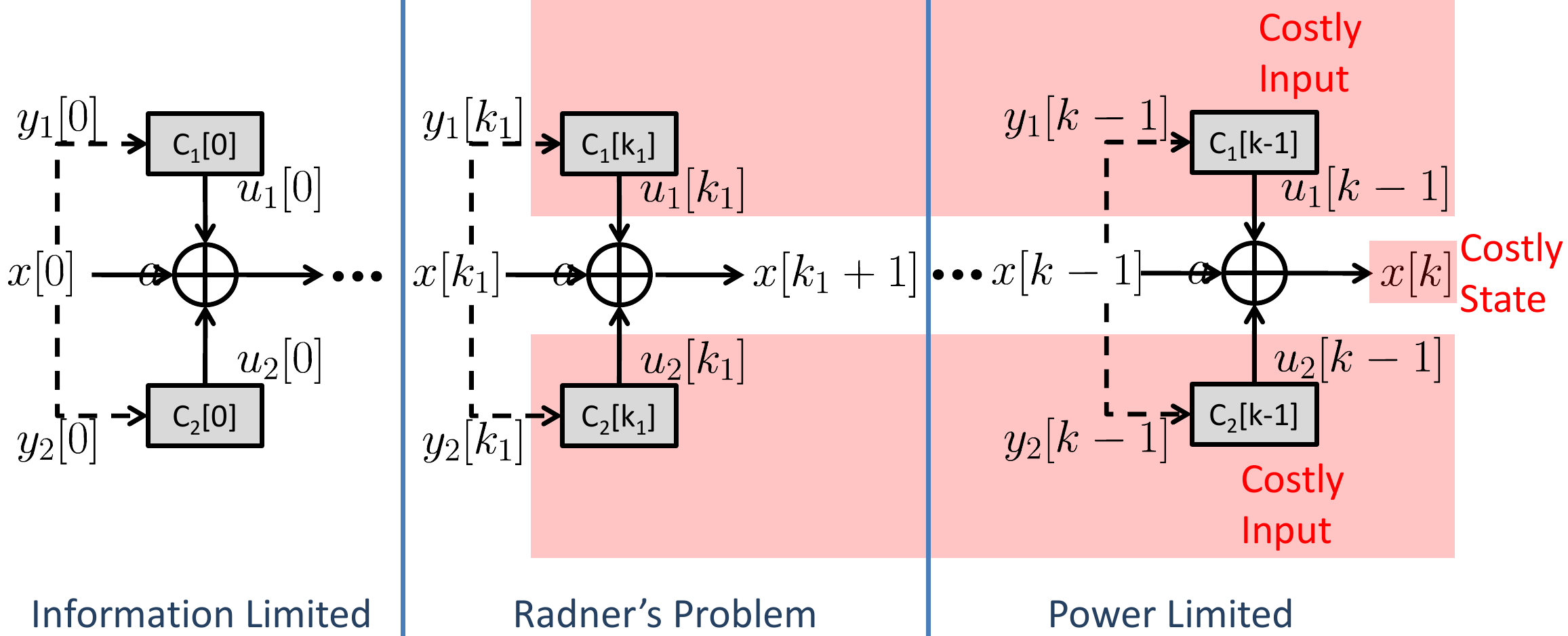}
\caption{The general finite-horizon problem structure to justify the infeasibility of $0$-stage signaling strategies. Like the one in Figure~\ref{fig:division}, the problem consists of three time intervals. However, unlike Figure~\ref{fig:division}, we can see Radner's problem between the information-limited and power-limited intervals.}
\label{fig:division2}
\end{center}
\end{figure}

\subsection{Lower bound on the optimal cost based on Radner's problem}
\label{sec:radnerlower}
As we discussed in Section~\ref{sec:caveat}, Radner's problem cannot be understood using the binary deterministic models and thereby is fundamentally different from Witsenhausen's counterexample. Essentially, it says the communication between controllers requires at least one step delay, and for the observations obtained at the same time step, nonlinear strategies do not improve the performance. Therefore, so-called `$0$-stage signaling' is impossible.

Sine Radner's problem is a sub-block of the infinite-horizon problem~\ref{fig:dyn}, we also need a lower bound based on Radner's problem to bound the infinite-horizon problem within a constant ratio. Figure~\ref{fig:division2} shows the general structure of the lower bound for the case. As we discussed in Figure~\ref{fig:division}, the information-limited interval from time step $0$ to $k_1$ is introduced due to the case $\sigma_{v1}^2 > 0$ and the power-limited interval from time step $k_1+1$ to $k$ is introduced due to the case $P_2 < \infty$.

However, between these two time intervals, we can see the difference. Even though the first controller has better observations and the second has worse observations, if this significant unbalance between two controllers lasts for only one time step, the implicit communication between two controllers is nearly impossible and nonlinear strategies cannot help that much. To capture this effect, we replace the MIMO Witsenhausen's problem with Radner's problem.

Like Lemma~\ref{cov:lemma1}, the following lemma gives a lower bound on the weighted average cost of Problem~\ref{prob:lqg} when Witsenhausen's interval is replaced by Radner's.

\begin{lemma}
Define a set $S_{L,2}$ as a set of $(k_1,k,\Sigma)$ such that
\begin{align}
&k_1, k \in \mathbb{N}, \Sigma \in \mathbb{R},\\
&k_1 \geq 1, k \geq k_1 + 1, \\
&0 \leq \Sigma \leq 
\left\{ 
\begin{array}{ll}
1& k_1 = 1 \\
\frac{a^{2(k_1-1)}\sigma_{v1}^2}{(1+\frac{\sigma_{v1}^2}{\sigma_{v2}^2})(\frac{a^{2(k_1-2)}(1-a^{-2(k_1-1)})}{1-a^{-2}})+\sigma_{v1}^2}
& k_1 \geq 2
\end{array}\right.
\end{align}
We also define $D_{L,2}(\widetilde{P_1},\widetilde{P_2};k_1,k,\Sigma)$ as follows:
\begin{align}
&D_{L,2}(\widetilde{P_1},\widetilde{P_2};k_1,k,\Sigma) \\
&:= \inf_{c_1, c_2 \in \mathbb{R}}(\sqrt{a^{2(k-k_1-1)}((a-c_1-c_2)^2 \Sigma + c_1^2 \sigma_{v1}^2 + c_2^2 \sigma_{v2}^2)}\\
&-\sqrt{\frac{a^{2(k-k_1-2)}(1-(2.5a^{-2})^{k-k_1-1})}{1-2.5a^{-2}} \frac{\widetilde{P_1}}{(1-2.5^{-1})2.5^{-1}}}-\sqrt{\frac{a^{2(k-k_1-2)}(1-(2.5a^{-2})^{k-k_1-1})}{1-2.5a^{-2}} \frac{\widetilde{P_2}}{(1-2.5^{-1})2.5^{-1}}})_+^2 +1 \\
&\mbox{s.t. } (1-2.5^{-1})c_1^2(\Sigma+\sigma_{v1}^2) \leq \widetilde{P_1} \\
&\quad (1-2.5^{-1})c_2^2(\Sigma+\sigma_{v2}^2) \leq \widetilde{P_2} 
\end{align}
Let $|a| \geq 2.5$. Then, for all $q, r_1, r_2 \geq 0$, the minimum cost \eqref{eqn:part11} of Problem~\ref{prob:lqg} is lower bounded as follows:
\begin{align}
&\inf_{u_1,u_2} \limsup_{N \rightarrow \infty} \frac{1}{N}
\sum_{0 \leq n < N} q \mathbb{E}[x^2[n]] + r_1 \mathbb{E}[u_1^2[n]] + r_2 \mathbb{E}[u_2^2[n]]\\
& \geq \sup_{(k_1,k,\Sigma) \in S_{L,2}}\min_{\widetilde{P_1},\widetilde{P_2} \geq 0} q D_{L,2}(\widetilde{P_1},\widetilde{P_2};k_1,k,\Sigma) + r_1 \widetilde{P_1} + r_2 \widetilde{P_2}.
\end{align}
For $k_1, k \in \mathbb{N}$, define $S_{L,3}$, $S_{L,4}$, $D_{L,3}(\widetilde{P_1},\widetilde{P_2};k_1)$ and $D_{L,4}(\widetilde{P_1},\widetilde{P_2},k)$ as follows:
\begin{align}
&S_{L,3} := \{k_1 \in \mathbb{N} \}\\
&S_{L,4} := \{k \in \mathbb{N} : k \geq 2 \}\\
&D_{L,3}(\widetilde{P_1},\widetilde{P_2};k_1) := \max(\frac{a^{2(k_1-1)}\sigma_{v1}^2}{(1+\frac{\sigma_{v1}^2}{\sigma_{v2}^2})(\frac{a^{2(k_1-2)}(1-a^{-2(k_1-1)})}{1-a^{-2}})+\sigma_{v1}^2},1) \label{eqn:lem:radner1}\\
&D_{L,4}(\widetilde{P_1},\widetilde{P_2};k) := ( \sqrt{a^{2(k-1)}} - \sqrt{\frac{a^{2(k-2)}}{1-2.5a^{-2}} \frac{\widetilde{P_1}}{1-2.5^{-1}}} - \sqrt{\frac{a^{2(k-2)}}{1-2.5a^{-2}} \frac{\widetilde{P_2}   }{1-2.5^{-1}}})_+^2. \label{eqn:lem:radner2}
\end{align}
Then, when $|a| \geq 2.5$, for all $q,r_1, r_2 \geq 0$, the minimum cost \eqref{eqn:part11} of Problem~\ref{prob:lqg} is also lower bounded as follows:
\begin{align}
&\inf_{u_1,u_2} \limsup_{N \rightarrow \infty} \frac{1}{N}
\sum_{0 \leq n < N} q \mathbb{E}[x^2[n]] + r_1 \mathbb{E}[u_1^2[n]] + r_2 \mathbb{E}[u_2^2[n]]\\
& \geq \sup_{k_1 \in S_{L,3}}\min_{\widetilde{P_1},\widetilde{P_2} \geq 0} q D_{L,3}(\widetilde{P_1},\widetilde{P_2};k_1) + r_1 \widetilde{P_1} + r_2 \widetilde{P_2}
\end{align}
and
\begin{align}
&\inf_{u_1,u_2} \limsup_{N \rightarrow \infty} \frac{1}{N}
\sum_{0 \leq n < N} q \mathbb{E}[x^2[n]] + r_1 \mathbb{E}[u_1^2[n]] + r_2 \mathbb{E}[u_2^2[n]]\\
& \geq \sup_{k \in S_{L,4}}\min_{\widetilde{P_1},\widetilde{P_2} \geq 0} q D_{L,4}(\widetilde{P_1},\widetilde{P_2};k) + r_1 \widetilde{P_1} + r_2 \widetilde{P_2}.
\end{align}
\label{cov:lemma2}
\end{lemma}
\begin{proof}
For simplicity, we assume $a \geq 2.5$, $k_1 \geq 2$, $k > k_1 +1$. The remaining cases when $a \leq -2.5$ or $k_1=1$ or $k = k_1 +1$ easily follow with minor modifications.

$\bullet$ Geometric Slicing: We apply the geometric slicing idea of Section~\ref{sec:lower:geo} to get a finite-horizon problem. By putting $\alpha=2.5^{-1}$ and $k_2=k_1$ to Lemma~\ref{lem:geo}, the average cost is lower bounded by
\begin{align}
&\inf_{u_1, u_2} (q \mathbb{E}[x^2[k]] \\
&+r_1 \underbrace{( (1-2.5^{-1}) \mathbb{E}[u_1^2[k_1]] + (1-2.5^{-1})2.5^{-1}\mathbb{E}[u_1^2[k_1+1]]+ \cdots+(1-2.5^{-1}) 2.5^{-k+1+k_1} \mathbb{E}[u_1^2[k-1]] )}_{:=\widetilde{P_1}}\\
&+r_2 \underbrace{( (1-2.5^{-1}) \mathbb{E}[u_2^2[k_1]] + (1-2.5^{-1})2.5^{-1}\mathbb{E}[u_2^2[k_1+1]]+ \cdots+(1-2.5^{-1}) 2.5^{-k+1+k_1} \mathbb{E}[u_2^2[k-1]] )}_{:=\widetilde{P_2}}
\end{align}
Like the proof of Lemma~\ref{cov:lemma2}, we denote the second and the third terms as $\widetilde{P_1}$ and $\widetilde{P_2}$ respectively.

$\bullet$ Power-Limited Interval: Denote
\begin{align}
&W_2 := a^{k-k_1-2} w[k_1+1]+\cdots+ a w[k-2]\\
&U_{12} := a^{k-k_1-2}u_1[k_1+1]+\cdots+u_1[k-1] \\
&U_{22} := a^{k-k_1-2}u_2[k_1+1]+\cdots+u_2[k-1] \\
\end{align}
Here, $U_{12}$ and $U_{22}$ correspond to the first and second controller's input in the power-limited intervals described in Figure~\ref{fig:division2}. We will first handle these power-limited inputs. Notice that
\begin{align}
x[k]&=a^{k-k_1-1}x[k_1+1]+a^{k-k_1-2}u_1[k_1+1]+\cdots+u_1[k-1]+a^{k-k_1-2}u_2[k_1+1]+\cdots+u_2[k-1]\\
&+a^{k-k_1-2}w[k_1+1]+\cdots+w[k-1].
\end{align}
Since $x[k_1+1]$ and $W_2$ are independent by causality, using Lemma~\ref{ach:lemmacauchy} we can lower bound $\mathbb{E}[x^2[k]]$ as
\begin{align}
&\mathbb{E}[x^2[k]]\\
&=\mathbb{E}[(a^{k-k_1-1}x[k_1+1]+U_{12}+U_{22}+W_2)^2]+1\\
&\geq ( \sqrt{\mathbb{E}[(a^{k-k_1-1}x[k_1+1]+W_2)^2]} - \sqrt{\mathbb{E}[U_{12}^2]} - \sqrt{\mathbb{E}[U_{22}^2]} )_+^2 +1\\
&= ( \sqrt{\mathbb{E}[(a^{k-k_1-1}x[k_1+1])^2] + \mathbb{E}[W_2^2]} - \sqrt{\mathbb{E}[U_{12}^2]} - \sqrt{\mathbb{E}[U_{22}^2]} )_+^2 +1\\
&\geq ( \sqrt{a^{2(k-k_1-1)}\mathbb{E}[x[k_1+1]^2]} - \sqrt{\mathbb{E}[U_1^2]} - \sqrt{\mathbb{E}[U_2^2]} )_+^2 +1. \label{eqn:radner1}
\end{align}
Here,  $\mathbb{E}[x[k_1+1]^2]$ is lower bounded as
\begin{align}
&\mathbb{E}[x[k_1+1]^2] = \mathbb{E}[(ax[k_1]+u_1[k_1]+u_2[k_1]+w[k_1])^2] \\
&=\mathbb{E}[(ax[k_1]+u_1[k_1]+u_2[k_1])^2] + \mathbb{E}[w[k_1]^2] \\
&\geq \mathbb{E}[(ax[k_1]+u_1[k_1]+u_2[k_1])^2]. \label{eqn:radner5}
\end{align}
In the last term, the effect of the power-limited inputs is separated out.

$\bullet$ Information-Limited Interval: Using Lemma~\ref{cov:lemma1}, we will bound the remaining uncertainty of the state after the information-limited interval. Since we will give all the disturbances except $w[0]$ as side-information, we denote the relevant observations as $y_1'[n]$ and $y_2'[n]$. Formally, denote
\begin{align}
&W_1 := a^{k_1-1}w[0] + \cdots + w[k_1-1]\\
&U_{11} := a^{k_1-1}u_1[0]+ \cdots + u_1[k_1-1]\\
&U_{21} := a^{k_1-1}u_2[0]+ \cdots + u_2[k_1-1]\\
&\overline{W} := (w[1],\cdots,w[k-1]) \\
&y_1'[n]:=a^{n-1}w[0]+v_1[n]\\
&y_2'[n]:=a^{n-1}w[0]+v_2[n]\\
&W_1' := W_1 - \mathbb{E}[W_1|y_1'[1:k_1-1],y_2'[1:k_1-1],\overline{W}] \\
&W_1'' := \mathbb{E}[W_1|y_1'[1:k_1-1],y_2'[1:k_1-1],\overline{W}] \\
&u_1'[k_1] := u_1[k_1] - \mathbb{E}[u_1[k_1]|y_1'[1:k_1-1],y_2'[1:k_1-1],\overline{W}]\\
&u_1''[k_1] := \mathbb{E}[u_1[k_1]|y_1'[1:k_1-1],y_2'[1:k_1-1],\overline{W}]\\
&u_2'[k_1] := u_2[k_1] - \mathbb{E}[u_2[k_1]|y_1'[1:k_1-1],y_2'[1:k_1-1],\overline{W}]\\
&u_2''[k_1] := \mathbb{E}[u_2[k_1]|y_1'[1:k_1-1],y_2'[1:k_1-1],\overline{W}].
\end{align}
Here, we have
\begin{align}
W_1' = a^{k_1-1}w[0]- \mathbb{E}[a^{k_1-1}w[0]|y_1'[1:k_1-1], y_2'[1:k_1-1]]
\end{align}
Since $w[0], y_1'[1:k_1-1], y_2'[1:k_1-1], \overline{W}$ are jointly Gaussian, $W_1'$ is independent from $ y_1'[1:k_1-1], y_2'[1:k_1-1], \overline{W}$. By Lemma~\ref{cov:lemma1} we have
\begin{align}
\mathbb{E}[W_1'^2] = \frac{a^{2(k_1-1)}\sigma_{v1}^2}{(1+\frac{\sigma_{v1}^2}{\sigma_{v2}^2})(\frac{a^{2(k_1-2)}(1-a^{-2(k_1-1)})}{1-a^{-2}})+\sigma_{v1}^2} \label{eqn:radner6}
\end{align}
This lower bounds the state disturbance due to $w[0]$ when it is observed by $y_1'[1:k_1-1]$ and $y_2'[1:k_1-1]$. Note that $y_1[1:k_1-1], y_2[1:k_1-1], \overline{W}$ is a function of $y_1'[1:k_1-1], y_2'[1:k_1-1], \overline{W}$. Therefore, $U_{11}$ and $U_{21}$ are also functions of $y_1'[1:k_1-1], y_2'[1:k_1-1], \overline{W}$ and \eqref{eqn:radner5} can be lower bounded as
\begin{align}
&\mathbb{E}[(ax[k_1]+u_1[k_1]+u_2[k_1])^2]\\
&=\mathbb{E}[(a(W_1+U_{11}+U_{12})+u_1[k_1]+u_2[k_1])^2]\\
&=\mathbb{E}[(aW_1' + u_1'[k_1] + u_2'[k_1])^2] + \mathbb{E}[(aW_1'' + aU_{11}+aU_{12} + u_1''[k_1] + u_2''[k_1])^2]\\
&\geq \mathbb{E}[(aW_1' + u_1'[k_1] + u_2'[k_1])^2]
\end{align}
In the last term, the effect of the information-limited inputs is separated out.

$\bullet$ Radner's Interval: Now we will reduce the last term to Radner's problem.
$u_1'[k_1]$ and $u_2'[k_1]$ are functions of $y_1[1:k_1],y_1'[1:k_1-1],y_2'[1:k_1-1],\overline{W}$ and $y_2[1:k_1],y_1'[1:k_1-1],y_2'[1:k_1-1],\overline{W}$ respectively. Here, $y_1'[1:k_1-1],y_2'[1:k_1-1],\overline{W}$ are independent from $W_1'$ and $y_1[1:k_1-1], y_2[1:k_1-1], \overline{W}$ is a function of $y_1'[1:k_1-1], y_2'[1:k_1-1], \overline{W}$. Therefore, only $y_1[k_1]$ at the first controller and $y_2[k_1]$ at the second controller are relevant to $W_1'$.
Therefore, by removing independent parts from $W_1'$ in $y_1[k_1]$, the sufficient statistic of $y_1[k_1]$ is
\begin{align}
&y_1[k_1]- (w[k_1-1]+aw[k_1-2]+\cdots+a^{k_1-2}w[1]) - \mathbb{E}[a^{k_1-1}w[0]|y_1'[1:k_1-1], y_2'[1:k_1-1]]\\
&-(u_1[k_1-1]+ a u_1[k_1-2]+ \cdots + a^{k_1-2}u_1[1])\\
&-(u_2[k_1-1]+ a u_2[k_1-2]+ \cdots + a^{k_1-2}u_2[1])\\
&=a^{k_1-1}w[0] - \mathbb{E}[a^{k_1-1}w[0]| y_1'[1:k_1-1], y_2'[1:k_1-1]] + v_1[k_1]\\
&= W_1' + v_1[k_1]
\end{align}
Likewise, $y_2[k_1]$ can be reduced to
\begin{align}
W_1' + v_2[k_1]
\end{align}
Therefore, by considering $W_1'$ as an initial state, $v_1[k_1]$ and $v_2[k_1]$ as observation noise of the first and second controller, we can map the problem with Radner's problem.
Here, we have the following power constraints on $u_1[k_1]$ and $u_2[k_1]$.
\begin{align}
&(1-2.5^{-1})\mathbb{E}[u_1^2[k_1]] \leq \widetilde{P_1}\\
&(1-2.5^{-1})\mathbb{E}[u_2^2[k_1]] \leq \widetilde{P_2}
\end{align}
Since a linear strategy is optimal in Radner's problem, by \eqref{eqn:radner6} we can conclude
\begin{align}
\mathbb{E}[(aW_1' + u_1'[k_1] + u_2'[k_1])^2] &\geq
\inf_{c_1, c_2 \in \mathbb{R}} (a-c_1-c_2)^2 \Sigma + c_1^2 \sigma_{v1}^2 + c_2^2 \sigma_{v2}^2 \label{eqn:radner2}\\
&\mbox{s.t. } (1-2.5^{-1})c_1^2(\Sigma+\sigma_{v1}^2) \leq \widetilde{P_1} \\
&\quad (1-2.5^{-1})c_2^2(\Sigma+\sigma_{v2}^2) \leq \widetilde{P_2}
\end{align}

$\bullet$ Final Lower bound:
Applying Lemma~\ref{lem:power} with paramters $a=a$ and $b=2.5^{-1}$, we can upper bound the power of the power-limited inputs.
\begin{align}
\mathbb{E}[U_1^2] &= \mathbb{E}[(a^{k-k_1-2}u_1[k_1+1]+\cdots + u_1[k-1])^2] \\
&\leq \frac{a^{2(k-k_1-2)}(1-(2.5a^{-2})^{k-k_1-1})}{1-2.5a^{-2}} \frac{\widetilde{P_1}}{(1-2.5^{-1})2.5^{-1}} \label{eqn:radner3}
\end{align}
and likewise
\begin{align}
\mathbb{E}[U_1^2]
&\leq \frac{a^{2(k-k_1-2)}(1-(2.5a^{-2})^{k-k_1-1})}{1-2.5a^{-2}} \frac{\widetilde{P_2}}{(1-2.5^{-1})2.5^{-1}} \label{eqn:radner4}
\end{align}
Finally, plugging \eqref{eqn:radner2}, \eqref{eqn:radner3}, \eqref{eqn:radner4} into \eqref{eqn:radner1} gives the first bound based on $D_{L,2}(\widetilde{P_1},\widetilde{P_2};k_1,k,\Sigma)$:
\begin{align}
\mathbb{E}[x^2[n]] &\geq \inf_{c_1, c_2}(\sqrt{a^{2(k-k_1-1)}((a-c_1-c_2)^2 \Sigma + c_1^2 \sigma_{v1}^2 + c_2^2 \sigma_{v2}^2)}\\
&-\sqrt{\frac{a^{2(k-k_1-2)}(1-(2.5a^{-2})^{k-k_1-1})}{1-2.5a^{-2}} \frac{\widetilde{P_1}}{(1-2.5^{-1})2.5^{-1}}}-\sqrt{\frac{a^{2(k-k_1-2)}(1-(2.5a^{-2})^{k-k_1-1})}{1-2.5a^{-2}} \frac{\widetilde{P_2}}{(1-2.5^{-1})2.5^{-1}}})_+^2 +1 \\
&\mbox{s.t. } (1-2.5^{-1})c_1^2(\Sigma+\sigma_{v1}^2) \leq \widetilde{P_1} \\
&\quad (1-2.5^{-1})c_2^2(\Sigma+\sigma_{v2}^2) \leq \widetilde{P_2}
\end{align}
The second bound based on $D_{L,3}(\widetilde{P_1},\widetilde{P_2};k_1)$ derived as follows. 
Since $\mathbb{E}[x^2[n]] \geq \mathbb{E}[w^2[n-1]]=1$, trivially $D_L(\widetilde{P_1},\widetilde{P_2}) \geq 1$.
Moreover, as justified above, we have
\begin{align}
\mathbb{E}[x^2[k_1]] \geq \mathbb{E}[(a^{k_1-1}w[0]- \mathbb{E}[a^{k_1-1}w[0]| y_1'[1:k_1-1], y_2'[1:k_1-1])^2] = \mathbb{E}[W_1'^2].
\end{align}
Therefore, by setting $k=k_1$ we get the second bound based on $D_{L,3}(\widetilde{P_1},\widetilde{P_2};k_1)$.

The last bound based on $D_{L,4}(\widetilde{P_1},\widetilde{P_2};k)$ of the lemma can be derived as follows.
\begin{align}
&\mathbb{E}[x^2[k]] \\
&\geq (\sqrt{\mathbb{E}[(a^{k-1}w[0]+ \cdots + w[k-1])^2]} - \sqrt{\mathbb{E}[(a^{k-1}u_1[0]+ \cdots + u_1[k-1])^2]} -\sqrt{\mathbb{E}[(a^{k-1}u_2[0]+ \cdots + u_2[k-1])^2]})_+^2 \\
&\geq ( \sqrt{a^{2(k-1)}} - \sqrt{\frac{a^{2(k-2)}}{1-2.5a^{-2}} \frac{\widetilde{P_1}}{1-2.5^{-1}}} - \sqrt{\frac{a^{2(k-2)}}{1-2.5a^{-2}} \frac{\widetilde{P_2}   }{1-2.5^{-1}}})_+^2
\end{align}
where the first inequality follows from Lemma~\ref{ach:lemmacauchy} and the second inequality follows from Lemma~\ref{lem:power}.
\end{proof}

In this lemma, the time-interval from $0$ to $k_1-1$ corresponds to the information-limited interval in Figure~\ref{fig:division2}. The time-interval from $k_1$ to $k_1+1$ corresponds to the Radner's interval in Figure~\ref{fig:division2}. The time-interval from $k_1+1$ to $k$ corresponds to the power-limited interval in Figure~\ref{fig:division2}.

\section{Constant Ratio Optimality}
\label{sec:proof:ratio}
Now, we have an upper and lower bound on $D(P_1, P_2)$. In this section, we will evaluate the bounds and prove Theorem~\ref{thm:1} which bounds the weighted average cost within a constant ratio. Even though the numerical evaluations are not elegant\footnote{The bounds can probably be improved and tightened. However, the main concern of this paper is not quantifying the exact cost, but qualitatively understanding the near-optimal strategies. The constant ratio optimality results are enough to justify our intuition behind the proposed strategies.}, these are enough to justify constant ratio optimality.

The upper bounds are written from the power-disturbance tradeoff perspective of Problem~\ref{prob:power} and denoted by $(D_U(P_1,P_2),P_1,P_2)$. The lower bounds in Lemma~\ref{cov:lemma3} and \ref{cov:lemma2} are given for the original weighted average-cost of Problem~\ref{prob:lqg},  which can be written as $(D_{L,i}(\widetilde{P_1},\widetilde{P_2}),\widetilde{P_1},\widetilde{P_2})$ from the power-disturbance perspective. The following lemma tells if these two tradeoff regions are within a constant ratio of each other as regions in $\mathbb{R}^3$, i.e. $\exists c\geq 1$ such that $(D_U(cP_1,cP_2),cP_1,cP_2) \leq c\cdot ( D_{L,i}(P_1,P_2),P_1,P_2)$, then the average cost can be characterized to within a constant ratio.

\begin{lemma}
For two functions $D_L(\widetilde{P_1}, \widetilde{P_2})$ and $D_U(P_1,P_2)$, let there exist $c \geq 1$ such that for all $x_1,x_2 \geq 0$
\begin{align}
D_U(c x_1, c x_2) \leq c \cdot D_L(x_1, x_2).
\end{align}
Then, for all $q, r_1 , r_2 \geq 0$, the following inequality holds.
\begin{align}
\min_{P_1, P_2 \geq 0} q D_U(P_1,P_2)+r_1 P_1 + r_2 P_2 \leq c(\min_{\widetilde{P_1}, \widetilde{P_2} \geq 0} q D_L(\widetilde{P_1},\widetilde{P_2}) + r_1 \widetilde{P_1} + r_2 \widetilde{P_2})
\end{align}
\label{rat:lemma1}
\end{lemma}
\begin{proof}
Let $P_1^\star$ and $P_2^\star$ achieve the minimum of the right term of the inequality, i.e.
\begin{align}
&\min_{\widetilde{P_1}, \widetilde{P_2} \geq 0} q D_L(\widetilde{P_1},\widetilde{P_2}) + r_1 \widetilde{P_1} + r_2 \widetilde{P_2} \\
&= q D_L(P_1^\star,P_2^\star) + r_1 P_1^\star + r_2 P_2^\star.
\end{align}
Then, we have
\begin{align}
&c (\min_{\widetilde{P_1}, \widetilde{P_2} \geq 0} q D_L(\widetilde{P_1},\widetilde{P_2}) + r_1 \widetilde{P_1} + r_2 \widetilde{P_2})\\
&=c\cdot(q D_L(P_1^\star,P_2^\star) + r_1 P_1^\star + r_2 P_2^\star) \\
&\geq q D_U(cP_1^\star,cP_2^\star) + r_1 (cP_1^\star)+ r_2 (cP_2^\star)\\
&\geq \min_{P_1,P_2} q D_U(P_1,P_2) + r_1 P_1 + r_2 P_2
\end{align}
where the first inequality comes from the assumption of the lemma. Thus, the lemma is proved.
\end{proof}

We will show that the proposed strategies of Defintion~\ref{def:lin} and \ref{def:sig} solve the weighted average cost problem of Problem~\ref{prob:lqg} to within a constant ratio. 
Let's call the case when $\sigma_{v2}^2 \leq \max(1,a^2 \sigma_{v1}^2)$ the \textit{weakly-degraded-observation} case since the gap between the two controllers' observation noises is not too huge and the second controller can observe what the first controller observed only after one-time step. Likewise, we will call the case when $\sigma_{v2}^2 > \max(1,a^2 \sigma_{v1}^2)$ the \textit{strongly-degraded-observation} case since the gap between the observation noises is larger.

The weakly-degraded-observation case will be discussed in Section~\ref{sec:proof:ratio1} and the strongly-degraded-observation case will be covered in Section~\ref{sec:proof:ratio2}.

\subsection{Weakly-Degraded-Observation case with $|a| \geq 2.5$}
\label{sec:proof:ratio1}
Let's first consider the weakly-degraded case when $\sigma_{v2}^2 \leq \max(1,a^2 \sigma_{v1}^2)$, which  corresponds to the left half plane of Figure~\ref{fig:region} of page~\pageref{fig:region}.  
In this case, the infeasibility of $0$-stage signaling discussed in Section ~\ref{sec:caveat} and \ref{sec:radnerlower} shows up and thus linear strategies are enough for constant-ratio optimality.

First, we evaluate the lower bound of Lemma~\ref{cov:lemma2} which involves Radner's problem.
\begin{corollary}
Let $|a| \geq 2.5$ and $\sigma_{v2}^2 \leq \max(1,a^2 \sigma_{v1}^2)$. Then, for all $q, r_1, r_2 \geq 0$, the minimum cost \eqref{eqn:part11} of Problem~\ref{prob:lqg} is lower bounded as follows:
\begin{align}
\inf_{u_1,u_2} \limsup_{N \rightarrow \infty}\frac{1}{N}
\sum_{0 \leq n < N} q \mathbb{E}[x^2[n]] + r_1 \mathbb{E}[u_1^2[n]] + r_2 \mathbb{E}[u_2^2[n]]
 \geq \min_{\widetilde{P_1},\widetilde{P_2} \geq 0} q D_L(\widetilde{P_1},\widetilde{P_2}) + r_1 \widetilde{P_1} + r_2 \widetilde{P_2}
\end{align}
where $D_L(\widetilde{P_1},\widetilde{P_2})$ satisfies the following conditions.\\
(a) If $\widetilde{P_1} \leq \frac{1}{400}a^2 \max(1,a^2 \sigma_{v1}^2)$ and $\widetilde{P_2} \leq \frac{1}{400}a^2 \max(1,a^2 \sigma_{v2}^2)$ then $D_L(\widetilde{P_1},\widetilde{P_2}) = \infty$.\\
(b) If $\widetilde{P_1} \leq \frac{1}{400}a^2 \max(1,a^2 \sigma_{v1}^2)$, for all $\widetilde{P_2}$, $D_L(\widetilde{P_1},\widetilde{P_2}) \geq  0.176 a^2 \sigma_{v2}^2+1$.\\
(c) For all $\widetilde{P_1}$ and $\widetilde{P_2}$, $D_L(\widetilde{P_1},\widetilde{P_2}) \geq 0.295  \cdot \max(1,a^2 \sigma_{v1}^2)$.
\label{rat:lemma2}
\end{corollary}
\begin{proof}
See Appendix~\ref{sec:lemma2}.
\end{proof}

(a) and (b) tell what happens if the first controller has little power (i.e. it must follow something close to a zero-input strategy). (a) shows if the second controller does not have enough power, the system becomes unstable. (b) shows even if the second controller has enough power, the state variance is lower bounded by the second controller's observation noise. (c) shows even if the first controller has enough power to apply a zero-forcing strategy, the state variance is lower bounded by the first controller's observation noise.

The following lemma analyzes the achievable disturbance by the simple linear strategy of Definition~\ref{def:lin}.
\begin{lemma}
Consider a single-controller scalar system
\begin{align}
&x[n+1]=ax[n]+u[n]+w[n]   \\
&y[n]=x[n]+v[n]
\end{align}
where $w[n]$ is i.i.d. $\mathcal{N}(0,1)$ and $v[n]$  is i.i.d. $\mathcal{N}(0,\sigma_v^2)$. For a given control strategy, let $D(P):=\limsup_{n \rightarrow \infty}\frac{1}{N}\sum_{0 \leq n < N} \mathbb{E}[x^2[n]]$ and $P := \limsup_{n \rightarrow \infty} \frac{1}{N}\sum_{0 \leq n < N} \mathbb{E}[u^2[n]]$. Then,
\begin{align}
(D(P),P) \leq (a^2 \sigma_v^2 + 1, a^4 \sigma_v^2 + a^2 \sigma_v^2 + a^2)
\end{align}
is achievable by a linear bang-bang controller, $u[n]=-ay[n]$. Therefore, in Problem~\ref{prob:power} the following power-disturbance tradeoffs are achievable.
\begin{align}
&(D(P_1,P_2),P_1, P_2) \leq (a^2 \sigma_{v1}^2 + 1, a^4 \sigma_{v1}^2 + a^2 \sigma_{v1}^2 + a^2,0), \\
&(D(P_1,P_2),P_1, P_2) \leq (a^2 \sigma_{v2}^2 + 1, 0, a^4 \sigma_{v2}^2 + a^2 \sigma_{v2}^2 + a^2).
\end{align}
\label{ach:lemmabb}
\end{lemma}
\begin{proof}
Put $u[n]=-ay[n]$ into the system equation. Then, we have
\begin{align}
x[n+1]&=ax[n]-ax[n]-av[n]+w[n]\\
&=-av[n]+w[n]
\end{align}
Thus, we conclude for $n \geq 1$
\begin{align}
\mathbb{E}[x^2[n]]=a^2 \sigma_v^2 +1
\end{align}
and
\begin{align}
\mathbb{E}[u^2[n]]&= a^2 \mathbb{E}[(x[n]+v[n])^2] \\
&=a^2(a^2 \sigma_v^2 + 1 + \sigma_v^2)
\end{align}
\end{proof}

Using Lemma~\ref{rat:lemma1}, Corollary~\ref{rat:lemma2} and Lemma~\ref{ach:lemmabb}, we can compare the upper and lower bounds to prove linear strategies suffice to achieve constant-ratio optimality in this region of problem parameters.
\begin{proposition}
There exists $c \geq 1200$ such that for all $a, q, r_1, r_2, \sigma_0,\sigma_{v1}, \sigma_{v2}$ satisfying $|a| \geq 2.5$ and $\sigma_{v2}^2 \leq \max(1,a^2 \sigma_{v1}^2)$, the following inequality holds:
\begin{align}
\frac{\inf_{u_1,u_2 \in L_{lin,bb}} \limsup_{N \rightarrow \infty} \frac{1}{N}\sum_{0 \leq n < N} \mathbb{E}[q x^2 [n] + r_1 u_1^2[n] + r_2 u_2^2[n]]}{ \inf_{u_1,u_2} \limsup_{N \rightarrow \infty} \frac{1}{N}\sum_{0 \leq n < N} \mathbb{E}[q x^2 [n] + r_1 u_1^2[n] + r_2 u_2^2[n]]}  \leq c.
\end{align}
\label{prop:2}
\end{proposition}
\begin{proof}
See Appendix~\ref{sec:prop:2}.
\end{proof}

\subsection{Strongly-Degraded-Observation case with $|a| \geq 2.5$}
\label{sec:proof:ratio2}
Let's consider the strongly-degraded-observation case when $\sigma_{v2}^2 > \max(1,a^2 \sigma_{v1}^2)$, which corresponds to the right half-plane of Figure~\ref{fig:division}. Since $|a| \geq 2.5$, we can find $s \in \mathbb{N}$ such that $a^{2(s-1)}\max(1,a^2 \sigma_{v1}^2) \leq \sigma_{v2}^2 \leq a^{2s}\max(1,a^2 \sigma_{v1}^2)$. We will show that the $s$-stage signaling strategy is required for constant-ratio optimality.


Since we need a matching lower bound to $s$-stage signaling strategies, we evaluate Lemma~\ref{cov:lemma2} which has a generalized Witsenhausen's counterexample in it.
\begin{corollary}
Let $|a| \geq 2.5$ and for some $s \in \mathbb{N}$, suppose
\begin{align}
a^{2(s-1)} \max(1,a^2 \sigma_{v1}^2) \leq \sigma_{v2}^2 \leq a^{2s} \max(1,a^2 \sigma_{v1}^2).
\end{align}
Then, for all $q, r_1, r_2 \geq 0$, the minimum cost \eqref{eqn:part11} of Problem~\ref{prob:lqg} is lower bounded as follows:
\begin{align}
\inf_{u_1,u_2} \limsup_{N \rightarrow \infty}\frac{1}{N}
\sum_{0 \leq n < N} q \mathbb{E}[x^2[n]] + r_1 \mathbb{E}[u_1^2[n]] + r_2 \mathbb{E}[u_2^2[n]]
 \geq \min_{\widetilde{P_1},\widetilde{P_2} \geq 0} q D_L(\widetilde{P_1},\widetilde{P_2}) + r_1 \widetilde{P_1} + r_2 \widetilde{P_2}.
\end{align}
where $D_L(\widetilde{P_1},\widetilde{P_2})$ satisfies the following conditions.\\
(a) When $\widetilde{P_1} \leq \frac{\sigma_{v2}^2}{70a^{2(s-1)}}$, then
$D_L(\widetilde{P_1},\widetilde{P_2}) \geq 0.008 a^2 \sigma_{v2}^2 + 1$.
\\
(b) When $\widetilde{P_1} \leq \frac{\sigma_{v2}^2}{70a^{2(s-1)}}$ and $\widetilde{P_2} \leq \frac{a^4 \sigma_{v2}^2}{28000}$, then $D_L(\widetilde{P_1},\widetilde{P_2})=\infty$.
\\
(c) When $\frac{\sigma_{v2}^2}{70a^{2(s-1)}} \leq \widetilde{P_1} \leq \frac{1}{20000}\max(a^2,a^4 \sigma_{v1}^2)$,\\
then $D_L(\widetilde{P_1},\widetilde{P_2}) \geq  0.2541 a^{2s}\widetilde{P_1}\exp(- \frac{50 a^{2(s-1)}\widetilde{P_1}}{\sigma_{v2}^2})+0.066 a^{2s}\max(1,a^2\sigma_{v1}^2) + 1$.
\\
(d) When $\frac{\sigma_{v2}^2}{70a^{2(s-1)}} \leq \widetilde{P_1} \leq \frac{1}{20000}\max(a^2,a^4 \sigma_{v1}^2)$ and
$\widetilde{P_2} \leq 0.0457a^{2(s+1)}\widetilde{P_1} \exp(-\frac{50 a^{2(s-1)}\widetilde{P_1}}{\sigma_{v2}^2}) + 0.0113 a^{2(s+1)} \max(1,a^2\sigma_{v1}^2) $, 
then $D_L(\widetilde{P_1},\widetilde{P_2})=\infty$.\\
(e) For all $\widetilde{P_1}$ and $\widetilde{P_2}$, $D_L(\widetilde{P_1},\widetilde{P_2}) \geq 0.295 \cdot \max(1,a^2 \sigma_{v1}^2)$.
\label{rat:lemma3}
\end{corollary}
\begin{proof}
See Appendix~\ref{sec:lemma3}.
\end{proof}

(a) and (b) tell what happens if the first controller has little power and thus is forced to be close to a zero-input strategy. Even if the second controller has enough power, the state variance is lower bounded by the second controller's observation noise. If the second controller does not have enough power to stabilize the system, the state diverges to infinity. (e) shows the opposite case when the first controller has enough power to apply zero-forcing strategy. However, even in this case, the state variance is lower bounded by the first controller's observation noise. (c) and (d) cover the case between these two extreme cases. (c) gives the lower bound that matches to the $s$-stage signaling strategy when the second controller has enough power. (d) shows that since the first controller does not stabilize the system with its signaling alone, the second controller's input power has to be large enough to stabilize the system.

Now, we evaluate the performance of the $s$-stage signaling analyzed in Lemma~\ref{ach:lemma9} of page~\pageref{ach:lemma9}.
\begin{corollary}
Consider Problem~\ref{prob:power} of page~\pageref{prob:power}, and let $|a| \geq 2.5$ and suppose  $a^{2(s-1)} \max(1,a^2 \sigma_{v1}^2) \leq \sigma_{v2}^2 \leq a^{2s}\max(1,a^2 \sigma_{v1}^2)$ for some $s \in \mathbb{N}$. Then, there exists an upper bound $D_U(P_1,P_2)$ on $D(P_1,P_2)$ i.e. $D(P_1,P_2) \leq D_U(P_1,P_2)$ for all $P_1, P_2 \geq 0$ satisfying the following:
\begin{align}
(D_U(P_1,P_2),P_1,P_2) \leq &( 832 a^{2s}P \exp(-\frac{50a^{2(s-1)}P}{\sigma_{v2}^2})+63a^{2s} \max(1,a^2 \sigma_{v1}^2), \\
& 80000P, 6656 a^{2(s+1)}P \exp(-\frac{50a^{2(s-1)}P}{\sigma_{v2}^2})+564 a^{2(s+1)} \max(1,a^2 \sigma_{v1}^2))
\end{align}
for $\frac{\sigma_{v2}^2}{70a^{2(s-1)}} \leq P \leq \frac{1}{20000}\max(a^2, a^4 \sigma_{v1}^2)$.
\label{rat:ach}
\end{corollary}
\begin{proof}
See Appendix~\ref{sec:rat:ach}
\end{proof}
Here we can notice that the performance is matching that of Corollary~\ref{rat:lemma3} (c), (d) in that the bounds on the state disturbance take the same form of a function on $P_1$ and system parameters.

Now, we can compare these two bounds to prove constant-ratio optimality.
\begin{proposition}
There exists $c \leq 1.5 \times 10^5$ such that for all $a, q, r_1, r_2, \sigma_0,\sigma_{v1}, \sigma_{v2}$ satisfying $|a| \geq 2.5$ and 
\begin{align}
a^{2(s-1)} \max(1,a^2 \sigma_{v1}^2) \leq \sigma_{v2}^2 \leq a^{2s} \max(1,a^2 \sigma_{v1}^2).
\end{align}
for some $s \in \mathbb{N}$, the following inequality holds:
\begin{align}
\frac{\inf_{u_1,u_2 \in L_{lin,bb} \cup L_{sig,s}} \limsup_{N \rightarrow \infty} \frac{1}{N}\sum_{0 \leq n < N} \mathbb{E}[q x^2 [n] + r_1 u_1^2[n] + r_2 u_2^2[n]]}{ \inf_{u_1,u_2} \limsup_{N \rightarrow \infty} \frac{1}{N}\sum_{0 \leq n < N} \mathbb{E}[q x^2 [n] + r_1 u_1^2[n] + r_2 u_2^2[n]]}  \leq c.
\end{align}
\label{prop:3}
\end{proposition}
\begin{proof}
See Appendix~\ref{sec:prop:3}
\end{proof}

Now, Theorem~\ref{thm:1} immediately follows from Propositions~\ref{prop:2} and \ref{prop:3}.

\begin{proof}
[Proof of Theorem~\ref{thm:1}]
Propostion~\ref{prop:2} covers the case when $\sigma_{v2}^2 \leq \max(1,a^2\sigma_{v1}^2)$.
Propostion~\ref{prop:3} covers the case when $\sigma_{v2}^2 > \max(1,a^2\sigma_{v1}^2)$, since in this case there exists $s \in \mathbb{N}$ such that $a^{2(s-1)} \max(1,a^2 \sigma_{v1}^2) \leq \sigma_{v2}^2 \leq a^{2s} \max(1,a^2 \sigma_{v1}^2)$.
\end{proof}

\section{Connection to Wireless Communication Theory}
\label{sec:wireless}
\begin{figure*}[htbp]
\begin{center}
        \begin{subfigure}[b]{0.7\textwidth}
                \centering
                \includegraphics[width=\textwidth]{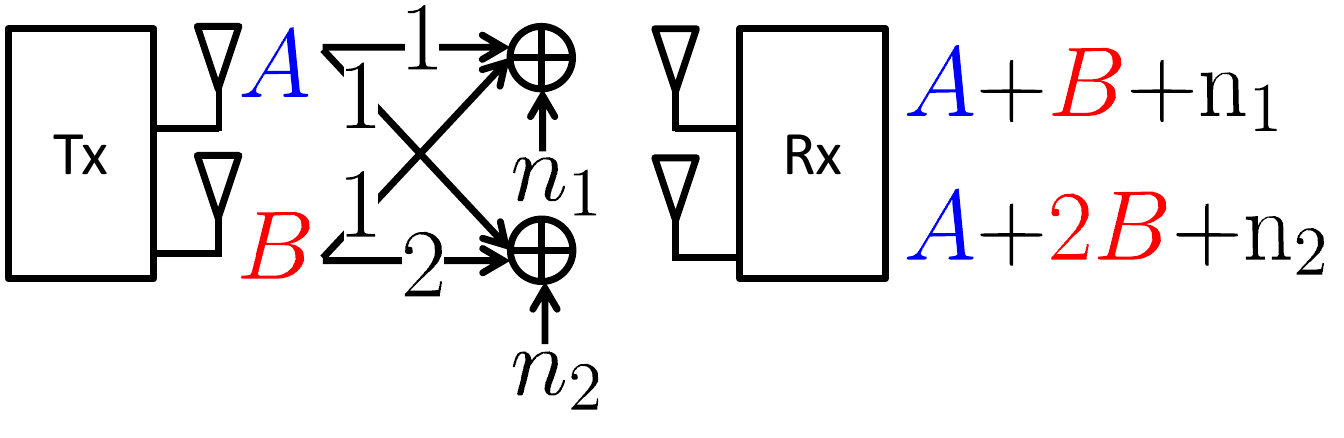}
                \caption{By transmitting different signals across the antennas, we can achieve `d.o.f. gain'. Generally, this scheme performs well in high-SNR.}
                \label{fig:MIMO_multi}
        \end{subfigure}
        \begin{subfigure}[b]{0.7\textwidth}
                \centering
                \includegraphics[width=\textwidth]{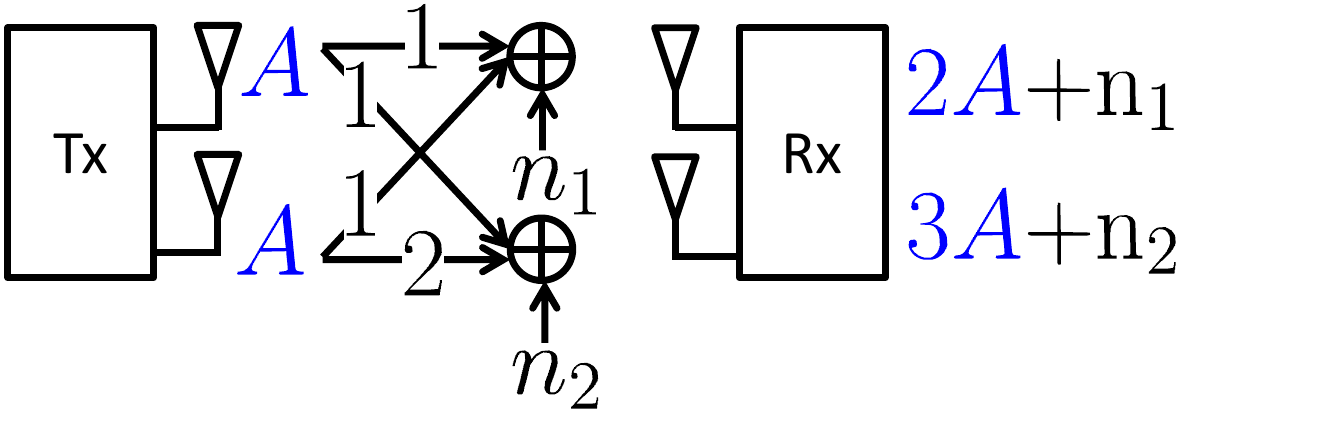}
                \caption{By transmitting the same signal across the antennas, we can achieve `power gain'. Generally, this scheme performs well in low-SNR.}
                \label{fig:MIMO_power}
        \end{subfigure}
\caption{MIMO Wireless Communication Problem}
\label{fig:MIMO}
\end{center}
\end{figure*}
Throughout the discussion, we have observed a lot of similarity between wireless communication and the decentralized LQG control problems considered in this paper. In this section, we will explore this point in more detail. At first glance, decentralized LQG control and wireless communication seem pretty distinct from each other. But the main result in this paper is actually a manifestation of a deeper connection.

The essence of wireless communication problems~\cite{Tse} can be summarized as follows: First, unlike wired communication, wireless communication systems share a common channel and as a result the signals from different transmitting antennas can interact with each other. Second, wireless communication systems involve uncertainties or randomness that come from channel fading or thermal noise in circuits. Third, to extend battery life and minimize interference to other transceivers, each transmitting antenna has a power constraint.

The way that wireless communication theory models capture this nature is very similar to stochastic control theory. First, the interaction between the signals is modeled by linear operations. Second, the uncertainty in the system is modeled by Gaussian random variables. Third, the power of the transmitting antennas is measured by a quadratic cost. If we remember that wireless communication systems are by nature distributed, wireless communication problems are essentially a special case of decentralized LQG control problems, except that wireless communication problems have the special objective of communication.

Like decentralized LQG problems, wireless network communication problems are still open~\cite{Etkin_Interference, Salman_Wireless} or nonconvex~\cite{yu2006dual}. However, wireless communication theorists found that it is helpful to divide cases  according to the SNR(Signal-to-Noise Ratio).  For a given communication scheme, the capacity of a channel is usually given as $\log(1+c_1 SNR+ \cdots + c_k SNR^{k})$. Therefore, when SNR is large (high-SNR case), the capacity is approximately $k \log SNR$ (where $k$ turns out to be the `d.o.f. gain' of the scheme). When SNR is small (low-SNR case), the capacity is approximately $c_1 SNR$ (where $c_1$ turns out to be the `power gain' of the scheme). Therefore, depending on the SNR the capacity of communication schemes are very different. Thus, wireless communication theory usually divides into two cases: (1) high-SNR (2) low-SNR.

Let's consider a $2 \times 2$ MIMO communication problem of Figure~\ref{fig:MIMO}. We can think of two basic ways of exploiting these two antennas. The first way is transmitting different signals across different antennas. As we can see in Fig.~\ref{fig:MIMO_multi}, in this case the receiver will have two variables and two (noisy) equations, and we can expect `MIMO gain' by solving for multiple variables. In wireless communication theory, this gain is called the `d.o.f.(degree-of-freedom) gain' and the scheme of Figure~\ref{fig:MIMO_multi} succeeds in increasing $k$ in the capacity formula.

As we mentioned in Section~\ref{sec:generalizeddof}, this concept can be extended to \textit{generalized d.o.f.} by allowing the transmitting powers of different antennas to scale differently~\cite{Etkin_Interference}. When the transmitting powers of different antennas scale differently, we can further divide a single receiving antenna according to ``signal levels". For the small signal level, all tranmitting antennas can affect it, but for the large signal level, only the few transmitting antennas with large power can affect it. In~\cite{Salman_Wireless}, binary deterministic models were proposed to capture this phenomenon by conceptualizing different bit-levels like different antennas, which we used in Section~\ref{sec:intui}.

The second way of using two antennas is transmitting the same signal across different antennas as shown in Fig.~\ref{fig:MIMO_power}. In this scheme, the receiver will have only one variable and we cannot expect the d.o.f. gain of solving for multiple variables. However, there is a gain to be had from aligning the signals. Let's assume all random variables, $A, B, n_1, n_2$, are Gaussian random variables with zero mean and unit variance, and compute the signal-to-noise ratio at the receiving antennas. The SNR of the first receive antenna in Fig.~\ref{fig:MIMO_multi} is $\frac{\mathbb{E}[(A+B)^2]}{\mathbb{E}[n_1^2]}=2$. On the other hand, the SNR of the first antenna in Fig.~\ref{fig:MIMO_power} is $\frac{\mathbb{E}[(2A)^2]}{\mathbb{E}[n_1^2]}=4$. Therefore, by transmitting the same signal over different antennas we can increase SNR of the received signals. This gain is known as `power gain' in wireless communication theory and the proposed scheme is good for increasing $c_1$ in the capacity formula. To exploit the power gain, the receiver has to introduce maximum-ratio combining~\cite{Tse}.

\begin{figure*}[htbp]
\begin{center}
\includegraphics[width=2.5in]{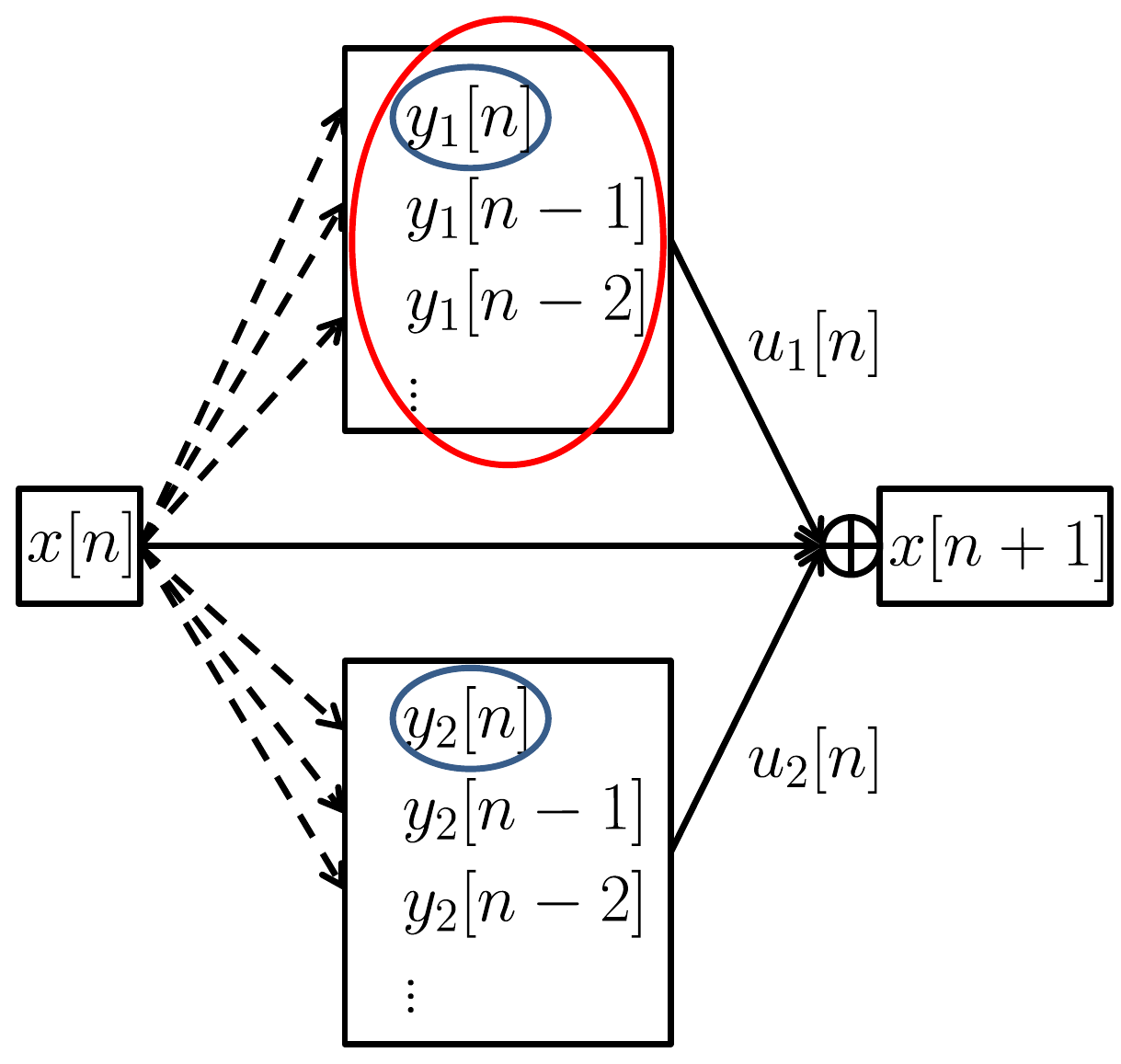}
\caption{As indicated by the blue circles, in the fast dynamics case the proposed scheme exploits  ``information flow" from both controllers but ignores all past observations. By contrast, as indicated by the red circle, in the slow dynamics case the proposed scheme exploits the information of only one controller but takes into account all past observations.}
\label{fig:MIMO_LQG}
\end{center}
\end{figure*}

How is this relevant for scalar decentralized LQG control problems? To control a plant we first have to gain information about its state. The quantitative behavior of this information flow (from a plant to controllers and finally back to the plant~\cite{Park_Equivalence}) is very similar to that in wireless communication systems. More precisely, according to the \textbf{eigenvalue} of the system, the information flow in the system shows a very different behavior. The system is deemed to be fast-dynamics when the eigenvalue is large ($|a| \geq 2.5$) and slow-dynamics when the eigenvalue is small ($|a| < 2.5$).

The main reason for this division is the relationship between the eigenvalue of the system and the SNR of the information flow for control. The discussion of Section~\ref{sec:intui} reveals that the SNR of implicit communication between two controllers will be bounded\footnote{Since the second controller can cancel all bits above its noise level at the next time step, the new information of the state cannot be amplified more than $|a|^2$ above the second controller's noise level. Thus, the SNR measured at the second controller is always bounded by $|a|^2$.} by the eigenvalue squared ($|a|^2$).
Therefore, when the eigenvalue is large, the SNR for implicit communication is also large. Therefore, from wireless communication theory we can expect that the (generalized) d.o.f. gain of the implicit communication is crucial. Likewise, when the eigenvalue is small, the SNR for implicit communication is also small and the power gain of the implicit communication is crucial. This is the slow dynamics case. To harness the power gain, we have to use Kalman filtering which corresponds to maximum-ratio combining in wireless communication~\cite{Tse}.

In short, even if the system is the simplest scalar system, we can think of two ways of sending information. One way is across different bit-levels and the other way is across different time-slots. Moreover, these multiple bit-levels and multiple time-slots can be deemed as MIMO antennas in wireless communication theory. In fast-dynamics, the MIMO antenna gain of multiple bit-levels dominates that of multiple time-slots. The proposed signaling strategies exploit the d.o.f. gain of the MIMO antennas over multiple bit-levels.

On the other hand, in slow-dynamics, the MIMO antenna gain of multiple time-slots is much more crucial. In \cite{Park_Approximation_Journal_Partii}, Kalman filtering is used to exploit the power gain of the MIMO antennas over multiple time-slots. 

Figure~\ref{fig:MIMO_LQG} visualizes the discussion so far. In fast-dynamics, the state is quickly changing and the SNR of implicit communication is high. Thus, the information from previous time steps is much less important than that of the current time step. However, to fully exploit the MIMO antenna gain of different bit-levels, the observations from both controllers has to be used.\footnote{Even though the strategy in Definition~\ref{def:sig} relies on the past controller inputs (therefore, the past observations), the role of the past inputs in the strategy is just to cancel their influence on the current time step, not providing information about the state.} On the other hand, in slow-dynamics, the state changes slowly and the SNR of implicit communication is low. Therefore, there is no huge incentive for implicit communication between controllers, and a strategy which fully exploits the observations of either one controller is enough to achieve constant-ratio optimality. However, the power gain from the past observations cannot be ignored and so Kalman filtering has to be used~\cite{Park_Approximation_Journal_Partii}.

It is worth mentioning that this fundamental difference between fast and slow dynamics was conjectured as early as the 1970s~\cite{Sandell_Survey} but it remains vague: \textit{``The
development of systematic procedures for appropriately
modeling large scale systems with slow and fast dynamics
has not received the attention it deserves. $\cdots$ one should look for
time scale separation (fast and slow dynamics)."} This paper is the first that has used this quantitatively. 

The division of fast and slow dynamics base on $2.5$ is somewhat surprising if we remind that Witsenhausen's counterexample corresponds to $a=1$ case in the infinite horizon problem.
In \cite{Pulkit_Witsen} it was shown that we need nonlinear strategy to achieve a constant ratio optimality in Witsenhasuen's counterexample, while in \cite{Park_Approximation_Journal_Partii} it is shown that in the slow dynamics case including $a=1$ linear strategies are enough for a constant ratio optimality. The main reason for this is that the infinite horizon problem is a sequential problem. Since the problem is sequential, we can think the infinite-horizon problem as an interlocking of a series of Witsenhausen's counterexamples. When $a=1$, the interference from the previous Witsenhausen's problem is too strong and we do not have to solve the current Witsenhausen's problem optimally.

\section{Discussion and Further Research}
\label{sec:discuss}
\begin{figure}[htbp]
\begin{center}
\includegraphics[width=1.6in]{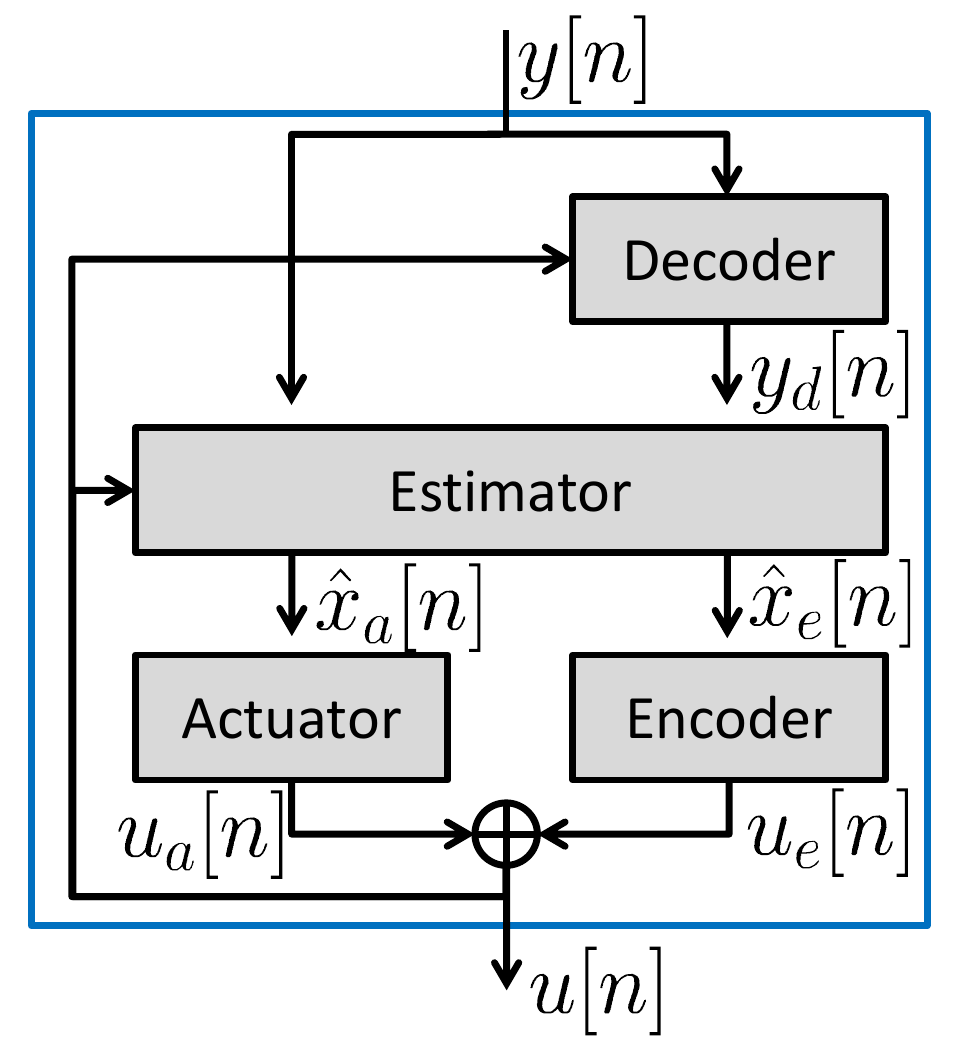}
\caption{Communication-Estimation-Control Separation Controllers}
\label{fig:str}
\end{center}
\end{figure}


\begin{figure}[htbp]
\begin{center}
\includegraphics[width=3in]{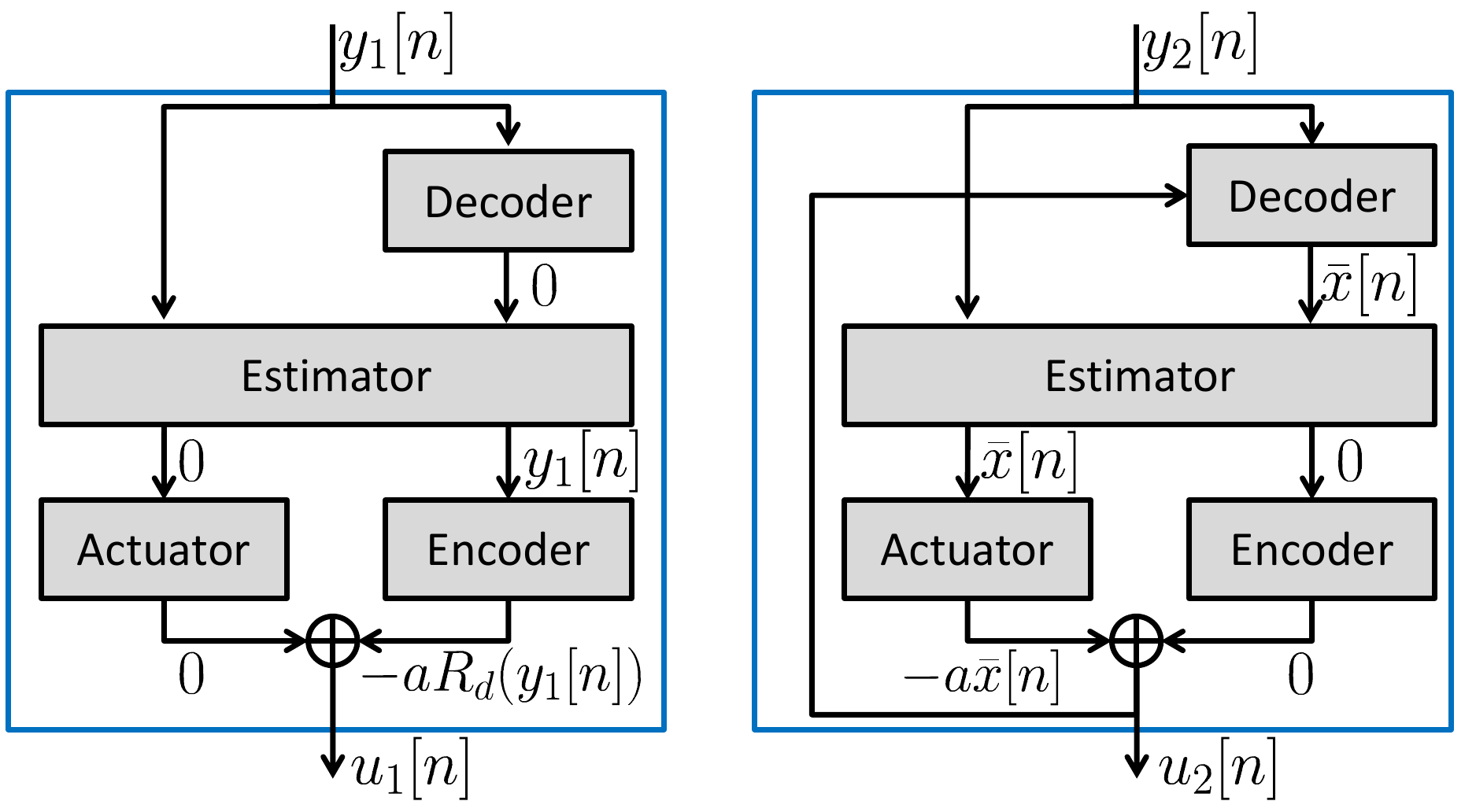}
\caption{Communication-Estimation-Control separation controller interpretation of $L_{sig,s}$ (Here,
$\bar{x}[n]=Q_{a^s d}(y_2[n]-R_{a^s d}(\sum_{1 \leq i \leq s} a^{i-1}u_2[n-i]))+R_{a^s d}(\sum_{1 \leq i \leq s} a^{i-1}u_2[n-i])$)}
\label{fig:str2}
\end{center}
\end{figure}

In the beginning of the paper, we summarized the two main contributions of the classical centralized LQG result. The first was linear controller optimality which narrows the search for the optimal strategy from the infinite-dimensional strategy space to the finite-dimensional linear strategy space. In Theorem~\ref{thm:1}, we gave the corresponding result for scalar decentralized LQG problems in an approximation sense by proposing a finite-dimensional strategy space.

The second contribution was more philosophical. Centralized LQG gave us the
separation principle for estimation and control. Therefore, a natural question is whether we can interpret the result of Theorem~\ref{thm:1} in terms of estimation-control separation, or if there is a conceptual missing block. The authors believe there is a missing fundamental design block, the ``communication" block.

Figure~\ref{fig:str} shows the proposed
communication-estimation-control separation controller, which we
believe, is \textbf{approximately} optimal. First, the controller
observes $y[n]$. Unlike the centralized case, $y[n]$ may contain
transmitted ``signals" from the other controllers. The decoder block
extracts such information and generates a new observation
$y_d[n]$. Based on both $y[n]$ and $y_d[n]$, the estimator block tries
to estimate the states. After the estimation, the controller can
either control the states\footnote{Even a single scalar state can be viewed
  as a collection of bits-positions.} by itself, or relay information
to the other controllers and let them control. $\hat{x}_a[n]$ is the
states that the controller wants to control by itself. Based on
$\hat{x}_a[n]$, the actuator generates the control action
$u_a[n]$. $\hat{x}_e[n]$ is the state that the controller wants to
encode for the other controllers. Based on $\hat{x}_e[n]$, the encoder generates the encoded signal $u_e[n]$. Finally, the control output is the superposition of $u_a[n]$ and $u_e[n]$.

Figure~\ref{fig:str2} interprets the strategy $L_{sig,s}$ based on the proposed controller structure. The strategy exploits the fact that the controller $1$ has a better observation than the controller $2$. Since the controller $1$'s control signal is expensive, it ``relays" its observation through the encoder rather than trying to control the state by itself. Then, the controller $2$ extracts the relayed information in the decoder block, and takes action based on it. We can notice that only encoders and decoders are nonlinear, while estimators and actuators are linear. Therefore, this structure fits the intuition that the essential nonlinearity comes from communication.


An extensive relationship between the implicit information flow for control and wireless information flow was discussed in Section~\ref{sec:wireless}. Some unique features of information flows for control was also noticed. We also found the counterpart of the classical notion of information theoretic cutset bounds~\cite{Cover} in dynamic programming context. The geometric slicing idea discussed in Figure~\ref{fig:geometric} can be thought as a cutset bound in a sense that it finds the informational bottleneck of the system. However, unlike traditional information theoretic cutsets, the geometric slicing idea divides the nodes by a weighted cut rather than a simple partitioning.


Even if this paper focused on the simplest toy scalar LQG problem with two controllers, the essential difficulty of decentralized problems --- nonconvex optimization over infinite-dimensional space --- was still there and we could finesse this difficulty by taking an approximation approach.
We believe the approaches and techniques developed in this paper will also be useful in more 
general problems with vector states and multiple controllers.
Moreover, in the process of such generalization, we will find more close relationship and parallelism between wireless information flows and control information flows. For example, the notion of the computation over communication channels~\cite{Bobak_Computation} or interference alignment~\cite{maddah2008communication, cadambe2008jafar} has to be properly understood in control contexts. Above all, by solving the problems only approximately, we ``may'' be able to make a breakthrough in this long-standing open problem, the decentralized LQG problem. 

\section{Acknowledgements}
The first author thanks Samsung Scholarship for the generous support during his graduate study. The authors would like to thank the CDC reviewers for their comments and NSF for grants NSF-07291222, NSF-0917212, and NSF-0932410.
\bibliographystyle{IEEEtran}
\bibliography{IEEEabrv,seyongbib}

\begin{thebibliography}{10}
\providecommand{\url}[1]{#1}
\csname url@rmstyle\endcsname
\providecommand{\newblock}{\relax}
\providecommand{\bibinfo}[2]{#2}
\providecommand\BIBentrySTDinterwordspacing{\spaceskip=0pt\relax}
\providecommand\BIBentryALTinterwordstretchfactor{4}
\providecommand\BIBentryALTinterwordspacing{\spaceskip=\fontdimen2\font plus
\BIBentryALTinterwordstretchfactor\fontdimen3\font minus
  \fontdimen4\font\relax}
\providecommand\BIBforeignlanguage[2]{{%
\expandafter\ifx\csname l@#1\endcsname\relax
\typeout{** WARNING: IEEEtran.bst: No hyphenation pattern has been}%
\typeout{** loaded for the language `#1'. Using the pattern for}%
\typeout{** the default language instead.}%
\else
\language=\csname l@#1\endcsname
\fi
#2}}

\bibitem{Park_easier}
S.~Y. Park and A.~Sahai, ``It may be easier to approximate decentralized
  infinite-horizon {LQG} problems,'' in \emph{Proceedings of the 51st IEEE
  Conference on Decision and Control (CDC)}, dec. 2012, pp. 2250--2255.

\bibitem{KumarVaraiya}
P.~R. Kumar and P.~Varaiya, \emph{Stochastic Systems: Estimation,
  Identification, and Adaptive Control}.\hskip 1em plus 0.5em minus 0.4em\relax
  New Jersey: Prentice Hall, 1986.

\bibitem{Witsenhausen_Counterexample}
H.~S. Witsenhausen, ``A counterexample in stochastic optimum control,''
  \emph{{SIAM} Journal on Control}, vol.~6, no.~1, pp. 131--147, Jan. 1968.

\bibitem{Ho_Teams}
Y.-C. Ho, M.~P. Kastner, and E.~Wong, ``Teams, signaling, and information
  theory,'' \emph{{IEEE} Trans. Automat. Contr.}, vol.~23, no.~2, Apr. 1978.

\bibitem{Sandell_Survey}
J.~Sandell, N., P.~Varaiya, M.~Athans, and M.~Safonov, ``Survey of
  decentralized control methods for large scale systems,'' \emph{Automatic
  Control, IEEE Transactions on}, vol.~23, no.~2, pp. 108 -- 128, apr 1978.

\bibitem{Radner_Team}
R.~Radner, ``Team decision problems,'' \emph{The Annals of Mathematical
  Statistics}, vol.~33, no.~3, pp. pp. 857--881.

\bibitem{Wiener}
N.~Wiener, \emph{Cybernetics, or Communication and Control in the Animal and
  the Machine}.\hskip 1em plus 0.5em minus 0.4em\relax Cambridge MIT Press,
  1948.

\bibitem{Witsenhausen_Separation}
H.~S. Witsenhausen, ``Separation of estimation and control for discrete time
  systems,'' \emph{Proceedings of the {IEEE}}, vol.~59, no.~11, pp. 1557--1566,
  Nov. 1971.

\bibitem{Yuksel_Nested}
S.~Yuksel, ``Stochastic nestedness and the belief sharing information
  pattern,'' \emph{Automatic Control, IEEE Transactions on}, vol.~54, no.~12,
  pp. 2773 --2786, dec. 2009.

\bibitem{Rotkowitz_characterization}
M.~Rotkowitz and S.~Lall, ``A characterization of convex problems in
  decentralized control,'' \emph{Automatic Control, IEEE Transactions on},
  vol.~50, no.~12, pp. 1984 --1996, dec. 2005.

\bibitem{Youla}
D.~Youla, H.~Jabr, and J.~Bongiorno, J., ``Modern {Wiener-Hopf} design of
  optimal controllers--part ii: The multivariable case,'' \emph{Automatic
  Control, IEEE Transactions on}, vol.~21, no.~3, pp. 319 -- 338, jun 1976.

\bibitem{Shah_Partial}
P.~Shah and P.~Parrilo, ``A partial order approach to decentralized control,''
  in \emph{Decision and Control, 2008. CDC 2008. 47th IEEE Conference on}, dec.
  2008, pp. 4351 --4356.

\bibitem{Lessard_state}
L.~Lessard and S.~Lall, ``A state-space solution to the two-player
  decentralized optimal control problem,'' in \emph{Communication, Control, and
  Computing (Allerton), 2011 49th Annual Allerton Conference on}, sept. 2011,
  pp. 1559 --1564.

\bibitem{Rotkowitz_information}
M.~Rotkowitz, ``On information structures, convexity, and linear optimality,''
  in \emph{Decision and Control, 2008. CDC 2008. 47th IEEE Conference on}, dec.
  2008, pp. 1642 --1647.

\bibitem{Bansal_affine}
R.~Bansal and T.~Basar, ``Stochastic teams with nonclassical information
  revisited: When is an affine law optimal?'' \emph{Automatic Control, IEEE
  Transactions on}, vol.~32, no.~6, pp. 554 -- 559, jun 1987.

\bibitem{Gastpar_Tocode}
M.~Gastpar, B.~Rimoldi, and M.~Vetterli, ``To code, or not to code: lossy
  source-channel communication revisited,'' \emph{Information Theory, IEEE
  Transactions on}, vol.~49, no.~5, pp. 1147 -- 1158, may 2003.

\bibitem{Wu_Witsenhausen}
Y.~Wu and S.~Verdu, ``Witsenhausen's counterexample: A view from optimal
  transport theory,'' in \emph{Decision and Control and European Control
  Conference (CDC-ECC), 2011 50th IEEE Conference on}, dec. 2011, pp. 5732
  --5737.

\bibitem{Sahai_Witsenhausen}
S.~K. Mitter and A.~Sahai, ``Information and control: Witsenhausen revisited,''
  \emph{Learning, Control and Hybrid Systems: Lecture Notes in Control and
  Information Sciences 241}, pp. 281--293, 1999.

\bibitem{Ho_Witsenhausen}
J.~Lee, E.~Lau, and Y.-C. Ho, ``The {Witsenhausen} counterexample: a
  hierarchical search approach for nonconvex optimization problems,''
  \emph{Automatic Control, IEEE Transactions on}, vol.~46, no.~3, pp. 382
  --397, mar 2001.

\bibitem{Baglietto_Numerical}
Baglietto, M.~Parisini, and R.~T.~Zoppoli, ``Numerical solutions to the
  {Witsenhausen} counterexample by approximating networks,'' \emph{{IEEE}
  Trans. Automat. Contr.}, vol.~46, no.~9, 2001.

\bibitem{Shamma_Learning}
N.~Li, J.~Marden, and J.~Shamma, ``Learning approaches to the {Witsenhausen}
  counterexample from a view of potential games,'' in \emph{Decision and
  Control, 2009 held jointly with the 2009 28th Chinese Control Conference.
  CDC/CCC 2009. Proceedings of the 48th IEEE Conference on}, dec. 2009, pp. 157
  --162.

\bibitem{Karlsson_Iterative}
J.~Karlsson, A.~Gattami, T.~Oechtering, and M.~Skoglund, ``Iterative
  source-channel coding approach to {Witsenhausen's} counterexample,'' in
  \emph{American Control Conference (ACC), 2011}, 29 2011-july 1 2011, pp. 5348
  --5353.

\bibitem{Pulkit_Witsen}
P.~Grover, S.~Y. Park, and A.~Sahai, ``The finite-dimensional {Witsenhausen}
  counterexample,'' \emph{{IEEE} Trans. Automat. Contr.}, To appear.

\bibitem{Tatikonda_Control}
S.~Tatikonda and S.~K. Mitter, ``Control under communication constraints,''
  \emph{{IEEE} Trans. Automat. Contr.}, vol.~49, no.~7, pp. 1056--1068, July
  2004.

\bibitem{Sahai_Anytime}
A.~Sahai and S.~Mitter, ``The necessity and sufficiency of anytime capacity for
  stabilization of a linear system over a noisy communication link - {Part I}:
  Scalar system,'' \emph{{IEEE} Trans. Inform. Theory}, vol.~52, no.~8, pp.
  3369--3395, Aug. 2006.

\bibitem{Yuksel_Optimal}
S.~Yuksel and T.~Basar, ``Optimal signaling policies for decentralized
  multicontroller stabilizability over communication channels,''
  \emph{Automatic Control, IEEE Transactions on}, vol.~52, no.~10, pp. 1969
  --1974, oct. 2007.

\bibitem{Elia_Bode}
N.~Elia, ``When {Bode} meets {Shannon}: Control-oriented feedback communication
  schemes,'' \emph{{IEEE} Trans. Automat. Contr.}

\bibitem{Nair_Communication}
G.~N. Nair and R.~J. Evans, ``Communication-limited stabilization of linear
  systems,'' in \emph{Proceedings of the 39th {IEEE} Conference on Decision and
  control}, Sydney, Australia, Dec. 2000, pp. 1005--1010.

\bibitem{Minero_Stabilization}
P.~Minero, M.~Franceschetti, S.~Dey, and G.~Nair, ``Data rate theorem for
  stabilization over time-varying feedback channels,'' \emph{Automatic Control,
  IEEE Transactions on}, vol.~54, no.~2, pp. 243 --255, feb. 2009.

\bibitem{Martins_Feedback}
N.~Martins and M.~Dahleh, ``Feedback control in the presence of noisy channels:
  Bode-like fundamental limitations of performance,'' \emph{Automatic Control,
  IEEE Transactions on}, vol.~53, no.~7, pp. 1604 --1615, aug. 2008.

\bibitem{Salman_Wireless}
S.~Avestimehr, S.~Diggavi, and D.~Tse, ``Wireless network information flow: A
  deterministic approach,'' \emph{{IEEE} Trans. Inform. Theory}, vol.~57,
  no.~4, pp. 1872 --1905, Apr. 2011.

\bibitem{Cover}
T.~M. Cover and J.~A. Thomas, \emph{Elements of Information Theory}.\hskip 1em
  plus 0.5em minus 0.4em\relax Wiley, 2006.

\bibitem{Ahlswede_Network}
R.~Ahlswede, N.~Cai, S.-Y. Li, and R.~Yeung, ``Network information flow,''
  \emph{IEEE Transactions on Information Theory}, vol.~46, no.~4, pp.
  1204--1216, July 2000.

\bibitem{Park_Approximation_Journal_Partii}
S.~Y. Park and A.~Sahai, ``An approximate solution to the decentralized
  two-controller infinite-horizon scalar {LQG} problem - part ii: Slow
  dynamics,'' In preparation.

\bibitem{Ranade_implicit}
G.~Ranade and A.~Sahai, ``Implicit communication in multiple-access settings,''
  in \emph{Information Theory Proceedings (ISIT), 2011 IEEE International
  Symposium on}, 31 2011-aug. 5 2011, pp. 998 --1002.

\bibitem{Park_Thesis}
S.~Y. Park, ``Information flows in linear systems,'' Ph.D. dissertation,
  University of California, Berkeley, Berkeley, CA, 2013.

\bibitem{Grover_distributed}
P.~Grover, A.~Wagner, and A.~Sahai, ``Information embedding meets distributed
  control,'' in \emph{Information Theory Workshop (ITW), 2010 IEEE}, jan. 2010,
  pp. 1 --5.

\bibitem{Choudhuri_witsenhausen}
C.~Choudhuri and U.~Mitra, ``On witsenhausen's counterexample: The asymptotic
  vector case,'' in \emph{Information Theory Workshop (ITW), 2012 IEEE}, sept.
  2012, pp. 162 --166.

\bibitem{Kim_amplification}
Y.-H. Kim, A.~Sutivong, and T.~Cover, ``State amplification,''
  \emph{Information Theory, IEEE Transactions on}, vol.~54, no.~5, pp. 1850
  --1859, may 2008.

\bibitem{Etkin_Interference}
R.~Etkin, D.~Tse, and H.~Wang, ``Gaussian interference channel capacity to
  within one bit,'' \emph{Information Theory, IEEE Transactions on}, vol.~54,
  no.~12, pp. 5534 --5562, dec. 2008.

\bibitem{Feller}
W.~Feller, \emph{An introduction to probability theory and its applications},
  6th~ed.\hskip 1em plus 0.5em minus 0.4em\relax New York, NY: Wiley, 1957.

\bibitem{Park_constant}
S.~Y. Park, P.~Grover, and A.~Sahai, ``A constant-factor approximately optimal
  solution to the witsenhausen counterexample,'' in \emph{Decision and Control,
  2009 held jointly with the 2009 28th Chinese Control Conference. CDC/CCC
  2009. Proceedings of the 48th IEEE Conference on}, dec. 2009, pp. 2881
  --2886.

\bibitem{Bertsekas}
D.~Bertsekas, \emph{Dynamic Programming and Optimal Control}, 3rd~ed.\hskip 1em
  plus 0.5em minus 0.4em\relax Athena Scientific, 2005.

\bibitem{Tse}
D.~Tse and P.~Viswanath, \emph{Fundamentals of Wireless Communication}.\hskip
  1em plus 0.5em minus 0.4em\relax Cambridge University Press, 2005.

\bibitem{Cover_feedback}
T.~Cover and S.~Pombra, ``Gaussian feedback capacity,'' \emph{Information
  Theory, IEEE Transactions on}, vol.~35, no.~1, pp. 37 --43, jan 1989.

\bibitem{yu2006dual}
W.~Yu and R.~Lui, ``Dual methods for nonconvex spectrum optimization of
  multicarrier systems,'' \emph{Communications, IEEE Transactions on}, vol.~54,
  no.~7, pp. 1310--1322, 2006.

\bibitem{Park_Equivalence}
S.~Y. Park and A.~Sahai, ``Network coding meets decentralized control:
  Capacity-stabilizability equivalence,'' In preparation.

\bibitem{Bobak_Computation}
B.~Nazer and M.~Gastpar, ``Computation over multiple-access channels,''
  \emph{{IEEE} Trans. Inform. Theory}, vol.~53, no.~10, pp. 3498--3516, Oct.
  2007.

\bibitem{maddah2008communication}
M.~A. Maddah-Ali, A.~S. Motahari, and A.~K. Khandani, ``Communication over mimo
  x channels: Interference alignment, decomposition, and performance
  analysis,'' \emph{Information Theory, IEEE Transactions on}, vol.~54, no.~8,
  pp. 3457--3470, 2008.

\bibitem{cadambe2008jafar}
V.~Cadambe and S.~Jafar, ``Interference alignment and degrees of freedom of the
  k-user interference channel,'' \emph{Information Theory, IEEE Transactions
  on}, vol.~54, no.~8, pp. 3425--3441, 2008.

\end{thebibliography}

\appendix

\subsection{Proof of Corollary~\ref{rat:lemma2} of Page~\pageref{rat:lemma2}}
\label{sec:lemma2}
Proof of (c):

Let's first consider when $\max(1,a^2 \sigma_{v1}^2)=a^2 \sigma_{v1}^2$. Since $a^2 \sigma_{v1}^2 \geq 1$, there exists $k_1 \geq 2$ such that 
\begin{align}
a^{2(k_1-2)} \leq a^2 \sigma_{v1}^2 < a^{2(k_1-1)},
\end{align}
and we choose such a $k_1$ as $k_1$ in Lemma~\ref{cov:lemma2}. Then, by \eqref{eqn:lem:radner1} of Lemma~\ref{cov:lemma2} we have
\begin{align}
&D_{L,3}(\widetilde{P_1},\widetilde{P_2};k_1) \\
&\geq \frac{a^{2(k_1-1)}\sigma_{v1}^2}{(1+\frac{\sigma_{v1}^2}{\sigma_{v2}^2})(\frac{a^{2(k_1-2)}}{1-a^{-2}})+\sigma_{v1}^2}\\
&\overset{(A)}{\geq} \frac{a^{2(k_1-1)}\sigma_{v1}^2}{\frac{2}{1-2.5^{-2}}a^{2(k_1-2)}+\sigma_{v1}^2}\\
&= \frac{a^2 \sigma_{v1}^2}{\frac{2}{1-2.5^{-2}}+\frac{\sigma_{v1}^2}{a^{2(k_1-2)}}}\\
&\overset{(B)}{\geq} \frac{a^2 \sigma_{v1}^2}{\frac{2}{1-2.5^{-2}}+1} \\
&\geq 0.295775... a^2 \sigma_{v1}^2 \\
&\geq 0.295 a^2 \sigma_{v1}^2. \label{eqn:ratio1}
\end{align}
(A): $\sigma_{v1}^2 \leq \sigma_{v2}^2$ and $|a| \geq 2.5$.\\
(B): $a^2 \sigma_{v1}^2 < a^{2(k_1-1)}$.

When $\max(1,a^2 \sigma_{v1}^2)=1$, by \eqref{eqn:lem:radner1} of Lemma~\ref{cov:lemma2} we have $D_{L,3}(\widetilde{P_1},\widetilde{P_2};k_1) \geq 1 \geq 0.295$.
\\

Proof of (b):

In Lemma~\ref{cov:lemma2}, choose $k_1$ in the same way as (c) and let $k=k_1 + 1$. Inspired by the proof of (c), we can safely choose $\Sigma = 0.295 \max(1,a^2 \sigma_{v1}^2)$. Then, by Lemma~\ref{cov:lemma2}, we notice that since $k-k_1-1=0$, the second and third square-root terms in $D_{L,2}(\widetilde{P_1},\widetilde{P_2};k_1,k,\Sigma)$ goes away and the bound reduces to
\begin{align}
D_{L,2}(\widetilde{P_1},\widetilde{P_2};k_1,k,\Sigma) &\geq \inf_{c_1, c_2} ( \sqrt{(a-c_1 -c_2)^2 \Sigma + c_1^2 \sigma_{v1}^2 + c_2^2 \sigma_{v2}^2})_+^2  + 1\\
&\mbox{s.t. }  (1-2.5^{-1}) c_1^2 (\Sigma+ \sigma_{v1}^2) \leq \widetilde{P_1} 
\end{align}
where $c_2$ can be chosen arbitrarily.

Here, since we assumed $\widetilde{P_1} \leq \frac{1}{400} a^2 \max(1,a^2 \sigma_{v1})$, we have
\begin{align}
&(1-2.5^{-1})c_1^2 (\Sigma + \sigma_{v1}^2) \leq \widetilde{P_1} \leq \frac{1}{400}a^2 \max(1, a^2\sigma_{v1}^2)\\
&(\Rightarrow) c_1^2(0.295 \max(1,a^2 \sigma_{v1}^2) + \sigma_{v1}^2) \leq  \frac{1}{400(1-2.5^{-1})} a^2 \max(1, a^2\sigma_{v1}^2) \\
&(\Rightarrow) c_1^2 \leq \frac{a^2 \max(1,a^2 \sigma_{v1}^2)}{400(1-2.5^{-1})(0.295 \max(1,a^2\sigma_{v1}^2) + \sigma_{v1}^2)} \leq \frac{a^2}{400 (1-2.5^{-1}) \cdot 0.295} \\
&(\Rightarrow) |c_1| \leq  0.118846...|a| \leq 0.119 |a|.
\end{align}
Therefore, 
\begin{align}
D_{L,2}(\widetilde{P_1},\widetilde{P_2};k_1,k,\Sigma)& \geq \inf_{c_1, c_2} ( \sqrt{(a-c_1 -c_2)^2 0.295 \max(1,a^2 \sigma_{v1}^2) + c_1^2 \sigma_{v1}^2 + c_2^2 \sigma_{v2}^2})_+^2  + 1\\
&\mbox{s.t. }  |c_1| \leq 0.119 |a| \\
&\geq \inf_{c_2} (a-0.119a -c_2)^2 0.295 \max(1,a^2 \sigma_{v1}^2) + c_2^2 \sigma_{v2}^2 + 1 \\
&\overset{(A)}{\geq} \inf_{c_2} (a-0.119a -c_2)^2 0.295  \sigma_{v2}^2 + c_2^2 \sigma_{v2}^2 + 1 \\
&= \inf_{\widetilde{c_2}} ( \sqrt{ (1-0.119-\widetilde{c_2})^2  0.295 + \widetilde{c_2}^2})^2 a^2 \sigma_{v2}^2 +1\\
&\overset{(B)}{=} 0.176808... a^2 \sigma_{v2}^2 + 1 \\
&\geq 0.176 a^2 \sigma_{v2}^2 +1.
\end{align}
(A): By the assumption $\max(1,a^2 \sigma_{v1}^2) \geq \sigma_{v2}^2$.\\
(B): By the numerical optimization of the quadratic function.
\\

Proof of (a):

(i) When $\max(1,a^2\sigma_{v2}^2)=a^2 \sigma_{v2}^2$

In Lemma~\ref{cov:lemma2}, we will choose $k_1$ in the same way as (c) and $k$ arbitrarily large. As above, we can safely choose $\Sigma = 0.295 a^2 \sigma_{v1}^2$. Applying the same arguments as (b), (c) to Lemma~\ref{cov:lemma2} gives 
\begin{align}
D_{L,2}(\widetilde{P_1},\widetilde{P_2};k_1,k,\Sigma) &\geq \inf_{c_1, c_2}(\sqrt{a^{2(k-k_1-1)}((a-c_1-c_2)^2 \Sigma + c_1^2 \sigma_{v1}^2 + c_2^2 \sigma_{v2}^2)}\\
&-\sqrt{\frac{a^{2(k-k_1-2)}(1-(2.5a^{-2})^{k-k_1-1})}{1-2.5a^{-2}} \frac{\widetilde{P_1}}{(1-2.5^{-1})2.5^{-1}}} \\
&-\sqrt{\frac{a^{2(k-k_1-2)}(1-(2.5a^{-2})^{k-k_1-1})}{1-2.5a^{-2}}  \frac{\widetilde{P_2}}{(1-2.5^{-1})2.5^{-1}}})_+^2 +1 \\
&\mbox{s.t. } (1-2.5^{-1})c_1^2(\Sigma+\sigma_{v1}^2) \leq \widetilde{P_1} \\
&\quad (1-2.5^{-1})c_2^2(\Sigma+\sigma_{v2}^2) \leq \widetilde{P_2} \\
&\geq
(\sqrt{a^{2(k-k_1-1)} 0.176 a^2 \sigma_{v2}^2 }\\
&-\sqrt{\frac{a^{2(k-k_1-2)}(1-(2.5a^{-2})^{k-k_1-1})}{1-2.5a^{-2}} \frac{\widetilde{P_1}}{(1-2.5^{-1})2.5^{-1}}} \\
&-\sqrt{\frac{a^{2(k-k_1-2)}(1-(2.5a^{-2})^{k-k_1-1})}{1-2.5a^{-2}}  \frac{\widetilde{P_2}}{(1-2.5^{-1})2.5^{-1}}})_+^2 +1.
\label{eqn:cor2:1}
\end{align}
Moreover, we have
\begin{align}
&\frac{a^{2(k-k_1-2)}(1-(2.5a^{-2})^{k-k_1-1})}{1-2.5a^{-2}} \frac{\widetilde{P_1}}{(1-2.5^{-1})2.5^{-1}}\\
&\overset{(A)}{\leq} \frac{a^{2(k-k_1-2)}}{1-2.5^{-1}} \frac{\widetilde{P_1}}{(1-2.5^{-1})2.5^{-1}} \\
&\overset{(B)}{\leq} \frac{a^{2(k-k_1-2)}}{(1-2.5^{-1})^2 2.5^{-1}} \frac{1}{400} a^2 \max(1,a^2 \sigma_{v1}^2) \\
&\leq 0.01736111... a^{2(k-k_1)} \sigma_{v2}^2 \\
&\leq 0.0174 a^{2(k-k_1)} \sigma_{v2}^2. \label{eqn:cor2:2}
\end{align}
(A): $|a| \geq 2.5$.\\
(B): We assumed $\widetilde{P_1} \leq \frac{1}{400}a^2 \max(1,a^2 \sigma_{v1}^2)=\frac{1}{400} a^4 \sigma_{v1}^2 \leq a^2 \sigma_{v2}^2$.

Likewise, we also have
\begin{align}
&\frac{a^{2(k-k_1-2)}(1-(2.5a^{-2})^{k-k_1-1})}{1-2.5a^{-2}} \frac{\widetilde{P_2}}{(1-2.5^{-1})2.5^{-1}}\\
&\leq 0.0174 a^{2(k-k_1)} \sigma_{v2}^2. \label{eqn:cor2:3}
\end{align}
Therefore, by plugging \eqref{eqn:cor2:2} and \eqref{eqn:cor2:3} into \eqref{eqn:cor2:1}, we get
\begin{align}
D_{L,2}(\widetilde{P_1},\widetilde{P_2};k_1,k,\Sigma) &\geq (\sqrt{0.176}-\sqrt{0.0174}-\sqrt{0.0174})_+^2 a^{2(k-k_1)} \sigma_{v2}^2 +1\\
&\geq 0.02 a^{2(k-k_1)} \sigma_{v2}^2 + 1.
\end{align}
Since $k$ can be chosen arbitrarily large and $|a| > 1$, $\lim_{k \rightarrow \infty} D_{L,2}(\widetilde{P_1},\widetilde{P_2};k_1,k,\Sigma)=\infty$.

(ii) When $\max(1,a^2\sigma_{v2}^2)=1$

We will choose $k$ arbitrarily large in \eqref{eqn:lem:radner2} of Lemma~\ref{cov:lemma2}. 
Here the assumptions $\widetilde{P_1} \leq \frac{1}{400} a^2 \max(1,a^2 \sigma_{v1}^2)$ and $\widetilde{P_2} \leq \frac{1}{400} a^2 \max(1,a^2 \sigma_{v2}^2)$ reduce to $\widetilde{P_1} \leq \frac{a^2}{400}$ and $\widetilde{P_2} \leq \frac{a^2}{400}$.

Therefore, by \eqref{eqn:lem:radner2} of Lemma~\ref{cov:lemma2} and $|a| \geq 2.5$, for all $k$ we have
\begin{align}
D_{L,4}(\widetilde{P_1},\widetilde{P_2};k) &\geq (\sqrt{a^{2(k-1)}} - \sqrt{ \frac{a^{2(k-2)}}{(1-2.5^{-1})^2} \frac{a^2}{400}} - \sqrt{\frac{a^{2(k-2)}}{(1-2.5^{-1})^2} \frac{a^2}{400}} )_+^2\\
& =(1 - \sqrt{\frac{1}{400(1-2.5^{-1})^2}} - \sqrt{\frac{1}{400(1-2.5^{-1})^2}})_+^2 a^{2(k-1)}\\
& \geq 0.6 a^{2(k-1)}
\end{align}
Since $k$ can be chosen arbitrarily large, $\lim_{k \rightarrow \infty} D_{L,4}(\widetilde{P_1},\widetilde{P_2};k)=\infty$.

\subsection{Proof of Proposition~\ref{prop:2}}
\label{sec:prop:2}

By Lemma~\ref{rat:lemma1}, if there exists $c \geq 1$ such that for all $\widetilde{P_1}, \widetilde{P_2} \geq 0$, $D_U(c \widetilde{P_1}, c\widetilde{P_2}) \leq c \cdot D_L(\widetilde{P_1},\widetilde{P_2})$, then for all $q, r_1, r_2 \geq 0$ we have
\begin{align}
\frac{\min_{P_1, P_2 \geq 0} q D_U(P_1,P_2)+r_1 P_1 + r_2 P_2}{\min_{\widetilde{P_1}, \widetilde{P_2} \geq 0} q D_L(\widetilde{P_1},\widetilde{P_2}) + r_1 \widetilde{P_1} + r_2 \widetilde{P_2}} \leq c
\end{align}
which finishes the proof. Therefore, we will only prove that such $c$ exists.

(i) When $\widetilde{P_1} \leq \frac{1}{400}a^2 \max(1,a^2 \sigma_{v1}^2)$ and $\widetilde{P_2} \leq \frac{1}{400}a^2\max(1,a^2 \sigma_{v2}^2)$\\
Lower bound: By Corollary~\ref{rat:lemma2} (a),
\begin{align}
D_L( \widetilde{P_1},\widetilde{P_2}) = \infty
\end{align}
Therefore, we do not need the corresponding upper bound.

(ii) When $\widetilde{P_1} \leq \frac{1}{400}a^2 \max(1,a^2 \sigma_{v1}^2)$ and $\widetilde{P_2} \geq \frac{1}{400}a^2 \max(1, a^2 \sigma_{v2}^2)$

Lower bound: By Corollary~\ref{rat:lemma2} (b),
\begin{align}
D_L(\widetilde{P_1},\widetilde{P_2}) \geq 0.176 a^2 \sigma_{v2}^2+1.
\end{align}
Upper bound: By Lemma~\ref{ach:lemmabb},
\begin{align}
(D_U(P_1,P_2),P_1,P_2) &\leq (a^2 \sigma_{v2}^2 +1,0, a^4\sigma_{v2}^2+ a^2 \sigma_{v2}^2 + a^2)\\
&\leq (a^2 \sigma_{v2}^2+1, 0, 3 a^2 \max(1,a^2 \sigma_{v2}^2)).
\end{align}
Ratio: Thus, $c$ is upper bounded by
\begin{align}
c \leq \max( \frac{1}{0.176}, \frac{3}{\frac{1}{400}})=1200.
\end{align}

(iii) When $\widetilde{P_1} \geq \frac{1}{400}a^2 \max(1,a^2 \sigma_{v1}^2)$\\
Lower bound: By Corollary~\ref{rat:lemma2} (c),
\begin{align}
D_L(\widetilde{P_1},\widetilde{P_2}) \geq 0.295 \max(1,a^2 \sigma_{v1}^2)
\end{align}
Upper bound: By Lemma~\ref{ach:lemmabb},
\begin{align}
(D_U(P_1,P_2),P_1,P_2) &\leq (a^2 \sigma_{v1}^2 +1, a^4\sigma_{v1}^2+ a^2 \sigma_{v1}^2 + a^2, 0) \\
&\leq (2 \max(1, a^2 \sigma_{v1}^2), 3a^2 \max(1,a^2 \sigma_{v1}^2) , 0)
\end{align}
Ratio: $c$ is upper bounded by
\begin{align}
c \leq \max( \frac{2}{0.295}, \frac{3}{\frac{1}{400}})=1200.
\end{align}
Therefore, by (i), (ii), (iii), the lemma is true and $c \leq 1200$.

\subsection{Proof of Corollary~\ref{rat:lemma3}}
\label{sec:lemma3}
For simplicity, we will only prove for the case when $\max(1,a^2 \sigma_{v1}^2)=a^2 \sigma_{v1}^2$. To prove the case when $\max(1,a^2 \sigma_{v1}^2)=1$, we can simply repeat the following proofs with parameters $k_1=1$ and $\Sigma=0.295$. We will also abbreviate $D_{L,1}(\widetilde{P_1},\widetilde{P_2};k_1,k_2,k,\sigma_{v2}',\alpha,\Sigma)$ to $D_{L,1}(\widetilde{P_1},\widetilde{P_2})$.
\\

Proof of (a):

Since $a^2 \sigma_{v1}^2 \geq 1$, there exists $k_1 \geq 2$ such that
\begin{align}
a^{2(k_1-2)} \leq a^2 \sigma_{v1}^2 < a^{2(k_1-1)},
\end{align}
and we will use such a $k_1$ in Lemma~\ref{cov:lemma3}.

Since the selection of $k_1$ is the same as in the proof of Corollary~\ref{rat:lemma2} (c), by \eqref{eqn:ratio1} we can select $\Sigma = 0.295 a^2 \sigma_{v1}^2$ in Lemma~\ref{cov:lemma3}. 

Let's further choose $k_2=k=k_1+s+1$, $\alpha=1$ and $\sigma_{v2}=\sigma_{v2}'$ in Lemma~\ref{cov:lemma3}. Then, since \eqref{eqn:witsenlower1}, \eqref{eqn:witsenlower2}, \eqref{eqn:witsenlower3} disappear, $D_{L,1}$ in Lemma~\ref{cov:lemma3} reduces to
\begin{align}
D_{L,1}(\widetilde{P_1},\widetilde{P_2}) \geq ( \sqrt{ \frac{a^{2(k_2-k_1)}\Sigma}{2^{2I'(\widetilde{P_1})}} } - \sqrt{ \frac{a^{2(k_2-k_1-1)}\widetilde{P_1}}{(1-2.5^{-1})^2} } )_+^2 +1. \label{eqn:lemma3:1}
\end{align}
We can bound $I''(\widetilde{P_1})$ as
\begin{align}
&I''(\widetilde{P_1}) \overset{(A)}{\leq}
\frac{1}{2} \log(1+ \frac{1}{(k_2-k_1-1)\sigma_{v2}^2}(
\frac{2a^{2(k_2-2-k_1)}}{1-2.5^{-2}} 0.295a^2 \sigma_{v1}^2
+
\frac{2a^{2(k_2-3-k_1)}}{1-2.5^{-2}} \frac{2.5\frac{1}{70}\frac{\sigma_{v2}^2}{a^{2(s-1)}}}{(1-2.5^{-1})^2}
))^{k_2-k_1-1}  \\
&=\frac{1}{2} \log(1+ \frac{1}{(k_2-k_1-1)\sigma_{v2}^2}(
\frac{2}{1-2.5^{-2}} 0.295a^{2s} \sigma_{v1}^2
+
\frac{2}{1-2.5^{-2}} \frac{2.5 \frac{1}{70}}{(1-2.5^{-1})^2} a^{-2} \sigma_{v2}^2 ))^{k_2-k_1-1}\\
&\overset{(B)}{\leq}
\frac{1}{2}\log(1+\frac{1}{k_2-k_1-1}(\frac{2}{1-2.5^{-2}}0.295+\frac{2}{1-2.5^{-2}}\frac{\frac{1}{70}\frac{1}{2.5}}{(1-2.5^{-1})^{2}}))^{k_2-k_1-1} \\
&=
\frac{1}{2}\log(1+\frac{0.7401738...}{k_2-k_1-1})^{k_2-k_1-1} \\
&\leq
\frac{1}{2}\log(1+\frac{0.7402}{k_2-k_1-1})^{k_2-k_1-1} \\
&\leq
\frac{1}{2}\log(e^{0.7402}).
\end{align}
(A): Assumptions $\widetilde{P_1} \leq \frac{\sigma_{v2}^2}{70a^{2(s-1)}}$ and $|a| \geq 2.5$.\\
(B): Assumptions $a^{2s}\sigma_{v1}^2 \leq \sigma_{v2}^2$ and $|a| \geq 2.5$.

Likewise, $I'(\widetilde{P_1})$ is upper bounded as
\begin{align}
&I'(\widetilde{P_1}) \leq \frac{1}{2}\log(e^{0.7402}) +
\frac{1}{2} \log(1+\frac{1}{\sigma_{v2}^2}(2a^{2(k_2-1-k_1)}\Sigma+2\frac{a^{2(k_2-2-k_1)}\widetilde{P_1}}{(1-2.5a^{-2})(1-2.5^{-1})})) \\
&\overset{(A)}{\leq}  \frac{1}{2}\log(e^{0.7402}) +  \frac{1}{2} \log(1+\frac{1}{\sigma_{v2}^2}(2a^{2s}0.295a^2 \sigma_{v1}^2+2 \frac{a^{2(s-1)}\frac{1}{70}\frac{\sigma_{v2}^2}{a^{2(s-1)}}}{(1-2.5^{-1})^2})) \\
&=  \frac{1}{2}\log(e^{0.7402}) + \frac{1}{2} \log( \frac{1}{\sigma_{v2}^2}( \sigma_{v2}^2 + 2 \times 0.295 a^{2(s+1)}\sigma_{v1}^2  + \frac{1}{35(1-2.5^{-1})^2}\sigma_{v2}^2) )\\
&\overset{(B)}{\leq}  \frac{1}{2}\log(e^{0.7402}) + \frac{1}{2} \log( \frac{1}{\sigma_{v2}^2}( a^{2(s+1)}\sigma_{v1}^2 + 2 \times 0.295 a^{2(s+1)}\sigma_{v1}^2  + \frac{1}{35(1-2.5^{-1})^2}a^{2(s+1)}\sigma_{v1}^2) ) \\
&= \frac{1}{2}\log(e^{0.7402}) +
\frac{1}{2} \log( \frac{a^{2(s+1)}\sigma_{v1}^2}{\sigma_{v2}^2}( 1 + 2 \times 0.295   + \frac{1}{35(1-2.5^{-1})^2}) ) \\
&= \frac{1}{2}\log(e^{0.7402}) +
\frac{1}{2} \log( \frac{a^{2(s+1)}\sigma_{v1}^2}{\sigma_{v2}^2} 1.669365...)\\
&\leq \frac{1}{2}\log(e^{0.7402}) +
\frac{1}{2} \log( \frac{a^{2(s+1)}\sigma_{v1}^2}{\sigma_{v2}^2} 1.6694). \label{eqn:lemma3:2}
\end{align}
(A): Assumptions $\widetilde{P_1} \leq \frac{\sigma_{v2}^2}{70a^{2(s-1)}}$ and $|a| \geq 2.5$.\\
(B): Assumption $\sigma_{v2}^2 \leq a^{2(s+1)} \sigma_{v1}^2$.

Moreover, since $\widetilde{P_1} \leq \frac{\sigma_{v2}^2}{70a^{2(s-1)}}$, we have
\begin{align}
\frac{a^{2(k_2-k_1-1)}\widetilde{P_1}}{(1-2.5^{-1})^2} \leq \frac{a^{2s}}{(1-2.5^{-1})^2} \frac{\sigma_{v2}^2}{70 a^{2(s-1)}} \leq \frac{a^2 \sigma_{v2}^2}{(1-2.5^{-1})^2 70}.\label{eqn:lemma3:3}
\end{align}
Therefore, by plugging \eqref{eqn:lemma3:2} and \eqref{eqn:lemma3:3} into \eqref{eqn:lemma3:1}, we have
\begin{align}
D_{L,1}(\widetilde{P_1},\widetilde{P_2}) &\geq ( \sqrt{\frac{a^{2(s+1)}0.295 a^2 \sigma_{v1}^2 }{e^{0.7402} \frac{a^{2(s+1)}\sigma_{v1}^2}{\sigma_{v2}^2} 1.6694}} - \sqrt{\frac{a^2 \sigma_{v2}^2}{(1-2.5^{-1})^270}} )_+^2 +1 \\
&\geq 0.008 a^2 \sigma_{v2}^2 +1.
\end{align}
\\

Proof of (b):

We choose $k_1, \Sigma, k_2, \alpha, \sigma_{v2}$ of Lemma~\ref{cov:lemma3} in the same way as the proof of (a) except $k$. Then, we will increase $k$ arbitrarily large. Then, Lemma~\ref{cov:lemma3} reduces to 
\begin{align}
D_{L,1}(\widetilde{P_1},\widetilde{P_2}) \geq (\sqrt{\frac{a^{2(k-k_1)}\Sigma}{2^{2I'(\widetilde{P_1})}}} - \sqrt{\frac{a^{2(k-k_1-1)}\widetilde{P_1}}{(1-2.5^{-1})^2}}
- \sqrt{\frac{a^{2(k-k_2-1)} 2.5^{k_2-k_1} \widetilde{P_1}}{(1-2.5^{-1})^2}}
- \sqrt{\frac{a^{2(k-k_2-1)} \widetilde{P_2}}{(1-2.5^{-1})^2}})_+^2 + 1. \label{eqn:lemma3:6}
\end{align}

Since the relevant parameters are the same, \eqref{eqn:lemma3:2} and \eqref{eqn:lemma3:3} in the proof of (a) still hold. Since $|a| \geq 2.5$ and $\widetilde{P_2} \leq \frac{a^4 \sigma_{v2}^2}{28000}$, we also have
\begin{align}
&\frac{a^{2(k-k_2-1)}2.5^{k_2-k_1}\widetilde{P_1}}{(1-2.5^{-1})^2}\\
&= (\frac{2.5}{a^2})^{k_2-k_1} \frac{a^{2(k-k_1-1)}\widetilde{P_1}}{(1-2.5^{-1})^2}\\
&\leq (\frac{1}{2.5})^2 \frac{a^{2(k-k_1-1)}\widetilde{P_1}}{(1-2.5^{-1})^2} \label{eqn:lemma3:4}
\end{align}
and
\begin{align}
\frac{a^{2(k-k_2-1)} \widetilde{P_2}}{(1-2.5^{-1})^2} \leq \frac{a^{2(k-k_2-1)}}{(1-2.5^{-1})^2} \frac{a^4 \sigma_{v2}^2}{28000}. \label{eqn:lemma3:5}
\end{align}
Therefore, by plugging \eqref{eqn:lemma3:2}, \eqref{eqn:lemma3:3}, \eqref{eqn:lemma3:4}, \eqref{eqn:lemma3:5} into \eqref{eqn:lemma3:6}, we have
\begin{align}
D_{L,1}(\widetilde{P_1},\widetilde{P_2}) & \geq ( \sqrt{\frac{a^{2(s+1)}0.295 a^2 \sigma_{v1}^2 }{e^{0.7402} \frac{a^{2(s+1)}\sigma_{v1}^2}{\sigma_{v2}^2} 1.6694}} - (1+ \frac{1}{2.5})\sqrt{\frac{a^2 \sigma_{v2}^2}{(1-2.5^{-1})^270}} - \sqrt{\frac{a^2 \sigma_{v2}^2}{(1-2.5^{-1})^2 28000}} )_+^2 a^{2(k-k_2)} + 1\\
& \geq 10^{-6} a^{2(k-k_2+1)} \sigma_{v2}^2 +1.
\end{align}
Since $k$ can be chosen arbitrarily large, $\lim_{k \rightarrow \infty}D_{L,1}(\widetilde{P_1},\widetilde{P_2}) = \infty$.
\\

Proof of (c):

We choose $k$, $k_1$, $k_2$, $\Sigma$ of Lemma~\ref{cov:lemma3} in the same way as the proof of (a), i.e. $\Sigma=0.295a^2 \sigma_{v1}^2$ and $k_2=k=k_1+s+1$. We put the remaining parameters $\alpha$ and $\sigma_{v2}'$ as $\alpha=\frac{1}{2}$, $\sigma_{v2}'^2=100a^{2(s-1)}\widetilde{P_1}$.
Then, Lemma~\ref{cov:lemma3} reduces to
\begin{align}
D_{L,1}(\widetilde{P_1},\widetilde{P_2}) \geq  \frac{1}{2}D'+\frac{1}{2}D''+1
\end{align}
where
\begin{align}
&D'=( \sqrt{c \frac{a^{2(k_2-k_1)}\Sigma}{2^{2I'(\widetilde{P_1})}}} - \sqrt{c \frac{a^{2(k_2-k_1-1)}\widetilde{P_1}}{(1-2.5^{-1})^2}} )_+^2  \label{eqn:lemma3:10} \\
&D''=( \sqrt{\frac{a^{2(k_2-k_1-1)} \Sigma}{2^{2I''(\widetilde{P_1})}}} - \sqrt{\frac{ a^{2(k_2-k_1-2)} 2.5  \widetilde{P_1} }{(1-2.5^{-1})^2}})_+^2. \label{eqn:lemma3:11}
\end{align}

Here, $I''(\widetilde{P_1})$ is upper bounded as:
\begin{align}
&I''(\widetilde{P_1}) =\frac{1}{2} \log(1+\frac{1}{(k_2-k_1-1)\sigma_{v2}^2}(\frac{2a^{2(k_2-2-k_1)}}{1-a^{-2}}\Sigma + \frac{2a^{2(k_2-3-k_1)} }{1-a^{-2}} \frac{2.5\widetilde{P_1}}{(1-2.5a^{-2})(1-2.5^{-1})}))^{k_2-k_1-1}\\
&\overset{(A)}{\leq} \frac{1}{2} \log(1+ \frac{1}{(k_2-k_1-1)\sigma_{v2}^2}(
\frac{2a^{2(k_2-2-k_1)}}{1-2.5^{-2}}0.295 a^2 \sigma_{v1}^2 + \frac{2a^{2(k_2-3-k_1)}}{1-2.5^{-2}}
\frac{2.5\frac{1}{20000}a^4 \sigma_{v1}^2}{(1-2.5^{-1})^2}
))^{k_2-k_1-1} \\
&= \frac{1}{2} \log(1+ \frac{1}{(k_2-k_1-1)\sigma_{v2}^2}(
\frac{2}{1-2.5^{-2}}0.295  + \frac{2}{1-2.5^{-2}}
\frac{2.5\frac{1}{20000}}{(1-2.5^{-1})^2}
)a^{2s}\sigma_{v1}^2)^{k_2-k_1-1} \\
&\overset{(B)}{\leq} \frac{1}{2} \log(1+ \frac{1}{k_2-k_1-1}(
\frac{2}{1-2.5^{-2}}0.295  + \frac{2}{1-2.5^{-2}}
\frac{2.5\frac{1}{20000}}{(1-2.5^{-1})^2}
))^{k_2-k_1-1} \\
&= \frac{1}{2} \log(1 + \frac{1}{k_2-k_1-1} 0.703207...)^{k_2-k_1-1} \\
&\leq  \frac{1}{2} \log(1 + \frac{1}{k_2-k_1-1}0.7033)^{k_2-k_1-1} \\
&\leq \frac{1}{2} \log(e^{0.7033}). \label{eqn:lemma3:12}
\end{align}
(A): $\widetilde{P_1} \leq \frac{1}{20000}a^4 \sigma_{v1}^2$ and $|a| \geq 2.5$.\\
(B): $a^{2s}\sigma_{v1}^2 \leq \sigma_{v2}^2$.

Likewise, $I'(\widetilde{P_1})$ is upper bounded as:
\begin{align}
&I'(\widetilde{P_1}) \leq  \frac{1}{2} \log(1+ \frac{1}{\sigma_{v2}'^2}(2a^{2(k_2-1-k_1)}\Sigma + 2 \frac{a^{2(k_2-2-k_1)}\widetilde{P_1}}{(1-2.5a^{-2})(1-2.5^{-1})})) + \frac{1}{2} \log(e^{0.7033})+ \frac{1}{2} \log(\frac{2 \pi e}{4})\\
&\leq \frac{1}{2} \log(1+ \frac{1}{\sigma_{v2}'^2}(2a^{2s}\Sigma+ 2 \frac{a^{2(s-1)}\widetilde{P_1}}{(1-2.5^{-1})^2}))+ \frac{1}{2} \log(e^{0.7033})+ \frac{1}{2} \log(\frac{2 \pi e}{4}) \\
&= \frac{1}{2} \log(1+ \frac{1}{\sigma_{v2}'^2}(2a^{2s}(0.295a^2\sigma_{v1}^2)+ 2 \frac{a^{2(s-1)}\widetilde{P_1}}{(1-2.5^{-1})^2})) + \frac{1}{2} \log(e^{0.7033})+ \frac{1}{2} \log(\frac{2 \pi e}{4})\\
&= \frac{1}{2} \log(\frac{1}{100a^{2(s-1)}\widetilde{P_1}}(100a^{2(s-1)}\widetilde{P_1}+2\cdot 0.295 a^{2(s+1)}\sigma_{v1}^2+ 2 \frac{a^{2(s-1)}\widetilde{P_1}}{(1-2.5^{-1})^2})) + \frac{1}{2} \log(e^{0.7033})+ \frac{1}{2} \log(\frac{2 \pi e}{4})\\
&\overset{(A)}{\leq} \frac{1}{2} \log(\frac{1}{100a^{2(s-1)}\widetilde{P_1}}(\frac{100a^{2(s+1)}\sigma_{v1}^2}{20000}+2\cdot 0.295 a^{2(s+1)}\sigma_{v1}^2+ 2 \frac{a^{2(s+1)}\sigma_{v1}^2}{20000(1-2.5^{-1})^2 })) + \frac{1}{2} \log(e^{0.7033})+ \frac{1}{2} \log(\frac{2 \pi e}{4})\\
&= \frac{1}{2}\log(\frac{1}{100}(0.595277...)\frac{a^{2(s+1)}\sigma_{v1}^2}{a^{2(s-1)}\widetilde{P_1}})+ \frac{1}{2} \log(e^{0.7033})+ \frac{1}{2} \log(\frac{2 \pi e}{4})\\
&\leq \frac{1}{2}\log(\frac{1}{100}(0.5953)\frac{a^{2(s+1)}\sigma_{v1}^2}{a^{2(s-1)}\widetilde{P_1}})+ \frac{1}{2} \log(e^{0.7033})+ \frac{1}{2} \log(\frac{2 \pi e}{4}). \label{eqn:lemma3:13}
\end{align}
(A): Assumption $\widetilde{P_1} \leq \frac{a^4 \sigma_{v1}^2}{20000}$.

Therefore, by plugging \eqref{eqn:lemma3:13} into \eqref{eqn:lemma3:10}, we get the following lower bound on $D'$:
\begin{align}
D' & \geq c(
\sqrt{\frac{0.295a^{2(s+2)}\sigma_{v1}^2}{e^{0.7033} \frac{2 \pi e}{4} \frac{1}{100} 0.5953 \frac{a^{2(s+1)}\sigma_{v1}^2}{a^{2(s-1)}\widetilde{P_1}} }}
- \sqrt{\frac{a^{2s}\widetilde{P_1}}{(1-2.5^{-1})^2}})_+^2\\
& = c(\sqrt{
\frac{0.295}{e^{0.7033}\frac{2 \pi e}{4} \frac{1}{100} 0.5953}
} -
\sqrt{\frac{1}{(1-2.5^{-1})^2}})_+^2 a^{2s} \widetilde{P_1}  \\
& = 0.532969... c a^{2s} \widetilde{P_1}\\
& \geq 0.5329 c a^{2s} \widetilde{P_1} \\
& = 0.5329 \frac{2 \cdot 10 a^{s-1}\sqrt{\widetilde{P_1}}}{\sqrt{2\pi}\sigma_{v2}} \exp(- \frac{100 a^{2(s-1)}\widetilde{P_1}}{2\sigma_{v2}^2}) a^{2s} \widetilde{P_1} \\
& \overset{(A)}{\geq}  0.5329 \frac{2 \cdot 10 }{\sqrt{2\pi}\sqrt{70}} \exp(- \frac{100 a^{2(s-1)}\widetilde{P_1}}{2\sigma_{v2}^2}) a^{2s} \widetilde{P_1}\\
& = 0.508202... a^{2s}\widetilde{P_1}\exp(- \frac{50 a^{2(s-1)}\widetilde{P_1}}{\sigma_{v2}^2}) \\
& \geq 0.5082a^{2s}\widetilde{P_1}\exp(- \frac{50 a^{2(s-1)}\widetilde{P_1}}{\sigma_{v2}^2}).
\end{align}
(A): Assumption $\widetilde{P_1} \geq \frac{\sigma_{v2}^2}{70a^{2(s-1)}}$.

Since $\widetilde{P_1} \leq \frac{a^4 \sigma_{v1}^2}{20000}$, we also have
\begin{align}
&\frac{a^{2(k_2-k_1-2)}2.5 \widetilde{P_1}}{(1-2.5^{-1})^2}  \\
&\leq \frac{a^{2(s-1)}2.5 \frac{a^4 \sigma_{v1}^2}{20000}}{(1-2.5^{-1})^2} \\
&= \frac{2.5a^{2(s+1)} \sigma_{v1}^2}{20000(1-2.5^{-1})^2} \label{eqn:lemma3:14}
\end{align}
Therefore, by plugging \eqref{eqn:lemma3:12} and \eqref{eqn:lemma3:14} into \eqref{eqn:lemma3:11}, $D''$ is lower bounded as:
\begin{align}
D'' &\geq ( \sqrt{ \frac{0.295a^{2(s+1)} \sigma_{v1}^2}{e^{0.7033}}} - \sqrt{ \frac{2.5 a^{2(s+1)}\sigma_{v1}^2}{20000(1-2.5^{-1})^2} })_+^2 \\
&= 0.132117... a^{2(s+1)} \sigma_{v1}^2\\
&\geq 0.1321 a^{2(s+1)} \sigma_{v1}^2
\end{align}
Finally, $D_{L,1}(\widetilde{P_1},\widetilde{P_2})$ is lower bounded as:
\begin{align}
D_{L,1}(\widetilde{P_1},\widetilde{P_2}) \geq 0.2541 a^{2s}\widetilde{P_1}\exp(- \frac{50 a^{2(s-1)}\widetilde{P_1}}{\sigma_{v2}^2}) + 0.066 a^{2(s+1)}\sigma_{v1}^2 + 1
\end{align}
\\

Proof of (d):

We choose $k_1, k_2, \Sigma, \alpha, \sigma_{v2}'$ of Lemma~\ref{cov:lemma3} in the same way as the proof of (c) except $k$. $k$ will be chosen arbitrarily large. Lemma~\ref{cov:lemma3} reduces to
\begin{align}
D_{L,1}(\widetilde{P_1},\widetilde{P_2}) \geq \frac{1}{2}a^{2(k-k_2)}D' + \frac{1}{2}a^{2(k-k_2)}D'' +1 \label{eqn:lemma3:44}
\end{align}
where
\begin{align}
D'&= (\sqrt{c \frac{a^{2(k_2-k_1)}\Sigma}{2^{2I'(\widetilde{P_1})}}} - \sqrt{ c \frac{a^{2(k_2-k_1-1)}\widetilde{P_1}}{(1-2.5^{-1})^2}}
- \sqrt{\frac{a^{-2}2.5^{k_2-k_1}\widetilde{P_1}}{(1-2.5^{-1})^2}}
- \sqrt{\frac{a^{-2}\widetilde{P_2}}{(1-2.5^{-1})^2}})_+^2 \\
D''&=( \sqrt{\frac{a^{2(k_2-k_1-1)} \Sigma}{2^{2I''(\widetilde{P_1})}}} - \sqrt{\frac{ a^{2(k_2-k_1-2)} 2.5  \widetilde{P_1} }{(1-2.5^{-1})^2}} - \sqrt{\frac{a^{-2}\widetilde{P_2}}{(1-2.5^{-1})^2}} )_+^2.
\end{align}
Denote $P':=\sqrt{\frac{a^{-2}2.5^{k_2-k_1}\widetilde{P_1}}{(1-2.5^{-1})^2}} +  \sqrt{\frac{a^{-2}\widetilde{P_2}}{(1-2.5^{-1})^2}}$. 
Then, following the same steps of the proof of (c), we can lower bound $D'$ and $D''$ as follows:
\begin{align}
D' &\geq (
\sqrt{0.5082a^{2s}\widetilde{P_1}\exp(- \frac{50 a^{2(s-1)}\widetilde{P_1}}{\sigma_{v2}^2}) }
 - \sqrt{P'}
)_+^2 \label{eqn:lemma3:40}\\
D'' &\geq (
\sqrt{0.1321 a^{2(s+1)} \sigma_{v1}^2}
-\sqrt{\frac{a^{-2}\widetilde{P_2}}{(1-2.5^{-1})^2}}
)_+^2 \\
&\geq ( \sqrt{0.1321 a^{2(s+1)} \sigma_{v1}^2} - \sqrt{P'})_+^2 \label{eqn:lemma3:41}
\end{align}

Here, we have
\begin{align}
&\frac{a^{-2}2.5^{k_2-k_1}\widetilde{P_1}}{(1-2.5^{-1})^2}\\
&\leq \frac{a^{-2}|a|^{s+1}\widetilde{P_1}}{(1-2.5^{-1})^2} \  (\because |a| \geq 2.5) \\
&\leq \frac{a^{-2}a^{2s}\widetilde{P_1}}{(1-2.5^{-1})^2} \  (\because s \geq 1) \\
&\leq \frac{a^{2s}\max(1,a^2\sigma_{v1}^2)}{20000(1-2.5^{-1})^2} \  (\because \widetilde{P_1} \leq \frac{\max(a^2,a^4 \sigma_{v1}^2)}{20000})
\end{align}
Thus,
\begin{align}
P'&\overset{(A)}{\leq} \sqrt{2(\frac{a^{-2}2.5^{k_2-k_1}\widetilde{P_1}}{(1-2.5^{-1})^2} + \frac{a^{-2}\widetilde{P_2}}{(1-2.5^{-1})^2})}\\
&\overset{(B)}{\leq} \sqrt{ \frac{a^{2s}\max(1,a^2\sigma_{v1}^2)}{10000(1-2.5^{-1})^2} + \frac{2}{(1-2.5^{-1})^2}(0.0457a^{2s}\widetilde{P_1} \exp(-\frac{50 a^{2(s-1)}\widetilde{P_1}}{\sigma_{v2}^2}) + 0.0113 a^{2s} \max(1,a^2\sigma_{v1}^2) ) }\\
&\leq \sqrt{ 0.253889a^{2s}\widetilde{P_1} \exp(-\frac{50 a^{2(s-1)}\widetilde{P_1}}{\sigma_{v2}^2})+ 0.0625 a^{2s} \max(1,a^2\sigma_{v1}^2)} \label{eqn:lemma3:42}
\end{align}
(A): Cauchy-Schwarz inequaility\\
(B): Assumptions on $\widetilde{P_1}$ and $\widetilde{P_2}$

By comparing \eqref{eqn:lemma3:42} with \eqref{eqn:lemma3:40} and \eqref{eqn:lemma3:41}, we can conclude that either $D'$ or $D''$ has to be greater than $0$. Moreover, since we can choose $k$ arbitrarily large in \eqref{eqn:lemma3:44}, $\lim_{k \rightarrow \infty} D_{L,1}(\widetilde{P_1},\widetilde{P_2})=\infty$.
\\

Proof of (e):

The same as (a) of Corollary~\ref{rat:lemma2}.

\subsection{Proof of Corollary~\ref{rat:ach}}
\label{sec:rat:ach}

For simplicity, we will only prove the case when $a \geq 2.5$. The proof for $a \leq -2.5$ easily follows by replacing $a$ with $|a|$.

We will evaluate Lemma~\ref{ach:lemma9} with the parameters $w_1 = \frac{a^s d}{6}$ and $d = \sqrt{\frac{320000P}{a^2}}$. Then, we can easily see that $(d,w_1) \in S_{U,1}$. Furthermore, $\frac{a^s d}{\sigma_{v2}}\geq 13$ since
\begin{align}
&\frac{a^s d}{ \sigma_{v2}} =\frac{a^{s-1}\sqrt{320000 P}}{\sigma_{v2}} \\
&\geq \frac{a^{s-1}}{\sigma_{v2}} \sqrt{4 \cdot 80000 \frac{\sigma_{v2}^2}{70a^{2(s-1)}}} \ (\because P \geq \frac{\sigma_{v2}^2}{70a^{2(s-1)}})\\
&= \sqrt{\frac{4 \cdot 80000}{70}} \geq 13.
\end{align}
Then, we will upper bound $D_{U,1}(d,w_1)$. First, let's bound the second term of $D_{U,1}(d,w_1)$ in \eqref{eqn:lemma9:2}. The second term is upper bounded as
\begin{align}
&\sum_{1 \leq i \leq \infty} 4 a^2 ( i a^s d + \frac{a^{s-1}d \frac{a}{a-1}+w_1}{2})^2 Q(\frac{(2i-1)a^s d - (a^{s-1}d \frac{a}{a-1}+w_1)}{2\sigma_{v2}})\\
&=4a^2 (a^s d + \frac{a^{s-1}d \frac{a}{a-1}+w_1}{2})^2 Q(\frac{a^s d - (a^{s-1}d \frac{a}{a-1}+w_1)}{2\sigma_{v2}})\\
 &+4a^2 (2 a^s d + \frac{a^{s-1}d \frac{a}{a-1}+w_1}{2})^2 Q(\frac{3 a^s d - (a^{s-1}d \frac{a}{a-1}+w_1)}{2\sigma_{v2}}) + \cdots \\
&=
4a^2 (a^s d + \frac{a^{s-1}d \frac{a}{a-1}+\frac{a^s d}{6}}{2})^2 Q(\frac{a^s d - (a^{s-1}d \frac{a}{a-1}+\frac{a^s d}{6})}{2\sigma_{v2}}) \\
&+ 4a^2 (2 a^s d + \frac{a^{s-1}d \frac{a}{a-1}+\frac{a^s d}{6}}{2})^2 Q(\frac{3 a^s d - (a^{s-1}d \frac{a}{a-1}+\frac{a^s d}{6})}{2\sigma_{v2}}) + \cdots \\
&\leq
4a^2 (a^s d + \frac{1}{2(2.5-1)}a^s d + \frac{1}{12}a^s d )^2 Q(\frac{1}{2\sigma_{v2}}(1-\frac{1}{2.5-1}-\frac{1}{6})a^s d) \\
&+
4a^2 (2 a^s d + \frac{1}{2(2.5-1)}a^s d + \frac{1}{12}a^s d )^2 Q(\frac{1}{2\sigma_{v2}}(3-\frac{1}{2.5-1}-\frac{1}{6})a^s d)+ \cdots \\
&=
4a^2 (a^s d)^2 (1+\frac{5}{12})^2 Q(\frac{1}{2\sigma_{v2}}(1-\frac{5}{6})a^s d) \\
&+4a^2 (a^s d)^2 (2+\frac{5}{12})^2 Q(\frac{1}{2\sigma_{v2}}(3-\frac{5}{6})a^s d) + \cdots  \label{eqn:ageq44}
\end{align}
Denote $k:=\frac{a^s d}{\sigma_{v2}}$. Since we already know $k \geq 13$, for all $n \geq 1$ we have 
\begin{align}
&\frac{(n+\frac{5}{12})^2 Q(\frac{1}{2 \sigma_{v2}}(2n-1-\frac{5}{6})a^s d) }{(n+1+\frac{5}{12})^2 Q(\frac{1}{2 \sigma_{v2}}(2n+1-\frac{5}{6})a^s d)} \\
&\geq
\frac{(n+\frac{5}{12})^2 ( \frac{1}{\frac{1}{2 \sigma_{v2}}(2n-1-\frac{5}{6})a^s d} - \frac{1}{(\frac{1}{2 \sigma_{v2}}(2n-1-\frac{5}{6})a^s d)^3} ) \exp(-\frac{(\frac{1}{2 \sigma_{v2}}(2n-1-\frac{5}{6})a^s d)^2}{2})}{(n+1+\frac{5}{12})^2 (\frac{1}{\frac{1}{2 \sigma_{v2}}(2n+1-\frac{5}{6})a^s d}) \exp(-\frac{(\frac{1}{2 \sigma_{v2}}(2n+1-\frac{5}{6})a^s d)^2}{2})} \ (\because Lemma~\ref{ach:lemma4})\\
&=
\frac{(n+\frac{5}{12})^2 ( \frac{1}{\frac{1}{2}(2n-1-\frac{5}{6}) k} - \frac{1}{(\frac{1}{2}(2n-1-\frac{5}{6})k)^3} ) \exp(-\frac{(\frac{1}{2}(2n-1-\frac{5}{6}) k)^2}{2})}{(n+1+\frac{5}{12})^2 (\frac{1}{\frac{1}{2}(2n+1-\frac{5}{6})k}) \exp(-\frac{(\frac{1}{2}(2n+1-\frac{5}{6})k)^2}{2})}\\
&\geq
(\frac{1+\frac{5}{12}}{2+\frac{5}{12}})^2
\frac{( \frac{1}{\frac{1}{2}(2n-1-\frac{5}{6}) k} - \frac{1}{(\frac{1}{2}(2n-1-\frac{5}{6})k)^3} ) }
{(\frac{1}{\frac{1}{2}(2n+1-\frac{5}{6})k})}
\frac{\exp(-\frac{( \frac{1}{12}k)^2}{2})}{\exp(-\frac{(\frac{13}{12}k)^2}{2})} \ (\because n \geq 1)\\
&\geq
(\frac{1+\frac{5}{12}}{2+\frac{5}{12}})^2
\exp(\frac{7}{12}k^2)
(\frac{\frac{1}{2}(2n+1-\frac{5}{6})}{\frac{1}{2}(2n-1-\frac{5}{6})}- \frac{\frac{1}{2}(2n+1-\frac{5}{6})}{(\frac{1}{2}(2n-1-\frac{5}{6}))^3 13^2}) \ (\because k \geq 13)\\
&\overset{(A)}{\geq} (\frac{1+\frac{5}{12}}{2+\frac{5}{12}})^2 \exp(\frac{7}{12}k^2) 0.99\\
&\geq 10^{42}\ (\because k \geq 13)
\end{align}
(A): When $n=1$, we can check the inequality by computation. When $n \geq 2$, we have
\begin{align}
&\frac{\frac{1}{2}(2n+1-\frac{5}{6})}{\frac{1}{2}(2n-1-\frac{5}{6})}- \frac{\frac{1}{2}(2n+1-\frac{5}{6})}{(\frac{1}{2}(2n-1-\frac{5}{6}))^3 13^2}\\
&\geq 1 - \frac{\frac{1}{2}(4+1-\frac{5}{6})}{(\frac{1}{2}(4-1-\frac{5}{6}))^3 13^2} \geq 0.99.
\end{align}

Thus, the terms in \eqref{eqn:ageq44} decrease faster than a geometric sequence with ratio $10^{-42}$ and thus can be upper bounded as
\begin{align}
\eqref{eqn:ageq44} \leq 4 a^2 (a^s d)^2 (1+\frac{5}{12})^2  Q(\frac{1}{2\sigma_{v2}}(1-\frac{5}{6})a^s d) \frac{1}{1-10^{-42}}. \label{eqn:ach:geo1}
\end{align}
The third term of $D_{U,1}(d,w_1)$ in \eqref{eqn:lemma9:2} can also be bounded similarly. We have 
\begin{align}
&\frac{(a^s d + \frac{a^s d}{2})^2 \frac{1}{2}}{(2 a^s d  + \frac{a^s d}{2})^2  Q(\frac{a^s d}{\sigma_{v2}})}
\geq \frac{(1+\frac{1}{2})^2}{(2+\frac{1}{2})^2 2 Q(13)}\ (\because \frac{a^s d}{\sigma_{v2}}\geq 13)\\
&\geq 10^{37}
\end{align}
and for $n \geq 2$
\begin{align}
&\frac{(n+\frac{1}{2})^2 Q(\frac{(n-1)ad}{\sigma_{v2}})}{((n+1)+\frac{1}{2})^2 Q(\frac{nad}{\sigma_{v2}})}\\
&\geq \frac{(n+\frac{1}{2})^2 ( \frac{1}{\frac{(n-1)ad}{\sigma_{v2}}} - \frac{1}{(\frac{(n-1)ad}{\sigma_{v2}})^3} )\exp(- \frac{(\frac{(n-1)ad}{\sigma_{v2}})^2}{2})}{((n+1)+\frac{1}{2})^2 (\frac{1}{\frac{n ad}{\sigma_{v2}}})  \exp(-\frac{(\frac{nad}{\sigma_{v2}})^2}{2})} \ (\because Lemma~\ref{ach:lemma4})\\
&= \exp(\frac{2n-1}{2}k^2) \frac{(n+\frac{1}{2})^2 (\frac{1}{(n-1)k} - \frac{1}{(n-1)^3 k^3}) }{((n+1)+\frac{1}{2})^2 (\frac{1}{nk})}\\
&\geq \exp(\frac{3}{2}13^2) \frac{(2+\frac{1}{2})^2}{(3+\frac{1}{2})^2}(\frac{n}{n-1}-\frac{n}{(n-1)^3 13^2}) 
\ (\because n \geq 2, k \geq 13)\\
&\overset{(A)}{\geq} \exp(\frac{3}{2}13^2)\frac{(2+\frac{1}{2})^2}{(3+\frac{1}{2})^2} 0.98\\
&\geq 10^{109}.
\end{align}
(A): Since $n \geq 2$, we have
\begin{align}
\frac{n}{n-1}-\frac{n}{(n-1)^3 13^2} \geq 1 - \frac{2}{(2-1)^3 13^2} \geq 0.98.
\end{align}

Therefore, the third term of $D_{U,1}(d,w_1)$ in \eqref{eqn:lemma9:2} is upper bounded by
\begin{align}
4 a^2 Q(\frac{\frac{a^s d}{6}}{2 \sqrt{a^{2(s-1)}\frac{a^2}{a^2-1}+a^{2s}\sigma_{v1}^2}})
 (a^s d + \frac{a^s d}{2})^2  \frac{1}{1-10^{-37}}. \label{eqn:ach:geo2}
\end{align}
By plugging \eqref{eqn:ach:geo1} and \eqref{eqn:ach:geo2} into \eqref{eqn:lemma9:2}, we can bound $D_{U,1}(d,w_1)$ of Lemma~\ref{ach:lemma9} as follows.
\begin{align}
D_{U,1}(d,w_1)
& \leq 2 a^{2s}(2 (\frac{d}{2})^2 (\frac{1}{1-\frac{1}{a}})^2 + 2 (\frac{1}{1-\frac{1}{a^2}}) + 2a^2 \sigma_{v1}^2 ) \\
&+ 4 a^2 (a^s d)^2 (1+\frac{5}{12})^2  Q(\frac{1}{2\sigma_{v2}}(1-\frac{5}{6})a^s d) \frac{1}{1-10^{-42}} \\
&+4 a^2 Q(\frac{\frac{a^s d}{6}}{2 \sqrt{a^{2(s-1)}\frac{a^2}{a^2-1}+a^{2s}\sigma_{v1}^2}})
 (a^s d + \frac{a^s d}{2})^2  \frac{1}{1-10^{-37}} \\
& + 2(a^2 (\frac{d}{2})^2)+1 \\
& \overset{(A)}{\leq} 2 a^{2s} (2 (\frac{d}{2})^2 (\frac{5}{3})^2 + \frac{50}{21}+2a^2 \sigma_{v1}^2)\\
&+4 a^2 (a^s d)^2 (\frac{17}{12})^2  Q(\frac{a^s d}{12 \sigma_{v2}}) \frac{1}{1-10^{-42}} \\
&+4 a^2 Q(\frac{\frac{a^s d}{6}}{2 \sqrt{a^{2(s-1)}\frac{25}{21}+a^{2s}\sigma_{v1}^2}})(a^sd+\frac{a^sd}{2})^2 \frac{1}{1-10^{-37}}\\
&+2(a^2 (\frac{d}{2})^2) + 1 \\
& \overset{(B)}{\leq} 2 a^{2s} (2 (\frac{d}{2})^2 (\frac{5}{3})^2 + \frac{50}{21}+2a^2 \sigma_{v1}^2)\\
&+4 a^2 (a^s d)^2 (\frac{17}{12})^2  Q(\frac{a^s d}{12 \sigma_{v2}}) \frac{1}{1-10^{-42}} \\
&+4 a^2 (a^sd+\frac{a^sd}{2})^2  Q(\frac{a^s d}{12 \sqrt{\frac{46}{21}} \sigma_{v2}}) \frac{1}{1-10^{-37}}\\
&+2(a^2 (\frac{d}{2})^2) + 1 \\
& \leq 2 a^{2s} (2 (\frac{d}{2})^2 (\frac{5}{3})^2 + \frac{50}{21}+2a^2 \sigma_{v1}^2)\\
&+4 a^2 (a^s d)^2 (\frac{17}{12})^2 Q(\frac{a^s d}{12 \sqrt{\frac{46}{21}} \sigma_{v2}}) \frac{1}{1-10^{-42}} \\
&+4 a^2 (a^sd+\frac{a^sd}{2})^2  Q(\frac{a^s d}{12 \sqrt{\frac{46}{21}} \sigma_{v2}}) \frac{1}{1-10^{-37}}\\
&+2(a^2 (\frac{d}{2})^2) + 1 \\
&= 1 + \frac{100}{21}a^{2s} + \frac{25}{9} a^{2s} d^2 + \frac{a^2 d^2}{2} + 4a^{2(s+1)}\sigma_{v1}^2 \\
&+
(4(\frac{17}{12})^2 \frac{1}{1-10^{-42}} + 9 \frac{1}{1-10^{-37}}) a^{2(s+1)}d^2
Q(\frac{a^s d}{12\sqrt{\frac{46}{21}}\sigma_{v2}}) \\
&\leq
1 + \frac{100}{21}a^{2s} + \frac{25}{9} a^{2s} d^2 + \frac{a^2 d^2}{2} + 4a^{2(s+1)}\sigma_{v1}^2 +
17.03 a^{2(s+1)}d^2
Q(\frac{a^s d}{12\sqrt{\frac{46}{21}}\sigma_{v2}}) 
\end{align}
\begin{align}
&= 1 + \frac{100}{21}a^{2s} + \frac{25}{9} 4 \cdot 80000 a^{2(s-1)}P + 2 \cdot 80000 P + 4a^{2(s+1)}\sigma_{v1}^2 +
17.03 \cdot 4 \cdot 80000 a^{2s} P \cdot
Q(\frac{a^{s-1} \sqrt{4 \cdot 80000 P}}{12\sqrt{\frac{46}{21}}\sigma_{v2}}) \\
&\overset{(C)}{\leq} 1 + \frac{100}{21}a^{2s} + \frac{25}{9} 16 a^{2(s-1)} \max(a^2,a^4 \sigma_{v1}^2) + 8 \max(a^2 ,a^4 \sigma_{v1}^2) + 4 a^{2(s+1)}\sigma_{v1}^2\\
&+17.03 \cdot 4 \cdot 80000 a^{2s}P \frac{1}{\sqrt{2 \pi} \frac{a^{s-1}\sqrt{4 \cdot 80000P}}{12 \sqrt{\frac{46}{21}} \sigma_{v2}}} \exp(- \frac{1}{2} \frac{a^{2(s-1)} 4 \cdot 80000P}{144 \cdot \frac{46}{21} \sigma_{v2}^2}) \\
&= 1 + \frac{100}{21}a^{2s} + \frac{25}{9} 16 a^{2(s-1)} \max(a^2,a^4 \sigma_{v1}^2) + 8 \max(a^2 ,a^4 \sigma_{v1}^2) + 4 a^{2(s+1)}\sigma_{v1}^2\\
&+17.03 \cdot 4 \cdot 80000  \frac{\sqrt{70}}{\sqrt{2 \pi} \frac{\sqrt{4 \cdot 80000}}{12 \sqrt{\frac{46}{21}} }} \exp(- (\frac{1}{2} \frac{ 4 \cdot 80000}{144 \cdot \frac{46}{21} }-50) \frac{a^{2(s-1)}P}{\sigma_{v2}^2}) a^{2s}P \exp(-\frac{50a^{2(s-1)}P}{\sigma_{v2}^2}) \\
&\overset{(D)}{\leq} 1 + \frac{100}{21}a^{2s} + \frac{25}{9} 16 a^{2(s-1)} \max(a^2,a^4 \sigma_{v1}^2) + 8 \max(a^2 ,a^4 \sigma_{v1}^2) + 4 a^{2(s+1)}\sigma_{v1}^2\\
&+17.03 \cdot 4 \cdot 80000  \frac{\sqrt{70}}{\sqrt{2 \pi} \frac{\sqrt{4 \cdot 80000}}{12 \sqrt{\frac{46}{21}} }} \exp(- (\frac{1}{2} \frac{ 4 \cdot 80000}{144 \cdot \frac{46}{21} }-50) \frac{1}{70})a^{2s}P \exp(-\frac{50a^{2(s-1)}P}{\sigma_{v2}^2}) \\
&\leq 1 + \frac{100}{21}a^{2s} + \frac{25}{9} 16 a^{2(s-1)} \max(a^2,a^4 \sigma_{v1}^2) + 8 \max(a^2 ,a^4 \sigma_{v1}^2) + 4 a^{2(s+1)}\sigma_{v1}^2\\
&+831.473...a^{2s}P \exp(-\frac{50a^{2(s-1)}P}{\sigma_{v2}^2}) \\
&\leq
1 + (\frac{100}{21} + \frac{25}{9} 16  + 8  + 4 ) a^{2s} \max(1,a^2 \sigma_{v1}^2)+831.473...a^{2s}P \exp(-\frac{50a^{2(s-1)}P}{\sigma_{v2}^2}) \\
&\leq
1 + 61.206... a^{2s} \max(1,a^2 \sigma_{v1}^2)+831.473...a^{2s}P \exp(-\frac{50a^{2(s-1)}P}{\sigma_{v2}^2}) \\
&\leq 1 + 62 a^{2s} \max(1,a^2 \sigma_{v1}^2)+832 a^{2s}P \exp(-\frac{50a^{2(s-1)}P}{\sigma_{v2}^2}) \\
&\leq 832 a^{2s}P \exp(-\frac{50a^{2(s-1)}P}{\sigma_{v2}^2}) + 63a^{2s} \max(1,a^2 \sigma_{v1}^2)
\end{align}
(A): $a \geq 2.5$.\\
(B): From the assumption $a^{2(s-1)}\max(1,a^2 \sigma_{v1}^2) \leq \sigma_{v2}^2 \leq a^{2s}\max(1,a^2 \sigma_{v1}^2)$, we have 
\begin{align}
a^{2(s-1)}\frac{25}{21}+ a^{2s} \sigma_{v1}^2 \leq \frac{46}{21}\max(a^{2(s-1)},a^{2s}\sigma_{v1}^2) \leq \frac{46}{21}\sigma_{v2}^2.
\end{align}
(C): $P \leq \frac{\max(a^2 ,a^4 \sigma_{v1}^2)}{20000}$ and Lemma~\ref{ach:lemma4}.\\
(D): $P \geq \frac{\sigma_{v2}^2}{70a^{2(s-1)}}$.

This justifies the upper bound on $D_U(P_1,P_2)$. By the definition of $d$ and Lemma~\ref{ach:lemma9}, $P_1$ is upper bounded by $80000P$. $P_2$ of Lemma~\ref{ach:lemma9} can be upper bounded as
\begin{align}
&P_2 \leq 8a^2 D_{U,1}(d,w_1) + \frac{7}{2}a^{2(s+1)}d^2+4a^2 \sigma_{v2}^2\\
&= 8a^2 D_{U,1}(d,w_1) + \frac{7}{2} 4 \cdot 80000a^{2s} P + 4 a^2 \sigma_{v2}^2 \\
&\leq 8a^2 (832 a^{2s}P \exp(-\frac{50a^{2(s-1)}P}{\sigma_{v2}^2})+63 a^{2s} \max(1,a^2 \sigma_{v1}^2)) \\
&+ \frac{7}{2} 4 \cdot 4 a^{2s}\max(a^2 ,a^4 \sigma_{v1}^2) + 4 a^2 a^{2s}\max(1,a^2 \sigma_{v1}^2) \\
&=  8a^2 (832 a^{2s}P \exp(-\frac{50a^{2(s-1)}P}{\sigma_{v2}^2}) + 70.5 a^{2s} \max(1,a^2 \sigma_{v1}^2)) \\
&= 6656 a^{2(s+1)}P \exp(-\frac{50a^{2(s-1)}P}{\sigma_{v2}^2}) + 564 a^{2(s+1)} \max(1,a^2 \sigma_{v1}^2).
\end{align}
where the inequality comes from the assumptions $P \leq \frac{\max(a^2 , a^4 \sigma_{v1}^2)}{20000} $ and $\sigma_{v2}^2 \leq a^{2s}\max(1,a^2 \sigma_{v1}^2)$. This finishes the proof.

\subsection{Proof of Proposition~\ref{prop:3}}
\label{sec:prop:3}
As the proof of Proposition~\ref{prop:2}, by Lemma~\ref{rat:lemma1} it is enough to show that there exists $c \geq 1$ such that $D_U(c \widetilde{P_1},c \widetilde{P_2}) \leq c \cdot D_L(\widetilde{P_1},\widetilde{P_2})$.

(i) When $\widetilde{P_1} \leq \frac{\sigma_{v2}^2}{70 a^{2(s-1)}}$ and $\widetilde{P_2} \leq \frac{a^4 \sigma_{v2}^2}{28000}$\\
Lower bound: By Corollary~\ref{rat:lemma3} (b)
\begin{align}
D_L(\widetilde{P_1},\widetilde{P_2})= \infty.
\end{align}
Therefore, we do not need the corresponding upper bound.

(ii) When $\widetilde{P_1} \leq \frac{\sigma_{v2}^2}{70 a^{2(s-1)}}$ and $\widetilde{P_2} \geq \frac{a^4 \sigma_{v2}^2}{28000}$\\
Lower bound: By Corollary~\ref{rat:lemma3} (a)
\begin{align}
D_L(\widetilde{P_1},\widetilde{P_2}) \geq 0.008 a^2 \sigma_{v2}^2 +1.
\end{align}
Upper bound: By Lemma~\ref{ach:lemmabb}
\begin{align}
(D_U(P_1,P_2),P_1,P_2) &\leq (a^2 \sigma_{v2}^2 +1,0, a^4\sigma_{v2}^2+ a^2 \sigma_{v2}^2 + a^2)\\
&\leq (a^2 \sigma_{v2}^2 +1,0, a^4\sigma_{v2}^2+ a^2 \sigma_{v2}^2 + a^2 \sigma_{v2}^2)\\
&\leq (a^2 \sigma_{v2}^2 +1,0, (1+\frac{2}{2.5^2})a^4\sigma_{v2}^2)\\
&\leq (a^2 \sigma_{v2}^2 +1,0, 1.32 a^4\sigma_{v2}^2).
\end{align}
Ratio: Thus, $c$ is upper bounded by
\begin{align}
c \leq \max(\frac{1}{0.008}, \frac{1.32}{\frac{1}{28000}}).
\end{align}

(iii) When $\frac{\sigma_{v2}^2}{70 a^{2(s-1)}} \leq \widetilde{P_1} \leq \frac{1}{20000}\max(a^2,a^4 \sigma_{v1}^2)$ and
$\widetilde{P_2} \leq 0.0457a^{2(s+1)}\widetilde{P_1} \exp(-\frac{50 a^{2(s-1)}\widetilde{P_1}}{\sigma_{v2}^2}) + 0.0113 a^{2(s+1)} \max(1,a^2\sigma_{v1}^2)$\\
Lower bound: By Corollary~\ref{rat:lemma3} (d)
\begin{align}
D_L(\widetilde{P_1},\widetilde{P_2}) = \infty
\end{align}
Therefore, we do not need the corresponding upper bound.

(iv) When $\frac{\sigma_{v2}^2}{70 a^{2(s-1)}} \leq \widetilde{P_1} \leq \frac{1}{20000}\max(a^2,a^4 \sigma_{v1}^2)$ and
$\widetilde{P_2} \geq 0.0457a^{2(s+1)} \widetilde{P_1} \exp(-\frac{50 a^{2(s-1)}\widetilde{P_1}}{\sigma_{v2}^2}) + 0.0113 a^{2(s+1)} \max(1,a^2\sigma_{v1}^2)$

Lower bound: By Corollary~\ref{rat:lemma3} (c)
\begin{align}
D_L(\widetilde{P_1},\widetilde{P_2}) \geq 0.2541 a^{2s} \widetilde{P_1}\exp(- \frac{50 a^{2(s-1)}\widetilde{P_1}}{\sigma_{v2}^2})+0.066 a^{2s}\max(1,a^2\sigma_{v1}^2) + 1
\end{align}
Upper bound: By Corollary~\ref{rat:ach}
\begin{align}
(D_U(P_1,P_2),P_1,P_2 ) \leq &( 63a^{2s} \max(1,a^2 \sigma_{v1}^2)+832 a^{2s}\widetilde{P_1} \exp(-\frac{50a^{2(s-1)}\widetilde{P_1}}{\sigma_{v2}^2}) ,80000\widetilde{P_1} \\
&, 6656 a^{2(s+1)}\widetilde{P_1} \exp(-\frac{50a^{2(s-1)}\widetilde{P_1}}{\sigma_{v2}^2}) + 564 a^{2(s+1)} \max(1,a^2 \sigma_{v1}^2))
\end{align}
Ratio: $c$ is upper bounded  by
\begin{align}
c \leq \max(\frac{832}{0.2541},\frac{63}{0.066}, 80000, \frac{6656}{0.0457}, \frac{564}{0.0113})
\end{align}

(v) When $\widetilde{P_1} \geq \frac{1}{20000}\max(a^2 ,a^4 \sigma_{v1}^2)$\\
Lower bound: By Corollary~\ref{rat:lemma3} (e)
\begin{align}
D_L(\widetilde{P_1},\widetilde{P_2}) \geq 0.295 \cdot \max(1,a^2 \sigma_{v1}^2)
\end{align}
Upper bound: By Lemma~\ref{ach:lemmabb}
\begin{align}
(D_U(P_1,P_2),P_1,P_2) &\leq (a^2 \sigma_{v1}^2 +1, a^4\sigma_{v1}^2+ a^2 \sigma_{v1}^2 + a^2, 0) \\
&\leq (a^2 \sigma_{v1}^2 +1, 2a^4\sigma_{v1}^2+ a^2, 0) \\
&\leq (2 \max(1, a^2 \sigma_{v1}^2), 3\max(a^2,a^4 \sigma_{v1}^2) , 0)
\end{align}
Ratio: $c$ is upper bounded by
\begin{align}
c \leq \max(\frac{2}{0.295}, \frac{3}{\frac{1}{20000}})
\end{align}

Therefore, by (i), (ii), (iii), (iv), (v), the lemma is true and $c \leq 1.5 \times 10^5$.

\subsection{Proof of Proposition~\ref{prop:1}}
\label{app:proposition}
Since $a$ goes to infinity, let $a \geq 10000$.  We will first show the best linear strategy performance is $\Theta(a^3)$.

$\bullet$ The best linear strategy performance is $\Theta(a^3)$: Following the same steps in the proof of Lemma~\ref{lem:geo}, we can lower bound the average cost as follows:
\begin{align}
&\inf_{u_1,u_2 \in L_{lin}'} \frac{1}{N} \sum_{0 \leq n < N}  \mathbb{E}[qx^2[n]+r_1 u_1^2[n]]
\\
&=\inf_{u_1, u_2 \in L_{lin}'} \frac{1}{N}
((
\frac{1}{2}r_1 \mathbb{E}[u_1^2[1]] + \frac{1}{2}r_1 \mathbb{E}[u_1^2[2]] + q\mathbb{E}[x^2[3]]
)+ (
\frac{1}{2}r_1 \mathbb{E}[u_1^2[2]] + \frac{1}{2}r_1 \mathbb{E}[u_1^2[3]] + q\mathbb{E}[x^2[4]]
) + \cdots \\
&+(
\frac{1}{2}r_1 \mathbb{E}[u_1^2[N-3]] + \frac{1}{2}r_1 \mathbb{E}[u_1^2[N-2]] + q\mathbb{E}[x^2[N-1]]
)
)\\
&\geq \frac{N-3}{N} \inf_{u_1,u_2 \in L_{lin}'} (\frac{1}{2}r_1 \mathbb{E}[u_1^2[1]] + \frac{1}{2}r_1 \mathbb{E}[u_1^2[2]]
+q \mathbb{E}[x^2[3]]
)\\
&= \frac{N-3}{N} \inf_{u_1,u_2 \in L_{lin}'} (\frac{1}{2} a \mathbb{E}[u_1^2[1]] + \frac{1}{2} a \mathbb{E}[u_1^2[2]]
+ \mathbb{E}[x^2[3]]
).
\end{align}
In the similar way of Proposition~\ref{prop:genie}, we can further justify that setting $w[1]=0$, $w[2]=0$ only decrease the quadratic cost. Then, at time $1$ we have
\begin{align}
&x[1]=w[0] \\
&y_1[1]=w[0] \\
&y_2[1]=w[0]+v_2[1]
\end{align}
Let
\begin{align}
&u_1[1]=k_{11} w[0]\\
&u_2[1]=k_{21}(w[0]+v_2[1])
\end{align}
At time $2$ we have
\begin{align}
&x[2]=ax[1]+u_1[1]+u_2[1]\\
&\quad\quad=aw[0]+k_{11}w[0]+k_{21}(w[0]+v_2[1])\\
&y_1[1]=aw[0]+k_{11}w[0]+k_{21}(w[0]+v_2[1])\\
&y_2[1]=aw[0]+k_{11}w[0]+k_{21}(w[0]+v_2[1])+v_2[2]
\end{align}
Therefore, we can put
\begin{align}
&u_1[2]= k_{12} w[0] + k_{13} v_2[1]\\
&u_2[3]= k_{22} (w[0]+v_2[1]) + k_{23}(aw[0]+k_{11}w[0]+v_2[2])
\end{align}
At time $3$ we have
\begin{align}
x[3]&=ax[2]+u_1[2]+u_2[2]\\
&=a^2 w[0] + a k_{11}w[0]+ak_{21}(w[0]+v_2[1])+k_{12} w[0] + k_{13}v_2[1] \\
&+ k_{22}(w[0]+v_2[1]) + k_{23}(aw[0]+k_{11}w[0]+v_2[2])\\
&=(a^2 + a k_{11} + k_{12}) w[0] + k_{13}v_2[1] + (ak_{21}+k_{22})(w[0]+v_2[1]) + k_{23}(aw[0]+k_{11}w[0]+v_2[2])\\
&=(a^2 + a k_{11} + k_{12} - k_{13}) w[0] + (ak_{21}+k_{22}+k_{13})(w[0]+v_2[1]) + k_{23}(aw[0]+k_{11}w[0]+v_2[2])
\end{align}
Therefore,
\begin{align}
\mathbb{E}[x^2[3]] \geq (a^2+ak_{11}+k_{12}-k_{13})^2 MMSE[w[0]|w[0]+v_2[1],aw[0]+k_{11}w[0]+v_2[2]]
\end{align}

(i) When $\mathbb{E}[u_1^2[1]]+\mathbb{E}[u_1^2[2]] \leq \frac{1}{16}a^2$

The condition implies
\begin{align}
&\mathbb{E}[(k_{11}w[0])^2] + \mathbb{E}[(k_{12}w[0]+k_{13}v_2[1])^2] \\
&=k_{11}^2 + k_{12}^2 + k_{13}^2 a  \leq \frac{1}{16}a^2
\end{align}
Thus,
\begin{align}
&|k_{11}| \leq \frac{1}{4}a\\
&|k_{12}| \leq \frac{1}{4}a\\
&|k_{13}| \leq \frac{1}{4} \sqrt{a}
\end{align}
Since $a \geq 10000$ we have
\begin{align}
a^2+ak_{11}+k_{12}-k_{13} &\geq a^2 - \frac{a^2}{4} - \frac{a^2}{4} - \frac{a^2}{4}=\frac{a^2}{4}
\end{align}
Moreover, we also have
\begin{align}
&MMSE[w[0]| w[0]+v_2[1], aw[0]+k_{11}w[0]+v_2[2]] \\
&\geq MMSE[w[0] | w[0]+v_2[1], \frac{5a}{4}w[0]+v_2[2]] \\
&\geq MMSE[w[0] | \frac{5a}{4}w[0]+v_2[1], \frac{5a}{4}w[0]+v_2[2]] \\
&= MMSE[w[0] | \frac{10a}{4}w[0]+v_2[1]+v_2[2]] \\
&= 1 - \frac{(\frac{10a}{4})^2}{(\frac{10a}{4})^2+2a}\\
&= \frac{2a}{(\frac{10a}{4})^2 + 2a} \\
&\geq \frac{2}{(\frac{10}{4})^2+2} \frac{1}{a} = \frac{8}{33a}
\end{align}
Therefore, in this case,
\begin{align}
\inf_{u_1,u_2 \in L_{lin}'} \frac{1}{2} a \mathbb{E}[u_1^2[1]] + \frac{1}{2} a \mathbb{E}[u_1^2[2]]
+ \mathbb{E}[x^2[3]] \geq \frac{1}{16} \cdot \frac{8}{33} a^3
\end{align}

(ii) When $\mathbb{E}[u_1^2[1]]+\mathbb{E}[u_1^2[2]] \geq \frac{1}{16}a^2$
In this case
\begin{align}
\inf_{u_1,u_2 \in L_{lin}'} \frac{1}{2} a \mathbb{E}[u_1^2[1]] + \frac{1}{2} a \mathbb{E}[u_1^2[2]]
+ \mathbb{E}[x^2[3]] \geq \frac{1}{32} a^3
\end{align}

Therefore, by (i),(ii),
\begin{align}
\inf_{u_1,u_2 \in L_{lin}'} \frac{1}{2} a \mathbb{E}[u_1^2[1]] + \frac{1}{2} a \mathbb{E}[u_1^2[2]]
+ \mathbb{E}[x^2[3]] \geq \frac{1}{66} a^3 \label{eqn:prop1}
\end{align}

$\bullet$ The optimal average cost is $O(a^2 \log a)$: Now, we will show the average cost, $O(a^2 \log a)$, is achievable by the nonlinear $1$-stage signaling strategy. Let $a \geq 20000$.
Since $\max(1,a^2\sigma_{v1}^2) \leq \sigma_{v2}^2 \leq a^2 \max(1,a^2 \sigma_{v1}^2)$ and $\frac{a}{70} \leq \frac{a \log a}{25} \leq \frac{1}{20000} a^2$, we can set $s=1$ and $P=\frac{a \log a}{25}$ in Corollary~\ref{rat:lemma3}.
%
Then, by Corollary~\ref{rat:lemma3}, the average cost is upper bounded as follows.
\begin{align}
&\inf_{u_1,u_2 \in L_{sig,1}} \frac{1}{N} \sum_{0 \leq n < N} \mathbb{E}[qx^2[n]+r_1 u_1^2[n]] \\
&\leq 832 a^2 \frac{a \log a}{25} \exp(- \frac{50 \frac{a \log a}{25}}{a})+ 63a^2 +  a \frac{80000a \log a}{25}\\
&\leq  832 \frac{a \log a}{25} +63a^2 + \frac{80000}{25}a^2 \log a\\
&\leq 3297 a^2 \log a \label{eqn:prop2}
\end{align}

In short, by \eqref{eqn:prop1} the optimal linear strategy cost is lower bounded by $\Omega(a^3)$ . By \eqref{eqn:prop2}, the nonlinear $1$-stage signaling strategy can achieve $O(a^2 \log a)$. Thus, their ratio diverges as $a$ goes to infinity, which finishes the proof.
\end{document}